\documentclass[english]{amsart}
\usepackage{amsxtra,amscd,amsthm,amsmath,amssymb}
\usepackage{longtable} \usepackage{bigstrut} \usepackage{enumerate}
\usepackage{verbatim}
\usepackage{graphicx}
\usepackage{xspace}
\usepackage{ifthen}
\usepackage{tabularx}
\usepackage[all,cmtip,line]{xy}

\usepackage{color}
\usepackage{mathrsfs}

\def\xcoch{\mathbb{X}_\bullet}
\def\xch{\mathbb{X}^\bullet}

\usepackage{amsthm}

\theoremstyle{plain}
\numberwithin{equation}{subsection}
\newtheorem{thm}[equation]{Theorem}
\newtheorem{ithm}[subsection]{Theorem}
\newtheorem{iprop}[subsection]{Proposition}
\newtheorem{icor}[subsection]{Corollary}
\newtheorem{prop}[equation]{Proposition}
\newtheorem{lemma}[equation]{Lemma}

\newtheorem{Remarknumb}[equation]{Remark}

\newtheorem{Remark}[equation]{Remark}

\newtheorem{cor}[equation]{Corollary}

\theoremstyle{remark}
\newtheorem{para}[equation]{\bf}
\theoremstyle{plain}
\renewcommand{\subsubsection}{\addtocounter{equation}{1}{\vskip 6pt \bf\arabic{section}.\arabic{subsection}.\arabic{equation}}}
\theoremstyle{definition}

\newcommand{\quash}[1]{}  %%Anything in \quash is ignored
\newcommand{\nc}{\newcommand}
\nc{\on}{\operatorname}
                                 
\newcommand{\ilim}{\underset\leftarrow{\rm lim}}
\newcommand{\dlim}{\underset\rightarrow{\rm lim}}
\newcommand{\lps}{[\![}
\newcommand{\rps}{]\!]}
\newcommand{\llps}{(\!(}
\newcommand{\lrps}{)\!)}
\newcommand{\rk}{{\rm rk}}
\newcommand{\ur}{{\rm ur}}
\renewcommand{\phi}{\varphi}
\newcommand{\loc}{{\rm loc}}
\newcommand{\Fil}{{\rm Fil}}
\newcommand{\SIsom}{{\underline{\rm Isom}}}

\newcommand{\End}{{\rm End}}
\newcommand{\Aut}{{\rm Aut}}
\newcommand{\Lie}{{\rm Lie\,}}
\newcommand{\der}{{\rm der}}
\newcommand{\ab}{{\rm ab}}
\newcommand{\an}{{\rm an}}
\newcommand{\sh}{{\rm sh}}
\newcommand{\SHom}{{\underline{\rm Hom}}}
\newcommand{\SAut}{{\underline{\rm Aut}}}
\newcommand{\Def}{{\rm Def}}
\newcommand{\Fr}{{\rm Fr}}
\newcommand{\fA}{{\mathscr A}}
\newcommand{\fB}{{\mathscr B}}

\newcommand{\eK}{{\sf K}}
\newcommand{\eE}{{\sf E}}
\newcommand{\eP}{{\sf P}}
\newcommand{\eI}{{\sf I}}

\newcommand{\eG}{{G}}

\newcommand{\rmM}{{\rm M}}

\newcommand{\gn}{{\mathfrak n}}

\newcommand{\mcQ}{{\mathcal Q}}

\newcommand{\RR}{{\mathbb R}}
\newcommand{\Sh}{{\rm Sh}}
\newcommand{\SSh}{{\mathscr S}}

\newcommand{\gS}{{\mathfrak S}}

\newcommand{\DD}{{\mathbb D}}
\newcommand{\wtDD}{{\widetilde{\mathbb D}}}

\newcommand{\fE}{{\mathscr E}}

\newcommand{\wtM}{{\widetilde M}}
\newcommand{\gm}{{\mathfrak m}}

\newcommand{\bbQ}{{\mathbb Q}}

\newcommand{\calA}{{\mathcal A}}
\newcommand{\calB}{{\mathcal B}}
\newcommand{\calC}{{\mathcal C}}

\newcommand{\calF}{{\mathcal F}}
\newcommand{\calG}{{\mathcal G}}
\newcommand{\ffG}{{\mathscr G}}
\newcommand{\calH}{{\mathcal H}}
\newcommand{\calI}{{\mathcal I}}

\newcommand{\calK}{{\mathcal K}}
\newcommand{\calL}{{\mathcal L}}

\newcommand{\gM}{{\mathfrak M}}

\newcommand{\calP}{{\mathcal P}}

\newcommand{\calT}{{\mathcal T}}
\newcommand{\calU}{{\mathcal U}}
\newcommand{\V}{{\mathcal V}}

\newcommand{\WW}{{\mathbb W}}

\newcommand{\whW}{{\widehat W}}
\newcommand{\bw}{{\mathbf w}}

\newcommand{\calZ}{{\mathcal Z}}

\newcommand{\ov}{\overline}
\newcommand{\TT}{{\mathcal T}}

\newcommand{\mF}{\ensuremath{\mathbb{F}}\xspace}

\nc{\al}{{\alpha}} \nc{\be}{{\beta}}

\nc{\ve}{{\varepsilon}} \nc{\Ga}{{\Gamma}}

\nc{\La}{{\Lambda}}

\def\xcoch{\mathbb{X}_\bullet}
\def\xch{\mathbb{X}^\bullet}
\def\0{\circ}
\def\x{\times}

\newcommand{\cal}{\mathcal}

\newcommand{\E}{{\mathcal E}}
\newcommand{\A}{{\mathcal A}}

\renewcommand{\AA}{{\mathbb A}}
\newcommand{\fa}{\mathfrak a}

\newcommand{\Mod}{\rm {Mod}}

\newcommand{\GG}{{\mathbb G}}

\newcommand{\C}{{\mathbb C}}
\newcommand{\R}{{\mathbb R}}
\newcommand{\Q}{{\mathbb Q}}

\newcommand{\iso}{\cong}

\newcommand{\et}{{\text{\rm \'et}}}
\newcommand{\B}{{\mathcal B}}

\newcommand{\Hom}{{\rm Hom}}
\newcommand{\Proj}{{\rm Proj}}

\newcommand{\Gg}{{\mathcal G}}

\newcommand{\Hh}{{\mathcal H}}
\newcommand{\Gm}{{{\mathbb G}_{\rm m}}}
\newcommand{\Z}{{\mathbb Z}}
\newcommand{\F}{{\mathcal F}}
\newcommand{\ti}{\tilde}
\newcommand{\Spec}{{\rm Spec \, } }
\newcommand{\Spf}{{\rm Spf } }
\newcommand{\ad}{{\rm ad } }
 \renewcommand{\O}{{\mathcal O}}

\newcommand{\UU}{{\mathcal U}}

\newcommand{\GL}{{\rm GL}}

\newcommand{\und}{\underline}

\renewcommand{\L}{{\mathcal L}}

\newcommand{\I}{{\mathcal J}}

\newcommand{\Res}{{\rm Res}}

\newcommand{\Rep}{{\rm Rep}}
\newcommand{\cris}{{\rm cris}}
\newcommand{\dR}{{\rm dR}}

\newcommand{\PP}{{\mathcal P}}

\newcommand{\rM}{{\rm M}}

\newcommand{\rH}{{\mathrm H}}

\def\thfill{\null\nobreak\hfill}

\def\endproof{\thfill\vbox{\hrule
  \hbox{\vrule\hbox to 5pt{\vbox to 5pt{\vfil}\hfil}\vrule}\hrule}}

\newcommand{\Om}{\Omega}
\renewcommand{\P}{{\cal P}}

\newcommand{\Gal}{{\rm Gal}}

\newcommand{\GSp}{{\rm GSp}}

\begin{document}

\title[ ]{Integral models of Shimura varieties with parahoric level structure}

\author[M. Kisin]{M. Kisin}
\address{Dept. of
Mathematics\\ Harvard University,
Cambridge, MA 02138\\ USA} \email{kisin@math.harvard.edu}
\thanks{M. Kisin is partially supported by NSF grant DMS-1301921}

\author[G. Pappas]{G. Pappas}
\thanks{G. Pappas is partially supported by  NSF grants DMS-1360733 and DMS-1701619.}

\address{Dept. of
Mathematics\\
Michigan State
University\\
E. Lansing\\
MI 48824\\
USA} \email{pappas@math.msu.edu}

\begin{abstract}
For a prime $p > 2,$ we construct integral models over $p$ for Shimura varieties with parahoric level structure, attached to Shimura data $(G,X)$ of abelian type, 
such that $G$ splits over a tamely ramified extension of $\Q_p.$ The local structure of these integral models is related to certain ``local models'', 
which are defined group theoretically. Under some additional assumptions, we show that these integral models satisfy a conjecture of Kottwitz which gives an explicit description 
for the trace of Frobenius action on their sheaf of nearby cycles.
\end{abstract}
\date{\today}

\maketitle

\vspace{-1ex}

\tableofcontents
\vspace{-2ex}


\bigskip

 \medskip
 \section*{Introduction} 

The aim of this paper is to construct integral models for a large class of Shimura varieties with parahoric level structure, 
namely for those which are of abelian type, and such that the underlying group $G$ splits over a tamely ramified extension. 
Recall that $(G,X)$ is said to be of {\em Hodge type} if the corresponding Shimura variety can be described as a moduli space of abelian varieties equipped with a certain family of Hodge cycles. 
The Shimura data of {\em abelian type} is a larger class, which can be related to those of Hodge type. They include almost all cases where $G$ is a classical group.
Our condition on the level structure allows many cases of Shimura varieties 
with non-smooth reduction at $p$.  

One application of such models is to Langlands' program \cite{Langlands} to compute the Hasse-Weil zeta 
function of a Shimura variety in terms of automorphic $L$-functions. The zeta function has a local factor at $p$ which is determined by 
the mod $p$ points of the integral model, as well as its local structure - specifically the sheaf of nearby cycles. The integral models we construct 
are related to moduli spaces of abelian varieties (at least indirectly), which makes it feasible to count their mod $p$ points as in the work of 
Kottwitz \cite{KottJAMS} (cf.~ also \cite{KisinLR}). On the other hand, their local structure is described in terms of ``local models'' 
which are simpler schemes given as orbit closures. In particular, we show that,
when $G$ is unramified, the inertia acts unipotently on the sheaf of nearby cycles, and our models verify a conjecture of Kottwitz, which determines the (semi-simple) trace of Frobenius action on their  nearby cycles rather explicitly. Such a local structure theory of integral models for Shimura varieties with parahoric level structure was first sought by Rapoport 
who took some of the first steps in extending the Langlands/Kottwitz method to the parahoric case \cite{RapoportAnnArbor, RapoportGuide}. 

To state our results more precisely, let $p$ be a prime, and $(G,X)$ a Shimura datum. For $\eK^\circ \subset G(\AA_f)$ a neat, compact open subgroup, 
the corresponding Shimura variety 
$$ \Sh_{\eK^\circ }(G,X) = G(\Q) \backslash X \times G(\AA_f)/\eK^\circ  $$ 
is naturally a scheme over the reflex field $\eE = E(G,X),$ which does not depend on the choice of $\eK^\circ.$ 
Let $\eK_p^\circ \subset G(\Q_p)$ be a parahoric subgroup, fix a compact open subgroup $\eK^p \subset G(\AA_f^p),$ and let 
$\eK^\circ = \eK_p^\circ\eK^p.$ We set 
$$ \Sh_{\eK^\circ_p}(G,X) = \ilim_{\eK^p} \Sh_{\eK^\circ}(G,X). $$ 

Fix $v|p$ a prime of $\eE,$ let $E = \eE_v$ and denote by $\kappa(v)$ the residue field 
of $E$. We say that a flat $\O_E$-scheme $S$ satisfies the {\it extension property}, if for any 
discrete valuation ring $R \supset \O_E$ of mixed characteristic $0,p,$ the map $S(R) \rightarrow S(R[1/p])$ is a bijection.

For the rest of the introduction, we assume that $p>2$, that $(G,X)$ is of abelian type, and that $G$ splits over a tamely ramified extension of $\Q_p$. We also assume that $p$ does not divide the order $|\pi_1(G^{\der})|$  of the (algebraic) fundamental group of the derived group $G^{\rm der}$
over   $\bar\Q_p$.

\begin{iprop}\label{introthmI} The $E$-scheme $\Sh_{\eK^\circ_p}(G,X)$ admits a $G(\AA_f^p)$-equivariant extension to a flat $\O_E$-scheme 
$\SSh_{\eK^\circ_p}(G,X),$ satisfying the extension property. Any sufficiently small compact open $\eK^p \subset G(\AA_f^p)$ 
acts freely on $\SSh_{\eK^\circ_p}(G,X),$ and the quotient 
$$\SSh_{\eK^\circ}(G,X) : = \SSh_{\eK^\circ_p}(G,X)/\eK^p$$ 
is a finite $\O_E$-scheme extending $\Sh_{\eK^\circ_p}(G,X)_E.$ 
\end{iprop}

In fact one can probably prove a result such as the proposition under weaker assumptions on $G.$ The main point of our results is to describe the local structure of 
these models in terms of orbit closures when $G$ splits over a tamely ramified extension of $\Q_p.$ To explain this, recall that the parahoric 
subgroup $\eK_p^\circ$ is associated to a point of the building $x \in \calB(G,\Q_p),$ which in turn defines, via the theory of Bruhat-Tits, 
a connected smooth group scheme $\calG^\circ$ over $\Z_p$, whose generic fibre is $G,$ 
and such that $\calG^\circ(\Z_p) \subset G(\Q_p)$ is identified with 
$\eK_p^\circ.$ It had been conjectured by Rapoport \cite[\S 6,7]{RapoportGuide} (see also \cite{PaJAG}),  
that $\Sh_{\eK^\circ}(G,X)$ admits an integral model $\SSh_{\eK^\circ}(G,X),$ whose singularities
are controlled by a ``local model'', ${\rm M}^{\loc}_{G,X},$ with an explicit group theoretic description. 
Although a general definition of ${\rm M}^{\loc}_{G,X}$ was not given in \cite{RapoportGuide}, it  was conjectured that ${\rm M}^{\loc}_{G,X},$ 
should be equipped with an action of $\calG^\circ,$ and that there should be a smooth morphism of stacks 
$$
 \lambda:  \SSh_{\eK^\circ}(G,X) \rightarrow [\calG^{c, \circ}\backslash {\rm M}^{\loc}_{G,X}], 
 $$
which is to say a ``local model diagram'' consisting of maps of $\O_E$-schemes 
\begin{equation*}
\xymatrix{
& \widetilde \SSh_{\eK_p^\circ} \ar[ld]_{\pi}\ar[rd]^{q} & \\
\quad\quad   \SSh_{\eK_p^\circ}(\eG, X) \quad\quad & &{\ \ \ {\rm M}^{\rm loc}_{G,X}\ }\, , \ \ \ \ 
  \ \ \ \ \ 
}
\end{equation*}
where $\pi$ is a $\calG^{c, \circ}$-torsor, and $q$ is smooth and $\calG^{c, \circ}$-equivariant.
Here, $\calG^{c, \circ}$ is the (smooth) quotient\footnote{This quotient by $\calZ_s$ is omitted in \cite{RapoportGuide} and other references. If $(G, X)$ is of Hodge type, then $\calZ_s=\{1\}$  and so $\calG^{c,\circ}=\calG^\circ$.} $\calG^\circ/\calZ_s$, where $\calZ_s$ is the Zariski closure  in $\calG^\circ$ of the maximal $\R$-split but $\Q$-anisotropic subtorus of the center of $G$.
%\mar{GP: This is one substantial correction: we need this quotient, otherwise we can't hope for the natural torsor not even over the generic fibre!} 
This conjecture was inspired by a similar result for Shimura varieties of PEL type that first appeared in Deligne and one of us (G.P) \cite{DelignePappas}, and de Jong \cite{deJongGamma0} for special cases, and in the book of Rapoport-Zink \cite{RapZinkBook} more generally. In particular, the work of Rapoport and Zink implies such a result for many Shimura varieties of PEL 
type with parahoric level structure but with an  ad-hoc definition of  ${\rm M}^{\rm loc}_{G,X}$   given  case-by-case. See the survey article \cite{PRS} for more information and 
for additional references.

When $G$ splits over a tamely ramified extension a candidate for ${\rm M}^{\loc}_{G,X}$ was constructed in \cite{PaZhu} by one of us (G.P) and Zhu. 
The construction of {\em loc.~cit.} is reviewed in \S 2, and uses the affine Grassmannian for $G.$ In \S 2.3 we show that it also has a more direct description 
as an orbit closure in a standard ({\em i.e.}~not affine) Grassmannian. We show that these local models ${\rm M}^{\loc}_{G,X}$ can be used to control the integral models 
$\SSh_{\eK^\circ}(G,X)$ in Proposition \ref{introthmI} \'etale locally: 

\begin{ithm}\label{introthmII} 
If $\kappa/\kappa(v)$ is a finite extension, and $z \in \SSh_{\eK^\circ}(G,X)(\kappa),$ then there exists $w \in {\rm M}^{\loc}_{G,X}(\kappa')$, with $\kappa'/\kappa$ a finite extension, such that 
there is an isomorphism of strict henselizations $$\O_{\SSh_{\eK^\circ}(G,X),z}^{\sh} \iso \O_{{\rm M}^{\loc}_{G,X},w}^{\sh}.
$$
\end{ithm}

The theorem, combined with results in \cite{PaZhu}, implies the following result about the local structure of $\SSh_{\eK^\circ}(G,X).$

\begin{icor}\label{introcorI} The special fibre $\SSh_{\eK^\circ}(G,X)\otimes \kappa(v)$ is reduced, and the 
strict henselizations of the local rings on  $\SSh_{\eK^\circ}(G,X)\otimes \kappa(v)$ have irreducible components 
which are normal and Cohen-Macaulay.

If $\eK^\circ_p$ is  associated to a point $x$ which is a special vertex in $\B(G, \Q_p^\ur)$,   then $\SSh_{\eK^\circ}(G,X)\otimes \kappa(v)$ is normal 
and Cohen-Macaulay.
\end{icor}

We often obtain a more precise result, involving a slightly weaker form of the local model diagram:  

\begin{ithm}\label{ithmIII} Suppose   that either $(G^{\ad}, X^{\ad})$ has no factor of type $D^{\mathbb H},$ or that $G$ is unramified over $\Q_p$ and $\eK_p^\circ$ is contained in a hyperspecial subgroup. Then there exists a local 
model diagram 
\begin{equation*}
\xymatrix{
& \widetilde \SSh_{\eK_p^\circ}^\ad \ar[ld]_{\pi}\ar[rd]^{q} & \\
\quad\quad   \SSh_{\eK_p^\circ}(\eG, X) \quad\quad & &{\ \ \ {\rm M}^{\rm loc}_{G,X}\ }\, , \ \ \ \ 
  \ \ \ \ \ 
}
\end{equation*}
where $\pi$ is a $\calG^{\ad\circ}$-torsor and $q$ is smooth and $\calG^{\ad\circ}$-equivariant. 
In particular, for any $z \in \SSh_{\eK^\circ}(G,X)(\kappa),$ there exists $w \in {\rm M}^{\loc}_{G,X}(\kappa)$ 
such that there is an isomorphism of henselizations $\O_{\SSh_{\eK^\circ}(G,X),z}^{\rm h} \iso \O_{{\rm M}^{\loc}_{G,X},w}^{\rm h}.$
\end{ithm}

Here, $\calG^{\ad\circ}$ denotes the connected smooth group scheme with generic fibre $G^{\ad}$ associated by Bruhat-Tits theory to the image  $x^{\rm ad}$ of the point $x$ under the canonical map $\B(G,\Q_p)\to \B(G^{\ad},\Q_p)$. Under our assumptions, $\calG^{\ad\circ}$ also acts on ${\rm M}^{\loc}_{G,X}$. In fact, the condition $p\nmid|\pi_1(G^{\der})|$ in the above theorem can be removed, although ${\rm M}^{\rm loc}_{G,X}$ then has to be replaced with a slightly different local model, 
attached to an auxiliary Shimura datum of Hodge type.
 
Below, we write  $\SSh = \SSh_{\eK^\circ}(G, X).$ Let $\bar E$ be an algebraic closure of $E,$ with residue field $k_{\bar E},$ and $F \subset \bar E$ a subfield with $F/E$ finite and such that $G_F$ is split. 
%\mar{GP: correction to Cor. 5. MK: Tweaked this, and the corollary below.}
The relationship with ${\rm M}^{\rm loc}_{G,X}$ and one of the results of \cite{PaZhu}, allows us to show the following result on the action of inertia on the sheaf 
of nearby cycles $ R\Psi^\SSh.$ 

\begin{icor}  For $\bar z \in \SSh(k_{\bar E}),$ the inertia subgroup $I_F$ of ${\rm Gal}(\bar E/F)$  acts unipotently on all the stalks 
$R\Psi^\SSh_{\bar z}.$ If $\eK_p^\circ$ is associated to a very special vertex\footnote{By definition
\cite{PaZhu}, this means that $x$ is a special vertex in $\calB(G, \Q_p)$ and is also special in $\calB(G, \Q_p^{\rm ur})$.
Such $x$ exist only when $G$ is quasi-split over $\Q_p$.}  
 $x\in \calB(G, \Q_p),$ then $I_F$
acts trivially on all the stalks $R\Psi^\SSh_{\bar z}$, $\bar z$ as above.
\end{icor}

In fact, we also give results about the semi-simple trace of Frobenius on the sheaf of nearby cycles of $\SSh_{\eK^\circ}(G,X)$. Under the assumptions
of Theorem \ref{ithmIII} we show, again using results of \cite{PaZhu}, that
this trace is given by a function which lies in the center of the parahoric Hecke algebra.
When $G$ is unramified, we can deduce that  $\SSh_{\eK^\circ}(G,X)$ verifies a more precise conjecture of 
Kottwitz (see \cite[\S 7]{HainesBernsteinCenter}). This was first shown by Haines-Ng\^o for 
 unramified unitary groups and for symplectic groups \cite{HainesNgoNearby}, and by Gaitsgory in the function field case \cite{GaitsgoryInv}. Let us give some details.
Since $G$ is unramified, $E$ is an unramified extension of $\Q_p.$ We denote by $E_r/E$ the unramified extension of degree $r,$ and by $\kappa_r$ its residue field.
Suppose that $\eK_p^\circ \subset G(\Q_p)$ is a parahoric subgroup, and set $P_r = \calG^\circ(\O_{E_r}).$ Let $\mu$ be a cocharacter of $G,$ in the conjugacy class of  $\mu_h,$ where $h \in X.$ 
One has the associated Bernstein function $z_{\mu,r}$ in the center of
the parahoric Hecke algebra $C_c(P_r\backslash G(E_r)/P_r).$

\begin{ithm}  (Kottwitz's conjecture) Suppose that $G$ is unramified over $\Q_p$, and that either $(G^{\ad}, X^{\ad})$ has no factor of type $D^{\mathbb H},$ or
$\eK_p^\circ$ is contained in a hyperspecial subgroup.  Let $r \geq 1$ and set $q = |\kappa_r|,$ 
and  $d= \dim \Sh_{\eK^\circ}(G, X).$
There is a natural embedding 
\begin{equation*}
\Gg^\circ(\mF_q)\backslash \rM^{\rm loc}_{G,X}(\mF_q)\hookrightarrow P_r\backslash G(E_r)/P_r.
\end{equation*}
For $y\in \SSh(\kappa_r)$
\begin{equation}%\label{kottwitz}
{\rm Tr}^{ss}({\rm Frob}_y, R\Psi^\SSh _{\bar y})=q^{d/2}z_{\mu, r}(w)
\end{equation}
where $w \in \rM^{\rm loc}_{G,X}(\kappa_r)$ corresponds to $y$ via the local model diagram.
\end{ithm}

We now explain the methods and organization of the paper in more detail. When $\eK_p^\circ$ is hyperspecial the integral models 
$\SSh_{\eK^\circ}(G,X)$ were constructed in \cite{KisinJAMS} and, as expected, turn out to be smooth. However, for more general parahoric 
level structures $\eK_p^\circ,$ many of the key arguments of \cite{KisinJAMS} break down or become much more complicated. 

In the first section, we prove various results about the parahoric group schemes $\calG^\circ,$ and torsors over them. 
To explain these, consider a faithful minuscule representation $\rho: G \rightarrow \GL(V).$ In \S 1.2, we explicate a result of 
Landvogt \cite{LandvogtCrelle}, and show that $\rho$ induces a certain kind of embedding of buildings $\iota: \calB (G,\Q_p) \hookrightarrow \calB (\GL(V),\Q_p).$
This is then used to show in \S 1.3, that for $x \in \calB (G,\Q_p),$ there is an closed embedding of   group schemes 
$\calG_x \rightarrow \mathcal{GL}(V)_{\iota(x)}.$ The existence of such an embedding is needed in exploiting Hodge cycles, to study integral 
models later in the paper. It replaces a general result for maps of reductive groups due to Prasad-Yu \cite{PrasadYuJAG}, which was used in \cite{KisinJAMS}.

In \S 1.4, we show that a $\calG^\circ$-torsor over the complement of the closed point in $\Spec (W(\mathbb F_q)\lps u \rps)$ extends to $\Spec (W(\mathbb F_q)\lps u \rps),$ 
and hence is trivial. As in \cite{KisinJAMS}, this result is used in an essential way in showing that the crystalline realizations of certain Hodge cycles have good $p$-adic integrality properties, 
and eventually in relating the local models ${\rm M}^{\loc}_{G,X}$ to the integral models $\SSh_{\eK_p^\circ}(G,X).$ When $\calG^\circ$ is reductive, this extension result was proved in \cite{ColliotSansuc}, 
and is a simple consequence of the analogous extension result for vector bundles. For general parahorics $\calG^\circ,$ the proof becomes much more involved, and uses in particular 
results of Gille \cite {GilleSerreIICompositio} and Bayer-Fluckiger -- Parimala \cite{BFluckParimalaInvent}, \cite{BFluckParimalaAnnals} on Serre's conjecture II. 
In fact, for this reason we prove the result only when $G$ has no factors of type $E_8.$

In \S 2, we recall the construction of the local models $\rM^{\loc}_{G,X}$ introduced in \cite{PaZhu}. Their definition involves the affine Grassmannian, however 
using the embedding $\iota$ mentioned above, we show that these local models can also be described as an orbit closure in a Grassmannian. This description is used in \S 3, 
to show that any formal neighborhood of a closed point of $\rM^{\loc}_{G,X}$ supports a family of $p$-divisible groups, equipped with a family of crystalline cycles. 
More precisely, let $K/\Q_p$ be a finite extension, $\ffG$ a $p$-divisible group over $\O_K,$ and $(s_{\alpha,\et})\subset T_p\ffG^\otimes$ a family of Galois invariant tensors in the Tate module 
$T_p\ffG,$ whose pointwise stabilizer can be identified with the parahoric group scheme $\calG^\circ \subset \GL(T_p\ffG)$ 
(in fact we deal also with non-connected stabilizers).  If $\DD$ denotes the Dieudonn\'e module of $\ffG,$ then the crystalline counterparts of the 
$(s_{\alpha,\et})$ are tensors $(s_{\alpha,0}) \subset \DD[1/p]^\otimes.$ Using the extension result of \S 1.4, mentioned above, we show that $(s_{\alpha,0}) \subset \DD^\otimes$ 
and that these tensors define a parahoric subgroup of $\GL(\DD)$ which is isomorphic to $\calG^\circ.$ This allows us to construct the required family 
of $p$-divisible groups over a formal neighborhood of $\rM^{\loc}_{G,X}.$ In \cite{KisinJAMS} this was done using an explicit construction of the universal deformation, due to 
Faltings. However this construction does not seem to generalize to the parahoric case, and we use instead a construction involving Zink's theory of displays 
\cite{ZinkDisplay} (\S 3.1, 3.2).

Finally in \S 4, we apply all this to integral models of Shimura varieties. We use the families of $p$-divisible groups over formal neighborhoods of $\rM^{\loc}_{G,X},$ 
to relate $\rM^{\loc}_{G,X}$ and $\SSh_{\eK_p^\circ}(G,X),$ when $(G,X)$ is of Hodge type. 
In particular, these results also cover the PEL cases of \cite{RapZinkBook} and 
our proof then circumvents the complicated case-by-case linear algebra arguments with lattice chains in {\em loc. cit.},~Appendix. (In some sense, the role of these linear algebra arguments  is 
now played   by 
the extension result of \S 1.4.)
To extend these results to the case of abelian type Shimura data, we follow 
Deligne's strategy \cite{DeligneCorvallis}, using connected Shimura varieties and the action of $G^{\ad}(\Q)^+.$ As in \cite{KisinJAMS} we use a moduli theoretic description 
of this action, in terms of a kind of twisting of abelian schemes. In the final subsection, we give the application to nearby cycles and Kottwitz's conjecture.

The application to integral models is somewhat complicated by the phenomenon that for $x \in \calB(G,\Q_p),$ the stabilizer group scheme $\calG_x,$ attached to $x$ 
by Bruhat-Tits, may not have connected special fibre. On the one hand, it is more convenient to work with the connected component of the identity $\calG^\circ_x;$ 
for example, the local model diagram in Theorem \ref{ithmIII} yields an isomorphism of henselizations only when $\pi$ is a torsor under a smooth {\em connected} group 
(using Lang's lemma). On the other hand, our arguments with Hodge cycles yield direct results only for integral model with level $\eK_p = \calG_x(\mathbb Z_p).$
We are able to overcome these difficulties in most, but not quite all cases, and this is the reason for the restriction on $G$ in Theorem \ref{ithmIII}.

\smallskip

\noindent{\it Acknowledgment.} We would like to thank T. Haines, K.-W.~Lan, G.~Prasad, M.~Rapoport and R.~Zhou for a number of useful suggestions.

 \section{Parahoric subgroups and minuscule representations}
\subsection{Bruhat-Tits and parahoric group schemes}
 
\begin{para}\label{notnsetup}  Let $p$ be a prime number.  If $R$ is an algebra over the $p$-adic integers $\Z_p$, we will denote by $W(R)$
the ring of Witt vectors with entries in $R$. Let $k$ be either a finite extension of $\mathbb F_p$ or an algebraic closure 
of $\mathbb F_p.$ Let $\bar k$ be an algebraic closure of $k.$ 
We set $W=W(k),$ $K_0={\rm Frac}(W),$ and $L = {\rm Frac}W(\bar k).$

In what follows, we let $K$ be either a finite totally ramified field extension of $K_0$,
or the equicharacteristic local field $k\llps \pi\lrps$ of Laurent
power series with coefficients in $k$.
We let $\bar K$ be an algebraic closure of $K$ 
with residue field $\bar k.$ We denote by $K^{\ur} \subset \bar K$ the maximal unramified extension of $K$ 
in $\bar K,$ and we write $\O = \O_K$ and $\O^{\ur} = \O_{K^{\ur}}$ for the valuation rings of $K$ and $K^{\ur}.$
\end{para}

\begin{para}\label{btstuff2} 
Let $G$ be a connected reductive group
over $K$.  We will denote by $\calB(G, K)$ the (extended) Bruhat-Tits building of $G(K)$
\cite{BTI, BTII, TitsCorvallis}. We will also consider the building $\calB(G^{\rm ad}, K)$ of the adjoint group;
the central extension $G\to G^{\rm ad}$ induces a natural $G(K)$-equivariant map $\calB(G, K)\to \calB(G^{\rm ad}, K)$
which is a bijection when $G$ is semi-simple. In particular, we can identify $\calB(G^{\rm der},K)$ with $\calB(G^{\rm ad}, K)$.

If $\Omega $ is a non-empty bounded subset of $ \cal B(G, K)$ which is contained in an apartment,
we will write $G(K)_\Omega=\{g\in G(K)\ |\ g\cdot x=x, \forall\ x\in \Omega\}$
for the pointwise stabilizer (``fixer")  of $\Omega$ in $G(K)$ and denote by $G(K)_\Omega^\circ$
  the ``connected stabilizer" (\cite[\S 4]{BTII}). When $\Omega=\{x\}$ is a point,  $G(K)_x^\circ$ is, by definition, the parahoric subgroup  
of $G(K)$ that corresponds to $x$. Similarly, if $\Omega$ is an open facet, 
$G(K)_\Omega^\circ$ is the parahoric subgroup that corresponds to the
facet $\Omega$. If $\Omega$ is an open facet and $x\in \Omega$, then $G(K)_\Omega^\circ=G(K)_x^\circ$.

Similarly, we can consider $G(K^{\rm ur})$, $G( K^{\rm ur})_\Om$ and $G(K^{\rm ur})^\circ_\Om$. By the main result of \cite{BTII}, there is a smooth affine group scheme $\Gg_\Om$ over $\Spec(\O)$
with generic fibre $G$ which is uniquely characterized by the property that $\Gg_\Om(\O^{\rm ur})=G(K^{\rm ur})_\Om$. 
By definition, we have  $G(K^{\rm ur})^\circ_\Om = \Gg^\circ_\Om(\O^{\rm ur}),$ where $\Gg^\circ_\Om$ is the connected component of $\Gg_\Om.$ 
We will call $\Gg^\circ_x$ a ``parahoric group scheme" (so these are, by definition, connected).
More generally, we will call $\Gg_\Om$ a ``Bruhat-Tits group scheme"
(even if $\Omega$ is not a facet).  

Denote by $ \bar\Omega \subset \calB(G^{\rm ad}, K)$ the image
of $\Omega$ under $\calB(G, K)\to \calB(G^{\rm ad}, K)$. We can then also consider
the subgroup $G(K)_{\bar\Omega } \subset G(K)$ fixing $\bar\Omega.$ We
have $G(K)_\Omega\subset G(K)_{\bar\Omega}$.   By \cite[Prop. 3 and Remarks 4 and 11]{HainesRapoportAppendix}, 
 $G(K^{\rm ur})^\circ_\Omega$ is the intersection of $G(K^{\rm ur})_{\bar\Omega }$ (and hence, also of 
$G(K^{\rm ur})_{\Omega}$) with the kernel $G(K^{\rm ur})_1$ of the Kottwitz homomorphism $\kappa_G: G(K^{\rm ur})\to \pi_1(G)_I$. 
It then follows that  $G(K)^\circ_\Omega$ is also the intersection of $G(K)_{\Omega }$  with the kernel   of the Kottwitz homomorphism.
As a result, using \cite[(1.7.6)]{BTII},
 we see that $\Gg_x^\circ$ only depends on $G$ and the image $\bar x$ of $x$ in $\calB(G^{\rm ad}, K)$.

If $G$ is semi-simple, simply connected, then $\kappa_G$ is trivial and we have $G(K)^\circ_\Om=G(K)_\Om$. 
\end{para} 

\begin{para}\label{centralIso}
We continue with the notations of the previous paragraph.
Let $\al: G\to \ti G$ be a central extension between connected
reductive groups over $K$ with kernel $Z$. By \cite[(4.2.15)]{BTII}, or \cite[Theorem 2.1.8]{LandvogtCrelle}, $\al$ induces a canonical $G(K)$-equivariant map 
$\al_*: \calB(G, K)\to \calB(\ti G , K)$. Set $\ti x =\al_*(x)$. Then 
$\al(G(K^{\rm ur})_x)\subset \ti G (K^{\rm ur})_{\ti x}$
and,   by \cite[(1.7.6)]{BTII}, $\al$ extends to group scheme homomorphisms
\begin{equation*}
\al: \Gg_x\xrightarrow{\  }  \ti\Gg_{\ti x} , \quad  \al: \Gg_x^\0\xrightarrow{\  }  \ti\Gg_{\ti x}^\0 .
\end{equation*}
We record the following for future use:

\begin{prop}\label{cIsoProp}
Suppose that $G$ splits over a tamely ramified extension
of $K$ and that $Z$ is either a torus or is finite of rank prime to $p$.
Then the schematic closure $\calZ$ of $Z$ in $\Gg_x^\0$ is smooth over $\Spec(\O)$ and
it fits in an (fppf) exact sequence
\begin{equation}
1\to \calZ\to  \Gg_x^\0\xrightarrow{\al  }  \ti\Gg_{\ti x }^\0\to 1 
\end{equation}
 of  group schemes over $\Spec(\O)$.
If $Z$ is a torus which is a direct summand of an induced torus,
then $\calZ=\calZ^\0$ is the connected Neron model of $Z$.
\end{prop}

\begin{proof}  
By base change, it is enough to show the Proposition when $k$ is algebraically closed.
Then both $G$ and $\ti G$ are quasi-split by Steinberg's theorem, and 
by our assumption, they split after a tame finite Galois extension $K'/K.$
Set $\Gamma={\rm Gal}(K'/K)$ which is a cyclic group.

Choose a maximal split torus in $G$ whose apartment contains $x,$ and let $T$ be its centralizer.
Since $G$ is quasi-split, $T$ is a maximal torus and we have an exact sequence
$$
1\to Z\to T\xrightarrow{\al} \ti T\to 1
$$
with $\ti T$ a maximal torus in $\ti G$. The central morphism $ \al:  G\to \tilde G$ induces an isomorphism 
between corresponding root subgroups $ U_{ a}$ and $\ti U_{ a}$.
 Denote by $ \calU_a$ and $\ti \calU_a$  their corresponding schematic closures  in 
$ \Gg_x^\0$  and $\ti \Gg_{\ti x}^\0$ respectively.
By \cite[(4.2.15), (4.3.2)]{BTII}, and the construction of $ \Gg_x^\0$  and $\ti \Gg_{\ti x}^\0$
in \cite[\S 4.6]{BTII}, the morphism $\al$ induces an isomorphism between $ \calU_a(\O)$ and $\ti \calU_a(\O)$, and therefore, by \cite[(1.7.6)]{BTII},
 between   $ \calU_a$ and $\ti \calU_a$. Also, by \cite[\S 4.4, \S 4.6]{BTII}, the schematic closure 
of $T$, resp. $\ti T$,  in $ \Gg_x^\0$, resp.  $\ti \Gg_{\ti x}^\0$, 
is  the connected Neron model $ \calT^\0$, resp. $ \ti\calT^\0 $ of $T$, resp. $\ti T$.
Assume we have an fppf exact sequence
\begin{equation}\label{exactZTT}
1\to \calZ\to  \calT^\0\xrightarrow{\al} \ti\calT^\0 \to 1
\end{equation}
 where $\calZ$ is smooth and is the schematic closure of $Z$ in 
$\calT^\0$. Then $\calZ$ is also the  schematic closure of $Z$ in 
$\Gg^\0_x$ and the quotient $\Gg^\0_x/\calZ$,  which is 
representable by \cite[\S 4]{Ana},  is a connected smooth group scheme 
which admits a  homomorphism
$\gamma: \Gg^\0_x/\calZ\to \ti\Gg^\0_{\ti x}$. 
Using again the construction of the parahoric group schemes via schematic root data (\cite[\S 3]{BTII}, 
\cite[\S 4.6]{BTII} for the quasi-split case), we see that $\gamma$
 is an isomorphism on an open neighborhood of the identity given 
 by the ``open big cell".
By \cite[(1.2.13)]{BTII}, $\gamma$ is an isomorphism and this 
proves the Proposition.

It remains to exhibit the exact sequence (\ref{exactZTT}).
 
If $Z$ is a torus the desired statement follows by the argument
in the proof of \cite[Lemma 6.7]{PappasRaTwisted}
which gives the analogous result in the equal characteristic case. 

Suppose that $Z$ is finite of rank prime to $p$. 
By base changing to $K'$ we obtain
$
1\to Z'\to T'\to \ti T'\to 1
$
with $T'$, $\ti T'$ split over $K'$;
here the prime indicates base change extension
to $K'$. This extends to an exact sequence of  
group schemes 
$$
1\to \calZ'\to  \calT'\to  \ti\calT'\to 1
$$
with $\calT'$, $ \ti\calT'$ split tori over $\O'$. 
Then $\calZ'= A(1)$,
a finite multiplicative group scheme 
with $\Gamma$-action which is the Zariski closure of $Z'$ in  $\calT'$.
(Here $A$ is a finite abelian group with $\Gamma$-action and
$A(1)=A\otimes_\Z\mu_n$, for $n={\rm exp}(A)$, $n$ prime to $p$.)
As $p \nmid n,$ we can see, using Hensel's lemma, that
we have an exact sequence of smooth group schemes
\begin{equation}\label{exactO'}
1\to {\rm Res}_{\O'/\O}( A(1))\to {\rm Res}_{\O'/\O}(\calT')\to {\rm Res}_{\O'/\O}(\ti\calT')\to 1.
\end{equation}

By taking the $\Gamma$-fixed (closed) subschemes we obtain the exact sequence 
$$
1\to {\rm Res}_{\O'/\O}( A(1))^\Gamma\to {\rm Res}_{\O'/\O}(\calT')^\Gamma\to {\rm Res}_{\O'/\O}(\ti\calT')^\Gamma. 
$$
Since $\#\Gamma$ is prime to $p$, by \cite[Prop. 3.1]{EdixhTame}, these 
fixed point (closed) subgroup schemes are also smooth over $\O$. The neutral components $\calT^\0$
and $\ti\calT^\0$
of $\calT:={\rm Res}_{\O'/\O}(\calT')^\Gamma$ and $\ti\calT:={\rm Res}_{\O'/\O}(\ti\calT')^\Gamma 
$ are the connected Neron models of $T$ and $\ti T$ respectively.

 Since $\O$ is strictly henselian, taking $\O$-valued points on (\ref{exactO'}) is exact. Using this together with
the long exact sequence of $\Gamma$-cohomology gives an exact sequence 
$$ 0 \to A^\Gamma \to \calT(\O) \to \ti\calT(\O) \to {\rm H}^1(\Gamma, A).$$
Since ${\rm H}^1(\Gamma, A)$ is finite, $\calT\to \ti\calT$ has open image, and induces  a surjection $\calT^\0 \to \ti\calT^\0$ 
between neutral components. Finally $\calZ = \ker(\calT^\0 \to \ti\calT^\0)$ is open in $\ker(\calT \to \ti\calT),$ and hence \'etale, 
which completes the construction of (\ref{exactZTT}).
\end{proof}

\begin{Remark}\label{cIsoRem}
{\rm    Using similar arguments as above, we can also see that, under the assumptions of Proposition \ref{cIsoProp}, the schematic closure of $Z$ in $\Gg_x$ is 
smooth over $\Spec(\O)$ and is equal to the kernel of $\al: \Gg_x\to \ti\Gg_{\ti x}$.
In general,  $\al: \Gg_x\to \ti\Gg_{\ti x}$
is not fppf surjective; this happens, for example, when $\Gg_x=\Gg_x^\0$
but $\ti\Gg_{\ti x}\neq \ti\Gg_{\ti x}^\0$.}
\end{Remark}

\end{para}

\begin{para}\label{buildGL}
{\sl The building $\B(\GL(V), K)$:} Suppose that $V$ is a finite dimensional $K$-vector space.
By \cite[Prop. 1.8, Th. 2.11]{BTclassI}, the points of the building $\B(\GL(V), K)$ are in 1-1 correspondence with graded
periodic lattice chains $(\{\Lambda\}, c)$: By definition, a periodic lattice chain 
is a non-empty set of $\O$-lattices $\{\Lambda \} $ in $V$ which 
is totally ordered by inclusion and stable under multiplication by scalars. 
A grading $c$ is a strictly decreasing function $c: \{\Lambda\}\to {\mathbb R}$ 
which satisfies
\begin{equation*}
c(\pi^n\Lambda)=c(\Lambda)+n
\end{equation*}
where $\pi$ is a uniformizer of $\O$. One can check (\emph{loc.~cit.}) that there
is an integer $r\geq 1$ (the period) and distinct lattices $\Lambda^i$, for $i=0,\ldots , r-1$,
such that
\begin{equation}
\pi\Lambda^0\subset \Lambda^{r-1}\subset \cdots \subset \Lambda^1\subset \Lambda^0
\end{equation}
and  $\{\Lambda\}=\{\Lambda^i\}_{i\in \Z}$, with $\Lambda^j$ defined by 
$\Lambda^{mr+i}=\pi^m\Lambda^i$ for $m\in \Z$, $0\leq i< r$. 

The stabilizer $\GL(V)_x$
of the point $x\in \B(\GL(V), K)$ that corresponds to $(\{\Lambda\}, c)$ is 
 the intersection $\bigcap^{r-1}_{i=0} \GL(\Lambda^i)$ in $\GL(V)$.
By \emph{loc. cit.} 3.8, 3.9, the  corresponding parahoric group scheme ${\mathcal {GL}}_x$ 
is the Zariski closure of the diagonally embedded 
$\GL(V)\hookrightarrow \prod^{r-1}_{i=0}\GL(V)$ in the 
product $\prod_{i=0}^{r-1} \GL(\Lambda^i)$. The group scheme
${\mathcal {GL}}_x$ can also be identified   
with the group scheme of automorphisms ${\rm Aut}(\Lambda^\bullet)$ 
of the (indexed) lattice chain $\Lambda^\bullet:=\{\Lambda^i\}_{i\in \Z}$.
This is true since  this group
of automorphisms is smooth (by \cite[Appendix to Ch. 3]{RapZinkBook}) 
and has the same $\O^{\ur}$-valued points
as ${\mathcal {GL}}_x$. In fact, in \cite{BTclassI}, one finds a similar description of the building $\B(\GL(V)_D, K)$
and the parahoric subgroups when $V$ is a finite dimensional (right) $D$-module,
where $D$ is a finite dimensional $K$-central division algebra. For $x\in \B(\GL(V)_D, K)$,
we will denote by $({\mathcal {GL}}(V)_D)_x$
the corresponding parahoric group scheme.

Note here that to simplify notation we will use the symbol $\GL(\Lambda)$ to denote both the abstract group
and the corresponding group scheme over $\Spec(\O)$; this should not lead to confusion.
\end{para}
 
 \begin{para}\label{buildGSp}

{\sl The building $\B(\GSp(V), K)$:} Suppose that $V$ is a finite dimensional $K$-vector space
with a perfect alternating bilinear form $\psi: V\times V\to K$. There is an involution
on the set of $\O$-lattices in $V$ given by $\Lambda\mapsto \Lambda^\vee:=\{v\in V\ |\ \psi(v, x)\in \O, \forall x\in \Lambda\}$.
In this case, 
the points of the building $\B(\GSp(V), K)$ are in 1-1 correspondence with ``almost self-dual" graded
period lattice chains $(\{\Lambda\}, c)$ and so $\B(\GSp(V), K)\subset \B(\GL(V), K)$. (This is a variant of a special
case of the results of \cite{BTclassII} that describe $\B({\rm Sp}(V), K)$.) Here, almost self-dual means that the set $\{\Lambda\}$ is stable under the involution
and that $c(\Lambda^\vee)=-c(\Lambda)+m$ for some $m\in \Z$, independent
of $\Lambda$.  In this case, there 
is an integer $r\geq 1$ and distinct lattices $\Lambda^i$, for $i=0,\ldots , r-1$,
such that
\begin{equation}
\Lambda^{r-1}\subset \cdots \subset  \Lambda^0\subset (\Lambda^0)^\vee\subset    \cdots \subset (\Lambda^{r-1})^\vee\subset \pi^{-1}\Lambda^{r-1},
\end{equation}
and for $a=0$, or $1,$ we have $(\Lambda^i)^\vee=\Lambda^{-i-a}$ for each $i.$
The complete chain $\{\Lambda\}$ consists of all scalar multiples of these lattices $\Lambda^i$ and $(\Lambda^i)^\vee$.
The stabilizer $\GSp(V)_x$
of the point $x\in \B(\GSp(V), K)$ that corresponds to $(\{\Lambda\}, c)$ is $\GSp(V)\cap \GL(V)_x$.
The  corresponding parahoric group scheme ${\mathcal {GSP}}_x$ 
is the schematic closure of the diagonally embedded 
$\GSp(V)\hookrightarrow \prod_{i=-(r-1)-a}^{r-1} \GL(V)$ in the 
product $\prod_{i=-(r-1)-a}^{r-1} \GL(\Lambda^i)$. As above, by \cite[Appendix to Ch. 3]{RapZinkBook},
this identifies
with the group scheme of similitude automorphisms ${\rm Aut}(\{\Lambda^i \}_{i\in \Z}, \psi_i)$
of the polarized lattice chain. Here $\psi_i: \Lambda^i\times \Lambda^{-i-a}\to \O$
are the perfect alternating forms given by $\psi$ and we consider automorphisms
that respect the $\psi_i$ up to common similitude.

Consider $V'=\oplus_{i=-(r-1)-a}^{r-1} V$ equipped with the perfect alternating $K$-bilinear form $\psi': V'\times V'\to K$ given as the orthogonal direct sum $\perp_{i=-(r-1)-a}^{r-1} \psi$. We have a natural ``diagonal" embedding 
$\GSp(V, \psi)\hookrightarrow \GSp(V',\psi')\subset \GL(V')$.
Consider the lattice $\Lambda'=
\oplus_{i=-(r-1)-a}^{r-1} \Lambda^i\subset V'$. Then, by the above, the group scheme 
${\mathcal {GSP}}_x$ is the schematic closure of $\GSp(V, \psi)$ in $\GL(\Lambda')$. 
By replacing $\Lambda'$ by a scalar multiple, we can assume that $\psi'$ takes integral values 
on $\Lambda'$, \emph{i.e.} that $\Lambda'\subset \Lambda'^\vee$ where the dual is with respect
to $\psi'$.
\end{para}

\subsection{Maps between Bruhat-Tits buildings}\label{subsec:mapsbetweenbuildings}

\begin{para}
In this section, we elaborate on Landvogt's results \cite{LandvogtCrelle} on embeddings 
of Bruhat-Tits buildings induced by (faithful) representations $\rho: G\to \GL(V)$. 
Here faithful means that the kernel of $\rho$ is trivial. Then $\rho$ gives a closed
 immersion of group schemes over $K$ (see for example \cite[Theorem 5.3.5]{ConradNotes}).  
 Landvogt shows that such a $\rho$ induces  a $G(K)$-equivariant ``toral" isometric embedding
$\B(G, K)\to \B(\GL(V), K)$ (see \emph{loc. cit.} for the definition of toral); such an embedding is not uniquely determined  by $\rho$
but also depends on the choice of the image of a given special point in $\B(G, K)$.
In this section we give a more specific construction of such an embedding when $\rho$ is minuscule, see below.
This construction will be used in \ref{embeddlocalmod} for showing that local models
embed in certain Grassmannians.
\end{para}

\begin{para}\label{splitEmbedd} First suppose that $G$ is split over $K$; denote  by $G_\O$ a reductive model 
over $\O$. Let $x_o$ be a hyperspecial
vertex   of the building $\B(G, K)$ with stabilizer the hyperspecial subgroup $ G_\O(\O)$.
Recall that there is canonical embedding 
$\B(G,K)\hookrightarrow\B(G, K^\ur)$ and we can also think of $x_o$ as a 
hyperspecial vertex of $\B(G, K^\ur)$. 

Suppose $\rho: G\to \GL(V)$ is a representation defined over $K$ (not necessarily faithful).
Suppose we can write $V=\oplus_i V_i$, where for each $i$,  $\rho_i: G\to \GL(V_i)$ is a $K$-representation which is irreducible
and hence, since $G$ is split, also geometrically irreducible. 
Notice here that $\rho_i$ factors
$$
G\xrightarrow{a_i} G_i\hookrightarrow \GL(V_i)
$$
where   $a_i$ is an epimorphism.
If $K=k\llps \pi\lrps$, we assume that $a_i$, for each $i$, 
induces  a separable morphism between each root subgroup of $G$
and its image in $G_i$. 
Here $G_i$ is also a split reductive group.
Suppose that, for each $i$,  $\Lambda_i\subset V_i$ 
is an $\O$-lattice such that 
$$
\rho_i(G_\O(\O^\ur))\subset \GL(\Lambda_i\otimes_\O\O^\ur).
$$
We would like to give a map 
of buildings
\begin{equation*}
\iota : \B(G, K^{\rm ur}) \to   \B(\GL(V), K^{\rm ur}),
\end{equation*}
such that  $\iota(x_o)$  
is the point $[\Lambda]$ in  $\B(\GL(V), K^\ur)$ which is given by the 
$\O^\ur$-lattice $\Lambda:= (\oplus_i\Lambda_i)\otimes_\O\O^{\rm ur}$. By definition, this 
is the point  $[\Lambda]:=(\{\pi^n\Lambda\}_{n\in \Z}, c_\Lambda)$, 
with  $c_\Lambda(\pi^n\Lambda)=n$.

\begin{prop}\label{splitBTembedding}  We assume that $G$ is split and let $x_o$, $\rho: G\to \GL(V)$ be as above. There exists a   ${\rm Gal}(K^{\rm ur}/K)$- and $G(K^{\rm ur})$-equivariant  toral map
\begin{equation}
\iota : \B(G, K^{\rm ur}) \to   \B(\GL(V), K^{\rm ur}),
\end{equation}  such that $\iota(x_o)$ is the point which corresponds to $\Lambda\otimes_\O\O^\ur=(\oplus_i\Lambda_i)\otimes\O^\ur$ as described above. 
Suppose in addition that $\rho: G\to \GL(V)$ is faithful. Then $\iota$ is   an isometric embedding
and is the unique ${\rm Gal}(K^{\rm ur}/K)$- and $G(K^{\rm ur})$-equivariant  toral 
embedding  with $\iota(x_o)$ as above. The map $\iota$ gives by restriction a $G(K)$-equivariant  toral isometric embedding  $\iota: \B(G, K) \to   \B(\GL(V), K)$.  
\end{prop}

\begin{proof} 
By
\cite[Theorem 2.1.8]{LandvogtCrelle} and its proof,   there is a canonical $G(K^\ur)$- and ${\rm Gal}(K^\ur/K)$-equivariant toral map   $a_i: \B(G, K^\ur)\to \B(G_i, K^\ur)$.
(When $K=k\llps \pi\lrps$, 
even though $K$  is not perfect, we see using the separability assumption, that the proof of \cite[Theorem 2.1.8]{LandvogtCrelle} extends.)
Under this, the image of the hyperspecial $x_o\in \B(G, K)$ is a hyperspecial $x_{o, i}\in  \B(G_i, K)$. Denote by $G_{i,\O}$ the  reductive group scheme over $\O$ that corresponds to $x_{o, i}$.
Using \cite[(1.7.6)]{BTII}, we see that
 $a_i$ extends to a group scheme homomorphism $a_{i,\O}: G_\O\to G_{i,\O}$. 
Recall that  $
\rho_i(G_\O(\O^\ur))\subset \GL(\Lambda_i\otimes_\O\O^\ur).
$

\begin{lemma}\label{pointembedding}
We   have
$G_{i,\O}(\O^\ur)\subset \GL(\Lambda_i\otimes_\O\O^\ur)$ and $G_i\hookrightarrow \GL(V_i)$ extends to a group scheme homomorphism  $G_{i,\O}\rightarrow \GL(\Lambda_i)$.
\end{lemma}
\begin{proof} 
Note  that $a_{i,\O}(\O^\ur): G_\O(\O^{\ur})\to G_{i,\O}(\O^{\ur})$ is not always surjective. For every   root subgroup $U_i$ of $G_i$, there is a root subgroup $U$
of $G$ such that $a_{i,\O |U}: U\to U_i$ is an isomorphism; this extends 
to an isomorphism of corresponding integral root subgroups $\calU_i$ and $\calU$. Therefore, the $\O^\ur$-valued points of each root subgroup $\calU_i$ of $G_{i,\O}$ belong to the image of $a_{i,\O}(\O^\ur)$
and therefore lie in $\GL(\Lambda_i\otimes_\O\O^\ur)$. Now let $T$ be a  maximal split torus of $G$  such that $x_o$ is in the apartment of $T$.  The image $T_i$ of $T$ under $a_i$ is a maximal spit torus of $G_i$ and $x_{o, i}$ is in the apartment of  $T_i$. 
Suppose that $\calT\simeq {\mathbb G}^r_{m,\O}\subset G_\O$, resp. ${\calT}_i\simeq  {\mathbb G}^{r_i}_{m,\O}\subset G_{i,\O}$, are the N\'eron models of $T$, resp. $T_i$.
Then $a_{i,\O}$ restricts to $\calT\to \calT_i$.
By our assumption, $\rho_i$ gives a group scheme homomorphism $\calT\to \GL(\Lambda_i)$, which amounts to a grading 
of $\Lambda_i$ by the character group $\xch(T)=\xch(\calT)\simeq \Z^r$ of $\calT$. Since the representation $G\to \GL(V_i)$ factors through $a_i$, the  non-zero graded pieces of $\Lambda_i$
appear only for characters in the subgroup $\xch(\calT_i)\subset \xch(\calT)$. This shows that
there is $\calT_i\to \GL(\Lambda_i)$ such that 
$\calT\to \GL(\Lambda_i)$ is the composition $\calT\to\calT_i\to \GL(\Lambda_i)$.
Hence, $\calT_i(\O^\ur)\subset \GL(\Lambda_i\otimes_{\O}\O^\ur)$.
Since $G_{i, \O}(\O^\ur)$ is generated by $\calU_i(\O^\ur)$ (for all  root subgroups) 
and $\calT_i(\O^\ur)$ (see e.g. \cite[4.6]{BTII}), we conclude that $G_{i,\O}(\O^\ur)\subset \GL(\Lambda_i\otimes_\O\O^\ur)$. The second statement then follows from 
 \cite[(1.7.6)]{BTII}. 
\end{proof}
\smallskip

We will now use \cite[Theorem 2.2.9]{LandvogtCrelle}
to produce a $G_i(K^\ur)$- and ${\rm Gal}(K^\ur/K)$-equivariant  toral isometric 
embedding of buildings 
$\B(G_i, K^\ur)\hookrightarrow \B(\GL(V_i), K^\ur)$ that maps $x_{o, i}$ to $y_i=[\Lambda_i\otimes_{\O}\O^\ur].$  
The point $y_i$
in the building of $\GL(V_i)$ satisfies the conditions (TOR), (STAB) and (CENT)
of \emph{loc. cit.}: We can easily check (TOR); (STAB) then follows from 
\emph{loc. cit.} Prop. 2.5.2, since both the groups are split. For the same reason,
(CENT)  trivially follows from (STAB). By \cite[Theorem 2.2.9, Prop. 2.2.10]{LandvogtCrelle}
it then follows that there exists a unique $G_i(K)$-equivariant toral isometric 
embedding of buildings 
$\B(G_i, K^\ur)\hookrightarrow \B(\GL(V_i), K^\ur)$ that maps $x_{o, i}$ to $y_i$.
This map is also ${\rm Gal}(K^\ur/K)$-equivariant, since the image $y_i$ is fixed 
by ${\rm Gal}(K^\ur/K)$.
By composing we now obtain a corresponding
$G(K^\ur)$- and ${\rm Gal}(K^\ur/K)$-equivariant  toral map $\iota_i: \B(G, K^\ur) \to \B(\GL(V_i), K^\ur)$.
By combining the maps above, we obtain a $G(K^\ur)$- and ${\rm Gal}(K^\ur/K)$-equivariant toral map
\begin{equation}\label{splitting}
\iota: \B(G, K^\ur)\xrightarrow{(\iota_i)_i} {\prod}_i \B(\GL(V_i), K^\ur)=\B({\prod}_i\GL(V_i), K^\ur)\subset \B(\GL(V), K^\ur).
\end{equation}
See \cite[Prop. 2.1.6]{LandvogtCrelle} for the equality in the middle, above. 
The last embedding in the display is obtained as follows: Since $\prod_i \GL(V_i)$ is a Levi subgroup
of $\GL(V)$, we can apply \cite[Prop. 2.1.5]{LandvogtCrelle} and obtain an embedding which sends
the point corresponding to $([\Lambda_i\otimes_\O\O^\ur])_i $ 
to the point given by the $\O^\ur$-lattice $\oplus_i (\Lambda_i\otimes_\O\O^\ur) \subset V\otimes_KK^\ur$. If $\rho$ is faithful, then $\iota$ is injective and so it gives an embedding.
The uniqueness then follows from \cite[Prop. 2.2.10]{LandvogtCrelle}. 
\end{proof}

\begin{Remark}\label{rem125}
{\rm 
a) When $\rho$ is faithful, the embedding $\iota$ as above can also be obtained
directly from the ``descent" of root valuation data 
of \cite[9.1.19 (c)]{BTI} by using that $\rho$ maps the hyperspecial subgroup 
$G_\O(\O^\ur)$ to $\GL(\oplus_i(\Lambda_i\otimes_{\O}\O^\ur))$.

b) For any $t\in \R$, we also have a $G(K^{\rm ur})$-equivariant  toral
map $t+\iota : \B(G, K^{\rm ur}) \to   \B(\GL(V), K^{\rm ur})$ determined by
$(t+\iota)(x)=(\{\pi^n\Lambda\}_{n\in \Z}, c_\Lambda+t)$. This map is also
${\rm Gal}(K^{\rm ur}/K)$-equivariant. For every $x\in \B(G, K^{\ur})$, 
$(t+\iota)(x)$ and $\iota(x)$ have the same stabilizer in $\GL(V\otimes_KK^\ur)$.

c) More generally, suppose that, for each  $i$, we have  a pair $(\Lambda_i, t_i)$  of a $\O$-lattice
$\Lambda_i\subset V_i$ and a real number $t_i\in \R$ which determine the point 
$(\{\pi^n\Lambda_i\}_{n\in\Z}, c_{\Lambda_i}+t_i)$ in the building $\B(\GL(V_i), K)$.
Suppose also that $
\rho_i(G_\O(\O^\ur))\subset \GL(\Lambda_i\otimes_\O\O^\ur)
$, for each $i$. Then the proof of Proposition \ref{splitBTembedding} extends to give a  
${\rm Gal}(K^{\rm ur}/K)$- and $G(K^{\rm ur})$-equivariant  toral   map
\begin{equation}
\iota : \B(G, K^{\rm ur}) \to   \B(\GL(V), K^{\rm ur}),
\end{equation}
such that  $\iota(x_o)$ is   the image of  $ (\{\pi^n(\Lambda_i\otimes_\O\O^\ur)\}_{n\in\Z}, c_{\Lambda_i}+t_i)_i$
under the Levi embedding 
${\prod}_i \B(\GL(V_i), K^\ur)=\B({\prod}_i\GL(V_i), K^\ur)\subset \B(\GL(V), K^\ur)$.
If $\rho$ is faithful, this map   is an isometric embedding and is unique.
Note that this $\iota(x_o)$ is not always hyperspecial. 
For example, if $K=\Q_p$ and $\rho: G={\mathbb G}_m^2 \hookrightarrow \GL_2$ is the  embedding of the diagonal torus, all points of the corresponding apartment can appear as $\iota(x_o)$. Indeed, all points of the apartment are translations of $[\Z_p e_1\oplus \Z_p  e_2]$  by some $(t_1, t_2)\in \R^2$. 

d) Observe that in Proposition \ref{splitBTembedding} and also in (c) above, the map of buildings $\iota$ factors through a ``Levi embedding'', with the Levi subgroup determined by a decomposition of the representation $V$ as a direct sum of irreducibles; we use this
in the proof of Proposition \ref{miniscule}.  In general, there are equivariant maps that do not factor this way.

}
\end{Remark}
\end{para}
 
\begin{para}
For the rest of this section, unless we explicitly discuss the case $K=k\llps \pi\lrps$,
we will assume that ${\rm char}(K)=0$.

Let $T\subset G $ be a maximal torus. We will say that $\rho$ is 
{\sl minuscule} if $\rho\otimes_K\bar K$ is isomorphic to a direct sum of irreducible representations
which are minuscule in the sense that the weights of the corresponding representation of ${\rm Lie}(G^{\rm der}_{\bar K})$
on $V_{\bar K}$  for the Cartan subalgebra ${\rm Lie}(T^{\rm der}_{\bar K})$
are conjugate under the Weyl group. 
(See \cite[Ch. VI, \S 1, ex. 24, \S 4, ex. 15]{BourbakiLie}). This notion is independent of the choice of $T$.
When $G={\rm SL}_2$ the irreducible minuscule representations 
are the standard and the trivial representation.  

\begin{prop}\label{minusculeLattice}
Suppose that $G$ is split over $K$ and that $\rho: G\to \GL(V)$ is minuscule and irreducible.
Assume  that $\Lambda $, $\Lambda'$ are two $\O$-lattices in $V$ such that 
$\rho(G_\O(\O^{\ur}))\subset \GL(\Lambda \otimes_\O\O^\ur)\cap \GL(\Lambda'\otimes_\O\O^\ur)$,
the intersection taking place in $\GL(V\otimes_KK^\ur)$.
Then $\Lambda $ and $\Lambda'$ are in the same homothety class, i.e.
  $\Lambda'=\pi^n\Lambda $, for some $n\in \Z$.
\end{prop}

\begin{proof}   By \cite[(1.7.6)]{BTII}, our assumption implies that $\rho$ extends to group scheme homomorphisms 
$\rho_\O: G_\O\to \GL(\Lambda)$, $\rho_\O': G_\O\to \GL(\Lambda').$ 
Let $T \subset G$ be a maximal split torus such that $x_o$ is in the apartment of $T,$ 
and let $\calT \subset G_{\O}$ be the N\'eron model of $T.$
The torus $\calT$ acts on $\Lambda,$ and we can decompose $\Lambda$  as direct sum of  weight spaces 
$\Lambda =\oplus_{\lambda\in W(\rho)} \Lambda_{ \lambda}.$ Since $\rho$ is minuscule, the set of weights $W(\rho)\subset \xch(T)$ is an 
orbit $W\cdot \lambda_0$ of a single highest weight $\lambda_0$ under the Weyl group and all the spaces $V_\lambda$ are one dimensional
(\cite[Ch. VIII, \S 7, 3]{BourbakiLie}). In particular, it follows that $\Lambda\otimes_\O k$ is an irreducible $G\otimes_\O k$-representation \cite[II 2.15]{JantzenBook}. 

After replacing $\Lambda'$ by a scalar multiple, we may assume that $\Lambda'\subset \Lambda,$ and that if 
$\bar\Lambda' \subset \Lambda \otimes_\O k$ denotes the image of $\Lambda' $ in 
$\Lambda \otimes_\O k,$ then $\bar\Lambda' \neq \{0\}.$ Then $\bar\Lambda' \subset \Lambda\otimes_\O k$ is a non-zero $G\otimes_\O k$-subrepresentation. 
%Since $\Lambda'_{\lambda_0}=\Lambda_{\lambda_0}$, we have $\bar \Lambda'\neq (0)$. 
As  $\Lambda\otimes_\O k$ is irreducible this implies $\bar\Lambda'=\bar\Lambda,$ and so $\Lambda'=\Lambda$, as desired.
\end{proof}

\begin{cor}\label{uptoTranslation} Assume that, in addition to the above assumptions, $\rho$ is faithful. If $\iota$ and $\iota'$ are $G(K^\ur)$-equivariant toral   embeddings $\B(G, K^\ur)\to \B(\GL(V), K^\ur)$, then there is $t\in \R$
such that $\iota'=t+\iota$.
\end{cor}

\begin{proof}
By \cite[Prop. 2.2.10]{LandvogtCrelle}, such $\iota$, $\iota'$ are determined by the points $\iota(x_o)$, 
$\iota'(x_o)$  in $\B(\GL(V), K^\ur)$. Their stabilizers subgroups both have to contain  $\rho(G_\O(\O^{\ur}))$
and so by Proposition \ref{minusculeLattice} they both have to be hyperspecial.
Since such hyperspecial points are determined  up to translation by a real number  by their stabilizer
subgroups,  Proposition \ref{minusculeLattice} implies
the result. 
\end{proof}
 
\end{para}

\begin{para} We continue to assume that $G$ is split over $K$ and that $\rho: G\to \GL(V)$
is a  $K$-representation.

Denote by $H$ the  split Chevalley form  of $G$ over $\Z_p$; fix a pinning 
$(T, B, \und e)=(T_H, B_H, \und e)$ of $H$ over $\Z_p$ and a corresponding hyperspecial
vertex $x_o$ of the building $\B(H, \Q_p)$ whose stabilizer is $H(\Z_p)$. Choose an isomorphism
$G\simeq H\otimes_{\Z_p}K$, then we can take $G_\O=H\otimes_{\Z_p}\O$.
Recall that if $K'$ is any $p$-adic local field extension of $K$, there is a canonical embedding 
$\B(H,\Q_p)\hookrightarrow\B(H, K')$ and so we can also think of $x_o$ as a 
hyperspecial vertex of $\B(G, K')$ for all such $K'$.

Let $V=\oplus_i V_i$, $\rho=\oplus_i\rho_i$, with $V_i=V(\lambda_i)\otimes_{\Q_p}K$, $V(\lambda_i)$ an irreducible Weyl module of highest weight $\lambda_i$ (for our choice of $T$, $B$) over $\Q_p$; fix a highest weight vector 
$v_i=v_{\lambda_i}$ in $V(\lambda_i)$ and consider the $\Z_p$-lattice $\Lambda_i\subset  V(\lambda_i)$ 
given as $\Lambda_i={\mathfrak U}^-_H\cdot v_i$ where ${\mathfrak U}^-_H$
is the subalgebra ${\mathfrak U}_H$ of the universal enveloping algebra of $H$ over $\Z_p$ generated 
by the negative root spaces  acting on $V(\lambda_i)$. This gives
$\rho_i: H\to \GL(\Lambda_i)$ (cf. \cite{JantzenBook}) and we can see that 
the assumptions of Proposition \ref{splitBTembedding} are satisfied
for the choice of lattices $ \Lambda_i\otimes_{\Z_p}\O\subset V_i=V(\lambda_i)\otimes_{\Q_p}K$.
Hence, we have 
\begin{equation}
\iota : \B(G, K^{\rm ur}) \to   \B(\GL(V), K^{\rm ur}),
\end{equation}  such that $\iota(x_o)$ is the point which corresponds to $\Lambda\otimes_{\Z_p}\O^\ur=(\oplus_i\Lambda_i)\otimes\O^\ur$ as described above. More generally, we will also consider maps $\iota$ that also depend on the choice of a collection of $t_i\in\R$, as in Remark \ref{rem125} (c). The choice above corresponds to $t_i=0$.
If $\rho$ is faithful, $\iota$ is an embedding.
\end{para}
 
\begin{para}
We  now allow $G$ to be non-split; however, we always suppose that $G$ splits over
a tamely ramified Galois extension $\ti K/K$ with Galois group $\Gamma={\rm Gal}(\ti K/K).$ 
We allow $\ti K/K$ to be infinite, but we assume that the inertia subgroup of $\Gamma$ is finite.  

Choose an isomorphism $\psi:   G\otimes_K\ti K\xrightarrow{\sim} H\otimes_{\Z_p} \ti K$  
which identifies $G(\ti K)$ and $H(\ti K)$ and write  $G(K)=H(\ti K)^\Gamma$
where the action of $\Gamma$ is given by $\gamma\cdot \ti h=c(\gamma)\cdot \gamma(\ti h)$
with $c: \Gamma\to {\rm Aut}(H)(\ti K)$   the cocycle $c(\gamma)=\psi\cdot \gamma(\psi)^{-1}$.
The cocycle $c$ represents 
the class  of the form $G$ of $H$ in ${\rm H}^1(\Gamma, {\rm Aut}(H)(\ti K))$. 
Our choice of pinning of $H$ allows us to write ${\rm Aut}(H)(\ti K)$
as a semi-direct product 
$$
{\rm Aut}(H)(\ti K)=H^{\rm ad}(\ti K)\rtimes \Xi
$$
where $\Xi=\Xi_H$ is the group of Dynkin diagram automorphisms (which is then identified with the subgroup
of automorphisms of $H$ that respect the chosen pinning). 

Under the assumption of tameness,
by work of Rousseau or \cite{PrasadYuInv}, the canonical map $\B(G, K)\hookrightarrow \B(G, \ti K)$ gives
identifications $\B(G, K)=\B(G, \ti K)^\Gamma=\B(H,  \ti K)^\Gamma$; the action of 
$\Gamma$ on $\B(H, \ti K)$ is induced by the action of $\Gamma$ on $H(\ti K)$ given above.

\end{para}

\begin{para}\label{reprStuff} We now assume that $G$ is as above and consider 
a representation $\rho: G\to \GL(V)$ (\emph{i.e.} defined over $K$).
In what follows, assuming in addition that $\rho$ is minuscule, we will construct a certain $G(K^\ur)$- and ${\rm Gal}(K^\ur/K)$-equivariant toral map
\begin{equation}
\iota: \B(G,  K^\ur)\to \B(\GL(V),  K^\ur)
\end{equation}
which also restricts to give a map $\iota: \B(G,  K )\to \B(\GL(V),  K)$.

Assume first that $\rho: G\to \GL(V)$ is irreducible over $K$;
we do not assume that $\rho$ is faithful.
We follow the arguments of  \cite{TitsCrelle}  or \cite{Satake}. (See, for example, 
the proof of Theorem 7.6 in \cite{TitsCrelle}).  
Let 
$$
D^0=\{\phi\in  {\rm End}_K(V)\ |\ \phi\cdot \rho(g)=\rho(g)\cdot \phi, \forall g\in G(\ti K)\}
$$
be the centralizer algebra of $\rho$, which is a division $K$-algebra. Then $V$ is a (right) module for the opposite $K$-algebra $D=(D^0)^{\rm opp}$.

The Galois group $\Gamma$ acts naturally on the set of dominant weights of $G$
as described in \cite[3.1]{TitsCrelle}. 
For a dominant weight $\lambda$, we denote 
by $ V_{\lambda, \ti K}$ the $\ti K$-subspace of $ V\otimes_K\ti K$ 
generated by all simple submodules of highest weight $\lambda$. Let $\lambda_1, \ldots ,\lambda_r$ be the dominant weights $\lambda$
for which $ V_{\lambda, \ti K}\neq 0$. This set is $\Gamma$-stable, and we have
$$
V\otimes_K \ti K=  \oplus_{i=1}^r V_{\lambda_i, \ti K}.
$$

The $\Gamma$-action on $V\otimes_K \ti K$ induces a transitive action on 
the set of summands $V_{\lambda_i, \ti K},$ which coincides with the one induced by the action of $\Gamma$ on $\{\lambda_i\}_i.$
As in \emph{loc. cit.}, we have
$$
V_{\lambda_i,\ti K}\simeq V(\lambda_i)^{\oplus d}\otimes_{\Q_p}\ti K,
$$
where $d$ is an integer not depending on $i.$ 
Denote by $\Gamma_1\subset \Gamma$ the stabilizer of $\lambda_1$; let $K_1$ 
be the corresponding field $K\subset K_1\subset \ti K$ and set $V_1=V_{\lambda_1, \ti K}$. 
The center of $D$ can be identified with $K_1$ and then $V$ becomes a $K_1$-vector space;
the epimorphism $V\otimes_{K}\ti K\to V\otimes_{K_1}\ti K$ gives an isomorphism
 $V\otimes_{K_1}\ti K\simeq V_1$. We obtain a $K_1$-representation
$$
\bar\rho_1: G_{K_1}\to \GL(V)_D
$$
which is absolutely irreducible and is such that $\bar\rho_1\otimes_{K_1}\ti K$ is identified with
the Weyl module representation $\rho_1: G_{\ti K}\cong H_{\ti K}\to \GL(V(\lambda_1)_{\ti K})$.
As in \emph{loc. cit.}, the original $K$-representation $\rho: G\to \GL(V)$
can be obtained from $\bar\rho_1$ by applying restriction of scalars twice:
\begin{equation}\label{resSca}
\rho={\rm Res}_{K_1/K}({\rm Res}_{D/K_1}\cdot \bar\rho_1).
\end{equation}
Here, ${\rm Res}_{D/K_1}: \GL(V)_D \hookrightarrow \GL_{K_1}(V) $ is given by forgetting the
$D$-module structure, and ${\rm Res}_{K_1/K}: \GL_{K_1}(V)\hookrightarrow \GL(V)$ by forgetting the 
$K_1$-module structure.  More precisely, $\rho$ is the composition of
\begin{equation}\label{restrictionScalarsB}
G\to {\rm Res}_{K_1/K}(G_{K_1})\xrightarrow{{\rm Res}_{K_1/K}(\bar\rho_1)} {\rm Res}_{K_1/K}( \GL(V)_D) 
\end{equation}
with
\begin{equation}\label{restrictionScalars}
{\rm Res}_{K_1/K}( \GL(V)_D) \to {\rm Res}_{K_1/K}(\GL(V)_{K_1})\to \GL(V).
\end{equation}
 
In fact, $\bar\rho_1: G_{K_1}\to \GL(V)_D$ is a $K_1$-form of the Weyl module $\rho_1$ as follows:
The group $\Gamma_1$ acts on 
$G(\ti K)=H(\ti K)$ with the action given by twisting via the cocycle 
$c_{|\Gamma_1}: \Gamma_1\to {\rm Aut}(H)(\ti K)$.  Denote by $J_1(\ti K)$ the subgroup of ${\rm Aut}(H)(\ti K)$ generated by $H^{\rm ad}(\ti K)$ 
together with $c(\gamma)$ for $\gamma\in \Gamma_1$. Since $\lambda_1$ is $\Gamma_1$-invariant, 
for every $a\in J_1(\ti K)$, the representation $\rho_1\circ a$ is again irreducible of highest weight $\lambda_1$, and so
there is $\theta(a)\in {\rm PGL}(V(\lambda_1)\otimes \ti K)$ such that $\rho_1\circ a=\theta(a)\circ \rho_1$; by Schur's lemma, $\theta(a)$ is uniquely determined and hence it gives a homomorphism
\begin{equation}
\theta: J_1(\ti K)\to {\rm PGL}(V(\lambda_1)\otimes \ti K).
\end{equation}
 As in the proof of  \cite[Theorem 3.3]{TitsCrelle},  the cocycle 
$$
c':=\theta\cdot c: \Gamma_1\to {\rm PGL}(V(\lambda_1)\otimes\ti K)
$$
defines the $K_1$-form ${\rm End}(V )_D={\rm End}(V(\lambda_1)\otimes \ti K)^{\Gamma_1}$ of ${\rm End}(V(\lambda_1))$ and $\rho_1: H_{\ti K}\to \GL(V(\lambda_1)\otimes\ti K)$
descends to 
$$
\bar\rho_1: G_{K_1}=(H\otimes\ti K)^{\Gamma_1}\to  \GL(V)_D= \GL(V(\lambda_1)\otimes\ti K)^{\Gamma_1}.
$$
Here, the $\Gamma_1$-fixed points are for the $\Gamma_1$-actions  
given using the cocycles $c$ and $c'=\theta\cdot c$.

From here and on we will assume that $\ti K$ contains $K^\ur$.  
\end{para}
 
 \begin{prop}\label{equivBT} Assume that $\rho$, or equivalently  that $\rho_1$, is minuscule.
We equip $\B(\GL(V(\lambda_1)), \ti K)$ with the action of $\Gamma_1$ induced by 
the standard action on $\GL(V (\lambda_1 )\otimes {\ti K})$  twisted by the cocycle $c' .$

Then the $G(\ti K)=H(\ti K)$-equivariant toral map
$$
\iota_1: \B(G, \ti K)=\B(H, \ti K)\to \B(\GL(V(\lambda_1)), \ti K)
$$
given as in the split case above is $\Gamma_1$-equivariant. 
 \end{prop}

\begin{proof}  
By the construction of $\iota_1$ as a composition 
$$
\B(H, \ti K)\to \B(H/{\rm ker}(\rho_1), \ti K)\to \B(\GL(V(\lambda_1)), \ti K),
$$ we see that after replacing $H$ by $H/{\rm ker}(\rho_1)$,
and $G_{K_1}$ by $G_{K_1}/{\rm ker}(\bar\rho_1)$, we are reduced to considering the situation in which we assume in addition that
 $\rho_1$ is faithful. By \cite{LandvogtCrelle}, there is a  $G(\ti K)$- {\sl and $\Gamma_1$-equivariant} 
toral isometric embedding
$$
\iota^L_1: \B(G, \ti K)=\B(H, \ti K)\to \B(\GL(V(\lambda_1)), \ti K). 
$$
Regard now both $\iota_1$ and $\iota^L_1$  as two $H(\ti K)$-equivariant toral isometric maps between the buildings of 
the split reductive groups $H(\ti K)$ and $\GL(V(\lambda_1)\otimes \ti K)$ over $\ti K$. 
By Corollary \ref{uptoTranslation}, we have $\iota_1=t+\iota^L_1$ with $t\in\R=\xcoch({\rm diag}(\Gm_{\ti K}))\otimes_\Z\R$.  Notice now that the Galois group $\Gamma_1$
acts trivially on $\xcoch({\rm diag}(\Gm_{\ti K}))$.  Since $\iota^L_1$ is $\Gamma_1$-equivariant 
this implies that $\iota_1$ is also $\Gamma_1$-equivariant and this concludes the proof.
\end{proof}

\begin{para}\label{iotaconstrBefore} We continue with the above notations and assume that $\rho$ is   minuscule. Recall $\ti K$ contains $K^\ur$;  let $I_1\subset \Gamma_1$
be the inertia subgroup. 
Using \cite{PrasadYuInv} and Prop. \ref{equivBT} we see that by restricting to $I_1$-fixed points, $\iota_1$
gives
\begin{equation}\label{iota1map}
\iota_1:  \B(G, K^\ur_1)\to \B(\GL(V)_D, K^\ur_1).
\end{equation}
The same construction also works for the translations $t+\iota_1$,  $t\in \R$.
 This gives
\begin{equation}
\B(G, K^\ur)\subset \B(G, K^\ur_1)\xrightarrow {t+\iota_1} \B(\GL(V)_D, K^\ur_1).
\end{equation}
Compose this with the standard equivariant embedding 
$$
\B( \GL(V)_D, K^\ur_1)=\B({\rm Res}_{K^\ur_1/K^\ur}( \GL(V)_D), K^\ur)\to \B(\GL(V), K^\ur)
$$
given by sending $\O_D$-lattices in $V$ to the corresponding $\O$-lattices 
in the $K$-vector space $V$ (by restriction of structure from $\O_D$ to $\O$).
 This composition gives a $G(K^\ur)$-equivariant toral map
  \begin{equation*}
  \iota: \B(G, K^\ur)\to \B(\GL(V), K^\ur),
  \end{equation*}
 which is also ${\rm Gal}(K^\ur/K)$-equivariant as desired. This concludes the construction of $\iota$ when $\rho$ is minuscule and irreducible over $K$.
 \end{para}
  
  \begin{para}\label{iotaconstr} In general, if $\rho: G\to \GL(V)$ is a minuscule $K$-representation,  write
  it as a direct sum of $K$-irreducible representations $\rho_j: G\to \GL(V_j)$ and then proceed to give a
   $G(K)$-equivariant toral map $\iota: \B(G, K^\ur)\to \B(\GL(V), K^\ur)$ by combining $\iota_j: \B(G, K^\ur)\to \B(\GL(V_j), K^\ur)$
   given above with the Levi canonical embedding as in (\ref{splitEmbedd}). If $\rho$ is faithful, 
   the map $\iota$ is injective. Hence, in this case, we obtain  a $G(K^\ur)$-equivariant toral embedding of buildings
   \begin{equation}\label{iotaK}
  \iota: \B(G, K^\ur)\to \B(\GL(V), K^\ur)
  \end{equation}
 which is also ${\rm Gal}(K^\ur/K)$-equivariant as desired.
  
  \end{para}
  
  \begin{para}\label{Laurentfield1}
   Consider now the case $K=k\llps \pi\lrps$.
Suppose we have a reductive group $G$ over $K$ which splits over a tamely ramified extension $\tilde K/K$ and 
  a representation $\rho: G\to \GL(V)$. Assume that $\rho$ is written as a direct sum of $G$-representations which are obtained by restriction of scalars as in (\ref{resSca}) of representations $\bar\rho_1$ given as twisted Weyl modules  for minuscule dominant weights $\lambda$. Here we assume that the twist is also given in the same way as $\bar\rho_1$ is given in
  the characteristic $0$ case of \ref{reprStuff}. 
  (Note that in this case, $G$-representations are not in general semi-simple
  modules; however, here we assume such a direct sum 
  decomposition and we are also giving the twisting construction as in \ref{reprStuff} 
  as part of our data. Also recall, a dominant weight $\lambda$ 
 for a Chevalley group is {\sl minuscule} if there is no other 
dominant weight $\mu$ with $\mu<\lambda$, where $\leq $ denotes the usual partial ordering of weights. This implies that the Weyl module $V(\lambda)_k$ 
is simple \cite[II, 2.15]{JantzenBook} and that its weights are the Weyl group orbit of $\lambda$.)

 Under the above assumptions, we can obtain  maps of buildings $\iota_1$ as in Proposition \ref{equivBT}, and then $\iota$ as in (\ref{iotaK}), by carrying out the same construction as above. (Note that, under our assumptions, $H\to H/{\rm ker}(\rho_1)$ is separable
on each root subgroup -for that see also the proof of Proposition \ref{miniscule} below
that reduces this to the case $H={\rm SL}_2$- and so we can apply Proposition \ref{splitBTembedding} as a step in our construction.) \end{para}

\subsection{Minuscule representations and group schemes} 

\begin{para} We continue to assume that $G$ splits over a tamely ramified extension $\ti K$ 
 of $K$ with Galois group $\Ga={\rm Gal}(\ti K/K)$. We assume that $\rho: G\hookrightarrow \GL(V)$ is a  
 faithful minuscule 
 representation of $G$ where $V$ is a finite dimensional $K$-vector space. 
Recall the $G(K)$-embedding
\begin{equation}
\iota: \B(G, K)\to \B(\GL(V), K)
\end{equation}
constructed in the previous paragraph. This depends on a choice of an isomorphism $\psi: G_{\ti K}\xrightarrow{\sim} H_{\ti K}$ and 
a hyperspecial vertex $x_o$ of $\B(H, K)$ together with  choices of, for each $K$-irreducible summand, a lattice $\Lambda_1={\mathfrak U}^-_H\cdot v_1$ given by the highest weight vector  $v_1\in V(\lambda_1)$ and a   grading $c_{\Lambda_i}+t_i$ of the lattice chain $\{\pi^n\Lambda_i \}_{n\in\Z}$ given by $t_i\in\R$.
The map $\iota$ appears as a restriction of a ${\rm Gal}(K^\ur/K)$-equivariant $G(K^\ur)$-embedding
$\iota: \B(G, K^\ur)\to \B(\GL(V), K^\ur)$.
\end{para}

 \begin{prop}\label{miniscule}
 For any $x\in \calB(G, K)$,   $\rho$ extends to 
 a closed immersion 
 $$
 \rho_x: \Gg_x\to {\mathcal {GL}}(V)_{\iota(x)}
 $$ of group schemes over $\Spec(\O)$.
 \end{prop}
 
 \begin{proof}  Let $y=\iota(x)$ and suppose that $\Lambda^\bullet_y=\{\Lambda^i_y\}_{i\in \Z}$ is the periodic chain of $\O$-lattices in $V$ that corresponds to $y$ and   is fixed by ${{\GL}}(V)_{y}$.
Then $G(K^\ur)_x=G(K^\ur)\cap \GL(V\otimes_K K^\ur)_y\hookrightarrow \GL(V\otimes_K K^\ur)_y=\GL(\Lambda^\bullet_y\otimes_\O \O^\ur)$.
Using \cite[(1.7.6)]{BTII}, we obtain a group scheme homomorphism
$$
\rho: \Gg_x\rightarrow {\mathcal {GL}}_y
$$
which we would like to show is a closed immersion. Denote by $\Gg'_x$ the schematic closure of $G$ 
in  ${\mathcal {GL}}_y$; this agrees with the scheme theoretic image of $\rho$ above.
Notice that $y=\iota(x)$ implies  $\Gg'_x(\O^\ur)=G(K^\ur)\cap \GL(V\otimes_KK^\ur)_y=G(K^\ur)_x=\Gg_x(\O^\ur)$.
Therefore, it is enough to show that the schematic closure $\Gg'_x$ of 
$G\hookrightarrow \GL(V)$ in ${\mathcal {GL}}_y$
is smooth or equivalently (by the description of ${\mathcal {GL}}_y$ 
recalled in \ref{buildGL}), that the schematic closure of $G\hookrightarrow \prod_{i=0}^{r-1}\GL(V)$  
(embedded diagonally) in the $\O$-group scheme $\prod_{i=0}^{r-1}\GL(\Lambda^i_y)$
is smooth. 

 1) We first suppose that $G$ is split over $K$. Fix a maximal $K$-split torus 
 $T\simeq {\mathbb G}_{\rm m}^r$ of $G$ such that $x$ belongs to the apartment
 $A(G, T, K)\subset \calB(G, K)$.  To start with, we also assume that $G$ is semi-simple, \emph{i.e.} $G=G^{\rm der}$.  
 We first assume that $\rho$ is actually irreducible.  The torus $T$ acts on $V$ via $\rho$ and we obtain the weight decomposition
\begin{equation}\label{weights}
V=\oplus_{\lambda\in W(\rho)}V_\lambda.
\end{equation}
 Since $\rho$ is minuscule,
the set of weights
$W(\rho)\subset \xch(T)$ is an  orbit $W\cdot \lambda_0$ of a single weight $\lambda_0$ 
under the Weyl group and all the spaces $V_\lambda$ are one dimensional
(\cite[Ch. VIII, \S 7, 3]{BourbakiLie}).
Set $T^\flat=\prod_{\lambda\in W(\rho)}\GL(V_\lambda)$ for the 
maximal torus of $\GL(V)$ that preserves the grading above.
We have $\rho(T)\subset T^\flat$.

For a root $a\in \Phi(G,T)$, we denote by $U_a$ 
the corresponding unipotent subgroup of $G$. Set $G_a=\langle U_a, U_{-a}\rangle$
for the subgroup of $G$ generated by $U_a$ and $U_{-a}$. This is isomorphic to either ${\rm SL}_2$ or ${\rm PSL}_2$.
The isomorphism   takes the standard unipotent subgroups $U_{\pm }$ of ${\rm SL}_2$ to $U_{\pm a}\subset G$.
Consider now the restriction $\rho : G_a\rightarrow \GL(V)$
and the composition with the central isogeny ${\rm SL}_2\rightarrow G_a$
$$
\rho_a:  {\rm SL}_2\to \GL(V).
$$
We claim that this is a minuscule representation of ${\rm SL}_2$:
Indeed, consider $V$ as a representation of ${\rm Lie}(G_a)\simeq sl_2$. It decomposes as follows
$$
V=\oplus_{[\lambda]}V_{[\lambda]}=\oplus_{[\lambda]} (\oplus_{\lambda'=\lambda+ka} V_\lambda).
$$
Here $[\lambda]$ runs over all equivalence classes of weights in $W(\rho)$
under: $\lambda'\sim \lambda $ if there is $k\in \Z$ with $\lambda'-\lambda=k a$.
By the general theory (\emph{e.g.} \cite[Ch. VIII, \S 7, 2, Prop. 3]{BourbakiLie}), $V_{[\lambda]}$ are representations of $sl_2=\langle X_{-a}, H_a, X_a\rangle$ (a standard Chevalley basis) and there are two cases:

a) $[\lambda]=\{\lambda\}$ has only one element,

b) $[\lambda]$ has two elements and we can then assume   it is of the form $\{\lambda, \lambda+a\}$.

In the first case, $V_{[\lambda]}$ is the trivial representation of $sl_2$;
in the second case, $V_{[\lambda]}$  is isomorphic to the standard representation of $sl_2$. 
Therefore, for each root $a\in \Phi(G,T)$, the composition $\rho_a: {\rm SL}_2\to \GL(V)$ is a minuscule representation. It now follows that  $\rho_a$ does not factor through ${\rm PSL}_2$ and so $G_a$ has to be
isomorphic to ${\rm SL}_2$.

By the construction of $\Gg_x$, the schematic closure $\TT$ of $T\subset G$ 
in $\Gg_x$ is smooth  and so it acts on $\Lambda^i_y$ for all indices $i$.
Since $\TT\simeq {{\mathbb G}_m^r}_{/\O}$,
 we obtain  decompositions
\begin{equation}\label{weights2}
\Lambda_y^i=\oplus_{\lambda\in W(\rho)}\Lambda^i_{\lambda, y}
\end{equation}
with $\Lambda^i_{\lambda, y}\subset V_\lambda$ rank $1$ $\O$-lattices in
$V_\lambda$. (This implies that the point $y$  lies in the apartment 
$A(\GL(V), T^\flat, K)\subset \B(\GL(V), K)$.)
We can now use   this to  reduce to the case that $G$ is ${\rm SL}_2$.
Write $U_{\pm a}$, $G_a\simeq {\rm SL}_2$ as before. 

We now allow $V$ to be reducible and write $V=\oplus_j V_j$ where $V_j$
are irreducible and minuscule. By the above applied to the irreducible $V_j$, we can write
$$
V_j=\oplus_{[\lambda]}V_{j, [\lambda]}=\oplus_{[\lambda]} (\oplus_{\lambda'=\lambda+ka} V_{j, \lambda})
$$
as before. We have
 \begin{equation}\label{weights3}
\Lambda_y^i=\oplus_j\oplus_{\lambda\in W(\rho_j)}\Lambda^i_{j, \lambda, y}
\end{equation}
with $\Lambda^i_{j, \lambda, y}\subset V_{j,\lambda}$ rank $1$ $\O$-lattices in
$V_{j, \lambda}$. Here, we also use the construction of $\iota$, see \ref{iotaconstr} and
\ref{splitBTembedding} and also Remark \ref{rem125} (d).  
 We can now see that the schematic closures $\UU_{\pm a}$ of 
$U_{\pm a}$ in $\prod_i\GL(\Lambda^i_y)$ are isomorphic to the schematic closures
$\UU_\pm$, 
of 
$$
U_{\pm}\subset {\rm SL}_2\xrightarrow {\ \rho\ }   \prod_{j, [\lambda]}\prod_i \GL(V_{j, [\lambda]})
$$
in the group scheme
$$
 \prod_{j, [\lambda]}\prod_i\, \GL(\Lambda^i_{j, [\lambda], y})
$$
where $\Lambda^i_{j, [\lambda],y}=\Lambda^i_{j, \lambda, y}$ or $\Lambda^i_{j, \lambda, y}\oplus \Lambda^i_{j, \lambda+a, y}$
(in cases (a) or (b) respectively).
Consider classes $[\lambda]$ for which the ${\rm SL}_2$ 
representation $V_{j, [\lambda]}$ is not trivial, as in (b) above. We choose a basis vector $e_{j,\lambda}$ of $V_{j,\lambda}$ and set
$f_{j,\lambda}=X_a\cdot e_{j,\lambda}$
which is a generator of $V_{j,\lambda +a}$.
The choice of basis $e_{j,\lambda}, f_{j,\lambda}$, of $V_{j,[\lambda]}$
gives an identification of $V_{j,[\lambda]}$ with the standard representation of ${\rm SL}_2$.
We have  
$$
\Lambda^i_{j,\lambda, y}=\pi^{n_{j,[\lambda], i}}\O \cdot e_{j,\lambda},\quad \Lambda^i_{j,\lambda+a, y}=\pi^{m_{j,[\lambda], i}}\O \cdot    f_{j,\lambda}, 
$$
for some $m_{j,[\lambda], i}, n_{j,[\lambda], i}\in \Z$, and so  
under this identification,  the lattices $\Lambda^i_{j, [\lambda], y}\subset V_{j, [\lambda]}$, for all $i$, are in the same apartment for $\GL(V_{j,[\lambda]}),$ namely the standard apartment for the chosen basis.
It now follows from \cite[3.6, and 3.9 (2)]{BTclassI} that the schematic closures of $U_{\pm}$   in 
$\prod_{j,[\lambda]}\prod_i\, \GL(\Lambda^i_{j, [\lambda], y})$ are smooth. 
Hence, the same is true for the schematic closures $\UU_{\pm a}$.
By the construction of the lattices $\Lambda_{j,\lambda}^i$, the schematic closure of $T$ in 
$ \prod_i\, \GL(\Lambda^i_{  y})$ is smooth. 
It follows by \cite[Thm. 2.2.3]{BTII} that the schematic closure $\Gg_x'$ of $G$ in $ \prod_i\, \GL(\Lambda^i_{  y})$
contains the smooth big open cell
$$
\prod_{a }\UU_{-a}\times {\mathcal T}\times \prod_{a }\UU_{a}.
$$
Hence, by \cite[Cor. 2.2.5]{BTII},  the schematic closure $\Gg_x'$ is smooth.

\smallskip

\begin{Remarknumb}
{\rm The above is similar to corresponding arguments in \cite[\S 10]{GYexc1}, \cite[\S 9]{GYexc2}. Our assumption that  $\rho: G\to \GL(V)$ is minuscule is used in an essential way in this proof.
For example, the assumption that the weight spaces have dimension one is used to reduce to
the case of ${\rm SL}_2$: In general, for
$G$   split semi-simple   and $\rho$
irreducible, consider  $\rho_a: {\rm SL}_2 \to \GL(V)$ as before which we write
as a direct sum of irreducible representations
$V=\oplus_{t}V_t$.
If ${\dim}(V_\lambda)\neq 1$, we might have two distinct summands $V_{t_1}$, $V_{t_2}$,
with $V_\lambda\cap V_{t_1}\neq (0)$, $V_\lambda\cap V_{t_2}\neq (0)$.
Then we cannot guarantee that $\Lambda^i_y$ is equal to the direct sum 
$\oplus_t (\Lambda^i_y\cap V_t)$.}
\end{Remarknumb}

2) Assume now that $G$ is still split over $K$ but is not necessarily semi-simple.
The argument above   extends to this more general case by observing the following.
The $\O^\ur$-points $\TT(\O^\ur)$ of the Zariski closure $\TT$ of $T$ in ${\mathcal {GL}}_y$ 
give the {\sl maximal} compact subgroup of $T(K^\ur)$. 
(Indeed, ${\rm Aut}(\Lambda^\bullet_y)\cap G(K^\ur)$ is equal to $G(K^\ur)_x=\Gg_x(\O^\ur)$
and since $x$ is in the apartment of $T$, the subgroup $G(K^\ur)_x$ contains the maximal 
compact subgroup of $T(K^\ur)$.) Then the Zariski closure
$\TT$ is smooth by \cite[Lemma 4.1]{PrasadYuJAG}. The rest is as before,
since the unipotent subgroups $U_a$ and their Zariski closures
$\UU_a$ are the same for both $G$ and $G^{\rm der}$.
\smallskip

 3) We now consider the general case in which $G$ splits over the  tamely 
 ramified Galois extension $\ti K$ of $K$  with group $\Gamma={\rm Gal}(\ti K/K)$.
By \cite{PrasadYuInv}, we have
 \begin{equation}\label{py}
 \B(G, K)= \B(G, \ti K)^\Gamma,
 \end{equation}
where on the right hand side, we have the fixed points of the natural action by $\Gamma$.
For a bounded subset $\Omega\subset \B(G,K)$, the Galois group $\Gamma$ acts on  
 $G(\ti K)_\Omega$. Since we are assuming $\ti K = \ti K^{\ur}$, by \cite[(1.7.6)]{BTII},
 this action comes from an action of the Galois group $\Gamma$ on the smooth
group scheme ${\rm Res}_{\ti \O/\O}(\Gg_{\Omega, \ti K})$. (Here, we use the subscript $\ti K$ 
to indicate that
$\Gg_{\Omega, \ti K}$ is the Bruhat-Tits group scheme
over $\ti\O$ which is associated to $\Omega$ considered as a subset of $\B(G, \ti K)$.)

\begin{prop}\label{edix}  As above, suppose that $\ti K/K$ is tamely ramified and Galois with 
 Galois group $\Gamma$. Then  we have
$$
({\rm Res}_{\ti\O/\O}(\Gg_{\Omega, \ti K}))^\Gamma\simeq\Gg_{\Omega,K},
$$
and in particular, a closed group scheme immersion
\begin{equation}\label{BCimmersion}
\Gg_{\Omega,K}\hookrightarrow  {\rm Res}_{\ti\O/\O}(\Gg_{\Omega, \ti K}).
\end{equation}
\end{prop}

\begin{proof} Since $(G(\ti K)_\Omega)^\Gamma=(G(\ti K)^\Gamma)_\Omega=G(K)_\Omega$ this follows from \cite[(1.7.6)]{BTII} using  that, by \cite{EdixhTame},
the group scheme on the left hand side is smooth over $\O$.  
\end{proof}
\smallskip
 
By our construction of $\iota$
we have a commutative diagram where the 
horizontal arrows are equivariant toral embeddings
\begin{equation}\label{1311dia}
\begin{matrix}
\B(G,\ti K) &\xrightarrow{\ } & \prod_j \B({\rm Res}_{K_{j1}/K}\GL(V_j)_{D_j}\otimes_{K}\tilde K, \ti K)\\
\uparrow &&\uparrow\\
\B(G,K)  &\xrightarrow{\ } & \prod_j\B({\rm Res}_{K_{j1}/K}\GL(V_j)_{D_j}, K) &\to\ \  \B(\GL(V), K),
\end{matrix}\ 
\end{equation}
and the vertical arrows are the natural embeddings. Here $K_{j1}$ is the field obtained from $V_j$ and the representation $\rho_j: G\to \GL(V_j)$ over $K$. 
By our construction, the top horizontal arrow is the $G(\ti K)$-map of buildings $(\iota_{j,\sigma})_{j,\sigma}$ that
corresponds to 
$$
\rho'_{\ti K}: G_{\ti K}\to  \prod\nolimits_{j}\prod\nolimits_{\sigma} \GL(V_j\otimes_{K_{j1} }\ti K)_{D_{j}\otimes_{K_{j1}}\ti K}\cong\prod\nolimits_{j}\prod\nolimits_\sigma \GL(V(\lambda_{j1}) \otimes_{\Q_p}\ti K).
$$
(Here $\sigma$ runs over all $K$-embeddings $K_{j1}\to \ti K$ and $\rho'_{\ti K}$ can be identified with the product over $j$ of the base changes of (\ref{restrictionScalarsB}) from $K$ to $\ti K$.)
Each factor corresponds to a minuscule irreducible $\ti K$-representation of the split group $G_{\ti K}$ and 
we can see that $\rho'_{\ti K}$ is faithful. The result in the split case implies that 
  $\rho'_{\ti K}$ induces a closed immersion
\begin{equation}
\rho'_{\ti K}: \Gg_{x,\ti K}\hookrightarrow \prod\nolimits_j\prod\nolimits_\sigma \mathcal {GL}(V(\lambda_j)\otimes_{\Q_p}\ti K) _{\iota_{j,\sigma}(x)}
\end{equation}
of smooth group schemes over $\O_{\ti K}$. Now, as in (\ref{BCimmersion}), we also have a closed immersion
$$
 {\rm Res}_{\O_{j1}/\O} ((\mathcal {GL}(V_j)_{D_j})_{\iota_j(x)}) \hookrightarrow \prod\nolimits_\sigma
{\rm Res}_{\ti\O/\O}(\mathcal {GL}(V(\lambda_j)\otimes_{\Q_p}\ti K) _{\iota_{j,\sigma}(x)}).
$$
Since, again by (\ref{BCimmersion}), $\Gg_{x}=\Gg_{x, K}\to {\rm Res}_{\ti\O/\O}\Gg_{x, \ti K}$ is a closed immersion,
we deduce that 
$$
\Gg_{x,  K}\to \prod\nolimits_j {\rm Res}_{\O_{j1}/\O} ((\mathcal {GL}(V_j)_{D_j})_{\iota_j(x)}) 
$$
is a closed immersion. The result now follows using  \cite[3.5, 3.9]{BTclassI}; this implies that the natural
$$
 {\rm Res}_{\O_{j1}/\O} ((\mathcal {GL}(V_j)_{D_j})_{\iota_j(x)}) \to \mathcal{GL}(V_j)_{\iota_j(x)}
$$
corresponding to restriction of scalars $$
{\rm Res}_{K_{j1}/K}( \GL(V_j)_{D_j}) \to {\rm Res}_{K_{j1}/K}(\GL(V_j)_{K_{j1}})\to \GL(V_j)$$
(cf.  (\ref{restrictionScalars})) is a closed immersion.
 \end{proof}
 
 \begin{para}\label{Laurentfield2}
 We can see that the statement  of Proposition \ref{miniscule} continues to hold, with the same proof, in the equicharacteristic case $K=k\llps \pi\lrps$
 provided we consider $\rho: G\to \GL(V)$ and a corresponding embedding $\iota: \B(G, K^\ur)\to \B(\GL(V), K^\ur)$ which are given as in \ref{Laurentfield1}. 
 \end{para}

\subsection{Extending torsors}

\begin{para}\label{tann}

We continue to use the notation introduced above. 
Let $\O_\E$ be the $p$-adic completion of $W\lps u \rps_{(p)}$; this is a henselian 
discrete valuation ring with residue field $k\llps u \lrps$ and
fraction field  $\E=\O_\E[1/p]=K_0\{\{u\}\}$. 
For simplicity, we set
$ D=\Spec(W\lps u \rps)$, $D^\x=D-\{(u,p)\}$ and  also 
$D[1/p]=D^\x[1/p]=\Spec(W\lps u \rps[1/p])$.
 \end{para}

\begin{para} Suppose $G$ is a connected reductive group over $K_0,$ and let $\Gg$ be a parahoric  Bruhat-Tits (smooth) group scheme over $W$
for $G$; \emph{i.e.} $\Gg=\Gg_x^\circ$ for a point $x$ in the Bruhat-Tits building $\calB(G, K_0)$
and $G=\Gg[1/p]=\Gg\otimes_W K_0$.

As above, we assume that $G$ splits over a tamely ramified extension of $K_0$.
We also assume that $G$ has no factors of type $E_8$. (For our purposes, this is an acceptable
assumption since it is satisfied for the reductive groups corresponding to Shimura varieties.)
The main result of this section is the proof of the following:
\end{para}

\begin{prop}\label{trivialityoftorsors}
Under the above assumptions, each $\Gg$-torsor over $D^\x$ is trivial, \emph{i.e.} we have $\rH^1(D^\x, \Gg)=(1)$.
\end{prop}

\begin{Remark}
{\rm When $x$ is hyperspecial (so in particular $G$ is quasi-split and
  split over an unramified extension of $K_0,$ this follows from \cite{ColliotSansuc} as shown in \cite{KisinJAMS}.
See also Remark \ref{alterKeyLemma} below.

Before giving the proof of the Proposition, we need the following two Lemmas.
In  the arguments below, all the cohomology groups/sets are for
the fppf topology. However, since all the coefficients here are given by smooth group
schemes we could also use the \'etale
topology with no change. 
}
\end{Remark}

\begin{lemma}\label{cohinducedtori}
Let $Q$ be an induced torus over $K_0,$ and $\mcQ^\circ$ its connected N\'eron model over $\O_{K_0}.$ 
Then we have 
$$
\rH^1(D[1/p], Q) = \{1\}$$ and 
$$\text{\rm Im}(\rH^2(D^\times, \mcQ^\circ) \to \rH^2(D[1/p], Q)) = \{1\}.$$ 
\end{lemma}
\begin{proof} By assumption $Q$ is a product of tori of the form $\Res_{K/K_0} \Gm,$ 
for $K/K_0$ a finite extension, and we may assume $Q = \Res_{K/K_0} \Gm.$ 
For the first claim we have 
$$ \rH^1(D[1/p], Q) = \rH^1(\Spec \O_K\lps u \rps[1/p], \Gm) = \{1\}$$
as $\O_K\lps u \rps[1/p]$ is a UFD. 

For the second claim, note that we have a tautological character $Q|_K \rightarrow \Gm,$ 
which extends to a map of smooth groups over $\O_K,$ $\mcQ^\circ|_{\O_K} \rightarrow \Gm$ 
since $\Gm$ is the connected N\'eron model of its generic fibre. Finally, we obtain 
a map $\mcQ^\circ \rightarrow \Res_{\O_K/\O_{K_0}} \Gm,$ and it suffices to show that 
$\rH^2(D^\times, \Res_{\O_K/\O_{K_0}} \Gm) = \{1\},$ or equivalently 
$\rH^2(D^\times_{\O_K}, \Gm) = \{1\}.$

By purity of the Brauer group (e.g \cite[Part II, Prop. 2.3]{DixExposesBG} or \cite[part III, Thm 6.1 (b)]{DixExposesBG}, 
and the fact that $\O_K\lps u \rps$ is strictly henselian, we have 
$$\rH^2(D^\times_{\O_K}, \Gm) = \rH^2(D_{\O_K}, \Gm) = \rH^2(\Gal(\bar k/k), \Gm).$$ 
Our assumptions on $k$ imply the final group is trivial.
\end{proof}

\begin{lemma}\label{trivialityoftorsorssimplyconnected}
Suppose that $G$ is quasi-split, semi-simple and simply connected with no factors of type $E_8$. Then $\rH^1(D[1/p], G) = \{1\}.$
\end{lemma} 

\begin{proof}
Note that $D[1/p]=\Spec(W\lps u \rps[1/p])$ is regular Noetherian of dimension $1$.
Set $\calK={\rm Frac}(W\lps u \rps)$. This is a field of cohomological 
dimension $2$: Indeed, if $\ell\neq p$, the $\ell$-cohomological dimension  (see \cite{SerreGaloisCoh})
${\rm cd}_\ell(\calK)$ of 
$\calK$ is $2$ by results of Gabber.
(This verifies a conjecture of M. Artin, see
[SGA4 X] or \cite[Exp. XVIII]{GabberPurity}.) On the other hand, ${\rm cd}_p(\calK)=2$ 
was shown by Kato, see \cite{Kuzumaki}, or \cite{GabberOrgo} for a
more general result. We now use results on Serre's conjecture II:
By \cite{GilleSerreIICompositio},  if $H$ is a semi-simple, simply connected quasi-split reductive group
 with no $E_8$ factors, then  $\rH^1(\calK, H)=(1)$. (This uses earlier more general results for groups of classical type,
by Bayer-Fluckinger--Parimala, see \cite{BFluckParimalaInvent}, \cite{BFluckParimalaAnnals}.
See also \cite{GilleSerreIIsurvey} for a survey.) Therefore, $\rH^1(\calK, G)=\{1\}$.

Now let $B \subset G$ be a Borel and $T \subset G$ a maximal torus.
Let $J\to  D[1/p]$ be a $G$-torsor. 
Since $\rH^1(\calK, G)=(1)$, $J$ has a section defined on a non-empty open subscheme $U$ of $D[1/p]$. 
This gives a section of the associated $G/B$-bundle $J\times^{G}G/B\to U$. 
Since $D[1/p]$ is affine, regular of Krull dimension $1$ and $J\times^{G}G/B\to U$ is proper,   this section extends to 
a section defined over $D[1/p]$. This defines a reduction of the structure group
of $J$ from $G$ to $B$, \emph{i.e.} a $B$-torsor $J'\to D[1/p]$ so that $J\simeq J'\times^B G.$
Now notice that all $B$-torsors over $D[1/p]$ are trivial. Indeed, $B$ is a successive
extension of the maximal torus $T$ and unipotent groups of the form ${\rm Res}_{K'/K_0}{\mathbb G}_a.$
By an argument as in Lemma \ref{cohinducedtori}, all torsors for these unipotent groups are trivial 
Similarly, since the torus $T$ is induced, 
$\rH^1(D[1/p], T)=\{1\}$ by Lemma \ref{cohinducedtori}. 
It follows that the $G$-torsor $J$ is trivial; hence $\rH^1(D[1/p], G)=\{1\}.$
\end{proof}

\begin{proof}[Proof of Proposition \ref{trivialityoftorsors}] 

Suppose  that $\I \to D^\x$ is a $\Gg$-torsor.  
We begin by considering the case when $k$ is algebraically
closed. Then, by Steinberg's theorem $G$ is quasi-split, \emph{i.e.}
it contains a Borel subgroup $B$ defined over $K_0$. 
The variety of Borel subgroups $G/B$ is projective over $K_0$.

\smallskip

{\sl Step 1.} {\sl The base change $\I_{\O_\E}\to \Spec(\O_\E)$
is   a trivial $\Gg\otimes_{\Z_p}\O_\E$-torsor. }
\smallskip

Indeed, the fibre $\I_{k\llps u \lrps}\to \Spec(k\llps u \lrps)$ is a trivial
$\Gg\otimes_{W}k\llps u \lrps$-torsor: This last group
is an extension of a connected reductive group 
by a unipotent group both defined over $k$. 
(Recall here that the special fibre of $\Gg=\Gg_x^\circ$ is connected.)
Since the cohomological dimension of $k\llps u \lrps$ is $1$ (\cite[II 3.3]{SerreGaloisCoh}),
the result then follows by Steinberg's
theorem (\cite[III. 2.3, Remark 1)]{SerreGaloisCoh})
and the fact that ${\rm H}^1(k\llps u \lrps, {\mathbb G}_a)=\{0\}$.
Since $\O_\E$ is henselian with residue field $k\llps u \lrps$
and $\I_{\O_\E}\to \Spec(\O_\E)$ is smooth, a section of $\I_{\O_\E}$ over $k\llps u \lrps$ lifts to a section over $\O_\E$.
 \smallskip
 
{\sl Step 2.} {\sl The base change $\I[1/p]\to D[1/p]$ is a trivial $G$-torsor.}
\smallskip

 By \cite{ColliotFlasques}, there is a flasque resolution
\begin{equation}\label{exactZ}
1\to Z\to \ti G\to G\to 1
\end{equation}
with $Z$ a flasque (central) torus, 
 $\ti G^{\rm der}$ semi-simple simply connected and 
at the same time
\begin{equation}\label{exactQ}
1\to \ti G^{\rm der}\to \ti G\to Q\to 1
\end{equation}
with $Q$ an induced torus (\emph{i.e.} $Q\simeq \prod_{i} {\rm Res}_{K_i/K_0}\Gm$
where $K_i/K_0$ are finite tamely ramified extensions). Recall that
a torus $Z$ over $K_0$ is called flasque if for every open subgroup $I'\subset I={\rm Gal}(\bar K_0/K_0)$
we have $\rH^1(I', \xcoch(Z))=0$.
Since $Z$ is central, it is contained in the centralizer $\ti T=Z(\ti S)$ of any maximal 
$K_0$-split torus $\ti S$ of $\ti G$. Actually, in this case we see (loc. cit.) that these centralizer maximal tori of 
both $\ti G^{\rm der}$ and $\ti G$ are induced. (This will be used later). 
Since we are assuming that $G$ splits after a tamely ramified (and hence
cyclic) extension of $K_0$ the flasque torus $Z$ is also a direct summand of an induced torus
(see  \cite[Prop. 1]{ColliotThInvent}; this uses a result of Endo-Miyata on permutation Galois modules)
so we have $Z\times_{K_0} Z'\simeq Q'$ for some torus $Z',$ and with $Q'$ an induced torus.

By Lemmas \ref{cohinducedtori} and \ref{trivialityoftorsorssimplyconnected}, we have $\rH^1(D[1/p], \tilde G) = \{1\}.$
Hence it suffices to show that the image of $\I$ in $\rH^2(D[1/p], Z)$ is trivial. By Proposition \ref{cIsoProp}, there is 
 an exact sequence of smooth group schemes over $W$
\begin{equation}\label{exactPtilde}
1\to {\calZ}^\circ\to \ti\Gg\to \Gg\to 1
\end{equation}
where ${\calZ}$ is the finite type Neron model (\cite[\S 4.4]{BTII}) of the torus $Z$ and $\ti\Gg=\ti\Gg_x^\0$ is the parahoric  group
scheme for the group $\ti G$ that corresponds to  $\bar x\in \calB(G^{\rm ad}, K_0)=\calB(\ti G^{\rm ad}, K_0)$.
Hence the image of $\I$ in $\rH^2(D[1/p], Z)$ is its image under the composite map 
$$ \rH^1(D^\times, \Gg) \to \rH^2(D^\times, \cal Z^\circ) \to \rH^2(D[1/p], Z).$$
This is trivial by Lemma \ref{cohinducedtori}, since $Z$ is a direct summand of an induced torus.
\smallskip
  
{\sl Step 3.}  {\sl We have    
$
G(W\lps u \rps[1/p])\backslash G(\E)/\Gg(\O_\E)=\{1\}
$.}
\smallskip

Assuming this, let us show that the $\Gg$-torsor $\I$ is trivial.
Indeed, from Steps 1 and 2  we have sections
$a_p$ and $a[1/p]$ of the torsor $\I\to D^\x$ 
over  $\O_\E$ and $D[1/p]$ respectively. Consider $g\cdot a_p=a[1/p]$, $g\in G(\E)$,
with both sections restricted on $\Spec(\E)$. The triviality of the double cosets above implies that we can modify these sections to achieve $g=1$. Now observe that
$D[1/p]\sqcup \Spec(\O_\E)\to D^\x$ is a 
cover in the fpqc topology; by Grothendieck's theorem on 
full faithfulness of fpqc descent data (\emph{e.g.} \cite[Chapter 6, Thm. 6 (a)]{BLRNeronModels}) we can conclude that the torsor $\I$ is trivial over $D^\x$, cf. \cite[Appendix]{GilleTorsors}.

We now turn to showing the triviality of the double cosets above.
Pick an alcove in the Bruhat-Tits building of $G$ over $K_0$ whose closure contains
$x$. The connected stabilizer of any point in that alcove  is an Iwahori subgroup $\calI$
for which $\calI(\O_\E)\subset \Gg(\O_\E)$. This shows that it is enough to assume
that  $\Gg=\calI$, \emph{i.e.} that $\Gg$ is an 
Iwahori  group scheme. 

Now denote by $ W_a$
the   affine Weyl group of $G(K_0)$, which is generated by the simple reflections
$s_i$, $i=1,\ldots , m,$ along the walls of the alcove. 
The reflections $s_i$ are all represented by elements of $G(K_0)$. There are  corresponding parahorics $\Gg_i$ such that $\calI(W)=\Gg(W)\subset \Gg_i(W)$  with 
$\Gg_{i}(W)=\Gg(W)\sqcup \Gg(W)s_i\Gg(W)$.
Then also 
$
\Gg_{i}(\O_\E)=\Gg(\O_\E)\sqcup \Gg(\O_\E)s_i\Gg(\O_\E)
$.
The cosets  $\Gg_i(\O_\E)/\Gg(\O_\E)$ are
parametrized by the $k\llps u \lrps$-valued points of the projective line ${\mathbb P}^1$:
To be more precise, let $G_i$ be the maximal reductive quotient of 
the special fibre $\overline\Gg_i=\Gg_i\otimes_W k$.  The
derived group $G^{\rm red}_i$ of $G_i$ is ${\rm SL}_2$ or ${\rm PSL}_2$.
Since $\Gg_i$ is smooth and $\O_\E$ is $p$-adically 
complete, reduction modulo $p$
gives surjective homomorphisms
$$
\Gg_i(\O_\E)\to \overline \Gg_i(k\llps u\lrps).
$$
Composing this with the surjection $\overline \Gg_i(k\llps u \lrps)\to G_i(k\llps u \lrps)$
gives surjective  homomorphisms
$$
 \Gg_i(\O_\E)\xrightarrow{q_i}G_i(k\llps u \lrps)\to 1.
$$
Now  the group $\Gg(\O_\E)$ is the inverse image 
by $q_i$ of a Borel subgroup $B_i(k\llps u \lrps)\subset G_i(k\llps u \lrps)$. This gives 
\begin{equation}\label{isoku}
\Gg_i(\O_\E)/\Gg(\O_\E)\xrightarrow{\sim} G_i(k\llps u \lrps)/B_i(k\llps u \lrps)\simeq {\mathbb P}^1(k\llps u \lrps).
\end{equation}
Similarly, since $\Gg_i$ is smooth over $W$, by Hensel's lemma, reduction modulo $p$ gives a surjective 
homomorphism $\Gg_i(W\lps u \rps)\to \overline \Gg_i(k\lps u \rps)$. As before, we obtain  similar surjective  homomorphisms
$$
 \Gg_i(W\lps u \rps)\xrightarrow{q_i}G_i(k\lps u \rps)\to 1
$$
which give 
\begin{equation}\label{isoW}
\Gg_i(W\lps u \rps)/\Gg(W\lps u \rps)\xrightarrow{\sim} {\mathbb P}^1(k\lps u \rps).
\end{equation}
These isomorphisms (\ref{isoku}) and (\ref{isoW}) are compatible via $W\lps u \rps\to \O_\E$,
which modulo $p$ induces $k\lps u \rps\to k\llps u \lrps$.
Since ${\mathbb P}^1$ is proper, 
${\mathbb P}^1(k\lps u \rps)={\mathbb P}^1(k\llps u \lrps).$
 Hence  
if $x_i\in \Gg_i(\O_\E)/\Gg(\O_\E)$, we can find $g_i\in \Gg_i(W\lps u \rps)\subset G(W\lps u \rps[1/p])$
so that $x_i= g_i \cdot \Gg(\O_\E)$. 
Hence 
\begin{equation}\label{approx}
\Gg_i(\O_\E)=\Gg_i(W\lps u \rps)\cdot \Gg(\O_\E).
\end{equation}

Now, for each $n\geq 1,$ and $i_1, \dots, i_n$ integers in $[1,m],$ consider the map
$$
\Gg_{i_1}(\O_\E)\times   \cdots \times \Gg_{i_n}(\O_\E)\to G(\E)/\Gg(\O_\E), \quad 
(y_1,\ldots , y_n)\mapsto y_1 \cdots y_n\cdot \Gg(\O_\E)
$$
which factors via the quotient by the action of $\Gg(\O_\E)^n$
given by 
$$
 (y_1,\ldots , y_n)\cdot (p_1,\ldots, p_n)=(y_1p_1, p_1^{-1}y_2p_2,\ldots, p_{n-1}^{-1}y_np_n).
$$
Start with $(y_1,\ldots , y_n)$ as above.
By (\ref{approx}), there is $p_1\in \Gg(\O_\E)$ so that $y_1p_1=g_1\in \Gg_{i_1}(W\lps u \rps)$.
Consider $p^{-1}_1y_2\in \Gg(\O_\E)\Gg_{i_2}(\O_\E)\subset \Gg_{i_2}(\O_\E)$.
Applying (\ref{approx}) again, we see that there is $p_2\in \Gg(\O_\E)$
so that $p^{-1}_1y_2p_2=g_2\in \Gg_{i_2}(W\lps u \rps)$. Continuing, 
we find $(g_1,\ldots, g_n)$ and $(p_1,\ldots, p_n)\in \Gg(\O_\E)^n$
with $g_{k}\in \Gg_{i_k}(W\lps u \rps)\subset G(W\lps u \rps[1/p])$
and
$
(y_1,\ldots , y_n)\cdot (p_1,\ldots, p_n)=(g_1,\ldots, g_n).
$
This gives
$$
y_1 y_2\cdots y_n\cdot \Gg(\O_\E)=g_1 g_2\cdots g_n\cdot 
\Gg(\O_\E).
$$

Denote by $G(\E)^1$ the subgroup of $G(\E)$ generated by 
all the parahoric subgroups $\Gg_{i}(\O_\E).$ The above calculation 
implies that the image of $G(W\lps u \rps[1/p])$ in $G(\E)/\Gg(\O_\E)$ 
contains $G(\E)^1/\Gg(\O_\E).$

The group  $G(\E)^1$ coincides with the 
subgroup generated by all parahoric subgroups of $G(\E)$ considered
in \cite[5.2.11]{BTII}. 
As above, choose a maximal split torus $S\subset G$ whose apartment contains $x,$ 
and let $\calT \subset \calG_x$ be the closure of its centralizer $T=Z_G(S)$. 
By \emph{loc. cit.} 5.2.4, $G(\E)^1$ is also the subgroup
generated by $\calT^\circ(\O_\E)$ and the $\E$-valued points
of the root subgroups of $G$. (Notice that
we have ${\mathcal T}^\circ(\O_\E)\subset \Gg(\O_\E)\subset G(\E)^1$.)
We now  have
\begin{equation}\label{g1}
G(\E)=T(\E)\cdot G(\E)^1.
\end{equation}
Now consider the quotient $T(\E)/{\mathcal T}^\circ(\O_\E)$.
The natural homomorphism $T(K_0)\to T(\E)$
gives a surjection $T(K_0)/\calT^\circ(W)\to T(\E)/\calT^\circ(\O_\E)$.
When the torus $T$ is induced this follows from \cite[4.4.14]{BTII}. In
our more general case, we have, as in Step 2, $T=\tilde T/Z$ with $\ti T$ induced and $Z$ 
flasque. As above, since
$Z$ is a direct summand of an induced torus,   
we have $\rH^1(\E, Z)=(1)$. Hence, $\ti T(\E)\to T(\E)$ is surjective and the 
desired surjectivity above then follows from the corresponding property for $\ti T.$ 
Therefore, (\ref{g1}) gives
\begin{equation}\label{tg1}
G(\E)=T(K_0)\cdot G(\E)^1.
\end{equation}
Since $T(K_0) \subset G(W\lps u \rps[1/p]),$ this completes the
proof of the proposition in the case when $k$ is algebraically closed.
\smallskip

{\sl Step 4.}  {\sl The proposition holds for any $k$.} 

\smallskip

Write $D_{\O_L} = \Spec \O_L\lps u \rps$ 
and $D^\times_{\O_L} \subset D_{\O_L}$ the complement of the closed point. 
Denote by $f: D_{\O_L} \rightarrow D$ the projection.
Then $f^*(\I)$ is equipped with a $\Gg$-equivariant descent datum for the morphism $f.$ By what we have already 
seen $f^*(\I)$ is a trivial $\Gg$-torsor over $D^\times_{\O_L},$ and we may consider this descent datum 
as a descent datum on $\Gg\times_{D^\times}D^\times_{\O_L}.$ Since $\Gg$ is affine this extends to an effective, $\Gg$-equivariant descent 
datum on $\Gg\times_DD_{\O_L},$ which produces a $\Gg$-torsor on $\tilde \I$ over $D$ extending $\I.$ 

Finally $\tilde \I$ has a section over the closed point of $D$ by Lang's lemma, and hence over $D$ by smoothness. 
It follows that $\tilde \I$ and hence $\I$ is a trivial $\Gg$-torsor.
\end{proof}

\begin{Remark} \label{alterKeyLemma}
{\rm Under the additional hypothesis that $G$ is split over $K_0$ 
and that the  subgroup $\Gg(W)=\Gg_x(W)$ is contained in a hyperspecial
subgroup $G_W(W)$ we can give a quicker proof of Proposition \ref{trivialityoftorsors}.
(Notice then that by \cite{HainesRapoportAppendix}, $\Gg^\circ_x=\Gg_x$,
since the Kottwitz invariant homomorphism vanishes on $\Gg_x(W)\subset G_W(W)$.)
We sketch the argument below:

Recall that there is a representation $G_W\hookrightarrow {\GL_n}_{/W}$ which is a closed 
immersion such that the quotient ${\GL_n}_{/W}/G_W$ is an affine scheme
(\cite{ColliotSansuc}, \cite{KisinJAMS}). Under our assumption, there is a parabolic subgroup 
$Q \subset \bar G_W=G_W\otimes_W k$ such that $\Gg(W) \subset G_W(W)$ is the preimage of $Q(k) \subset G_W(k).$
In this case, $\Gg$ is given as the dilatation (\cite{BLRNeronModels}, \cite{YuModels}) of $G_W\to \Spec(W)$ along the subgroup $Q\subset G_W\otimes_Wk$ of its 
special fibre. 

We can now   write $Q$ as the scheme theoretic intersection of $G_W\otimes_Wk$
and a parabolic subgroup $Q'$ in ${\GL_n}_{/k}.$ The dilation of ${\GL_n}_{/W}\to \Spec(W)$ along $Q'$
is a parahoric subgroup ${\mathcal {GL}}_y$ scheme for $\GL_n$ which is given as the stabilizer of a 
corresponding lattice chain. We have a closed group scheme immersion $\Gg\hookrightarrow 
{\mathcal {GL}}_y$ such that the quotient ${\mathcal {GL}}_y/\Gg$ is affine: Indeed, 
the quotient ${\mathcal {GL}}_y/\Gg$ can be identified as the dilatation of the affine scheme 
${\GL_n}_{/W}/G_W$
along the closed subscheme $Q'/Q$ of its closed fibre. Such dilatations of affine schemes are
also affine. Now use, as in \cite{ColliotSansuc}, \cite{KisinJAMS}, the fact that any morphism $D^\x \to X$ with $X$ affine extends to $D\to X$ to reduce the proof  to the case that the group is $G=\GL_n$ and the parahoric subgroup $\Gg={\mathcal {GL}}_y$. 

When $\Gg={\mathcal {GL}}_y$,  a $\Gg$-torsor over a scheme $T$ is given (cf. \cite[Appendix to Chapter 3]{RapZinkBook}) by a periodic chain $(\calF_i, \phi_i)_i$ of locally free rank $n$ $\O_T$-modules $\calF_i$ with maps $\phi_i: \F_i\to \F_{i+1}$ such that, for all $i$, the quotients $\F_{i+1}/\phi_i(\calF_i)$ are locally free $\O_T/p\O_T$-modules of fixed rank $r_i$ (which depends on our choice of $y$). 
(By ``periodic" one means that there is $a\geq 1$ such that $\F_{i+a}=\F_i$ 
and the composition $\phi_{i+a-1}\cdots   \phi_{i+1}\cdot \phi_i$ is multiplication by $p$,
for all $i$.)
Since $D$ is regular Noetherian of dimension $2$ and $D-D^\x$ has codimension $2$, 
a periodic chain $(\F_i, \phi_i)$ over $D^\x$ uniquely extends to 
a periodic chain $(\ti \calF_i, \ti\phi_i)$ over $D$. If $(\F_i, \phi_i)$ satisfies the above condition on the quotients, then so does the extended chain $(\ti \calF_i, \ti\phi_i)$:
Indeed, by the above, the $W\lps u \rps$-modules $\ti\calF_i/\ti\phi_i(\ti\F_i)$ have projective dimension $1$, and are annihilated by $p$.  By the Auslander-Buchsbaum theorem 
$\ti\calF_i/\ti\phi_i(\ti\F_i)$  has only trivial $u$-torsion; therefore,   it is free over $W\lps u \rps/pW\lps u \rps=k\lps u \rps$. This establishes that the $\Gg$-torsor over $D^\x$ extends to a $\Gg$-torsor over $D$ which then has to be trivial as before.}
\end{Remark}

 \section{Local Models}

\subsection{The local models}
 
 \begin{para}\label{locmodsetup}
We now recall the definition of the local models from \cite[\S 7]{PaZhu}. 
We continue to use the notation of the previous section, but we 
 assume that  $K/\Q_p$ is a finite unramified extension of $\Q_p.$ 
 Suppose that $(G,  \{\mu\}, \eK)$ is a triple, with 
 \begin{itemize}
 \item $G$ a connected reductive group over $K$,
 
 \item $\{\mu\}$
 a conjugacy class of a geometric cocharacter $\mu: \Gm_{\ov\Q_p}\to G_{\ov\Q_p}$, and

 \item $\eK\subset G(K)$ a parahoric subgroup which is the connected stabilizer
 of the point $x\in\calB(G, K)$.
 \end{itemize}
 
 We assume that $G$ splits over a tame extension of $K$ and that $\mu$
 is minuscule.
\footnote{Recall that $\mu$ is  minuscule if $\langle a, \mu\rangle \in \{-1,0,1\}$ for every root $a$ of $G_{\ov\Q_p}$.}

 Suppose that $E\subset \ov\Q_p$ is the local reflex field, \emph{i.e.} the extension of $K$ which is the field of definition of the conjugacy class $\{\mu\}$. 
 
In \cite[\S 3]{PaZhu}, there is a construction of a smooth affine group scheme $\und \Gg$ over $\O[u]$ which specializes to the parahoric group scheme $\Gg:=\Gg_x^\circ$ over $\O$ after the base change $\O[u]\to \O$
 given by $u\mapsto p$ (\emph{loc. cit.} \S 4), and such that $\und G = \und \Gg|_{\O[u,u^{-1}]}$ is reductive. There is a corresponding ind-projective ind-scheme 
(the global affine Grassmannian)
${\rm Gr}_{\und \Gg, \AA^1}\to
\AA^1=\Spec(\O[u])$ (\emph{loc. cit.} \S 6). The   base change
${\rm Gr}_{\und \Gg, \AA^1}\times_{\AA^1} \Spec(K)$ under $\O[u]\to K$ given by $u\mapsto p$
can be identified with the affine Grassmannian ${\rm Gr}_{G, K}$ of $G$  over $K$.
(Recall that ${\rm Gr}_{G, K}$ represents the fpqc sheaf associated to the quotient $R\mapsto G(R\llps t\lrps/G(R\lps t \rps)$;
the identification is via $t=u-p$.)

The cocharacter $\mu$ defines a 
 projective homogeneous space $G_{\bar\Q_p}/P_{\mu^{-1} }$
 over $\bar\Q_p$. Here, $P_\nu$ denotes the parabolic subgroup 
 that corresponds to the coweight $\nu$; by definition, the Lie algebra ${\rm Lie}(P_\nu)$
 contains all the root subgroups $U_a$ for roots $a$ such that $a\cdot \nu:\Gm\to\Gm$ is a non-negative power of the identity character. Since the conjugacy class $\{\mu\}$ is defined over $E$ we can see that 
 this homogeneous space has a canonical model $X_\mu$ defined over $E$ 
(notice however, that $X_\mu$ might not have any $E$-rational point).  If $G$ is quasi-split, then 
 $\{\mu\}$ has a representative $\mu: \Gm_E\to G_E$ which is actually defined over $E$ \cite[1.1.3]{KottTwisted} 
 then we can write $X_\mu=G_E/P_{\mu^{-1}}$ which has an $E$-rational
 point.

Since $\mu$ is minuscule, the corresponding affine Schubert variety with $\ov\Q_p$-points $S_\mu(\ov\Q_p)=G(\ov\Q_p\lps t\rps )\mu(t)
G(\ov\Q_p\lps t\rps )/G(\ov\Q_p\lps t\rps)$ in the affine Grassmannian
${\rm Gr}_{G, K}\times_K  \ov\Q_p$ is closed, see \cite[p. 146]{PRS}. Our assumption that the conjugacy class $\{\mu\}$ is defined over $E$  
implies that $S_\mu(\ov\Q_p)$
 is ${\rm Gal}(\ov\Q_p/E)$-equivariant and so it corresponds to a
closed subvariety $S_\mu$ of the ind-projective ${\rm Gr}_{G, K}\times_KE$. The natural left action of $G(\ov\Q_p\lps t\rps )$ on $S_\mu(\ov\Q_p)$ is transitive and the stabilizer  
of $\mu(t)$ is $H_\mu:=G(\ov\Q_p\lps t\rps ) \cap \mu(t)G(\ov\Q_p\lps t\rps ) \mu(t)^{-1}$. Let $T$ be a maximal torus of $G_{\ov\Q_p}$
which contains the image of $\mu$. Then $H_\mu $ contains $T(\ov\Q_p\lps t\rps)$.
Since $\mu$ is minuscule, we can see that $H_\mu$ contains the kernel of the canonical homomorphism $ U_a(\ov\Q_p\lps t\rps)\to U_a(\ov\Q_p) $, 
for all roots $a$ of $G_{\ov\Q_p}$. We conclude   that $H_\mu$ contains the kernel of   $G(\ov\Q_p\lps t\rps ) \to G(\ov\Q_p)$;
by definition, 
 the image of $H_\mu$ in $G(\ov\Q_p)$ is equal to $P_{\mu^{-1}}(\ov\Q_p)$.
Hence, $S_\mu$ can be $G_E$-equivariantly identified with
   $X_\mu$.

The local model ${\rm M}^{\rm loc}_{G, \{\mu\}, x}:={\rm M}^{\rm loc}(G, \{\mu\})_x$  is by definition the Zariski closure of $X_\mu\subset {\rm Gr}_{G, K}\times_KE$
in ${\rm Gr}_{\und G, \AA^1}\times_{\AA^1} \Spec(\O_E)$ where the base change $\O[u]\to \O_E$ is given by $u\mapsto p$. By its construction, ${\rm M}^{\rm loc}_{G, \{\mu\}, x}$ is a projective flat scheme over $\Spec(\O_E)$
which admits an action of the group scheme $\Gg\times_{\O}\O_E$. 
We recall:
\end{para}

\begin{thm}(\cite[Theorem 9.1]{PaZhu})\label{normalThm}
Suppose in addition that $p$ does not divide the order of the
(algebraic) fundamental group 
${\rm \pi}_1(G^{\rm der})$
of the derived group of $G$. Then the scheme
${\rm M}^{\rm loc}_{G, \{\mu\}, x}$
is normal. The geometric special fibre of ${\rm M}^{\rm loc}_{G, \{\mu\}, x}$
is reduced and
admits a stratification with locally closed smooth strata; 
the closure of each stratum is normal and Cohen-Macaulay.\endproof
\end{thm}

\begin{cor}\label{normalbcCor}
Under the above assumptions,
the base change ${\rm M}^{\rm loc}_{G, \{\mu\}, x}\otimes_{\O_E}\O_{L}$ is  normal,
 for every finite extension $L/E$.
\end{cor}

\begin{proof}
Using Theorem \ref{normalThm}, we see that the special fibre of ${\rm M}^{\rm loc}_{G, \{\mu\}, x}\otimes_{\O_E}\O_{L}$ is reduced. The result then 
follows as in \cite[Prop. 9.2]{PaZhu}. 
\end{proof}
\smallskip

 \begin{para}\label{quasisplitconstruction} 
 For simplicity, we will often write $\rM^{\rm loc}_G$ instead of ${\rm M}^{\rm loc}_{G, \{\mu\}, x}$, when $\{\mu\}$ and $x$ 
 are understood.
When the data $(G, \{\mu\}, \eK)$  are obtained, as in the next chapters, from global Shimura data $(G,  X, \prod_l\eK_l)$
 after the choice of a prime $v|p$ of the reflex field $E(G, X)$, 
 we will often write $\rM^{\loc}_{G, X}$ instead of $\rM^{\loc}_{G, \{\mu\}, x}$. 
 In particular, in this we take $\mu$ to be in the conjugacy class of $\mu_h$ for $h\in X$.
 
We now recall some points of 
the construction of the group schemes $\und G$ and $\und \Gg$.
We will only need these details when the reductive group $G$ is quasi-split 
over $K$; then this construction   is somewhat
more straightforward and
proceeds as follows. (Notice that Steinberg's theorem 
 implies that
 we can make sure that this assumption is always satisfied
 after enlarging the unramified extension $K/\Q_p$.)

Choose a  maximal $K$-split torus $S$ of $G.$ Since $G$ is quasi-split,
the centralizer $T=Z_G(S)$
is a maximal torus of $G$. Also choose a Borel subgroup $B$ of $G$ defined over $K$ 
that contains $S$ and consider the corresponding 
based root datum ${\mathcal R}^+:=(\xcoch(T), \Delta, \xch(T), \Delta^+)$ where $\Delta\subset \Phi$ is the set of simple
roots that corresponds to $B$ in the root system $\Phi=\Phi(G, T)$. Denote again by $H$ the  
split Chevalley form of $G$ over $\Z_p,$ and choose a pinning $(T_H, B_H, \und e)$ of $H$ over $\Z_p$.
The corresponding based root datum of $H$ agrees with ${\mathcal R}^+$. 
Set $\Xi :=\Xi_H={\rm Aut}({\mathcal R}^+)$. 

The quasi-split group $G$ over $K$
is described by a $\Xi_H$-torsor over $K$; this splits over a tame finite extension 
$K\subset \ti K\subset \ov \Q_p$ and can thus by described via a 
group homomorphism $\xi: {\rm Gal}(\ov\Q_p/K)\to \Xi $ that factors through $\Gamma_K={\rm Gal}(\ti K/K)$. 
As explained in \cite[3.2, 3.3]{PaZhu}, $\xi$ corresponds to 
$\und \xi: \pi_1(\Spec(\O[u^{\pm 1}]), \Spec(\ov\Q_p))\to \Xi$, with $\O[u^{\pm 1}]\to \ov\Q_p$ given by $u\mapsto p$,
\emph{i.e.} to a $\Xi$-torsor over $\O[u^{\pm 1}]$. (This specializes to the above $\Xi$-torsor 
over $K$ after $u\mapsto p$.) This  now gives 
 a quasi-split reductive group scheme $\und G$ over $\O[u^{\pm 1}]$ which can be described explicitly as follows: 
  Our choice of pinning $( T_H, B_H, e)$ of $H$ identifies
$\Xi $   with the group of automorphisms of $H$ that respect the pinning. Now  
there is an $K$-isomorphism
\begin{equation}
\psi: G\xrightarrow{\sim} ({\rm Res}_{\ti K/K}(H\otimes_\O{  \ti K}))^{\Gamma}
\end{equation}
where $\gamma\in \Gamma$ acts on the right hand side via $\xi(\gamma)\otimes\gamma$.
Set 
\begin{equation}
\und G=({\rm Res}_{\O_0[v^{\pm 1}]/\O[u^{\pm 1}]}(H\otimes_{\O[u^{\pm 1}]} \O_0[v^{\pm 1}])^{\Gamma}
\end{equation}
where $\O_0[v^{\pm 1}]/\O[u^{\pm 1}]$ is the $\Gamma$-cover  
which is described in \cite[3.2]{PaZhu} and which  specializes to $\ti K/K$ after base changing by 
$u\mapsto p$. 

Now for $x$ in the building $\B(G, K)$  pick the torus $S$ so that $x$ is in the apartment $A(G, S, K)$ of $S$;
the construction in \cite[Theorem 4.1]{PaZhu} gives a smooth affine group scheme $\und \Gg$ over $\O[u]$ that extends $\und G$ and 
specializes to $\Gg$ after base changing by $\O[u]\to K$, $u\mapsto p$. 
Let $\kappa$ denote either $K$ or $k$. Then, \cite[4.1]{PaZhu} provides
an identification of the apartment $A(G, S, K)$ in $\B(G, K)$ with an apartment in the building ${\mathcal B}( \und G\otimes_{\O[u^{\pm 1}]}\kappa\llps u \lrps,  \kappa\llps u \lrps)$
of the group $\und G_{\kappa\llps u \lrps}:=\und G\otimes_{\O[u^{\pm 1}]}\kappa\llps u \lrps$;
here $\kappa\llps u \lrps$ is considered as a discretely valued field with uniformizer $u$ and residue field $\kappa$. Then   $\und \Gg\otimes_{\O[u]}\kappa\lps u \rps$ is the parahoric 
group scheme over $\kappa\lps u \rps$ which is the connected stabilizer of the point $x_{\kappa\llps u \lrps}$  corresponding to $x$
under this identification.
  \end{para}

\subsection{Local models and central extensions.}
\label{locmodcentralPar}

\begin{para}
The results of this paragraph will be used only in \S \ref{subsec:AbelianType}. 
We start with the following:
\begin{prop}\label{extendcisoProp} Suppose that $\al: G_1\to  G_2$ is a central extension
of reductive groups over $\Q_p$ and let $x_1\in \B(G_1, \Q_p)$, $ x_2=\al_*(x_1)\in \B( G_2, \Q_p)$.
Assume that $G_1$, $G_2$ split over a tamely ramified extension of $\Q_p$ and
denote by $\Gg_i$, $i=1, 2$, the corresponding parahoric group scheme
$\Gg^\0_{x_i}$ over $\Spec(\Z_p)$.
The group scheme homomorphism
$\al: \Gg_1\to  \Gg_2$ extends to a group scheme homomorphism  $\und\al:   \und\Gg_1 \to \und\Gg_2
 $
  over $X=\Spec(\Z_p[u])$. This gives $\und\al_*: {\rm Gr}_{\und\Gg_1, X}\to 
{\rm Gr}_{ \und\Gg_2 , X}$ and by specializing at $u=p$, we obtain
a morphism
$
\al_*: {\rm Gr}_{ \und\Gg_1 , \Z_p}\to {\rm Gr}_{ \und{ \Gg}_2 , \Z_p}.
$

\end{prop} 
 
 \begin{proof} We will use the notations and constructions of \cite[\S 2, 3, 4]{PaZhu}.  
 Suppose that $H_i$ are the split forms of $G_i$ over $\Z_p$; we can choose pinnings $(T_i, B_i, \und e_i)$
 of $H_i$ and a central isogeny $\beta: H_1\to H_2$ that respects the pinnings in the sense that we have
 $\beta(T_1)\subset T_2$, $\beta(B_1)\subset B_2$, $\beta(\und e_1)=\und e_2$. Let $Z_0\subset T_1$ be the kernel 
 of $\beta$. The quasi-split form $G_1^{\rm qs}$ is given by a group homomorphism $\Gamma\to \Xi_1$
 whose image lies in the subgroup $\Xi_1'\subset \Xi_1$ that preserves $Z_0$. We have $\Xi_1'\to \Xi_2$ and 
 this gives $\Gamma\to \Xi_2$ which defines the quasi-split form $G^{\rm qs}_2$ together with  
 a central $\al^{\rm qs}: G_1^{\rm qs}\to G^{\rm qs}_2$. 
 The construction of the quasi-split groups $\und G_i^{\rm qs}$ over $\Spec(\Z_p[ u^{\pm 1}])$ in \emph{loc. cit.}, see also above,
 shows that $\al^{\rm qs}$ extends to a central isogeny $\underline\al^{\rm qs}:
 \und G^{\rm qs}_1\to \und G^{\rm qs}_2$. To obtain $\und G_i$ over $\Spec(\Z_p[ u^{\pm 1}])$, we set
 for an $\Z_p[u^{\pm 1}]$-algebra $R$ (see \cite[3.3.4]{PaZhu})
 $$
 \und G_i(R)=(\und G^{\rm qs}(\Z_p^\ur[u^{\pm 1}]\otimes_{\Z_p[u^{\pm 1}]}R))^{{\rm Gal}(\Z_p^\ur/\Z_p)}
 $$
 where the action of Frobenius $\sigma \in  {\rm Gal}(\Z_p^\ur/\Z_p)$ is given
 by ${\rm Int}({\bf g}_i)\cdot\sigma$ with ${\rm Int}({\bf g}_i)$ a certain element $ \und G^{\rm qs}_{i, \rm ad}(\Z_p^\ur[u^{\pm 1}])$.
 Using $G^{\rm qs}_{1, \rm ad} =G^{\rm qs}_{2, \rm ad} $, we can see that we obtain
 a central isogeny $\und \alpha: \und G_1\to \und G_2$ over $\Spec(\Z_p[ u^{\pm 1}])$.
 It remains to see that $\und\alpha$ extends to a group scheme homomorphism
 $\und\al:   \und\Gg_1 \to \und\Gg_2
 $ between the parahoric group schemes $\und\Gg_i$
  over $X=\Spec(\Z_p[u])$. As in \cite[4.2.1]{PaZhu}, it is enough to show that 
  the base change  $\und\alpha \otimes_{\Z_p[u^{\pm 1}]}\Q_p\llps u \lrps $
  extends to a group homomorphism between the   parahoric group 
  schemes over $\Q_p^\ur\lps u\rps$ that correspond to the points $x_{i, \Q_p\llps u\lrps}$
  in the building of $\und G( \Q_p\llps u\lrps)$ that correspond to $x_i$, as in \cite[4.1.3]{PaZhu}; 
  this then follows from the construction in \emph{loc. cit.}.
 The rest then follows from this and the definition of the affine Grassmannians
 ${\rm Gr}_{\und\Gg, X}$ in \cite[6.2]{PaZhu}.
 \end{proof}
 
\begin{para}
Suppose $G$ is a reductive group over $\Q_p$ which splits over a tamely ramified 
extension, and denote by $\ad: G\to G^{\ad}$ the natural homomorphism. 
If $x\in \B(G,\Q_p)$ with $\bar x=\ad_*(x)\in \B(G^{\ad}, \Q_p)$, we have a morphism 
$$
\ad_*: {\rm M}^{\rm loc}_{G, \{\mu\}, x}\to{\rm M}^{\rm loc}_{G^{\ad}, \{\mu_{\ad}\}, {\bar x}}\otimes_{\O_{E_{\ad}}}\O_E
$$
 which is given using the definition of the local model and 
Proposition \ref{extendcisoProp} applied to   $\ad : G\to G^{\ad}$.
For simplicity, we will denote the parahoric group scheme for $G$
that corresponds to $\bar x$ by $\Gg$, we will also use $\Gg^{\rm ad}$, resp. $\Gg^{\rm der}$, 
for the corresponding parahoric group schemes for  $G^\ad$, resp. $G^\der$.
\end{para}

\begin{prop}\label{rhoProp} Assume $\pi_1(G^{\rm der} )$ 
has order prime to $p$. Then the morphism $\ad_*$ induces an isomorphism 
$$
 \ad^\sim_*: {\rm M}^{\rm loc}_{G, \{\mu\}, x} \xrightarrow{\sim} \left({\rm M}^{\rm loc}_{G^{\ad}, \{\mu_{\ad}\}, {\bar x}}\otimes_{\O_{E_\ad}}\O_E\right)^{\sim}
$$
where the target is the normalization
of the base change of ${\rm M}^{\rm loc}_{G^{\ad}, \{\mu_{\ad}\}, {\bar x}}$.
The isomorphism $ \ad^\sim_*$ is equivariant with respect to $\ad: \Gg\to \Gg^{\ad}$
and hence, the natural action of $\Gg$ on ${\rm M}^{\rm loc}_{G, \{\mu\}, x }$ factors
through an action of $\Gg^\ad$.
\end{prop}

\begin{proof} Since $\Gg^\ad$ is smooth, the natural action
of $\Gg^\ad$ on ${\rm M}^{\rm loc}_{G^{\ad}, \{\mu_{\ad}\}, {\bar x}}$
extends to the normalization of the base change. 
By the definitions, the morphism $\ad_*$ 
is equivariant with respect to $\ad: \Gg\to \Gg^{\ad}$.
Since by Theorem \ref{normalThm}, ${\rm M}^{\rm loc}_{G, \{\mu\}, x}$
is normal, $\ad_*$ induces a morphism $\ad^\sim_*$ as above which is also 
equivariant. From the definition, one sees that $\ad_*\otimes_{\O_E}E$
is an isomorphism. Using \cite[Cor. 6.6]{PaZhu}, we see that the morphism $\ad_*\otimes_{\O_E}k$ is given by restricting  
the corresponding natural morphism ${\rm Gr}_{P_k}\to {\rm Gr}_{P^{\rm ad}_k}$  of affine Grassmannians. Here
the group schemes $P_k=\und\Gg\otimes_{\Z_p[u]}k\lps u\rps$ and $P^{\rm ad}_k=\und\Gg^{\rm ad}\otimes_{\Z_p[u]}k\lps u\rps$ are   as in \emph{loc. cit.}. 
We can now see that 
the induced morphism from each connected component of ${\rm Gr}_{P_k}$ to ${\rm Gr}_{P^{\rm ad}_k}$ gives a finite to one map on $k$-valued points.  (See \cite[\S 6 (a), (b)]{PappasRaTwisted}, especially the proof of (6.19) there).  Hence, the restriction $\ad_*\otimes_{\O_E}k$ is quasi-finite. 
Since both its source and target are normal and proper, it follows, using Zariski's main theorem, that $\ad^\sim_*$ is an isomorphism.
\end{proof}
\end{para}

\begin{para}
We assume that we have two triples $(G, \{\mu\}, \eK)$, $(G', \{\mu'\}, \eK')$,
over $K=\Q_p$ as in \ref{locmodsetup} that, in addition, satisfy the following:
\begin{itemize}
\item[a)]  There is a central isogeny $\al: G^{\rm der}\to G'^{\rm der}$
which induces an isomorphism $\al^{\rm ad}: (G^{\rm ad},  \{\mu_{\rm ad}\})\xrightarrow{\sim} (G'^{\rm ad}, \{\mu'_{\rm ad}\})$,

\item[b)] The parahoric subgroups $\eK\subset G(\Q_p)$, $\eK'\subset G'(\Q_p)$, correspond to 
points $x\in \B(G, \Q_p)$, $x'\in \B(G', \Q_p)$, that map to the same 
point $\bar x$ in $\B(G^{\rm ad}, \Q_p)=\B(G'^{\rm ad}, \Q_p)$, where the identification
is via $\al^{\rm ad}$ as in (a),

\item[c)] The prime $p$ does not divide the order of $\pi_1(G'^{\rm der} )$.
\end{itemize}

Under the assumptions (a)-(c), we will compare the local models 
${\rm M}^{\rm loc}_{G, \{\mu\}, x}$ and ${\rm M}^{\rm loc}_{G', \{\mu'\}, x'}$.
Let $E$, resp. $E'$,  the reflex field of $(G, \{\mu\})$, resp. $(G',  \{\mu'\})$, and denote 
by $E_{\rm ad}$ the reflex field of $(G^{\rm ad},  \{\mu_{\rm ad}\})$.
Using (a) above, we obtain $E_{\rm ad}\subset E, E'$.

Denote by $C$ the kernel of $\al$.
By (c), $C$ is a finite group scheme of rank prime to $p$.
For simplicity, we will denote the parahoric group schemes 
that correspond to $\bar x$ by $\Gg$, $\Gg'$, etc. The central isogeny extends to a group scheme
homomorphism $\al: \Gg^{\rm der}\to \Gg'^{\rm der}$. We have $\Gg^{\rm ad}=\Gg'^{\rm ad}$, and by Proposition \ref{cIsoProp},
 \begin{equation}\label{quotientGrpSch}
\Gg'^{\rm der}\simeq\Gg^{\rm der}/\calC
\end{equation}
 where
$\calC$  is the (smooth) schematic closure of $C$ in
$\Gg^{\rm der}$ and the isomorphism is induced by $\al$.

\begin{prop}\label{compLocModProp}
Under the assumptions (a)-(c), 
there is an isomorphism
$$
 {\rm M}^{\rm loc}_{G, \{\mu\}, x} \otimes_{\O_{E }}\O_{EE'}\xrightarrow{\sim} {\rm M}^{\rm loc}_{G', \{\mu'\}, x'}\otimes_{\O_{E' }}\O_{EE'}
$$
which is equivariant with respect to $\al: \Gg^\der \to  \Gg'^\der $. 
\end{prop}

Here the source, resp. target, of the isomorphism admits an action of the group scheme $\Gg^{\der} $, resp. 
$\Gg'^{\der} $, by restricting the natural action of $\Gg $, resp. of $\Gg' $.

\begin{proof} Under the  assumptions (a)-(c), the order  of  
$\pi_1(G^{\rm der})\subset \pi_1(G'^{\rm der})$
is also prime to $p$.  Hence, Proposition \ref{rhoProp}  applies to both $G$ and $G'$
to produce   isomorphisms $\ad^\sim_*$, $\ad'^\sim_*$. Consider now
$$
( \ad'^\sim_*)^{-1}\cdot\tau\cdot \ti\ad_*:
{\rm M}^{\rm loc}_{G, \{\mu\}, x}\otimes_{\O_{E }}\O_{EE'}\xrightarrow{\sim} {\rm M}^{\rm loc}_{G', \{\mu'\}, {x'}}\otimes_{\O_{E' }}\O_{EE'}.
$$
Here we use the natural isomorphism
$$
\tau: \left(M_{\ad}\otimes_{\O_{E_\ad}}\O_E\right)^{\sim}\otimes_{\O_E}\O_{EE'}\xrightarrow{\sim}
\left(M_{\ad}\otimes_{\O_{E_\ad}}\O_{E'}\right)^{\sim}\otimes_{\O_{E'}}\O_{EE'}.
$$
which exists since, by Corollary \ref{normalbcCor}, both its source and target are normal, and therefore agree
with the normalization of $M_{\ad} \otimes_{\O_{E_\ad}}\O_{EE'}$.
(In this we set $M_{\ad}={\rm M}^{\rm loc}_{G^{\ad}, \{\mu_{\ad}\}, {\bar x}}$
for simplicity.)
It remains to show the claimed equivariance property. Using flatness, we
see that it is enough to check this on the generic fibres; there it follows
easily from the definitions.
\end{proof}

\begin{comment}
{\rm Notice, that under   the  assumptions (a)-(c), we can see more directly
(\emph{i.e} without using Theorem \ref{normalThm}) 
that the local model ${\rm M}^{\rm loc}(G, \{\mu\})_x$
admits an action of $\Gg'^{\der}\otimes_{\Z_p}{\O_E}$. On its generic fibre, the central subgroup $C_E\subset G^\der_E$ acts trivially;
hence its schematic closure $\calC\otimes_{\Z_p}\O_E$ in $\Gg^\der\otimes_{\Z_p}\O_E$ acts trivially
on  ${\rm M}^{\rm loc}(G, \{\mu\})_x$. The result follows since, by Proposition \ref{cIsoProp},
$(\Gg^\der\otimes_{\Z_p}\O_E)/(\calC\otimes_{\Z_p}\O_E)\simeq
 (\Gg^{\der}/\calC)\otimes_{\Z_p}\O_E=\Gg'^\der\otimes_{\Z_p}\O_E$.}
 \end{comment}

\end{para}

 \subsection{Embedding  local models in Grassmannians}
 \label{embeddlocalmod}
 
 \begin{para}\label{1.5.1} Here we assume that $K=\Q_p$ and that 
 $(G,  \{\mu\}, \eK)$ is as above.
Suppose we also  have a faithful symplectic representation $\rho: G\to  \GSp(V)\subset \GL(V)$.
We suppose that the composite $\rho\cdot \mu$ is conjugate
 to the minuscule coweight $\mu_{0}$ of $\GSp(V)$
 given by $a\mapsto {\rm diag}(a^{(g)}, 1^{(g)})$
 and that the symplectic 
representation $\rho$ is minuscule
(cf.  table \cite{DeligneCorvallis} 1.3.9).
We also assume that $G\subset \GL(V)$ contains the diagonal torus $\Gm$ of scalars.
We will call such a $\rho$ a (local) Hodge embedding.

Choose an $\Q_p$-split torus $A$ such that
$x\in A(G, A, \Q_p)\subset \B(G, \Q_p)$; choose also a maximal $ \Q^{\rm ur}_p$-split torus $S$
in $G$ that contains $A$ and is defined over $\Q_p$
(such a torus exists by \cite[5.1.12]{BTII}); since $G_{ \Q^{\rm ur}_p}$ is 
quasi-split, $T=Z_G(S)$ is a maximal torus of $G$ which is defined over $\Q_p$
and splits over $\ti K$. Suppose we also choose a pinning $(T_H, B_H, \und e)$ of the split Chevalley
form $H$ of $G$ over $\Z_p$. Again, since $G$ splits over $\ti K$ and $G_{ \Q^{\rm ur}_p}$ is quasi-split,
we can choose $\psi:  G_{\ti K}\xrightarrow{\sim}H_{\ti K}$
that maps  $T_{\ti K}$ to  $(T_H)_{\ti K}$ and is such that the Borel subgroup 
$\psi^{-1}((B_H)_{\ti K})\subset G_{\ti K}$ is defined over $  \Q^{\rm ur}_p$.
Then for $\gamma$ in the inertia $I_{\ti K}={\rm Gal}(\ti K/(\ti K\cap  \Q^{\rm ur}_p))$,
 $c(\gamma):=\psi\cdot \gamma(\psi)^{-1}\in {\rm Aut}(H)(\ti K)$
preserves  $(T_H)_{\ti K}$ and $(B_H)_{\ti K}$. Furthermore, by composing $\psi$
with the (conjugation) action of an element of $T_{H^{\rm ad}}(\ti K)$ we can suppose that $c(\gamma)$ is a diagram automorphism,
\emph{i.e.} it preserves the pinning $(T_H, B_H, \und e)\times_{\Z_p}\ti K$.
Recall now that starting with the pinning $(T_H, B_H, \und e)$
of $H$, the isomorphism $\psi$, the choice of irreducible summands $V_j$, and for each such summand, the choice of a highest weight vector $v_1$ and the lattice chain gradings  given by the translations $t\in\R$,
 we have constructed    
in the previous paragraph a 
$G(\Q^\ur_p)$-equivariant toral embedding
 \begin{equation}
 \iota: \B(G, \Q^\ur_p)\xrightarrow{ } \B({\rm GL}(V), \Q^\ur_p)
 \end{equation}
 which is also ${\rm Gal}(\Q_p^\ur/\Q_p)$-equivariant.
 Note that there is also a canonical equivariant toral embedding
 $$
 s: \B({\rm GSp}(V), \Q^\ur_p)\xrightarrow{}\B({\rm GL}(V), \Q^\ur_p).
 $$

\begin{lemma}\label{lemmaSympl}
There is a choice of the above data such that  $\iota$  factors
$$
\B(G, \Q^\ur_p)\xrightarrow{j} \B({\rm GSp}(V), \Q^\ur_p)\xrightarrow{s}\B({\rm GL}(V), \Q^\ur_p).
$$
\end{lemma}

\begin{proof} We will use results of Satake \cite{Satake} on symplectic representations. 
Consider the similitude character $\chi: G\subset \GL(V)\to \Gm$ and denote by $G_1\subset G$
its kernel so that $\rho(G_1)\subset {\rm Sp}(V)$. We have $\B(G, \Q_p^\ur)=\B(G_1,\Q_p^\ur)\times \R$,
$\B(\GSp(V), \Q_p^\ur)=\B({\rm Sp}(V), \Q^\ur_p)\times\R$ and we can see that it is enough to show that 
there is a choice of data as above such that the corresponding $\iota$ maps $\B(G_1,\Q_p^\ur)$ to
$\B({\rm Sp}(V),\Q_p^\ur)$. Following \cite{Satake}, 
we canonically decompose $V=\oplus_{a} V_a$ as the direct sum of its $\Q_p$-primary $G$-summands. 
(Recall that a $G$-representation $W$ is called $\Q_p$-primary when for every two absolutely irreducible $G$-summands 
$W_1$ and $W_2$ of $W\otimes_{\Q_p}\bar\Q_p$, there is   $\sigma\in {\rm Gal}(\bar\Q_p/\Q_p)$ such that
$W_1\simeq \sigma(W_2)$ as $G$-representations.) As in \emph{loc. cit.}, there are three different types of $\Q_p$-primary components of $V$ that can be distinguished as follows.
If $W_a$ is a  geometrically irreducible $G_1$-summand of $V_a\otimes_{\Q_p}\bar\Q_p$ we have:
\begin{itemize}
\item[a)]  $ W^\vee_a\simeq   W_a$,   
\item[b)] $ W_a^\vee\simeq \sigma(W_a)$, for some $\sigma\in {\rm Gal}(\bar\Q_p/\Q_p)$,
or,
\item[c)] $ W^\vee_a\not\simeq   \sigma(W_a)$, for all $\sigma\in {\rm Gal}(\bar\Q_p/\Q_p)$.
\end{itemize}
Using our construction of $\iota$ and the discussion in \cite[2.1]{Satake}
we can easily reduce to the case that either $V=V_a$ is $\Q_p$-primary and of type (a) or (b), or $V=V_a\oplus V_a^\vee$
with $V_a$ of type (c).  Assume that $V=V_a$ and is of type (a). 
Note that $V$ here does not have to be $\Q_p$-irreducible
but we can write $V=V'^{\oplus m}$ where  $V'$ is $\Q_p$-irreducible. The set-up
of \ref{reprStuff} applies to the $\Q_p$-irreducible $G_1$-representation $V'$.
By \cite[Theorem 1]{Satake}
(using the notation of \ref{reprStuff} for $V'$)  we see that there is a field extension $K_1/K$, a central division algebra $D$ over $K_1$ with an
involution of the first kind,  a right $D$-module $V'_1$ and a left $D$-module $V_2$, with $m=\dim_D(V_2)$,
together with a non-degenerate $\varepsilon$-hermitian, resp. $(-\varepsilon)$-hermitian, form $h_1$ on $V'_1$, resp. $h_2$ on $V_2$,
such that the following is true: The restriction of $\rho$ to $G_1$ factors as the composition of
$$
G_1\to {\rm Res}_{K_1/K}(G_{1, K_1})
$$
with the restriction of scalars of 
$$
G_{1, K_1}\xrightarrow{(\rho_1, 1)} {\rm U}(V'_1/D, h_1)\times {\rm U}(D\backslash V_2, h_2)\xrightarrow{\otimes} {\rm Sp}(V_1).
$$
and the natural ${\rm Res}_{K_1/K} ({\rm Sp}(V_1))\to {\rm Sp}(V)$. (Here, $V_1=V'_1\otimes_D V_2$ is $V$ which has a $K_1$-module structure; $V_1$ supports a non-degenerate
alternating form given via the $D/K_1$ trace of the tensor product of $h_1$ and ${}^th_2$, see \emph{loc. cit.}. Recall that   $V'_1\otimes_{K_1}\bar K$ is an irreducible Weyl module.) By the main result of \cite{LandvogtCrelle} there exists an equivariant toral map 
$\iota'_1: \B(G_1, K_1^\ur)\to \B({\rm U}(V'_1/D, h_1), K_1^\ur)$; in fact, this is obtained by taking ${\rm Gal}(\ti K/K_1^\ur)$-fixed points of a ${\rm Gal}(\ti K/K_1^\ur)$-equivariant toral map $\iota'_1(\ti K): \B(G_1, \ti K)\to \B({\rm U}(V'_1/D, h_1), \ti K)$. 
 Using the uniqueness argument in the 
proof of Proposition \ref{equivBT} we see that the map $\iota'_1$, 
when composed with $\B({\rm U}(V'_1/D, h_1), K_1^\ur)\subset \B(\GL_D(V'_1), K^\ur_1)$ 
agrees with  $t+\iota_1$ as in (\ref{iota1map}), for a suitable choice of translation $t$. The result now follows from the construction of $\iota$ and the above.
The argument for type (b) is similar. Finally, in  type (c)  the alternating form on $V$ is given by the duality between the Lagrangian subspaces
$V_a$ and $V_a^\vee$ in $V$. This case is simpler and is also left to the reader.
\end{proof}
\end{para}

\begin{para}
For $x\in \B(G, \Q_p)$ as before, consider the parahoric group scheme ${\mathcal {GSP}}_z$
of ${\rm GSp}(V)$ that corresponds to $z=j(x)$. As before, set $y=\iota(x)$.
Since $z=s(y)$, the corresponding (periodic) lattice chain $\Lambda^\bullet_y$
is self-dual. We have affine smooth group scheme homomorphisms
\begin{equation}
\rho: \Gg_x\xrightarrow{\ } {\mathcal {GSP}}_z\xrightarrow{\ }{\mathcal {GL}}_y.
\end{equation}
By Proposition \ref{miniscule}, $\Gg_x\xrightarrow{\ }  {\mathcal {GL}}_y$
and therefore $\Gg_x\xrightarrow{\ } {\mathcal {GSP}}_z$ is a closed immersion.

The corresponding local model ${\rm M}^{\rm loc}_{\rm GSp}:={\rm M}^{\rm loc}_{{\rm GSp}(V), \{\mu_{0}\}, z}$ for the group ${\rm GSp}(V)$,
its standard minuscule coweight $\mu_0$ and the periodic self dual
lattice chain $\Lambda^\bullet_z$ that corresponds to $z$ was considered by G\"ortz in \cite{GortzSymplectic}; in this case, this agrees with the 
corresponding local model of \cite{PaZhu} as explained in {\em loc.~cit..}
The generic fibre of ${\rm M}^{\rm loc}_{\rm GSp}$
over $\Q_p$ is the Lagrangian Grassmannian ${\rm LGr}(V)$
of maximal isotropic subspaces in $V$.   
The standard embedding ${\mathcal {GSP}}_z\xrightarrow{\ }{\mathcal {GL}}_y$
induces a closed immersion ${\rm M}^{\rm loc}_{\rm GSp} 
\hookrightarrow {\rm M}^{\rm loc}_{{\rm GL}(V), \{\mu_{0}\}, y}$.
 Since the composition  of $\mu$ with $\rho$
is conjugate to the standard minuscule coweight $\mu_0$ of ${\rm GSp}(V)$
the embedding $\rho$ induces a closed immersion
\begin{equation}\label{immGeneric}
X_\mu\hookrightarrow 
{\rm LGr}(V)\otimes_{\Q_p} E.
\end{equation}
\end{para}
 
\begin{prop}\label{immProp}
With the above assumptions and notations, 
(\ref{immGeneric}) extends to a closed immersion
\begin{equation}
{\rm M}^{\rm loc}_G={\rm M}^{\rm loc}_{G, \{\mu\}, x}\hookrightarrow 
{\rm M}^{\rm loc}_{{\rm GSp}(V), \{\mu_{0}\}, z} \otimes_{\Z_p}\O_{E}.
\end{equation}
\end{prop}

\begin{proof}  
For simplicity, we set $L=\Q_p^\ur$ and $E'=EL$. It is enough to show that  
the base change $(\ref{immGeneric})\times_E E'$
extends to a closed immersion
\begin{equation}\label{clImm}
{\rm M}^{\rm loc}_{G, \{\mu\}, x}\otimes_{\O_E}\O_{E'}\hookrightarrow 
{\rm M}^{\rm loc}_{{\rm GL}(V), \{\mu_{0}\}, y} \otimes_{\Z_p}\O_{E'}.
\end{equation}
Indeed, assuming this, we  easily
verify that (\ref{clImm}) descends over $\O_E$ by checking the descent condition on the generic fibre. 

Now recall that, by  construction (\cite{PaZhu}),   we have
${\rm M}^{\rm loc}_{G, \{\mu\}, x}\otimes_{\O_E}\O_{E'}={\rm M}^{\rm loc}_{G_L, \{\mu\}, x}$, where
${\rm M}^{\rm loc}_{G_L, \{\mu\}, x}$ is the local model for the triple $(G_L,  \{\mu\}, x)$
over $L$. (Here, we use  the obvious extension of the definition of local models over $L=\Q_p^\ur$.)
Using the above we now see that it will be enough to show the closed immersion claim
for the local model over $\O_{E'}$ associated to the triple $(G_{L},  \{\mu\}, x )$
 and the (faithful)  representation $\rho_{L}:=\rho\otimes_{\Q_p}L: G_{L}\hookrightarrow {\rm GL}(V_L)$
 over $L$ obtained by base change.

 As in  \ref{subsec:mapsbetweenbuildings}, 
write $\rho=\prod_j \rho_j$ with $\rho_j:G\to \GL(V_j)$ irreducible over $\Q_p$. 
We return to the set-up of   \ref{reprStuff} for $\rho_j$
over the base field $\Q_p$; we can choose the field $\tilde \Q_p$ there so that $  L\subset \ti\Q_p$.
We have   representations
$$
\rho_{j1, \ti \Q_p}:  G_{\ti\Q_p}\cong H_{\ti\Q_p}\to \GL(V(\lambda_{j1})\otimes_{\Q_p}\ti\Q_p).
$$
Let $\Gamma_{j1}$ be the subgroup of ${\rm Gal}(\ti\Q_p/\Q_p)$ fixing the weight $\lambda_{j1}$ and 
$\Q_{p,j1}$ the corresponding subfield of $\ti\Q_p$. Set $I_{j1}$ for the  subgroup of $\Gamma_{j1}$ with fixed field $ L\Q_{p, j1}$. For simplicity, set $L_{j}=L\Q_{p, j1}\supset L$.
After taking fixed points, i.e. descending, by the action of $I_{j1} $ described in  \ref{reprStuff} we obtain
$$
\rho_{j1, L_{j}}:  G_{L_{j}}\cong (H_{\ti\Q_p})^{I_{j1}}\to (\GL(V(\lambda_{j1})\otimes_{\Q_p}\ti\Q_p))^{I_{j1}}.
$$
Recall that $G_{L}$ is quasi-split by Steinberg's theorem. In fact, we can assume that the action of $I_{j1}$ preserves the Borel subgroup $\psi^{-1}((B_H)_{\ti\Q_p})$. Then the argument 
in the proof of Theorem 3.3 in \cite{TitsCrelle} shows that the group $I_{j1}$ acts via a cocycle $I_{j1}\to \GL(V(\lambda_{j1})\otimes_{\Q_p}\ti\Q_p)$which lifts the cocycle $c'$ of \ref{reprStuff} (see also Step 1 below). This allows us to view $\rho_{j1, L_{j1}}$ as a representation
$$
\rho'_{j, L_j}:   G_{L_{j}} \to\GL(V'_{j})
$$
where $V'_{j}= (V(\lambda_{j1})\otimes_{\Q_p}\ti\Q_p)^{I_{j1}}$ is a $L_{j}$-vector space
such that $V'_{j}\otimes_{L_{j}}\ti\Q_p\cong V(\lambda_{j1})\otimes_{\Q_p}\ti\Q_p$.
Consider the composition 
$$
\rho'_{j, L}:  G_{L}\to {\rm Res}_{L_j/L} (G_{L_j})\xrightarrow{{\rm Res}_{L_j/L}(\rho'_{j, L_j})} {\rm Res}_{L_j/L}(\GL(V'_j))\to  \GL(V'_{j, L}),
$$
where  $V'_{j, L}$ is, by definition, $V'_{j}$ regarded as a $L$-vector space by restriction of scalars.

The base change $\rho_L:=\rho\otimes_{\Q_p}L$ can be identified with
$$
\rho_L: G_L\xrightarrow{\prod_j\prod_\tau \rho'_{j, L}} \prod\nolimits_j \prod\nolimits_\tau\GL(V'_{j, L})\subset \GL( V_L)
$$
where $V_L:=\oplus_j\oplus_\tau V'_{j, L}$ and $\tau$ runs over a finite set of $\Q_p$-automorphisms $\tau: L\to L$ that depends on $j$ and is in bijection with the orbit
$\{\tau(\lambda_{j1})\}$. As in Proposition \ref{equivBT} we obtain an equivariant
map of buildings 
\begin{equation*} 
\iota'_{j}:  \calB(G, L_j)\to \B({\rm GL}(V'_j), L_j) 
\end{equation*}
which as in \ref{iotaconstrBefore} produces
a $G(L)$-equivariant map of buildings 
\begin{equation}\label{iotajL}
\iota_{j,L}:  \calB(G,L)\to \B({\rm GL}(V'_{j, L}), L) ,
\end{equation}
corresponding to $\rho'_{j, L}: G_L\to \GL(V'_{j, L})$. 
Set  $y'_{j}:=\iota_{j, L}(x)$. The image of $(\tau(y'_{j, L}))_{j,\tau}$
under the natural equivariant embedding
$$
\prod\nolimits_j \prod\nolimits_\tau\B(\GL(V'_{j, L}), L)\subset \B(\GL( V_L), L)
$$
 is $y=\iota(x)\in \B(\GL(V), \Q_p)\subset \B(\GL(V_L), L)$.

For simplicity, we set $\O=\O_L=\O^\ur$ and denote by $k$ the residue field of $\O_L$.
Using \cite[Prop. 8.1]{PaZhu} and the above, we see that it is enough to show that
$\rho_L: G_L\hookrightarrow {\rm GL}(V_L)$ extends to a group scheme
homomorphism $\und\rho_{\O[u]}: \und \Gg\to  {\rm GL}(N_\bullet)$ over $\Spec(\O[u])$
(for some periodic $\O[u]$-lattice chain $N_\bullet$) which satisfies the following condition  
from \emph{loc. cit.} 8.1.1:

(*)  The Zariski closure of  
of $\und\Gg\otimes_{\O[u]}k\llps u \lrps$
in ${\rm GL}(N_\bullet \otimes_{\O[u]}k\lps u \rps)$ is a smooth group scheme
which stabilizes the point $x_{k\llps u \lrps}$ and $\und\rho_{\O[u]}\otimes_{\O[u]}k\lps u\rps$   identifies the group scheme
$\und\Gg\otimes_{\O[u]}k\llps u \lrps=\Gg_{x_{k\llps u \lrps}}^\circ$ with the neutral component
of that Zariski closure.

(The homomorphism $\und\rho_{\O[u]}$ then produces a corresponding morphism between local
models as in \cite{PaZhu}. Actually, \cite[8.1]{PaZhu} discusses embeddings into group
schemes related to ${\rm GSp}$ instead of $\GL$ but the argument is the same.) In fact, we will first show that, for all $j$, 
$\rho'_{j, L_j}: G_{L_j}\to \GL(V'_{j})$, and $\rho'_{j, L}: G_L\to \GL(V'_{j, L})$ as above,  suitably extend.
 Then we will deduce that $\rho_L$ also extends
in the desired way. 
We will do this in several steps:

\begin{par}{\it Step 1.} We first show that, for all $j$, $\rho'_{j, L_j}$ and $\rho'_{j, L}$ extend to representations over Laurent polynomial rings with coefficients in $\O$.  If $e_j$ is the (ramification) degree of $L_j/L$, we consider the cover $\O[u]\to \O[v]$, $u\mapsto v^{e_j}$. We identify the generic fibre of the specialization of this cover 
under $u\mapsto p$ with $L_j/L$. 
Recall that we start with a point $x$ in the building $\calB(G, \Q_p)\subset \calB(G, L)$ which lies 
 in the apartment $A( G, S, L)$ of the $L$-split torus $S$. We have chosen a pinning $(T_H, B_H, \und e)$
 of the Chevalley split form $H$ of $G$ over $\Z_p$ which gives a hyperspecial point $x_o$ of $\B(H,\Q_p)$
 in the apartment of the standard torus $T_H$. We have also chosen
 the isomorphism $\psi: G_{\ti \Q_p}\xrightarrow{\sim} H_{\ti \Q_p}$ 
as in (\ref{1.5.1}). In particular, 
 $T_{\ti \Q_p}$ maps isomorphically  under $\psi$ to 
 the standard torus $ (T_H)_{\ti \Q_p}$ and, in fact,  $c(\gamma)=\psi\cdot \gamma(\psi)^{-1}$ 
preserves the pinning, \emph{i.e.} it is 
a diagram automorphism.

Recall that  $\rho_{j1,\ti \Q_p}$ is given by a Weyl module $V(\lambda_{j1})\otimes_{\Q_p}\ti\Q_p$
for the highest weight $\lambda_{j1}$ of $H$. For simplicity, we will write $\lambda_j$ instead of $\lambda_{j1}$.
Recall we fix a vector $v_j\in V(\lambda_j)$ 
of highest weight $\lambda_j$ and consider the $\Z_p$-lattice $\Lambda_j\subset  V(\lambda_j) $ 
given by $\Lambda_j={\mathfrak U}^-_H\cdot v_j$ as before. Consider the $\O$-lattice $\calL_j= \Lambda_j\otimes_{\Z_p}\O$ in $V(\lambda_j)\otimes_{\Q_p}L$; we then have 
a representation 
$$
 \rho_{j, o}: H_\O\to {\rm GL}(\calL_j)
$$
over $\O$ such that $\rho_{j, o}\otimes_\O \ti\Q_p=\rho'_{j, L_j}\otimes_{L_j} \ti \Q_p$.
Every $\gamma$ in the inertial group $I_{j1} ={\rm Gal}(\ti\Q_p/L_j)\subset \Xi$ preserves $\lambda_j$.
Hence, 
$\rho_{j, o}\otimes_\O\ti \Q_p$ and $(\rho_{j, o}\cdot \gamma)\otimes_\O\ti \Q_p$ are equivalent representations
and so there is $A_\gamma\in {\rm GL}(V(\lambda_j)\otimes_{\Q_p}{\ti \Q_p})$ with $ \gamma(g)A_\gamma=A_\gamma g$
for all $g\in H(\ti \Q_p)$. In fact, this identity makes sense and is still true
for all $g\in {\mathfrak U}_H$. The matrix $A_\gamma$ takes $v_{j}$ to 
a multiple of $v_{j}$; we can normalize $A_\gamma$ to
assume that $A_\gamma\cdot v_{j}=v_{ j}$.
Then $A_\gamma$ is uniquely determined.  Since the action of $\Xi$ on $H$ is by diagram automorphisms,    $\gamma$ preserves ${\mathfrak U}^-_H$.
Hence
$$
A_\gamma ( \Lambda_j)=A_\gamma({\mathfrak U}_{H}^-\cdot v_j)\subset \gamma\cdot ({\mathfrak U}_{H}^-)\cdot A_\gamma v_j\subset 
 {\mathfrak U}_{H}^- \cdot v_{j}=\Lambda_{j},
$$ 
\emph{i.e.} $A_\gamma$ preserves $\calL_j$
and hence $A_\gamma$ gives an equivalence of the  
$\O$-representations $\rho_{j, o}$ and $\rho_{j, o}\cdot \gamma$. 
 We thus obtain $A : I_{j1}\to {\rm GL}(\calL_j)$, $A(\gamma)=A_\gamma$,
which we can see is a group homomorphism. 
Therefore we obtain a group scheme
homomorphism
\begin{equation}
\und\rho_{j, o}: ({\rm Res}_{\O[w]/\O[v]} (H\otimes_\O\O[w]))^{I_{j1}}\to 
{\rm GL}((\L_j\otimes_\O\O[w])^{I_{j1}}).
\end{equation}
If $\lambda\in \xch(T_H)$ is a weight of $T_H$ and $\L_{j,\lambda}$ is the corresponding weight
 space of $\L_j$, so that $\L_j=\oplus_\lambda \L_{j,\lambda}$, 
 then $A_\gamma (\L_{j,\lambda})=\L_{j,\gamma\lambda}$.
Since the $I_{j1}$-cover $\O[w]/\O[v]$ is tame, the $\O[v]$-module
$(\L_j\otimes_\O\O [w])^{I_{j1}}$ is finitely generated and projective and hence free
(\emph{e.g.} by \cite{Seshadri}), of rank $d'_j={\rm dim}_{L_j}(V'_j)$;
Similarly, its direct summands $((\oplus_{\gamma\lambda, \gamma\in I_{j1}}\L_{j,\gamma\lambda})\otimes_{\O}{\O[w]})^{I_{j1}}$ (the sum is for the weights in
a $I_{j1}$-orbit) are $\O[v]$-free. Choose a basis $\und b$ that respects this decomposition; this
allows us to identify the target   ${\rm GL}((\L_j\otimes_\O\O[w])^{I_{j1}})$ with
${\rm GL}_{d'_j}(\O[v])$. By restricting $\und\rho_{j, o}$ to $\O[v^{\pm 1}]$
we obtain 
a representation
\begin{equation}\label{sourceref}
 \und\rho'_{j, \O[v^{\pm 1}]} :  ({\rm Res}_{\O[w^{\pm 1}]/\O[v^{\pm 1}]} (H\otimes_\O\O[w^{\pm 1}]))^{I_{j1}}\to   {\rm GL}_{d'_j}(\O[v^{\pm 1}])
\end{equation}
that extends $\rho'_{j, L_j}$. Set $d_j=d'_je_j$. By the definition of $\und G$ (see \ref{quasisplitconstruction}), the source of (\ref{sourceref}) is isomorphic to $\und G\otimes_{\O[u^{\pm 1}]}\O[v^{\pm 1}]$, and so we have a group scheme homomorphism 
\begin{equation}\label{undtores}
\und G\to {\rm Res}_{\O[v^{\pm 1}]/\O[u^{\pm 1}]}(({\rm Res}_{\O[w^{\pm 1}]/\O[v^{\pm 1}]} (H\otimes_\O\O[w^{\pm 1}]))^{I_{j1}} ).
\end{equation}
To obtain an extension 
$$
\und\rho'_{j, \O[u^{\pm 1}]} :\und G\to \GL_{d_j} (\O[u^{\pm 1}])
$$
 of $\rho'_{j, L}$ we now 
compose (\ref{undtores}) with  ${\rm Res}_{\O[v^{\pm 1}]/\O[u^{\pm 1}]} (\und\rho'_{j, \O[v^{\pm 1}]} )$ followed by the homomorphism ${\rm GL}_{d'_j}(\O[v^{\pm 1}])\to \GL_{d_j}(\O[u^{\pm 1}])$ given by restriction of scalars from $\O[v^{\pm 1}]$ to $\O[u^{\pm 1}]$.
 Notice that  
\begin{equation}\label{tamenb}
\O[v]\simeq \O[u]^{e_j}
\end{equation}
 as $\O[u]$-modules and so the target of the last map can be indeed identified with $\GL_{d_j}(\O[u^{\pm 1}])$.
 (Here and in other places,
``extends" is meant in the sense that there is an equivalence between the
 base change of $\und\rho'_{j, \O[u^{\pm 1}]} $
 by $u\mapsto p$ and $\rho'_{j, L}$.) We see that with the choice of basis of $V'_{j, L}$ obtained by specializing $\und b$ by $\O[u]\to L$, $u\mapsto p$, and using (\ref{tamenb})
 above, the image
$\rho'_{j,L}(S)$ of $S$ is contained in the standard maximal torus of ${\rm GL}_{d_j}$.
Then
$$
\iota_{j,L}:  \calB(G,L)\to \B({\rm GL}(V'_{j, L}), L)=\B({\rm GL}_{d_j}, L),
$$
maps the apartment of the torus $S$ to the apartment of the standard maximal 
torus of ${\rm GL}_{d_j}$.  
\end{par}

 \begin{par}{\it Step 2.}
We will now show that $\und\rho'_{j, \O[u^{\pm 1}]}$
extends to a homomorphism 
$$
\und \rho'_{j, \O[u]}: \und \Gg\to  {\rm Aut}_{\O[u]}(N_\bullet)
$$
of group schemes over $\Spec(\O[u])$. Here
  $N_\bullet=N_{j, \bullet}\subset \O[u^{\pm 1}]^{d_j}$ is a periodic chain of  finitely generated $\O[u]$--free rank $d_j$ 
 submodules of $\O[u^{\pm1}]^{d_j}$ as in \cite[5.2]{PaZhu}. Set $y:=y_{j, L}=\iota_{j, L}(x)$; this 
 choice will  allow us to determine the chain $N_\bullet$.
 (For simplicity, in what follows, we sometimes omit the subscript $j$).  This is done as follows: Recall that we have chosen a basis over $\O[u]$ that
 allows us to identify the apartments of the standard torus of $\GL_d$ over $L$, $L\llps u \lrps$ and $k\llps u \lrps$, and that 
 $y$ is on the apartment of this torus over $L$. The identification gives a point $y_{L\llps u \lrps}$ for ${\rm GL}_d(L\llps u \lrps)$ which is in the apartment of this standard torus;
this then corresponds to a $L\lps u \rps$-lattice chain $\Lambda'_\bullet$ in $L\llps u \lrps^d$ and we take 
$N_\bullet=\Lambda'_\bullet\cap \O[u^{\pm 1}]^d$. We can see that $N_\bullet$ has the desired properties
to form a periodic $\O[u]$-lattice chain.
The construction of $\und G$ also gives a point $x_{L\llps u \lrps}$ in the building for $\und G\otimes_{\O[u^{\pm 1}]}L\llps u \lrps$.
The point $x_{L\llps u \lrps}$, by the same reason, then also maps to the point $y_{L\llps u \lrps}$ for ${\rm GL}_d(L\llps u \lrps)$.
(The point $y_{L\llps u \lrps}$ is also the image of $x_{L\llps u \lrps}$ by a map 
$\iota_{L\llps u\lrps }: \B(\und G, L\llps u\lrps)\to \B(\GL_d, L\llps u\lrps)$ that can be defined
as before using our choices.)
Since for $\und\Gg=\Spec(\calA)$ we have $\calA=\calA[u^{-1}]\cap (\calA\otimes_{\O[u]}L\lps u \rps)$,
it will be enough to show that $\und\rho'_{j, \O[u^{\pm 1}]}\otimes_{\O[u^{\pm 1}]}L\llps u \lrps$
extends to a group scheme homomorphism of the corresponding parahoric group
schemes over $L\lps u \rps$; this now follows from our choice of $N_\bullet$ above.
We can now see that
$$
\und\rho_{\O[u]}:= \prod\nolimits_j\prod\nolimits_\tau \und\rho'_{j, \O[u]}: \und\Gg\to \calH:=\prod\nolimits_j\prod\nolimits_\tau {\rm Aut}_{\O[u]}( N_{j, \bullet})
$$
extends the base change $\rho_L=\rho\otimes_{\Q_p}L: G_L\to \prod_j\prod_\tau\GL(V'_{j, L})\subset \GL(V_L)$. 
\end{par}

\begin{par}{\it Step 3.}
It remains to show that
 $ 
\und\rho_{\O[u]} 
$  satisfies condition (*) above.

This will be obtained using the results and arguments of the previous paragraphs
by observing that $\und\rho_{\O[u]}\otimes_{\O[u]}k\llps u \lrps$
is   minuscule. In fact, by construction, this representation satisfies the assumptions described in \ref{Laurentfield1}.

Set $F=k\llps u\lrps$, $\ti F=k\llps w\lrps$. 
As in \ref{Laurentfield1} we see  that $\und \rho_{\O[u]}\otimes_{\O[u]}\ti F$
produces a $\und G(\ti F)=H(\ti F)$-equivariant and ${\rm Gal}(\ti F/F)$-equivariant
toral embedding
$$
\iota_{\ti F}: \B(H, \ti F)\to \B(\GL_n, \ti F),
$$
with $n=\dim_L(V_L)$. This embedding is obtained 
using the decomposition into irreducibles and the descent data 
given as above.   By its construction, $\iota_{\ti F}$ has  the following property:
It maps the apartment of the standard torus of $H(\ti F)$ to the standard apartment
of $\GL_n(\ti F)$ compatibly with the identifications of apartments over $\ti F$
and $\ti \Q_p$ and with the maps between the buildings over $\ti \Q_p$
as above. It also sends $x_F$ to $y_F$ (where these are points are determined from $x$ and $y$
by our choices above as in \ref{quasisplitconstruction}). We now see that \ref{Laurentfield2}, which is a version of Proposition \ref{miniscule} 
in the equicharacteristic case, implies the desired statement. 
\end{par}
\end{proof}

\begin{para}\label{newSymplectic}
We now return to the previous set-up, as in \ref{1.5.1}.
As in  \ref{buildGSp},  ${\mathcal {GSP}}_z$ is the stabilizer of a periodic self-dual  (with respect to the form $\psi$)
 lattice chain $\calL=\{\La^i\}_{i\in \Z}$ in $V$. 
Index the chain as in  \ref{buildGSp}; in particular,  assume that $(\La^i)^\vee=\La^{-i-a}$ with $a=0$ or $1$.
Set $V'=\oplus_{i=-(r-1)-a}^{r-1} V$ equipped with the perfect alternating $K$-bilinear form $\psi'$ 
as in  \ref{buildGSp}.  
Consider the lattice $V'_{\Z_p}=
\oplus_{i=-(r-1)-a}^{r-1} p\Lambda^i\subset V'$; then
 $V'_{\Z_p}\subset V'^\vee_{\Z_p}$. 
The closed immersion $\Hh_z\hookrightarrow    {\rm GSp}(V'_{\Z_p}, \psi')$
composed with $\rho: \Gg_x\hookrightarrow \Hh_z$  gives a closed group scheme immersion 
$$
\rho': \Gg_x\hookrightarrow   {\rm GSp}(V'_{\Z_p}, \psi')\subset \GL(V'_{\Z_p}).
$$
This shows that by composing $\rho$ with the embedding above,
we can assume that the point $y$ is hyperspecial. 
The corresponding   local model ${\rm M}^{\rm loc}_{{\rm GL}(V') , \{\mu'_0\}, y}$ over $\Z_p$  is 
the smooth  Grassmannian ${\rm Gr}(V'_{\Z_p})$ classifying 
subbundles $\calF\subset V'_{\Z_p}\otimes_{\Z_p}\O_S$ of 
rank ${\rm dim}_{\Q_p}(V')$. 
We thus obtain

\begin{cor}\label{embedCor}
Assume $\rho: G\to \GSp(V, \psi)$ comes from a Hodge embedding as above.
We can find a new Hodge embedding $\rho': G\to \GSp(V', \psi')$ and a  
lattice $V'_{\Z_p}\subset V'$ with $V'_{\Z_p}\subset V'^\vee_{\Z_p}$,
such that $\rho'$ induces a closed immersion 
\begin{equation}
{\rm M}^{\rm loc}_{G, \{\mu\}, x}\hookrightarrow 
{\rm Gr}(V'_{\Z_p})  \otimes_{\Z_p}\O_{E}
\end{equation}
of schemes over $\O_E$. 
\end{cor}

   \end{para}

\section{Deformations of $p$-divisible groups}

\subsection{A construction for the universal deformation} 

\begin{para} We continue to use the notations introduced in (\ref{notnsetup}). In particular, we write $W = W(k)$ and $K_0 = W[1/p].$
Unless we mention otherwise, we assume $p>2$.   The aim of this section is to construct the versal deformation space of a $p$-divisible group over $k$ 
using Zink's theory of displays.
\end{para}

\begin{para} Let $R$ be a complete local ring with residue field $k,$ and maximal ideal $\gm.$ 

Recall  \cite{ZinkWindows} \S 2 (see also \cite{LauRelations}), that we have a subring 
$\whW(R) = W(k)\oplus \WW(\gm) \subset W(R),$ where $\WW(\gm) \subset W(R)$ 
consists of Witt vectors $(w_i)_{i \geq 1}$ such that $w_i \in \gm$ and $\{w_i\}_{i \geq 1}$ goes to $0$ $\gm$-adically.
We write $\varphi$ for the Frobenius on $\whW(R)$ and $V$ for the Verschiebung.

Let $I_R \subset \whW(R)$ denote the kernel of the projection $\whW(R) \rightarrow R.$ We recall that the Verschiebung 
$V$ on $\whW(R)$ maps $\whW(R)$ isomorphically to $I_R,$ and we write $V^{-1}: I_R \rightarrow \whW(R)$ for the inverse map.
Note that 
$$\varphi(I_R) = \varphi(V(\whW(R)) = (\varphi V)(\whW(R)) = p \whW(R).$$

\end{para}

\begin{para}\label{windows}
Recall \cite{ZinkDieudonne} that a Dieudonn\'e display over $R$  is a tuple 
$(M,M_1, \Phi, \Phi_1)$ where 

\begin{enumerate}[(i)]
\item $M$ is a finite free $\whW(R)$-module.
\item $M_1\subset M$ is an $\whW(R)$-submodule such that 
$$
I_RM\subset M_1\subset M
$$
and $M/M_1$ is a projective $R$-module.
\item $\Phi: M\to M$ is a $\phi$-semi-linear map 
\item $\Phi_1: M_1 \rightarrow M$ is a $\varphi$-semi-linear map whose image generates 
$M$ as a $\whW(R)$-module, and which satisfies 
$$ \Phi_1( V(w) m) = w \Phi(m); \quad w \in \whW(R), \, m \in M.$$
\end{enumerate}

We will sometimes write $\bar M = M/I_RM$ and $\bar M_1 = M_1/I_RM.$ 
We think of $\bar M$ as a filtered $R$-module, with $\Fil^0 \bar M = \bar M,$ and $\Fil^1 \bar M = \bar M_1.$

If we take $w=1$ and $m \in M_1,$ in the equation in (iv) above, we obtain 
$$ \Phi(m) = \varphi V(1) \Phi_1(m) = p\Phi_1(m).$$

We will be particularly interested in cases where $W(R),$ and hence $\whW(R),$ is $p$-torsion free.
This condition holds when $R$ is $p$-torsion free, or when $p\cdot R = 0,$ and $R$ is reduced.
In this case, the tuple $(M,M_1,\Phi, \Phi_1)$ 
is determined by $(M,M_1, \Phi_1)$ satisfying (i), (ii) and (iv) above. Indeed, we define $\Phi$ by setting 
$\Phi(m) = \Phi_1(V(1)m)$ for $m \in M.$ Then for $w \in \whW(R)$ and $m \in M$ we have 
$$ p\Phi_1( V(w) m) = \Phi_1(V(w)V(1) m) = pw \Phi_1(V(1)m)) = pw\Phi(m),$$ 
and hence $\Phi_1( V(w) m) = w \Phi(m)$ as $\whW(R)$ is $p$-torsion free.

When $W(R)$ is $p$-torsion free, we will also refer to the tuple $(M,M_1,\Phi_1)$ satisfying  (i), (ii) and (iv) 
as a Dieudonn\'e display over $R.$
\end{para}

\begin{para}
Let $(M, M_1, \Phi,\Phi_1)$ be a Dieudonn\'e display  over $R.$ 
The condition (ii) implies that we may write $M$ as a sum of $\whW(R)$-submodules $M = L \oplus T$ such that $M_1= L \oplus I_RT.$
Such a  direct sum is called a {\em normal decomposition} for $M.$

Denote by $\wtM_1$ the image of the $\whW(R)$-module homomorphism 
$$\phi^*(i): \phi^*M_1:=\whW(R)\otimes_{\phi, \whW(R)}M_1\to \phi^*M=\whW(R)\otimes_{\phi, \whW(R)}M$$ induced by the inclusion 
$i: M_1\to M$.

Note that $\wtM_1$ and the notion of a normal decomposition depends only on $M$ and the submodule $M_1$ and not on $\Phi$ and $\Phi_1.$
\end{para}

\begin{lemma}\label{enough}  Suppose that $W(R)$ is $p$-torsion free. Let $M$ be a free $\whW(R)$-module, and $M_1 \subset M$ 
a submodule, with $I_RM \subset M_1$ and $M/M_1$ a projective $R$-module, and let $M = L\oplus T$ be a normal decomposition for $M.$ 
Then 

a) The $\whW(R)$-module $\wtM_1$ is isomorphic to $\varphi^*(L)\oplus p\varphi^*(T)\simeq \whW(R)^d,$  with $d = \rk_{\whW(R)} M,$ 
and in particular depends only on the reduction of  $(M,M_1)$ modulo $p.$

b) If $(M, M_1, \Phi_1)$ is a Dieudonn\'e display over $R,$ then the linearization of $\Phi_1$,
$\Phi_1^{\#}: \phi^*M_1\to M$ factors as a composition
$$
\Phi_1^{\#}: \phi^*M_1\to \wtM_1\xrightarrow{\Psi} M
$$
with $\Psi$ an $\whW(R)$-module isomorphism.

c) Conversely, suppose we are given 
$$
\Psi: \wtM_1:={\rm Im}(\phi^*M_1\to \phi^*M)\xrightarrow{\sim} M.
$$
There there is a unique Dieudonn\'e  display over $R,$ $(M, M_1,\Phi_1)$ which produces our given $(M, M_1, \Psi)$
via the construction in (b).
\end{lemma}

\begin{proof} (a) follows immediately from that fact that $\varphi(I_R) = p\whW(R).$

For (b) we first show that $\Phi_1^{\#}: \phi^*M_1\to M$ factors through $\wtM_1$.
It is enough to show that $p\Phi^{\#}_1$ vanishes on the kernel $K$ of $\phi^*(i): \phi^*M_1\to \phi^*M$.
But $p\Phi_1=\Phi_{|M_1}$ and so $p\Phi^{\#}_1=\Phi^{\#}\circ \phi^*( i)$; this obviously vanishes on $K$.
We write $\Phi^{\#}_1=\Psi\circ (\phi^*M_1\to \wtM_1)$ with a surjective $\Psi: \wtM_1\simeq M$ 
which is necessarily an isomorphism, as $\wtM_1$ and $M$ are free over $\whW(R)$ of the same rank.

For (c) define $\Phi_1: M_1\to M$ by
$$
\Phi_1(m_1)=\Psi(1\otimes m_1)
$$
where $1\otimes m_1$ denotes the image of $1\otimes m_1\in \whW(R)\otimes_{\phi, \whW(R)}M_1=\phi^*M_1$ in $\phi^*M$.
Them $\Phi_1$ is clearly $\phi$-linear and its linearization $\Phi^{\#}_1: \phi^*M_1\to M$ is surjective.
\end{proof}

\begin{para}\label{basechangermk} Let $R \rightarrow R'$ be a morphism of complete local rings with residue field $k.$ A Dieudonn\'e display $(M, M_1, \Phi,\Phi_1)$ over $R,$ has a base change 
to $R'$ (cf.~\cite{ZinkDisplay} Defn.~20), given by 
$M_{R'} = M\otimes_{\whW(R)}\whW(R'),$ and 
$$M_{R',1} = \ker(M_{R'} \rightarrow M/M_1\otimes_RR') = \text{\rm Im}(M_1\otimes_{\whW(R)}\whW(R') \rightarrow M_{R'}) + I_{R'}M_{R'}.$$
Then $\Phi$ on $M_{R'}$ is defined as the $\varphi$ semi-linear extension of $\Phi$ on $M.$ 
The map $\Phi_1$ on $M_{R',1}$ is the unique $\varphi$-semilinear map $M_{R',1} \rightarrow M_{R'}$ 
which satisfies 
$$ \Phi_1(w\otimes m_1) = \varphi(w)\otimes \Phi_1(m_1) \quad w \in \whW(R'), \, m_1 \in M_1$$
and 
$$\Phi_1(V(w)\otimes m) = w \otimes \Phi(m) \quad w \in \whW(R'). \, m
\in M.
$$
The existence and uniqueness of such a map follows, as in {\em loc.~cit.}, from the existence of a normal decomposition.
In particular, if $R \rightarrow R'$ is surjective, we have the notion of a deformation to $R$ of a display over $R'.$

If $W(R)$ and $W(R')$ are $p$-torsion free, then using a normal decomposition one finds that there is a natural isomorphism 
$\wtM_{R',1} \iso \wtM_1\otimes_{\whW(R)} \whW(R'),$ and the diagram 
$$\xymatrix{
\varphi^*(M_1) \ar[r]\ar[d] & \wtM_1 \ar[r]^\Psi\ar[d] & M \ar[d] \\
\varphi^*(M_{R',1}) \ar[r]  & \wtM_{R',1} \ar[r]^{\Psi_{R'}} & M_{R'}
}
$$
commutes. Here $\Psi_{R'}$ denotes the map associated to the Dieudonn\'e display over $R'$ by Lemma \ref{enough}.
\end{para}

\begin{para}
Let $\ffG$ be a $p$-divisible group over $R,$ and denote by $\DD(\ffG)$ its 
contravariant Dieudonn\'e crystal. By the main theorem of \cite{ZinkDieudonne}, $\DD(\ffG)(\whW(R))$ has a natural structure of Dieudonn\'e display over $R,$ and the functor 
$\ffG \mapsto \DD(\ffG)(\whW(R))$ induces an anti-equivalence between $p$-divisible groups over $R,$ and Dieudonn\'e $\whW(R)$-display over $R.$ 
More precisely, the equivalence of {\it loc.~cit.} uses the covariant Dieudonn\'e crystal, and we compose the functor defined there with Cartier duality.
Under this anti-equivalence, base change for Dieudonn\'e displays, defined in the previous paragraph corresponds to base change for $p$-divisible groups \cite{LauRelations} Thm.~3.19.
\end{para}

\begin{para}\label{basicconstruction} Let $\ffG_0$ be a $p$-divisible group over $k.$ We now use the above to construct the versal deformation space of $\ffG_0.$

Let $\DD = \DD(\ffG_0)(W),$ and let $(\DD, \DD_1, \Phi, \Phi_1)$ be the Dieudonn\'e display corresponding to $\ffG_0.$ 
By Lemma \ref{enough}, this data is given by an isomorphism $\Psi_0: \wtDD_1= \varphi^*(\DD_1)  \iso \DD.$

The filtration on $\DD(\ffG_0)(k)$ corresponds to a parabolic subgroup $P_0 \subset \GL(\DD\otimes_Wk).$ 
Fix a lifting of $P_0$ to a parabolic subgroup $P \subset \GL(\DD).$ 
Write $\rmM^{\loc}= \GL(\DD)/P$ and denote by $\widehat{\rmM}^{\loc} = \Spf R,$ the completion of 
$\GL(\DD)/P$ along the image of the identity in $\GL(\DD\otimes_Wk),$ so that $R$ is a power series ring over $W.$ 

Set $M = \DD\otimes_W \whW(R),$ and let $\bar M_1\subset M/I_RM$ be the direct summand corresponding to the parabolic subgroup 
$gPg^{-1} \subset \GL(\DD)$ over $\widehat{\rmM}^{\loc},$ where $g \in (\GL(\DD)/P)(R)$ is the universal point. We denote by 
$M_1 \subset M$ the preimage of $\bar M_1$ in $M.$ 
Let $\Psi: \wtM_1 \iso M$ be an $\whW(R)$-linear isomorphism which reduces to $\Psi_0$ mod $\gm_R.$ 
Then $(M,M_1, \Psi)$ corresponds to a Dieudonn\'e display over $R,$ and hence to a $p$-divisible group $\ffG_R$ over $R$ which deforms $\ffG_0.$ 
\end{para}

\begin{lemma}\label{commdiagram} Let $\fa_R = \gm_R^2+pR.$
There is a canonical commutative diagram 
$$\xymatrix{
\wtM_1\otimes_{\whW(R)} \whW(R/\fa_R) \ar[r]\ar[d]^\sim & \varphi^*(M_{R/\fa_R})\ar@{=}[d]\\
\wtDD_1\otimes_W\whW(R/\fa_R) \ar[r] & \varphi^*(\DD)\otimes_W\whW(R/\fa_R)
}
$$
where the horizontal maps are induced by the natural inclusions $\wtM_1 \rightarrow \varphi^*(M_R)$ and $\wtDD_1 \rightarrow \varphi^*(\DD).$
\end{lemma}
\begin{proof} $L\oplus T$ be a normal decomposition for $(M_{R/\fa_R}, M_{R/\fa_R,1}),$ and let $L_0\oplus T_0$ be the induced normal 
decomposition for $(\DD, \DD_1).$ Observe that the Frobenius on $\whW(R/\fa_R)$ factors as
$$\whW(R/\fa_R) \rightarrow W \overset \varphi \rightarrow W \rightarrow W(R/\fa_R).$$
Hence the submodule $\varphi^*(T) \subset \varphi^*(M_{R/\fa_R})$ is identified with $\varphi^*(T_0)\otimes_W\whW(R/\fa_R) \subset \DD\otimes_W\whW(R/\fa_R).$
An analogous remark applies to $L.$

For any $\Z_p$-module $N$ write $p\otimes N = p\Z_p\otimes_{\Z_p} N.$ Then 
\begin{multline}
\wtM_1\otimes_{\whW(R)} \whW(R/\fa_R) \iso \varphi^*(L) \oplus p\otimes \varphi^*(T) \\ 
\iso (\varphi^*(L_0)\oplus p\otimes \varphi^*(T_0))\otimes_W\whW(R/\fa_R) 
= \wtDD_1\otimes_W\whW(R/\fa_R).
\end{multline}
This produces the left isomorphism in the lemma, and one checks immediately the diagram commutes, and is independent of the choice of normal decomposition.
\end{proof}

\begin{para}\label{defnconstant} We say that $\Psi$ is constant modulo $\fa_R$ if the composite map 
$$ \wtDD_1\otimes_W\whW(R/\fa_R) \iso \wtM_1\otimes_{\whW(R)} \whW(R/\fa_R) \underset \sim{\overset \Psi \longrightarrow} M_{R/\fa_R} \iso \DD\otimes_W\whW(R/\fa_R)$$
is equal to $\Psi_0\otimes 1.$
\end{para}

\begin{lemma}\label{constructionofversalring} 
If $\Psi$ is constant mod $\fa_R$ then the deformation $\ffG_R$ of $\ffG_0$ is versal.
\end{lemma}
\begin{proof} We have two displays over $R/\fa_R.$ One obtained from $(M, M_1, \Phi,\Phi_1)$ 
by the base change $R \rightarrow R/\fa_R,$ and one obtained from $(\DD,\DD_1,\Phi,\Phi_1)$ by the base change 
$k \rightarrow R/\fa_R.$ We denote the corresponding morphisms $\Phi_1$ by $\Phi_1$ and $\Phi_{1,0}$ respectively.

Let $\hat M_{R/\fa_R,1} \subset M_{R/\fa_R}$ be the submodule
\begin{multline}
\hat M_{R/\fa_R,1} = M_{R/\fa_R,1} + \WW(\gm_R/\fa_R)M_{R/\fa_R} \\ =\DD_1\otimes_W\whW(R/\fa_R)+\WW(\gm_R/\fa_R)M_{R/\fa_R} \subset M_{R/\fa_R}.
\end{multline}
We regard $R/\fa_R \rightarrow k$ as a thickening with trivial divided powers. 
By \cite{ZinkDieudonne} Thm.~3, the morphisms $\Phi_1$ and $\Phi_{1,0}$ extend uniquely to $\varphi$-semilinear maps 
$$\hat \Phi_1, \hat \Phi_{1,0}: \hat M_{R/\fa_R,1}  \rightarrow M_{R/\fa_R}.$$
We claim that if $\Psi$ is constant mod $\fa_R$ then $\hat \Phi_1 = \hat \Phi_{1,0}.$ 
Assuming this, the lemma follows from \cite{ZinkDieudonne} Thm 4, and the versality of the filtration $\bar M_1 \subset M/\fa_R = \DD\otimes_WR.$
(As well as, of course, the main theorem of {\em loc.~cit.} giving the equivalence between displays and $p$-divisible groups.)

To show the claim, note that we may regard $\gm_R/\fa_R$ as a
$\whW(R/\fa_R)$-submodule of $\WW(\gm_R/\fa_R),$ by sending $a \in \gm_R/\fa_R$ to $[a].$
Let $L \oplus T$ be  a normal decomposition for $(M_{R/\fa_R}, M_{R/\fa_R,1}).$ Then 
$\hat M_{R/\fa_R,1} = \fa_R T \oplus L \oplus I_{R/\fa_R}T,$ and $\hat \Phi_1$ is given by sending $\fa_RT$ to $0,$ 
and on  $L \oplus I_{R/\fa_R}T,$ is given by the map 
$$ L \oplus I_{R/\fa_R}T \rightarrow \varphi^*(L) \oplus p\otimes \varphi^*(T) = \wtM_1\otimes_{\whW(R)} \whW(R/\fa_R) \overset \Psi \rightarrow M_{R/\fa_R}.$$
In particular, we see that there is a natural map $\varphi^*(\hat M_{R/\fa_R,1}) \rightarrow  \wtM_1\otimes_{\whW(R)} \whW(R/\fa_R),$ which is independent 
of the choice of $\Psi,$ and that the linearization $\hat \Phi_1^{\#}$ of $\hat \Phi_1$ factors through this map and is induced by $\Psi.$
As in the proof of Lemma \ref{commdiagram}, this map depends only on $T_0$ and not on $T.$
An analogous remark applies to $\hat \Phi_{1,0}.$ Thus, we obtain a
diagram 
$$\xymatrix{
\hat M_{R/\fa_R,1} \ar[r]\ar@{=}[d]  & \wtM_1\otimes_{\whW(R)} \whW(R/\fa_R) \ar[r]^{\phantom{tttttttttt}\Psi}\ar[d]^\sim & M_{R/\fa_R} \ar@{=}[d] \\
\hat M_{R/\fa_R,1} \ar[r] & \wtDD_1\otimes_W\whW(R/\fa_R) \ar[r]^{\Psi_0\otimes 1} & \DD\otimes_W\whW(R/\fa_R)
}
$$
where the composite horizontal maps are $\hat \Phi_1$ and  $\hat \Phi_{1,0}$ respectively, the left square commutes, and 
the right square commutes if $\Psi$ is constant mod $\fa_R.$ This proves the claim.
\end{proof}

\begin{para} We assume from now on that $\Psi$ is constant mod $\fa_R,$ so that $\ffG_R$ is versal. Equivalently, $M = M_R$ 
is versal for deformations of displays.
Let $S' \rightarrow S$ be a surjection of $W$-algebras. If $\ffG_S$ is a $p$-divisible group over $S,$ we denote by 
$\Def(\ffG_S; S')$ the set of isomorphism classes of deformations of $\ffG_S$ to $S'.$ 
We will apply this when $S' \rightarrow S$ has nilpotent kernel, in which case a deformation of $\ffG_S$ to $S'$ has no automorphisms. 
If $f: A \rightarrow S$ is a map of $W$-algebras, we denote by $\Def(f;S')$ the set of lifts of $f$ to a map $A \rightarrow S'.$
For any ring $A$ we denote by $A[\epsilon] = A[X]/X^2$ the dual numbers over $A.$
\end{para}

\begin{lemma}\label{versalchar0} Let $K/K_0$ be a finite extension, with ring of integers $\O_K.$ 
Let $\xi: R \rightarrow \O_K$ be a map of $W$-algebras, and $\ffG_{\xi}$ the induced $p$-divisible group over $\O_K$. 
The map 
$$\Def(\xi; \O_K[\epsilon]) \rightarrow \Def(\ffG_{\xi}; \O_K[\epsilon])$$
is a bijection.
\end{lemma}
\begin{proof} Let $R_{\O_K} = R\otimes_W\O_K.$ Let $\xi_K: R_{\O_K} \rightarrow \O_K$ be the induced map of $\O_K$-algebras,
and $I = \ker(\xi_K) \subset R_{\O_K}.$ Then $\Def(\xi; \O_K[\epsilon])$ is in bijection with the set of lifts of $\xi_K$ to a map of 
$\O_K$-algebras $R_{\O_K} \rightarrow \O_K[\epsilon],$ and the latter set is naturally in bijection with $\Hom_{\O_K}(I/I^2, \epsilon\cdot \O_K).$
In particular, $\Def(\xi; \O_K[\epsilon])$ is naturally a free $\O_K$-module. Similarly $\Def(\ffG_{\xi}; \O_K[\epsilon])$ is naturally a free $\O_K$-module; 
it may be identified with the tangent space to a point in a Grassmannian over $\O_K.$
One checks easily that the map in the lemma is a map of $\O_K$-modules. 

Now let $\xi_0: R \rightarrow k'$ be the map to the residue field $k'$ over $\O_K.$ We again denote by $\ffG_0$ the base change of 
$\ffG_0$ to $k'.$ Consider the diagram 
$$\xymatrix{
\Def(\xi; \O_K[\epsilon]) \ar[d]\ar[r] & \Def(\ffG_{\xi}; \O_K[\epsilon]) \ar[d] \\
\Def(\xi_0; k'[\epsilon]) \ar[r] & \Def(\ffG_0; k'[\epsilon]).
}
$$
Here the vertical maps are given by specializing maps, respectively $p$-divisible groups via the map $\O_K \rightarrow k'.$
The map at the bottom is obtained from the map of $\O_K$-modules in the top row by applying $\otimes_{\O_K}k.$ This is obvious 
for the term on the left, and for the term on the right it follows from the description of $\Def(\ffG_{\xi}, \O_K[\epsilon])$ 
and $\Def(\ffG_0, k'[\epsilon])$ in terms of Grothendieck-Messing theory. Since $\ffG_R$ is versal the map on the bottom is an isomorphism, 
and hence so is the map of free $\O_K$-modules at the top.
\end{proof}

\begin{para} We end with the following lemma, which will be needed later
\end{para}
\begin{lemma}\label{cansection} Let $A$ be a complete local, $p$-torsion free, $W$-algebra, with $\gm_A^N \subset pA$ for some integer $N.$ 
Let $M_A = (M_A,M_{A,1}, \Phi,\Phi_1)$ be a deformation of the Dieudonn\'e display $\DD$ to a Dieudonn\'e display over $A.$  
Then there is a unique, Frobenius equivariant map 
$$ \DD\otimes_{\Z_p}\Q_p \rightarrow M_A\otimes_{\Z_p}\Q_p$$
lifting the identity on $\DD.$
\end{lemma}
\begin{proof} Lift the identity on $\DD,$ to an arbitrary $W$-linear map $s: \DD \rightarrow M_A.$ 
Denote by $\Phi_0$ the Frobenius on $\DD.$ Then 
$$\Phi\circ s \circ \Phi_0^{-1} - s \in \WW(\gm_A)\Hom_W(\DD, M_A)\otimes_{\Z_p}\Q_p.$$
Hence, it suffices to give a complete separated topology $\tau$ on $\whW(A)\otimes_{\Z_p}\Q_p$ such that for any 
$x \in \WW(\gm_A),$ $p^{-m}\varphi^m(x) \rightarrow 0.$ Indeed, then the sum
$$\tilde s = s + \sum_{m=0}^\infty \Phi^m\circ(\Phi\circ s \circ \Phi^{-1}_0 - s)\circ\Phi_0^{-m} = s+\sum_{m=0}^\infty\Phi^{m+1}\circ s \circ \Phi_0^{-m-1} - \Phi^m\circ s \circ \Phi_0^{-m} $$
converges to an element of $\Hom_W(\DD, M)\otimes_{\Z_p}\Q_p$ in $\tau$ and is Frobenius invariant.

To define $\tau$ we consider the maps defined by the Witt polynomials 
$$ \whW(A)\otimes_{\Z_p}\Q_p \overset {\prod_n {\bw_n}} \longrightarrow \prod_{n \geq 0} A[1/p].$$ 
Now $A$ is equipped with its $p$-adic topology. This induces a topology on $A[1/p],$ 
and hence on $\prod_{n \geq 0} A[1/p].$ We take $\tau$ to be the coarsest topology such that the above map is continuous. 

To see that $\tau$ has the required property, let $x = (x_0, x_1, \dots) \in \WW(\gm_A).$ Then we have to show that for $n \geq 0$ 
$$ p^{-m}\bw_n(\varphi^m(x)) = p^{-m}\bw_{n+m}(x) = \sum_{i=0}^{n+m} p^{i-m}x_i^{p^{n+m-i}} \rightarrow 0$$
as $m \rightarrow \infty.$ For $0 \neq y \in A[1/p],$ write $v_p(y)$ for the greatest integer such that $yp^{-v_p(y)} \in A.$ 
Let $a$ be any positive integer. Since $x \in \WW(\gm_A),$ $\{x_i\}_{i \geq 1}$ goes to $0$ $\gm_A$-adically and hence $p$-adically. 
Thus there exists $i_0 > 0$ such that $v_p(x_i) > a$ for $i > i_0.$ Then also $v_p(p^{i-m}x_i^{p^{n+m-i}}) > a.$ 
For $i \leq i_0$, $v_p(p^{i-m}x_i^{p^{n+m-i}}) > a,$ for $m$ sufficiently large. 
\end{proof}

\begin{para} Suppose $A = \O_K,$ where $K/K_0$ is a finite, totally ramified extension with uniformizer $\pi.$ 
Then we may apply the previous lemma, and obtain 
$$ \DD\otimes_{\Z_p}\Q_p \rightarrow M\otimes_{\Z_p}\Q_p \rightarrow (M/I_{\O_K}M)\otimes_{\Z_p}\Q_p.$$
The right hand side is a filtered $K$-vector space, and the composite gives $\DD\otimes_{\Z_p}\Q_p$ 
the structure of a weakly admissible $\varphi$-module. 
It is the weakly admissible module corresponding to the $p$-divisible group attached to $M=M_{\O_K}.$
This is easily deduced from \cite{Breuilpdiv}, Prop.~5.1.3, using the map 
$S \rightarrow \whW(\O_K)$ given by $u \mapsto [\pi],$ where, as in {\em loc.~cit}, $S$ the $p$-adic completion 
of $W[u,E(u)^i/i!]_{i \geq 1}$.
\end{para}

\subsection{Deformations with crystalline cycles} \label{crystallinedefs}
\begin{para} 
We continue to use the notation above. For the remainder of \S 3, we assume that $k$ is algebraically closed, as this simplifies the discussion. 
The reader can check that for any $k,$ the same results go through after replacing $k$ by a finite extension.

For any ring $A$ and a finite free $A$-module $N,$ we denote by $N^\otimes$ the direct sum 
of all $A$-modules which can be formed from $N$ by using the operations of taking tensor products, duals, 
and symmetric and exterior powers. 

Suppose that $(s_{\alpha,0}) \subset \DD^{\otimes}$ is a collection 
of $\phi$-invariant sections whose images in $\DD^{\otimes}\otimes_Wk$ lie in $\Fil^0.$ Suppose that the pointwise stabilizer of $(s_{\alpha,0})$ 
is a smooth subgroup $\calG \subset \GL(\DD),$ with $\calG\otimes_{\Z_p}\Q_p$ a connected reductive group $G.$
As usual, we denote by $\Gg^\circ$ the neutral component of $\Gg$ which is also a smooth affine group scheme 
over $\Z_p$. 

We assume the following conditions hold:
\begin{equation}\label{constancyoftensors}
\text{There exists an isomorphism $\varphi^*(\DD) \iso \DD$ taking $s_{\alpha,0}\otimes 1$ to $s_{\alpha,0}.$}
\end{equation}
\begin{equation}\label{trivialityofGtorsors}
{\rm H}^1(D^\times,\calG^\circ) = \{1\},
\end{equation}
where, as before $D^\times$ denotes the complement of the closed point in $D = \Spec \gS,$ where $\gS = W \lps u \rps.$
\begin{equation}\label{Gcontainsscalars}
\text{$G \subset \GL(\DD\otimes_{\Z_p}\Q_p)$ contains the scalars.}
\end{equation}
\end{para}

\begin{para}\label{localmodelseconddfn} 
Recall, that we fixed a parabolic subgroup $P \subset \GL(\DD)$ lifting the parabolic subgroup $P_0 \subset \GL(\DD\otimes_Wk),$ 
corresponding to the filtration on $\DD\otimes_Wk,$ and we wrote $\rM^{\loc} = \GL(\DD)/P,$ and $\Spf R = \widehat \rM^{\loc}$ 
for the completion of $\rM^{\loc}$ along the identity.

Let $K'/K_0$ be a finite extension,  and $y: R \rightarrow K'$ a $K'$-valued point such that $s_{\alpha,0} \in \Fil^0 y^*(\bar M^\otimes).$ 
Then the stabilizer of $y^*(\bar M_1)\subset y^*(\bar M)$ is a parabolic subgroup $P_y \subset \GL(\DD)$, defined over $K',$ 
which is conjugate to $P,$ and induced by a $G$-valued cocharacter $\mu_y$ \cite{KisinJAMS} Lemma 1.4.5.
\footnote{The last line of the proof of {\em loc.~cit} requires a correction, since the group scheme $\Aut^\otimes(\omega)$ is not in general reductive, so one cannot apply (1.1.3) to it. One should replace that line by:

Let $\langle D \rangle^{\otimes,G}$ denote the Tannakian category generated by the $G$-representation $D.$ 
By Tannakian duality, $\langle D \rangle^{\otimes,G}$ is a subcategory of $\langle D \rangle^{\otimes},$ and the filtration 
on $\langle D \rangle^{\otimes}\otimes_{K_0}K$ induces a filtration on $\langle D \rangle^{\otimes,G}\otimes_{K_0}K.$
Hence the filtration on $D\otimes_{K_0}K$ is $G$-split by (1.1.3).}
Let $E \subset K'$ be the local reflex field of $\mu_y.$ That is, $E \supset K_0$ 
is the field of definition of the $G$-conjugacy class of $\mu_y.$ Then the orbit 
$G\cdot y \subset (\GL(\DD)/P)_{K'}$ is defined over $E.$ 
Let  $\rM^{\loc}_G$ be the closure $\rM^{\loc}_G$ of $G\cdot y \subset (\GL(\DD)/P)_{\O_E}.$ 
Note that $\rM^{\loc}_G$ depends only on the $G$-conjugacy class of $\mu_y,$ and not on $y.$

Let $\bar M$ be the vector bundle $\DD\otimes_W\O_{\rmM^{\loc}}$ on $\rmM^{\loc}$ and $\bar M_1 \subset \bar M$ the subbundle 
corresponding to universal parabolic subgroup of $\GL(\DD)$ over $\rmM^{\loc}.$ These restrict to the bundles denoted by the same symbols  
on $\widehat{\rmM}^{\loc}.$ Since $G$ fixes the $s_{\alpha,0}$ pointwise, we have $s_{\alpha,0} \in \Fil^0 \bar M^{\otimes}$ over $\rM^{\loc}_G.
$

We denote by 
$\widehat{\rM}^{\loc}_G = \Spf R_G$ the completion of $\rM^{\loc}_G$ along (the image of) the identity in 
$\GL(\DD\otimes_Wk).$ Then $\widehat{\rM}^{\loc}_G$ depends only the  $G$-conjugacy class of $\mu_y,$ 
and the specialization of $y$ in  $(\GL(\DD)/P)\otimes k.$ By construction $R_G,$ is a quotient of 
$R_E = R\otimes_W\O_E.$

We remark that, in the special situation considered in \S 2, the definition of $\rM^{\loc}_G$ given there for $\mu=\mu_y^{-1}$ agrees with the one in this section by 
Proposition \ref{immProp}. More precisely, if $G$ is defined over $\Q_p,$ then in \S 2, $\rM^{\loc}_G$ was defined as a scheme over 
the integers of the reflex field of $\mu_y$ over $\Q_p,$ and its base change to $\O_E \supset W(k)$ is what we denote $\rM^{\loc}_G$ in the present subsection.

Let $K/ K_0,$ be a totally ramified, finite extension. Let $\pi$ be a uniformizer of $K,$ with Eisenstein polynomial $E(u).$ 
We regard $\O_K$ as a $\gS$-algebra via $u \mapsto \pi.$ Write $M_{\gS} = \DD\otimes_W\gS.$ Note that this is an exception to our usual convention, 
for which, for a ring $A,$ $M_A$ is a $\whW(A)$-module.
\end{para}

\begin{lemma}\label{isomoverclosedpoints}
Let $\xi: R_G \rightarrow \O_K$ be an $\O_K$-valued point, and let 
$F \subset M_{\gS}$ denote the preimage of $\xi^*(\bar M_1).$ Then $F$ is a free $\gS$-module and 
\begin{enumerate} 
\item $s_{\alpha,0} \in F^{\otimes} \subset F^\otimes[1/E(u)] = M_{\gS}^{\otimes}[1/E(u)].$ 
\item The scheme $\SIsom_{(s_{\alpha,0})}(F, M_{\gS})$ consisting of isomorphisms respecting the tensors $(s_{\alpha,0}),$ 
is a trivial $\calG$-torsor over $\gS.$ 
\end{enumerate}
\end{lemma}
\begin{proof}
Since $M_{\gS}/F$ is a free $\O_K$-module and 
$\O_K = \gS/E(u)\gS$ has projective dimension $1$ over $\gS,$ $F$ is free over $\gS.$

Clearly, the remaining two statements hold over $D[1/E(u)]=\Spec \gS[1/E(u)]$ 
since $F[1/E(u)]=M_{\gS}[1/E(u)]$; in fact, $\SIsom_{(s_{\alpha,0})}(F|_{D[1/E(u)]}, M_{\gS}|_{D[1/E(u)]})$
%$\SIsom_{(s_{\alpha,0})}(F, M_{\gS})|_{D[1/E(u)]}$
is a trivial $\Gg$-torsor. 
By \cite{KisinJAMS} Lemma 1.4.5 there exists a $G$-valued cocharacter 
$\mu,$ defined over $K,$ such that $\xi^*(\bar M_1) \subset \xi^*(\bar M)$ is the filtration induced by 
$\mu.$ Let $\widehat \gS_0$ denote the completion of $\gS[1/p]$ at the ideal $E(u)\gS.$
Then $\widehat \gS_0$ is a $K$-algebra and 
$F\otimes_{\gS}\widehat \gS_0 =E(u)\mu(E(u))^{-1}\cdot M_{\widehat \gS_0}.$ Hence  
$F\otimes_{\gS}\widehat \gS_0 = g\cdot M_{\widehat\gS_0}$ for 
some $g \in G(\widehat \gS_0),$ since we are assuming that $G$ contains the subgroup of scalars.
As the $s_{\alpha,0}$ are $G$-invariant, this implies that (1) holds over 
$\Spec \widehat \gS_0$ and hence over $D^\times,$ and that  $\SIsom_{(s_{\alpha,0})}(F|_{D^\times}, M_{\gS}|_{D^\times})$
is a $\Gg$-torsor. Moreover, we have $s_{\alpha,0} \in F^\otimes = \Gamma(D^\times, F^\otimes),$ which proves (1).

Next we show that the $\Gg$-torsor $\SIsom_{(s_{\alpha,0})}(F, M_{\gS})|_{D^\times}$ can be reduced to 
a $\Gg^\circ$-torsor. There is an exact sequence of \'etale sheaves on $\Spec W$
\begin{equation}
1\to \Gg^\circ\to \Gg\to i_*(\Gamma)\to 1
\end{equation}
where $\Gamma$ is finite (constant) on $\Spec(k)$ and $i: \Spec k\hookrightarrow \Spec W$ is the natural
immersion. Taking \'etale cohomology over $D^\times$ gives an exact sequence of pointed sets
\begin{equation}
\rH^1(D^\times, \Gg^\circ)\to \rH^1(D^\times, \Gg)\to \rH^1(D^\times\otimes_W k, \Gamma).
\end{equation}
We saw above that the restriction of $\SIsom_{(s_{\alpha,0})}(F, M_{\gS})|_{D^\times}$ to 
$\Spec \gS[1/E(u)],$ and hence to $D^\times\otimes_Wk = \Spec k\llps u\lrps,$ is a trivial $\Gg$-torsor. 
Thus $\SIsom_{(s_{\alpha,0})}(F, M_{\gS})|_{D^\times}$ can be reduced to a $\Gg^\circ$-torsor.

Since $\rH^1(D^\times,\calG^\circ) = \{1\},$ the torsor $\SIsom_{(s_{\alpha,0})}(F, M_{\gS})|_{D^\times}$ is trivial, 
and there is an isomorphism $F \iso M_{\gS}$ over $D^\times$ respecting the $s_{\alpha,0}.$
Such an isomorphism necessarily extends over $D,$ which proves the two statements.
\end{proof}

\begin{lemma}\label{isomoverclosedpointsII}
Let $\xi: R_G \rightarrow \O_K.$ Then 
\begin{enumerate}
\item $s_{\alpha,0} \in \wtM_{\O_K,1}^\otimes = \wtM_1\otimes_{\whW(R)} \whW(\O_K)^\otimes.$ 
\item The scheme 
$$\calT_{\xi} = \SIsom_{(s_{\alpha,0})}(\wtM_1\otimes_{\whW(R)}\whW(\O_K), M\otimes_{\whW(R)}\whW(\O_K))$$ 
is a $\calG$-torsor over $ \whW(\O_K).$
\end{enumerate}
\end{lemma}
\begin{proof} As remarked in \ref{basechangermk}, we have
$$\wtM_1\otimes_{\whW(R)} \whW(\O_K) = \wtM_{\O_K,1} \subset \varphi^*(\xi^*(M)) = \varphi^*(M_{\O_K}) $$ 
so the first statement in (1) makes sense. 

Let $\gS \rightarrow \whW(\O_K)$ be the unique Frobenius equivariant map 
lifting the identity on $\O_K.$ This is given by $u \mapsto [\pi].$ 
Choose a decomposition $M_{\gS} = L \oplus T$ as $\gS$-modules such that 
$F = L \oplus E(u)T.$ Applying $\otimes_{\gS}\whW(\O_K)$ to $L\oplus T$ gives a normal decomposition 
of $(M_{\O_K}, M_{\O_K,1}).$ 

Note that $E([\pi])$ is not a zero divisor in $\whW(\O_K).$ To see note this that 
$$\bw_n(E([\pi])) = \bw_0(\varphi^n(E([\pi]))) = \bw_0(E[\pi^{p^n}]) = E(\pi^{p^n}),$$ 
which is non-zero for $n \geq 1.$ Thus if $z \in \whW(\O_K)$ satisfies $z\cdot E([\underline \pi]) = 0,$ 
then $\bw_n(z) = 0$ for $n \geq 1.$ But this implies $\bw_n(\varphi(z)) = 0$ for $n \geq 0,$ so $\varphi(z) = 0$ 
and hence $z=0.$ Thus 
\begin{multline}
\varphi^*(F)\otimes_{\gS} \whW(\O_K) = \varphi^*(L \otimes_{\gS} \whW(\O_K) \oplus E(u)T \otimes_{\gS} \whW(\O_K)) \\
= \varphi^*(L \otimes_{\gS} \whW(\O_K)) + p\varphi^*(T\otimes_{\gS} \whW(\O_K)) \iso \wtM_{\O_K.1}.
\end{multline}
Thus both parts of the lemma follow from the corresponding statements in Lemma \ref{isomoverclosedpoints}, and the fact that the $s_{\alpha,0}$ are $\varphi$-invariant.
\end{proof}

\begin{cor}\label{isomoverR} 
Suppose that $R_G$ is normal. Then 
$$s_{\alpha,0} \in \wtM_{R_G,1}^\otimes = \wtM_1^\otimes\otimes_{\whW(R)}\whW(R_G)$$ 
and 
$$ \calT = \SIsom_{(s_{\alpha,0})}(\wtM_1\otimes_{\whW(R)}\whW(R_G), M\otimes_{\whW(R)}\whW(R_G))$$ 
is a trivial $\calG$-torsor over $\whW(R_G).$
\end{cor}
\begin{proof} Since $R_G$ is normal, it follow from \cite{deJongDieudonne} Proposition 7.3.6, that 
if $s \in R_G[1/p]$ is an element such that $\xi(s) \in \O_K$ for every 
finite extension $K/E$ and $\xi: R_G \rightarrow \O_K,$ then $s \in R_G.$

Now suppose $f \in \whW(R_G)$ is non-zero. Then $f$ is divisible by $p$ in $\whW(R_G)$ if and only 
if $\xi(f)$ is divisible by $p$ for every $\xi$ as above. To see this note that $p^{-1}f \in \whW(R_G)$ 
if and only certain universal polynomials in $\bw_n(f),$ $n=0,1,2,\dots$, with coefficients in $\mathbb Z[1/p]$ 
take values in $R_G.$ By what we just saw, this is equivalent to asking that the same polynomials in 
$\bw_n(\xi(f))$ take values in $\O_K$ for all $\xi,$ which is the same as $p^{-1}\xi(f) \in \whW(\O_K).$
Now the first claim of the Corollary follows from (1) of Lemma \ref{isomoverclosedpointsII}.

By Lemma \ref{isomoverclosedpointsII}, for every $\xi$ as above, 
$\xi^*(\calT)$ is a trivial $\calG$-torsor.  By \cite[Thm. 4.1.2]{RaynaudGruson} 
this implies that $\calT$ is flat over $\whW(R_G)$, as $\cap_{\xi} \ker(\whW(\xi)) = 0.$ 
Moreover, $\calT$ has a non-empty fibre over the closed point of $\Spec \whW(R_G).$ 
Hence $\calT$ is a $\calG$-torsor, which is necessarily trivial as $k$ is algebraically closed.
\end{proof}

\begin{para}\label{Frobeniuswithcycles} For the remainder of the section we assume that $R_G$ is normal, so that 
the conditions of Corollary \ref{isomoverR}(2) are satisfied. 

Let $\fa_{R_E} = \gm_{R_E}^2 + \pi_E R_E,$ where $\pi_E \in \O_E$ is a uniformizer.
Note that $R_E/\fa_{R_E} = R/\fa_R.$  
Choose an isomorphism $\Psi_{R_G}: \wtM_{R_G,1} \iso M_{R_G}$ which respects the $s_{\alpha,0},$ 
and such that $\Psi_{R_G}$ is constant modulo $\fa_{R_E}$ in the sense that the reduction of $\Psi_{R_G}$ modulo $\fa_{R_E}$ 
is induced by the isomorphism $\Psi_0\otimes 1$ of \ref{defnconstant}.  
Note that this is possible as $\calG$ is smooth, and if $\Psi_{R_G}$ is constant modulo $\fa_{R_E}$ then the map 
$\wtM_{R_G,1}\otimes_{\whW(R_G)}\whW(R_G/\fa_{R_E}) \rightarrow M_{R_G/\fa_{R_E}}$ does respect the $s_{\alpha,0}.$
Finally, lift $\Psi_{R_G}$ to any isomorphism $\Psi: \wtM_{R_E,1} \iso M_{R_E}$ which is constant mod $\fa_{R_E}.$

As in \ref{basicconstruction}, $(M_{R_E},\wtM_{R_E,1},\Psi)$ gives rise to a Dieudonn\'e display over $R_E,$ and hence to a 
$p$-divisible group $\ffG_{R_E}$ over $R_E.$  If $R'_E$ is a versal deformation $\O_E$-algebra for $\ffG_0,$ then 
$\ffG_{R_E}$ is induced by a map $j:R'_E \rightarrow R_E.$ 
Since $R_E/\fa_{R_E} = R/\fa_R,$ it follows from Lemma \ref{constructionofversalring} that $j$ is an isomorphism mod $\gm_{R'_E}^2+\pi_E.$ 
Hence $j$ is a surjection, and hence an isomorphism, as both rings are smooth over $\O_E$ of the same dimension. In particular, 
$\ffG_{R_E}$ is versal. 
\end{para}

\begin{lemma}\label{cansectionwithcycles} With the notation and assumptions of Lemma \ref{cansection}, 
suppose that $t \in \DD^\otimes,$ and $\tilde t \in M_A^\otimes$ are Frobenius invariant with $\tilde t$ lifting $t,$ and that there is an $W$-linear section 
$s:\DD \rightarrow M_A$ sending $t$ to $\tilde t.$ Then the map of Lemma \ref{cansection}
$$ \DD\otimes_{\Z_p}\Q_p \rightarrow M_A\otimes_{\Z_p}\Q_p$$
sends $t$ to $\tilde t.$
\end{lemma}
\begin{proof} The proof of Lemma \ref{cansection} shows that the map there 
is given by a convergent sum
$$s + \sum_{m=0}^\infty\Phi^{m+1}\circ s \circ \Phi_0^{-m-1} - \Phi^m\circ s \circ \Phi_0^{-m}.$$
Since $t$ and $\tilde t$ are Frobenius invariant, this map sends $t$ to $s(t) = \tilde t.$
\end{proof}

\begin{lemma}\label{factorizationI} Let $K/E$ be a finite extension, $\xi: R_G \rightarrow \O_K,$ and let $M_{\O_K}$ be the Dieudonn\'e display over $\O_K$ induced by $\xi.$ 
Let $M_{\O_K[\epsilon]}$ be any deformation of $M_{\O_K}$ to a Dieudonn\'e display over $\O_K[\epsilon],$ and let $\tilde s_{\alpha}$ denote the image of $s_{\alpha,0}$ under the map 
\begin{equation}\label{cansplittingeqn}
\DD^\otimes\otimes_{\Z_p}\Q_p \rightarrow M_{\O_K[\epsilon]}^\otimes\otimes_{\Z_p}\Q_p
\end{equation}
given by Lemma \ref{cansection}. 
Then the following conditions are equivalent 
\begin{enumerate}
\item The deformation $M_{\O_K[\epsilon]}$ is induced by a lift $\tilde \xi:R_G \rightarrow \O_K[\epsilon]$ of $\xi.$
\item Any lift $\tilde \xi: R_E \rightarrow \O_K[\epsilon]$ of $\xi$ which induces $M_{\O_K[\epsilon]}$ factors through $R_G.$
\item $\tilde s_{\alpha}$ maps to an element $s_{\alpha} \in \Fil^0 (\bar M_{\O_K[\epsilon]})^\otimes\otimes_{\Z_p}\Q_p.$
\item $\tilde s_{\alpha}\in M_{\O_K[\epsilon]}^\otimes,$ maps to $s_{\alpha} \in \Fil^0 (\bar M_{\O_K[\epsilon]})^\otimes.$
\end{enumerate}
\end{lemma}
\begin{proof} 
We first check that (1) implies (4). By construction, $M_{R_G} = \DD\otimes_W{\whW(R_G)},$ and under this identification the tensors $s_{\alpha,0} \in M_{R_G}^\otimes$ are Frobenius 
invariant, and their images in $\bar M_{R_G}^\otimes$ lie in $\Fil^0.$ In particular, if (1) holds, we obtain in this way Frobenius invariant tensors $\tilde s_{\alpha}' \in M_{\O_K[\epsilon]}^\otimes,$ 
which map to $\Fil^0\bar M_{\O_K[\epsilon]}^\otimes.$ To show (4), we have to check $\tilde s_{\alpha}' = \tilde s_{\alpha}.$ 
If $s: \DD \rightarrow M_{\O_K[\epsilon]}$ denotes the tautological inclusion, then $s$ sends $s_{\alpha,0}$ to $\tilde s_{\alpha}',$ 
so (4) follows by Lemma \ref{cansectionwithcycles}

We obviously have (4) implies (3), and (2) implies (1) so it remains to show that (3) implies (2). For this we show that the space of
lifts $\tilde \xi$ such that (3) holds, is an $\O_K$-module, and that its rank is equal to the dimension of $R_G[1/p].$ 
By \cite{ZinkDieudonne} Thm 3,4, for any deformation $M_{\O_K[\epsilon]}$  of $M_{\O_K}$ as a display, there is a Frobenius equivariant identification 
\begin{equation}\label{crystallinedisplay}
M_{\O_K[\epsilon]} \iso M_{\O_K}\otimes_{\whW(\O_K)}\whW(\O_K[\epsilon])
\end{equation}
of the underlying $\whW(\O_K[\epsilon])$-modules,  
and isomorphism classes of deformations correspond bijectively to lifts of $\bar M_{\O_K,1} \subset \bar M_{\O_K}$ to a direct summand 
$\bar M_{\O_K[\epsilon],1} \subset \bar M_{\O_K[\epsilon]}.$

To see which deformations satisfy the condition in (3), we identify $M_{\O_K}\otimes_{\Z_p}\Q_p$ with $\DD\otimes_W\whW(\O_K)[1/p]$ using the isomorphism of Lemma \ref{cansection}. 
Combing this with (\ref{crystallinedisplay}), $M_{\O_K[\epsilon]}\otimes_{\Z_p}\Q_p$ is Frobenius equivariantly identified with $\DD\otimes_W\whW(\O_K[\epsilon])[1/p].$ As above, one sees that 
$s_{\alpha,0}$ is taken to $\tilde s_{\alpha}$ under this identification. In particular, this identifies 
$\bar M_{\O_K[\epsilon]}\otimes \Q_p$ with $\DD_{K[\epsilon]} = \DD\otimes_WK[\epsilon],$ and gives $\DD_{K_0[\epsilon]} = \DD\otimes_WK_0[\epsilon],$ the structure 
of a weakly admissible filtered $\varphi$-module, which is a self extension of $\DD_{K_0} = \DD\otimes_WK_0.$ 
The filtration on  $\DD_{K[\epsilon]}$ is obtained by translating the constant filtration arising from the filtration on $\DD_K = \DD\otimes_WK,$ by an element $1+\epsilon h$ for some 
$h\in \End_K(\DD_K).$ Let $P_{\xi}\subset \GL(\DD_K)$ be the subgroup respecting the filtration on $\DD_K,$ and consider the map 
$$ \End_{K_0}(\DD_{K_0}) \rightarrow (\DD^\otimes)_{\alpha}; \quad g \mapsto g(s_{\alpha,0}).$$
Since $s_{\alpha,0}$ is Frobenius invariant and in $\Fil^0,$ this is a map of weakly admissible, filtered $\varphi$-modules, 
and hence is strict for filtrations. It follows that if $g(s_{\alpha,0}) \in \Fil^0 \DD^\otimes$ for all $\alpha,$ then 
$g \in \Lie G + \Lie P_{\xi} \subset \End_K(\DD_K).$ We apply this to the element $h.$ 
If $s_{\alpha,0} \in (1+\epsilon h)\Fil^0\DD^\otimes_K,$ we have $h\cdot s_{\alpha,0} \in \Fil^0 \DD_K^\otimes,$ 
so $h \in \Lie G + \Lie P_{\xi}.$ Thus we may assume $h \in \Lie G.$ Conversely, if $h \in \Lie G$ then $s_{\alpha,0} \in (1+\epsilon h)\Fil^0\DD^\otimes_K.$

Thus the set of lifts of $\bar M_{\O_K,1} \subset \bar M_{\O_K}$ to a direct summand $\bar M_{\O_K[\epsilon],1} \subset \bar M_{\O_K[\epsilon]}$ such that 
$s_{\alpha,0} \in \Fil^0 \bar M_{\O_K[\epsilon]}^\otimes$ can be identified with an $\O_K$-module of rank equal to 
$$\dim \Lie G/(\Lie P_{\xi}\cap \Lie G) = \dim R_G[1/p].$$
Let $\Def_G(\xi; \O_K[\epsilon])$ denote the set of lifts of $\xi$ to a map $R_G \rightarrow \O_K[\epsilon],$ and let $\Def_G(M_{\O_K}, \O_K[\epsilon])$ denote the set of isomorphism classes of
deformations of $M_{\O_K}$ to a Dieudonn\'e display over $\O_K[\epsilon]$ satisfying (3). Then we have a map 
$$ \Def_G(\xi; \O_K[\epsilon]) \rightarrow \Def_G(M_{\O_K}; \O_K[\epsilon])$$
which is injective by the versality of $\ffG_{R_E},$ and Lemma \ref{versalchar0}. We have just seen that both source and target are $\O_K$-modules of the same rank. 
If $\tilde \xi: R_E \rightarrow \O_K[\epsilon]$ is a lift of $\xi$ inducing $M_{\O_K[\epsilon]} \in \Def_G(M_{\O_K}; \O_K[\epsilon]),$ this shows that for $n$ large enough, 
the composite of $\tilde \xi$ with the map $\O_K[\epsilon] \rightarrow \O_K[\epsilon]$ given by $\epsilon \mapsto p^n\epsilon$ factors through $R_G.$ 
Since this last map is injective, this shows that $\tilde \xi$ factors through $R_G$ and proves that (3) implies (2).
\end{proof}

\begin{prop}\label{factorizationII} Let $K/E$ be a finite extension, and $\ffG_{\O_K}$ a deformation of $\ffG_0$ satisfying the following conditions.
\begin{enumerate}
\item Under the canonical isomorphism
$\DD(\ffG_{\O_K})(\O_K)\otimes_{\O_K}K \iso \DD\otimes_{K_0}K = \DD_K,$ the filtration on $\DD(\ffG_{\O_K})(\O_K)\otimes_{\O_K}K$ 
is induced by a $G$-valued cocharacter conjugate to $\mu_y.$
\item The $s_{\alpha,0}$ lift to Frobenius invariant tensors $\tilde s_{\alpha} \in \DD(\ffG_{\O_K})(\whW(\O_K))^\otimes,$ and there is an isomorphism 
$$ \DD\otimes_W\whW(\O_K) \iso \DD(\ffG_{\O_K})(\whW(\O_K))$$
taking $s_{\alpha,0}$ to $\tilde s_{\alpha}.$
\end{enumerate}
Then any morphism $\xi: R_E \rightarrow \O_K$ inducing $\ffG_{\O_K}$ factors through $R_G.$
\end{prop}
\begin{proof}
The map in (2) induces a map 
$$ \DD \rightarrow  \DD(\ffG_{\O_K})(\whW(\O_K))\otimes_{\whW(\O_K)} W = \DD(\ffG_0)(W) = \DD.$$
which takes $s_{\alpha,0}$ to $s_{\alpha,0},$ hence is given by an element of $\calG(W).$ Lifting this element to $\calG(\whW(\O_K)),$ we may modify the map 
in (2) and assume it lifts the identity on $\DD.$ Then Lemma \ref{cansectionwithcycles} implies that the $\tilde s_{\alpha}$ 
are the images of $s_{\alpha,0}$ under the canonical map given by Lemma \ref{cansection}. 
Denote by $s_{\alpha} \in \DD(\ffG_{\O_K})(\O_K)^\otimes$ the image of $\tilde s_{\alpha}.$

Let $\DD_{\O_K} = \DD\otimes_W\O_K,$ and consider the filtration on $\DD_{\O_K}$ induced by the isomorphism 
$\DD_{\O_K} \iso \DD(\ffG_{\O_K})(\O_K)$ arising from the map in (2). This map takes $s_{\alpha,0}$ to $s_{\alpha},$ and hence differs 
from the isomorphism in (1) by an element of $G(K).$ In particular, the induced filtration on $\DD_{\O_K}$ corresponds to a parabolic 
subgroup $G$-conjugate to $P_y,$ so $s_{\alpha} \in \Fil^0 \DD_{\O_K}^\otimes.$ 
As this filtration lifts the one on $\DD\otimes_Wk,$ it corresponds to a point $y': R_G \rightarrow \O_K.$
As $R_G$ depends only on the reduction of $y,$ and the conjugacy class of $\mu_y,$ we may assume $y = y'$ (and $K'=K$) in order to simplify notation.

Let $(M_{\O_K}, M_{\O_K,1}, \Phi_1,\Phi)$ be the Dieudonn\'e display corresponding to $y,$ and $\Psi: \wtM_{\O_K,1} \iso M_{\O_K}$ the isomorphism  associated by Lemma \ref{enough}.
Recall that, by construction, $M_{\O_K}$ is identified with $\DD\otimes_W\whW(\O_K),$ and $\Psi$ takes $s_{\alpha,0}$ to $s_{\alpha,0}.$
By what we have just seen, $\ffG_{\O_K}$ arises from a morphism $\Psi': \wtM_{\O_K,1} \iso M_{\O_K},$ which takes $s_{\alpha,0}$ to $s_{\alpha,0}$ 
(because the $\tilde s_{\alpha}$ are fixed by Frobenius), and reduces to $\Psi_0: \widetilde {\DD}_1 \iso \DD.$

We now construct a Dieudonn\'e display over $S = \O_K\lps T \rps.$ First consider the base change of $(M_{\O_K}, M_{\O_K,1}, \Phi_1,\Phi)$ to $S,$ 
$(M_S, M_{S,1}, \Phi_1,\Phi).$ 
The map $\O_K\lps T \rps \rightarrow \O_K\times_k \O_K$ given by $T \mapsto (0,\pi)$ 
is surjective, and hence so is $\whW(\O_K\lps T \rps) \rightarrow \whW(\O_K)\times_W \whW(\O_K).$
Hence, by Corollary \ref{isomoverR}, there exists an isomorphism $\Psi_S: \wtM_{S.1} \iso M_S$ which takes $s_{\alpha,0}$ to $s_{\alpha,0},$ and specializes 
to $(\Psi, \Psi')$ under $T \mapsto (0,\pi).$ We take $M_S$ to be the Dieudonn\'e display over $S$ associated to $\Psi_S$ by Lemma \ref{enough}.

Now let $\xi: R_E\rightarrow \O_K$ be a map inducing $\ffG_{\O_K}.$ By versality, we may lift the map $(y,\xi) :R_E \rightarrow \O_K\times_k\O_K$ 
to a map $\tilde \xi: R_E \rightarrow S$ which induces $M_S,$ and we may identify the 
Dieudonn\'e display $M_S$ with the base change of $M_{R_E}$ by $\tilde \xi.$ We will show that $\tilde \xi$ factors through $R_G,$ which implies that $\xi$ does also. 

For $n \geq 1,$ let $S_n = S/T^n,$ and denote by $M_{S_n}$ the base change of $M_S$ to $S_n.$ 
Let $I_n = \ker(R_E \overset {\tilde \xi} \rightarrow S \rightarrow S_n),$ and let $J_G = \ker(R_E \rightarrow R_G).$
Let $\gn = \ker(y: R_G \rightarrow \O_K),$ and $J_G^n = \ker(R_E \rightarrow (R_G/\gn^n)[1/p]).$ 
By Lemma \ref{cansectionwithcycles}, under the canonical map given by Lemma \ref{cansection}, $s_{\alpha,0} \in \DD^\otimes$ 
is mapped to $s_{\alpha,0} \in M_{S_n}^\otimes.$ It follows that, for the Dieudonn\'e display $M_{R_E/I_n\cap J_G^n},$ the map of Lemma \ref{cansection} sends 
 $s_{\alpha,0}$ to $s_{\alpha,0}.$ In particular, by Lemma \ref{factorizationI}, any map $R_E \rightarrow \O_K[\epsilon]$ which factors through 
$R_E/I_n\cap J_G^n,$ factors through $R_G.$

Now let $I = \ker {\tilde \xi},$ and $R_G' = R_E/I\cap J_G,$ and consider a map $\theta: R_G'[1/p] \rightarrow K[\epsilon]$ which lifts $y.$ 
After replacing $\epsilon$ by $p^{-n}\epsilon$ we may assume that $\theta$ induces a map $R_G' \rightarrow \O_K[\epsilon].$
We also write $\theta$ for the induced map $R_E \rightarrow K[\epsilon].$ Since $(I\cap J_G)[1/p] = \cap_n ((I_n\cap J_G^n)[1/p]),$ 
and the decreasing sequence of $E$-subspaces $\theta((I_n\cap J_G^n)[1/p]) \subset \epsilon\cdot K$ must stabilize, we see that 
$\theta(I_n \cap J_G^n) = 0$ for some $n.$ Hence $\theta$ factors through $R_G$ by what we saw above.  This implies that the tangent 
spaces of $R_G'[1/p]$ and $R_G[1/p]$ at $\gn$ are equal. Since $R_G[1/p]$ is regular, we have 
$$ \dim_{\gn} R_G'[1/p] \leq \dim_{\kappa(\gn)} \gn/\gn^2 = \dim_{\gn} R_G[1/p] \leq \dim_{\gn} R_G'[1/p]. $$
(Here $\dim_{\gn}$ denotes the dimension at $\gn.$) It follows that $\dim_{\gn} R_G'[1/p] = \dim_{\kappa(\gn)} \gn/\gn^2,$ which implies that 
$R_G'[1/p]$ is regular at $\gn,$ and $R'_{G,\gn}[1/p] = R_{G,\gn}[1/p].$ In particular,  
$$ J_G\otimes_{\Z_p}\Q_p \subset I_n\otimes_{\Z_p}\Q_p,$$
so $J_G \subset (I_n\otimes_{\Z_p}\Q_p) \cap R_E = I_n,$ as $R_E/I_n$ is $p$-torsion free. Finally, $J_G \subset \cap_n I_n = I,$ so $\tilde \xi$ factors 
through $R_G$ and so does $\xi.$
\end{proof}.

\subsection{Deformations with \'etale cycles}
\begin{para} We continue to use the notation above, so in particular $k$ is algebraically closed. Set $\Gamma_K = \Gal(\bar K/K).$ 
Denote by $\Rep^{\cris}_{\Gamma_K}$ the category of crystalline 
$\Gamma_K$-representations, and by $\Rep^{\cris\circ}_{\Gamma_K}$ the category 
of $\Gamma_K$-stable $\Z_p$-lattices spanning a representation in $\Rep^{\cris}_{\Gamma_K}.$ 
For $V$ a crystalline representation, recall Fontaine's functors  
$$ D_{\cris}(V) = (B_{\cris}\otimes_{\Q_p}V)^{\Gamma_K} \quad {\rm and} \quad D_{\dR}(V) = (B_{\dR}\otimes_{\Q_p}V)^{\Gamma_K}.$$ 

Fix a uniformiser $\pi \in K,$ and let $E(u) \in W[u]$ be the 
Eisenstein polynomial for $\pi.$ We have the $\varphi$-equivariant inclusion 
$\gS \hookrightarrow \whW(\O_K)$ introduced above. 
As above, we denote by $D^\times$ the complement of the closed point in
$\Spec \gS.$ 

Let $\Mod^{\varphi}_{/\gS}$ denote the category of finite free $\gS$-modules $\gM$  
equipped with a Frobenius semi-linear isomorphism 
$$ 1\otimes\varphi: \varphi^*(\gM)[1/E(u)] \iso \gM[1/E(u)]. $$

For $\gM \in \Mod^{\varphi}_{/\gS}$ and $i$ an integer, we set 
$$\Fil^i\varphi^*(\gM) = \varphi^*(\gM) \cap (1\otimes\varphi)^{-1}(E(u)^i\gM).$$
If we view $K$ as a $\gS$-algebra via $u \mapsto \pi,$ then this induces a filtration on 
$\varphi^*(\gM)\otimes_{\gS}K.$
\end{para}

\begin{thm}\label{classificationthm} There exists a fully faithful tensor functor 
$$ \gM: \Rep^{\cris\circ}_{\Gamma_K} \rightarrow \Mod^{\varphi}_{/\gS},$$ 
which is compatible with formation of symmetric and exterior powers, 
and such that $L \mapsto \gM(L)|_{D^\times}$ is exact.
If $L$ is in $\Rep^{\cris\circ}_{\Gamma_K},$ $V = L\otimes_{\Z_p}\Q_p,$ and $\gM = \gM(L),$ 
then
\begin{enumerate}
\item  There are canonical isomorphisms
$$ \text{$D_{\cris}(V) \iso \gM/u\gM[1/p]$ and $D_{\dR}(V) \iso \varphi^*(\gM)\otimes_{\gS}K$} $$
where the first isomorphism is compatible with Frobenius, and the second isomorphism is compatible with filtrations. 
\item If $L = T_p\ffG^{\vee} := \Hom_{\Z_p}(T_p\ffG, \Z_p)$ for a $p$-divisible group $\ffG$ over $\O_K,$ then there is 
a canonical isomorphism 
$$ \DD(\ffG)(\whW(\O_K)) \iso \whW(\O_K)\otimes_{\gS,\varphi}\gM$$
such that the induced map 
$$ \DD(\ffG)(\O_K) \iso \O_K\otimes_{\gS}\varphi^*(\gM) \rightarrow D_{\dR}(T_p\ffG^{\vee})$$ 
is compatible with filtrations.
In particular, if $\ffG_0 = \ffG\otimes k$ then 
$\DD(\ffG_0)(W)$ is canonically identified with $\varphi^*(\gM/u\gM).$
\end{enumerate}
\end{thm}
\begin{proof} Except for the claim that $L \mapsto \gM(L)|_{D^\times}$ is exact, 
this follows from \cite{KisinJAMS} 1.2.1, 1.4.2
\footnote{Note also that $T_p\ffG^*$ should be replaced by the linear dual 
$T_p\ffG^{\vee}$ of $T_p\ffG$ in (1.4.2), (1.4.3) and (1.5.11) of the published version of {\it
  loc.~cit}.}. More precisely, let $S$ be the $p$-adic completion of $W[u, E(u)^i/i!]_{i \geq 1}.$ 
Then the first isomorphism in (2) is constructed in {\em loc.~cit.} with $S$ in place of 
$\whW(\O_K),$ and we obtain the isomorphism in (2) using the continuous extension 
$S \hookrightarrow \whW(\O_K)$ of $\gS \rightarrow \whW(\O_K).$

To see the exactness of $L \mapsto \gM(L)|_{D^\times},$ 
let $L^\bullet$ be an exact sequence in $\Rep^{\cris\circ}_{\Gamma_K}.$ 
We have to show that $\gM(L^{\bullet})|_{D^\times}$ is exact.  
Let $Q$ be a cohomology group of $\gM(L^{\bullet})|_{D^\times}.$

By (1) the support of $Q$ on $D^\times\otimes_{\Z_p}\Q_p$ is disjoint from the ideal $E(u),$ and in particular 
is contained in a finite number of closed points. There is an isomorphism $\varphi^*(Q)[1/E(u)] \iso Q[1/E(u)],$ 
which implies that the support of $Q$ on $D^\times\otimes_{\Z_p}\Q_p$ is empty. Finally, the support of $Q$ 
does not contain the ideal $(p),$ by \cite{KisinJAMS} 1.2.1(2), so $Q=0.$
\end{proof}

\begin{para}\label{etaletensorsetup} 
Suppose that $L$ is in $\Rep^{\cris\circ}_{\Gamma_K},$ and let $s_{\alpha,\et} \in L^\otimes$ 
be a collection of $\Gamma_K$-invariant tensors, which define  a subgroup $\calG \subset \GL(L),$ 
which is smooth over $\Z_p$ with reductive generic fibre $G.$  
Applying the functor $\gM$ we obtain corresponding tensors $\tilde s_{\alpha} \in \gM(L)^{\otimes}.$

Note that, since the $s_{\alpha,\et}$ are $\Gamma_K$-invariant, the action $\Gamma_K$ on $L$ gives rise 
to a representation 
$$ \rho: \Gamma_K \rightarrow \calG(\Q_p).$$
Recall the Kottwitz homomorphism $\kappa_G: G(K_0)\to \pi_1(G)_{I},$ where $I = \Gal(\bar K/K_0).$ 
We have the following Lemma, due to Wintenberger.
\end{para}

\begin{prop}\label{WintenProp}
(\cite{WintenFcrist}) The image of the crystalline representation $$\rho : \Gamma_K \to G(\Q_p)$$ is contained in $\ker \kappa_G.$
\end{prop}
\begin{proof} This is proved in \cite[Lemme 1]{WintenFcrist} when $K =  K_0,$ however the proof there goes over verbatim 
without this assumption. Note that we are using here that $k$ is algebraically closed.
\end{proof}

\begin{lemma}\label{keylemma} Suppose that $\rH^1(D^\times, \calG^\circ) = \{1\},$ and $\rho$ factors through $\calG^\circ(\Z_p).$ 
Then there is an isomorphism 
$$ L\otimes_{\Z_p} \gS \xrightarrow{\sim} \gM(L)$$ 
taking $s_{\alpha}$ to $\tilde s_{\alpha}.$
\end{lemma}
\begin{proof}  

  Write $\O_{\Gg^\circ}=\dlim_{i \in J} L_i$ with $L_i \subset \O_{\Gg^\circ}$ of finite $\Z_p$-rank and $\Gg^\circ$-stable, as 
in \cite{BroshiJPAA}, Lemma 3.1. Let $\gM(\O_{\Gg^\circ}):=\dlim_i \gM(L_i).$ Since $L \mapsto \gM(L)|_{D^\times}$ is an exact 
faithful tensor functor by Theorem \ref{classificationthm}, it follows from \cite{BroshiJPAA}, Thm.~4.3, that 
$\gM(\O_{\Gg^\circ})|_{D^\times}$ is a sheaf of algebras on $D^\times$ and that 
$\calP^\circ=\und{\Spec}(\gM(\O_{\Gg^\circ}))_{|D^\times}$ is naturally a $\calG^\circ$-torsor. 
If we carry out the same construction with $\calG$ in place of $\calG^\circ$ we obtain 
a $\calG$-torsor $\calP$ over $D^\times.$  By construction, there is a $\calG^\circ$-equivariant map 
$\calP^\circ \rightarrow \calP,$ so $\calP$ is obtained from $\calP^\circ$ by pushing out by 
$\calG^\circ \rightarrow \calG.$ Our assumptions imply that $\calP^\circ$ is trivial and hence so is $\calP.$

Now let $\calP' \subset \SHom(L\otimes_{\Z_p}\gS, \gM(L))$ be the scheme of isomorphisms $L\otimes_{\Z_p} \gS \iso \gM(L)$ taking $s_{\alpha}$ to $\tilde s_{\alpha}.$ By \cite{BroshiJPAA}, Thm.~4.5 there is a natural isomorphism 
$\gM(L) \iso \calG\backslash \calP \times L$ where $\calG$ acts on $\calP\times L$ via $g\cdot(p,e) = (pg^{-1},ge).$
This implies that there is a  $\calG$-equivariant 
inclusion $\calP \subset \calP'|_{D^\times},$ so $\calP'|_{D^\times} = \calP$ is a trivial $\calG$-torsor. 
Hence $\calP'$ has a section over $D^\times,$ and the resulting isomorphism necessarily extends to $\Spec\gS.$ 
\end{proof}

\begin{cor}\label{keylemmaparahoric} Suppose that $G$ splits over a tamely ramified extension, and has no factors of type $E_8,$ and 
that $\calG = \calG_x$ for some $x\in \B(G,\Q_p),$ so that $\calG^\circ$ is a parahoric group scheme. Then there is an isomorphism 
$$ L\otimes_{\Z_p} \gS \xrightarrow{\sim} \gM(L)$$ 
taking $s_{\alpha}$ to $\tilde s_{\alpha}.$
\end{cor}
\begin{proof} By \cite{HainesRapoportAppendix} Prop.~3, $\calG^\circ(\Z_p) = \calG(\Z_p)\cap \ker\kappa_G.$ 
Hence, by Proposition \ref{WintenProp},  the action of $\Gamma_K$ on $L$ factors through 
$\Gg^\circ(\Z_p).$ Moreover, $\rH^1(D^\times, \calG^\circ) = \{1\},$ by Proposition \ref{trivialityoftorsors},
so the Corollary follows from Lemma \ref{keylemma}.
\end{proof}

\begin{para} Keep the assumptions introduced in (\ref{etaletensorsetup}).
Suppose that $L = T_p\ffG^{\vee},$ where $\ffG$ is a $p$-divisible group over $\O_K$ with special fibre $\ffG_0.$
We denote by $s_{\alpha,0} \in \Fil^0D_{\cris}(T_p\ffG^{\vee}\otimes_{\Z_p}\Q_p)^\otimes$ the $\varphi$-invariant tensors 
corresponding to $s_{\alpha,\et}$ via the $p$-adic comparison isomorphism.  

Assume from now on that $\calG = \calG_x$ for some $x\in \B(G,\Q_p),$ and that $G$ splits over tamely ramified extension, and 
has no factors of type $E_8.$ 
\end{para}

\begin{prop}\label{keylemmapdiv} We have $s_{\alpha,0} \in \DD(\ffG_0)(W)^\otimes,$ where 
we view $\DD(\ffG_0)(W) \subset D_{\cris}(T_p\ffG^{\vee}\otimes_{\Z_p}\Q_p)$ via the isomorphisms of Theorem \ref{classificationthm},
and the $s_{\alpha,0}$ lift to $\varphi$-invariant tensors $\tilde s_{\alpha} \in \DD(\ffG)(\whW(\O_K))$ 
which map into $\Fil^0\DD(\ffG)(\O_K)^\otimes.$

There exists an isomorphism 
$$ \DD(\ffG)(\whW(\O_K)) \iso \whW(\O_K)\otimes_{\Z_p}T_p\ffG^{\vee}$$
taking $\tilde s_{\alpha}$ to $s_{\alpha,\et}.$ 
In particular, there exists an isomorphism 
$$\DD(\ffG_0)(W) \iso W\otimes_{\Z_p}T_p\ffG^{\vee}$$ 
taking $s_{\alpha,0}$ to $s_{\alpha,\et}.$ 
\end{prop}
\begin{proof} Let $\gM = \gM(T_p\ffG^{\vee}),$ and $\tilde s_{\alpha} \in \gM^\otimes$ the tensors 
corresponding to $s_{\alpha,\et}.$ We may view $\tilde s_{\alpha}$ in $\DD(\ffG)(\whW(\O_K))^\otimes$ 
via the isomorphism 
$$\DD(\ffG)(\whW(\O_K)) \iso \whW(\O_K)\otimes_{\varphi,\gS} \gM$$ 
of Theorem \ref{classificationthm}. which also implies that these elements specialize to $s_{\alpha,0} \in \DD(\ffG_0)(W)^\otimes,$ and map into 
$\Fil^0\DD(\ffG)(\O_K)^\otimes.$

By Proposition \ref{keylemmaparahoric} there is an isomorphism 
$\gM \iso T_p\ffG^{\vee}\otimes_{\Z_p}\gS$ taking $\tilde s_{\alpha}$ to $s_{\alpha},$ and the remaining 
statements in the lemma follow from the isomorphism $\DD(\ffG)(\whW(\O_K)) \iso \whW(\O_K)\otimes_{\varphi,\gS} \gM.$
\end{proof}

\begin{para} Set $\DD = \DD(\ffG_0)(W).$ 
By Proposition \ref{keylemmapdiv}, we may identify the subgroup $\calG_W \subset \GL(\DD)$ defined by the $s_{\alpha,0}$ with $\calG\otimes_{\Z_p}W.$ 
This identification is independent of the choice of isomorphism $\DD\iso W\otimes_{\Z_p}T_p\ffG^{\vee}$ up to $\calG_W$-conjugacy. 
We write $G_{K_0} = \calG_W\otimes_WK_0.$
\end{para}

\begin{cor}\label{keylemmapdivcor} With the above assumptions and notation, let $s_{\alpha} \in \DD(\ffG)(\O_K)^\otimes$ denote the image of 
$\tilde s_{\alpha}.$ Then there exists an isomorphism 
$$\DD(\ffG)(\O_K) \iso \DD\otimes_W\O_K$$
taking $s_{\alpha}$ to $s_{\alpha_0}$ and lifting the identity on $\DD(\ffG_0)(k).$
In particular, there is a $G_{K_0}$-valued cocharacter $\mu_y$ such 
\begin{enumerate}
\item The filtration on $\DD\otimes_WK$ induced by the canonical isomorphism 
$$ \DD\otimes_WK \iso \DD(\ffG)(\O_K)\otimes_{\O_K}K$$
is given by a $G_{K_0}$-valued cocharacter $G_{K_0}$-conjugate to $\mu_y.$
\item $\mu_y$ induces a filtration on $\DD$ which lifts the filtration on 
$\DD\otimes_Wk = \DD(\ffG_0)(k).$
\end{enumerate}
\end{cor}
\begin{proof} By Proposition \ref{keylemmapdiv}, there is an isomorphism 
$i: \DD(\ffG)(\O_K) \iso \DD\otimes_W\O_K$ taking $s_{\alpha}$ to $s_{\alpha,0}.$ 
Since the scheme of such isomorphisms forms a $\calG_W$-torsor, we may assume that 
this isomorphism lifts the identity on $\DD(\ffG_0)(k),$ and we consider the induced 
filtration on $\DD\otimes_W\O_K.$ As above, since $s_{\alpha} \in \Fil^0\DD(\ffG)(\O_K)^\otimes,$ 
this filtration is given by a $G_{K_0}$-valued cocharacter $\mu_y,$ which satisfies (2) by construction.

As $i$ differs from the canonical map $\DD\otimes_WK \iso \DD(\ffG)(\O_K)\otimes_{\O_K}K$ 
by the action of an element   of $G_{K_0}(K),$ $\mu_y$ satisfies (1).
\end{proof}

\begin{para}\label{basicconstructionagain} 
Keep the assumptions above. 
We apply the construction of sections \ref{basicconstruction} and \ref{crystallinedefs} to the 
$p$-divisible group $\ffG_0$ equipped with the tensors $s_{\alpha,0}.$ Thus $P_0 \subset \GL(\DD\otimes_Wk)$ is a parabolic 
corresponding to the filtration on $\DD\otimes_Wk,$ and $P \subset \GL(\DD)$ a lifting of $P_0.$
The filtration in Corollary \ref{keylemmapdivcor} is given by a point $y \in \GL(\DD)/P,$ which reduces to $P_0,$ 
and we have the formal completions of the local models 
$$
\widehat {\rm M}^{\loc} = \widehat {\rm M}^{\loc}_y= \Spf R, \quad{\rm and}\quad  \widehat{\rm M}^{\loc}_G = 
\widehat{\rm M}^{\loc}_{G,y} = \Spf R_G,
$$ defined over $\O_E,$ 
corresponding to the orbit $G\cdot y \subset (\GL(\DD)/P)_{\O_K}$ which is defined over the reflex field $E/K_0$ of 
$\mu_y.$ 

Note that, by Proposition \ref{keylemmapdiv}, $(\DD, (s_{\alpha}))$ satisfies the condition (\ref{constancyoftensors}).
As $\calG_W = \calG\otimes_{\Z_p}W,$ we have $\rH^1(D^\times, \Gg^\circ_W) = \{1\},$ and (\ref{trivialityofGtorsors}) is satisfied.
We also assume from now on that 
\begin{equation}\label{finaletalecondition}
\text{$R_G$ is normal and $G$ contains the scalars.}
\end{equation} 
Then the assumptions in \ref{Frobeniuswithcycles}(2) are satisfied, and 
we may fix an isomorphism $\Psi: \wtM_1 \iso M$ lifting $\Psi_0,$ such that
$\Psi$ is constant modulo $\fa_{R_E},$ and such that its base change $\Psi_{R_G}$ to $\whW(R_G)$ respects the $s_{\alpha,0}.$
\end{para}

\begin{prop}\label{factorization} 
Let $\ffG'$ be a deformation of $\ffG_0$ defined over some finite extension $K/E$ such that 
\begin{enumerate} 
\item The filtration on $\DD\otimes_{K_0}K$ corresponding to $\ffG'$ is given by a $G$-valued 
cocharacter which is $G$-conjugate to $\mu_y.$
\item There exists Galois invariant tensors $s_{\alpha,\et}' \in (T_p\ffG^{\prime\vee})^\otimes$ 
which correspond to $s_{\alpha,0}$ under the $p$-adic comparison isomorphism.
\end{enumerate}
Then any morphism $R_E \rightarrow \O_{K}$ which induces $\ffG'$ factors through $R_G.$
\end{prop}
\begin{proof} By Lemma \ref{keylemmapdiv}, $\ffG'$ satisfies the conditions (1) and (2) of 
Proposition \ref{factorizationII}, which implies the present Proposition.
\end{proof}

\section{Shimura varieties and local models}

\subsection{Shimura varieties of Hodge type} 
\begin{para}\label{introHodgetype}

Let $\eG$ be a connected reductive group over $\mathbb Q$ and $X$ a conjugacy class of maps 
 of algebraic groups over $\RR$ 
$$ h: \mathbb S = \Res_{\C/\RR} \GG_m \rightarrow \eG_{\RR},$$ 
such that $(\eG,X)$ is a Shimura datum \cite{DeligneCorvallis} \S 2.1. 

For any $\C$-algebra $R,$ we have $R\otimes_{\RR}\C = R \times c^*(R)$ where $c$ 
denotes complex conjugation, and we denote  
by $\mu_h$ the cocharacter given on $R$-points by 
$$ R^\times \rightarrow (R \times c^*(R))^{\times} = (R\otimes_{\RR}\C)^\times = \mathbb S(R) \overset h \rightarrow \eG_{\C}(R).$$ 
We set $w_h = \mu_h^{-1}\mu_h^{c-1}.$

Let $\AA_f$ denote the finite adeles over $\Q,$ and $\AA^p_f \subset \AA_f$ the subgroup 
of adeles with trivial component at $p.$ Let $\eK = \eK_p\eK^p \subset \eG(\AA_f)$ 
where $\eK_p \subset \eG(\Q_p),$ and $\eK^p \subset \eG(\AA^p_f)$ are compact open subgroups.   

If $\eK^p$ is sufficiently small then 
$$\Sh_{\eK}(\eG,X)_{\C} = \eG(\Q) \backslash X \times \eG(\AA_f)/\eK $$
has a natural structure of an algebraic variety over $\mathbb C,$
which has a model,  $\Sh_{\eK}(\eG,X)$ over a number field ${\sf E} = E(\eG,X),$ which is the minimal field 
of definition of the conjugacy class of $\mu_h.$ 
We will always assume in the following that $\eK^p$ is sufficiently small that the quotient above exists as an algebraic variety.

We will sometimes consider the $\eE$-schemes 
$$
 \Sh(\eG,X) = \ilim\, \Sh_{\eK}(\eG,X),
 $$ 
and 
$$
 \Sh_{\eK_p}(\eG,X) = \ilim\, \Sh_{\eK}(\eG,X),
 $$ 
where $\eK$ runs through all compact open subgroups in the first limit and through all compact open subgroups with a fixed factor $\eK_p$ at $p$ in the second limit. 
These exist as the transition maps are finite, hence affine.
\end{para}

\begin{para}
Fix a $\Q$-vector space $V$ with a perfect alternating pairing $\psi.$ 
For any $\Q$-algebra $R,$ we write $V_R = V\otimes_{\Q}R.$ 
Let $\GSp = \GSp(V,\psi)$ be the corresponding group of symplectic similitudes, and let $S^{\pm}$ be the Siegel 
double space, defined as the set of maps $h: \mathbb S \rightarrow \GSp_{\RR}$ such that 
\begin{enumerate}
\item The $\C^\times$-action on $V_{\RR}$ gives rise to a Hodge structure of type $(-1,0), (0,-1):$ 
$$V_{\C} \iso V^{-1,0} \oplus V^{0,-1}.$$
\item $(x,y) \mapsto \psi(x, h(i)y)$ is (positive or negative) definite on $V_{\RR}.$
\end{enumerate}
\smallskip
\end{para}

\begin{para}
For the rest of this subsection we will assume that there is an embedding of Shimura data 
$\iota: (\eG,X) \hookrightarrow (\GSp, S^{\pm}).$ 
We will sometimes write  $G$ for $G_{\Q_p} = \eG\otimes_{\Q}\Q_p,$ when there is no risk of confusion.
We will assume from now on that the following conditions hold 
\begin{equation}\label{mainassumptions}
\text{\rm $G$ splits over a tamely ramified extension of $\Q_p$ and that $p\nmid |\pi_1(G^{\rm der} )|.$} 
\end{equation}

Fix  $x\in \B(G, \Q_p)$ and let $\calG=\Gg_x$ be the smooth $\Z_p$-group scheme with generic fibre $G,$ which is the stabilizer of $x,$ 
so that $\Gg^\circ$ is a parahoric group scheme.
\end{para}

\begin{para}
The table \cite{DeligneCorvallis} 1.3.9 shows that the symplectic 
representation $\iota$ is minuscule.
In \S \ref{subsec:mapsbetweenbuildings} we constructed a toral embedding 
$\B(G, \Q_p) \rightarrow \B(\GSp, \Q_p)$ associated to $\iota.$ 
For simplicity, we again denote by $\iota$ this embedding of buildings. 
Let $\mathcal{GSP}$ be the smooth $\Z_p$-group scheme defined by $\iota(x),$ and let $V_{\Z_p} \subset V_{\Q_p}$ 
be the $\Z_p$-lattice corresponding to the image of $x$ in $\B(\GL(V_{\Q_p}), \Q_p).$

 By Lemma \ref{miniscule} $\iota$ induces a closed 
embedding of $\Z_p$-group schemes $\calG \hookrightarrow \mathcal{GSP}.$ 
By the discussion in (\ref{newSymplectic}) and Corollary \ref{embedCor}, 
after replacing $\iota$ by another symplectic embedding, we may and do assume 
that $\mathcal{GSP}$ is the group scheme corresponding to  a lattice $V_{\Z_p} \subset V_{\Q_p}$ such that 
$V_{\Z_p} \subset V_{\Z_p}^{\vee},$ and that $\iota$ induces an embedding of local models ${\rm M}^{\loc}_{G,X} \hookrightarrow {\rm M}^{\loc}_{\GSp,S^{\pm}}.$

These models have a more concrete description: Let $P_{h^{-1}} \subset \GL(V_{\Z_p})$ be a parabolic defined over $\Z_p,$ and corresponding to a cocharacter in the conjugacy class of 
$\mu_h^{-1}$ for $h \in X.$  Let $\mu$ be a $G$-valued cocharacter,  defined over $\bar \Q_p,$ and in the $G$-conjugacy class of $\mu_h .$  
The orbit $G\cdot y \subset \GL(V_{\Z_p})/P_{h^{-1}}$, where $y$ is the filtration defined by $\mu^{-1}$, depends only on $X$ and not on the choice of $\mu,$ and is defined over $E.$ 
By Proposition \ref{immProp}, the $\O_E$-scheme $\rM^{\rm loc}_{G,X}$ agrees, (as a subscheme of $\rM^{\loc}_{\GSp,S^{\pm}}$) 
with the closure of $G\cdot \mu \subset \GL(V_{\Z_p})/P_{h^{-1}}.$
\end{para}

\begin{para}
Let $V_{\Z_{(p)}} = V_{\Z_p} \cap V_{\Q},$ and fix a $\Z$-lattice $V_{\Z} \subset V_{\Q}$ such that $V_{\Z}\otimes_{\Z}\Z_{(p)} = V_{\Z_{(p)}}$ 
and $V_{\Z} \subset V_{\Z}^\vee.$ The choice of lattice $V_{\Z}$ gives rise to an interpretation of $\Sh_{\eK'}(\GSp, S^{\pm})$ as a moduli 
space of polarized abelian varieties.

Consider the Zariski closure $G_{\Z_{(p)}}$ of $\eG$ in $\GL(V_{\Z_{(p)}})$; then $G_{\Z_{(p)}}\otimes_{\Z_{(p)}}\Z_p\cong \Gg$.
Set $\eK_p = \calG(\Z_p),$ and $\eK'_p = \mathcal{GSP}(\Z_p).$ We set $\eK = \eK_p\eK^p$ and similarly for $\eK'.$ 
By  \cite{KisinJAMS} Lemma 2.1.2, for any compact open subgroup $\eK^p \subset G(\AA^p_f)$ there exists 
$\eK^{\prime p} \subset \GSp(\AA^p_f)$ such that $\iota$ induces an embedding over $\sf E$
$$
\Sh_{\eK}(\eG,X) \hookrightarrow \Sh_{\eK'}(\GSp, S^{\pm}).
$$

\end{para}

\begin{para} We now introduce Hodge cycles. Fix a collection of  tensors $(s_{\alpha}) \subset V_{\Z_{(p)}}^\otimes$ 
whose stabilizer is $G_{\Z_{(p)}}.$ 
This is possible by \cite{KisinJAMS} Lemma 1.3.2. 

Let $h: \A \rightarrow \Sh_{\eK}(G,X)$ denote the restriction to $\Sh_{\eK}(G,X)$ of the universal abelian scheme, 
and let $\V = R^1h_*\Omega^\bullet$ be the de Rham cohomology of $\A.$ As in \cite{KisinJAMS} \S 2.2, the $s_{\alpha}$ give rise to a 
collection of absolute Hodge cycles $s_{\alpha,\dR} \in \V^\otimes,$ defined over the reflex field $\eE.$ 

Now let $\kappa \supset \eE$ be a field of characteristic $0,$ and $\bar \kappa$ an algebraic closure of $\kappa.$ 
Fix an embedding $\Q_p \hookrightarrow \C$ and an embedding of $E$-algebras 
$\sigma: \bar \kappa \hookrightarrow \C.$
Let $x \in \Sh_K(G,X)(\kappa)$ and denote by $\A_x$ the corresponding 
abelian variety over $\kappa.$  Denote by $H^1_B(\A_x(\C),\Q)$ the Betti cohomology of $\A_x(\C).$ 
Write $H^1_{\dR}(\A_x)$ for its de Rham cohomology and $H^1_{\et}(\A_{x,\bar \kappa}) = H^1_{\et}(\A_{x,\bar \kappa},\Q_p)$ 
for the $p$-adic \'etale cohomology of $\A_{x,\bar \kappa} = \A_x\otimes_\kappa\bar \kappa.$ 
The embedding $\sigma$ induces isomorphisms 
$$ H^1_{\dR}(\A_x)\otimes_{\kappa,\sigma}\C \iso H^1_B(\A_x(\C),\Q)\otimes_{\Q}\C \iso H^1(\A_{x,\bar \kappa},\Q_p)
\otimes_{\Q_p}\C. $$ 
Let $s_{\alpha,\dR,x}$ denote the fibre of $s_{\alpha,\dR}$ over $x,$ and 
$s_{\alpha,\et,x}  \in H^1_{\et}(\A_{x,\bar \kappa})^\otimes$ the image of $s_{\alpha,\dR,x}$ under the composite of the above 
two isomorphisms. As in \cite{KisinJAMS} Lemma 2.2.1 one sees that $s_{\alpha,\et,x}$ is $\Gal(\bar\kappa/\kappa)$-invariant and, in particular, independent of the choices made above.
\end{para}

\subsection{Integral models}\label{integralmodels}

\begin{para}\label{basicsetup}We keep the notation and assumptions introduced above.

Fix a prime $v|p$ of ${\eE},$ and let $\O$ be the ring of integers of $\eE,$ and $k_v$ the residue field of $v.$
The choice of lattice $V_{\Z}$ gives rise to an interpretation of $\Sh_{\eK'}(\GSp, S^{\pm})$ as a moduli space of 
polarized abelian varieties, and hence to a natural integral model $\SSh_{\eK'}(\GSp, S^{\pm})$ over $\Z_p,$ 
and hence over $\O_{(v)}.$  
We denote by $\SSh^-_{\eK}(\eG,X)$ the closure of $\Sh_{\eK}(\eG,X)$ in the $\O_{(v)}$-scheme 
$\SSh_{\eK'}(\GSp, S^{\pm}),$ and by $\SSh_{\eK}(\eG,X),$ the normalization of $\SSh_{\eK}(\eG,X)^-.$

Fix an algebraic closure $\bar \Q_p$ of $\Q_p,$ and an embedding $v: \eE \hookrightarrow \bar \Q_p.$ 
Let $E=\eE_v,$ so that $E$ is the local reflex field of $(\eG, \{\mu_h\}).$ 
We denote by $k$ the residue field of $\bar \Q_p $ and 
write $W = W(k)$ and $K_0 = W[1/p].$ Set $E ^{\ur} = E \cdot K_0$ in the completion of $\bar \Q_p.$ 

Let $K/E^{\ur}$ be a finite extension, and let $x \in \Sh_{\eK}(\eG,X)(K)$ be a point which admits 
a specialization $\bar x \in \SSh_{\eK}^-(\eG,X)(k).$ 
Let $\ffG_x$ denote the $p$-divisible group over the $\O_K$-valued point 
corresponding to $x,$ and $\ffG_{\bar x}$ its special fibre. Write $\DD_{\bar x} = \DD(\ffG_{\bar x})(W).$ 
Let $s_{\alpha,0} \in \DD_{\bar x}^\otimes\otimes_{\Z_p}\Q_p$ be the $\varphi$-invariant tensors corresponding to 
$s_{\alpha,\et,x}$ under the $p$-adic comparison isomorphism, and $G_{K_0} \subset \GL(\DD_{\bar x}\otimes_{\Z_p}\Q_p)$ 
the group defined by $(s_{\alpha,0}).$ The filtration on $\DD_K$ corresponding to $\ffG_x$ corresponds to a parabolic in $G_{K_0}\otimes_{K_0}K$ 
by \cite{KisinJAMS} Lemma 1.4.5 (in the terminology of {\it loc.~cit.}, this filtration is $G$-split).
This is induced by a $G$-valued cocharacter which lies in the $G$-conjugacy class of $\mu_h^{-1}$.

We use the notation of \ref{basicconstructionagain} applied with $\ffG = \ffG_x.$ 
Thus we have a parabolic subgroup $P \subset \GL(\DD_{\bar x}),$ and a point $y = y(x) \in (\GL(\DD_{\bar x})/P)(K),$ 
which specializes to $P_0 = P\otimes_W k,$ and is induced by a $G$-valued cocharacter $\mu_y,$ which is conjugate to $\mu_{h^{-1}}.$
We obtain formal local models $\widehat{\rm M}^{\loc}_y = \Spf R$ and $\widehat{\rm M}^{\loc}_{G,y} = \Spf R_G$ defined over $\O_E,$ 
the latter being obtained by completing the orbit closure $\rM^{\loc}_G: = \overline{G_{K_0}\cdot y(x)} \subset \GL(\DD_{\bar x})/P$ 
at the specialization of $y.$  
\end{para}

\begin{prop}\label{structureofbranches} Let $\widehat U_{\bar x}$ be the completion of $\SSh_{\eK}^-(G,X)_{\O_{E^{\ur}}}$ at $\bar x.$ 
Then the irreducible component of $\widehat U_{\bar x}$ containing $x$ is isomorphic to $\widehat{\rm M}^{\loc}_{G,y}$ 
as formal schemes over $\O_{E^{\ur}}.$
\end{prop}
\begin{proof} Recall that we are assuming that $G$ splits over a tamely ramified extension, and that $\calG^\circ$ is a parahoric group scheme. 
Note that $G_{K_0}  \subset \GL(\DD_{\bar x}\otimes_{\Z_p}\Q_p)$ contains the scalars, since 
$G \subset \GL(V_{\Q})$ contains the image of $w_h,$ and $R_G$ is normal by Theorem \ref{normalThm}. 
It follows that the conditions imposed in the construction of (\ref{basicconstructionagain}) are satisfied, and  we can equip $\DD\otimes_WR$ with the structure of a 
Dieudonn\'e display over $R$ satisfying the conditions in \ref{Frobeniuswithcycles}.

In particular, this construction allows us to view $R$ as a versal deformation ring for $\ffG_{\bar x},$ so there exists 
a map $\Theta:\widehat U_{\bar x} \rightarrow \widehat{\rm M}^{\loc}_y$ such that the $p$-divisible group corresponding to the chosen 
Dieudonn\'e display over $R$  pulls back to the $p$-divisible group over $\widehat U_{\bar x}$ arising from the universal 
family of abelian schemes over $\widehat U_{\bar x}.$ By the Serre-Tate theorem, $\Theta$ is a closed embedding, and 
it suffices to show that it  factors through $\widehat {\rm
  M}^{\loc}_{G,y}$ since both $\widehat{\rm M}^{\loc}_{G,y}$ and
$\widehat U_{\bar x}$ have the same dimension.

Let $K' \supset E^{\ur}$ be any finite extension and $x' \in \widehat U_{\bar x}(K')$ a point lying on the same irreducible 
component of $\widehat U_{\bar x}$ as $x.$ The same argument as in \cite{KisinJAMS} Proposition 2.3.5 shows 
that $s_{\alpha,\et,x'}$ corresponds to $s_{\alpha,0}$ under the $p$-adic comparison isomorphism for the $p$-divisible 
group $\ffG_{x'}$: Let $\mathcal U$ and $\widehat U_{\bar x}^{\an}$ denote the analytic spaces over $E^{\ur}$ attached to 
$\Spf R$ and $\widehat U_{\bar x}$ respectively \cite[\S 7]{deJongDieudonne}. 
Since $\DD_R = \DD\otimes_WR$ underlies an $F$-isocrystal on $R,$ the sections $s_{\alpha,0}$ extend uniquely to parallel sections $\tilde s_{\alpha,0} \in \DD_R^\otimes|_{\mathcal U}$ 
\cite[3.1]{KatzTravaux}. The isomorphism $H^1_{\cris}(\A_{\bar x}/W)\otimes K' \iso H^1_{\dR}(\A_{x'})$ takes $s_{\alpha,0}$ to $\tilde s_{\alpha,0}|_{x'}$ \cite[\S 2.9]{BerthelotOgus}. 
Now $(\tilde s_{\alpha,0} - s_{\alpha,dR})|_{\widehat U_{\bar x}^{\an}}$ is a parallel section of $\DD_R^\otimes|_{\widehat U_{\bar x}^{\an}}$ which vanishes at $x$ by construction. Hence it vanishes 
on the irreducible component of $\widehat U_{\bar x}^{\an}$ containing $x,$ and in particular at $x',$ so that  $\tilde s_{\alpha,0}|_{x'} = s_{\alpha,\dR,x'}.$ 
Finally $s_{\alpha,dR,x'}$ and $s_{\alpha,\et,x'}$ correspond under the $p$-adic comparison isomorphism \cite{Blasiusmotives}.

Since the filtration on $\DD_{\bar x}\otimes_{K_0}K'$ corresponding to $\ffG_{x'}$ is given by a 
cocharacter which is conjugate to $\mu_h^{-1}.$ It follows from Lemma \ref{factorization} that $x'$ is induced by a 
point of $\widehat{\rm M}^{\loc}_{G,y}.$ Since this holds for any $x',$ it follows that $\Theta$ factors through $\widehat{\rm M}^{\loc}_{G,y}.$
\end{proof}

\begin{para} Let $k_E$ denote the residue field of $E.$ If $k'/k_E$ is an extension of perfect fields, and $z \in \rM^{\loc}_{G,X}(k'),$ we denote 
by $\rM^{\loc}_{G,X}\otimes_{W(k_E)}W(k')$ by $\rM^{\loc}_{G,X,z},$ and by 
$\widehat \rM^{\loc}_{G,X,z},$ the completion of $\rM^{\loc}_{G,X,z},$ at the image of $z.$
\end{para}

\begin{cor} Let $\bar x \in \SSh_{\eK}(\eG,X)$ be a closed point of characteristic $p.$ Then there exists 
$z \in \rM^{\loc}_{G,X}(k)$ such that $\widehat U_x$ is isomorphic to $\widehat \rM^{\loc}_{G,X,z}$ over $\O_{E^{\ur}}.$  
\end{cor}
\begin{proof} By Proposition \ref{keylemmapdiv} we may identify the subgroup $\calG_W \subset \GL(\DD)$ with the 
pullback to $\O_{E^{\ur}}$ of $\calG \subset \GL(V_{\Z_p}).$ Hence the lemma follows from Proposition \ref{structureofbranches}. 
The fact that $y$ corresponds to a parabolic of $G$ was already remarked above. 
\end{proof}

\begin{para} 
We continue to assume that $G$ splits over a tamely ramified extension of $\Q_p$ and that 
$p$ does not divide the order of $\pi_1(G^{\rm der} )$.
The relationship between the integral model $\SSh_{\eK}(\eG,X)$ and local models can be globalized. 
To explain this, recall that we have the bundle $\V = R^1h_*\Omega^\bullet$ over $\Sh_{\eK}(\eG, X)$ given by first de Rham cohomology
of the universal abelian scheme and a collection of absolute Hodge cycles $s_{\alpha,\dR} \in \V^\otimes,$ all defined over the reflex field $\eE.$ The bundle $\V$ extends to a bundle $\und \V$ over the $\O_{{\sf E}_{(v)}}$-scheme $\SSh_{\eK}(\eG, X)$. 

Consider now the $G$-torsor $\widetilde{\Sh}_{\eK}(\eG, X)$ over $\Sh_{\eK}(\eG, X)$ classifying
 trivializations  $f: V^{\vee}  \xrightarrow{\sim} \V$ that preserve the tensors, \emph{i.e.} with $f^{\otimes}(s_{\alpha})=s_{\alpha, {\rm dR}}$.

\begin{prop}\label{torsorSh} 
 The $s_{\alpha, {\rm dR}}$ extend to tensors  $s_{\alpha, {\rm dR}}\in \und \V^\otimes$ over $\SSh_{\eK}(\eG, X)$.
 The scheme $\widetilde {\SSh}_{\eK}(\eG, X)$ that classifies trivializations $f: V_{\Z_{(p)}}^{\vee}\xrightarrow{\sim} \und \V$ with
  $f^{\otimes}(s_{\alpha})=\und s_{\alpha, {\rm dR}}$, is a $\Gg$-torsor over $\SSh_{\eK}(\eG, X)$.
\end{prop}

\begin{proof} As in the proof of Corollary \ref{isomoverR},
since $\SSh_{\eK}(\eG, X)$ is normal, to show that   $s_{\alpha, {\rm dR}}$    belongs to $\und \V^\otimes$,
it is enough to check that for every $\xi: \Spec(\O_K)\to \SSh_{\eK}(\eG, X)$ with $K$ a finite extension of $E$, 
$s_{\alpha, {\rm dR},\xi}$ is in $\xi^*(\und \V)$ (where we again denote by $\xi$ the $K$-valued point corresponding to $\xi$). 
A result of Blasius and Wintenberger \cite{Blasiusmotives} asserts that the $p$-adic comparison 
isomorphism takes $s_{\alpha,\et,\xi}$ to $s_{\alpha,\dR,\xi}$ 
\footnote{Indeed this result was already used implicitly via the citation of \cite{KisinJAMS} in the proof of 
Lemma \ref{structureofbranches} above.}. 
Let $\ffG_{\xi}$ denote the pullback via $\xi$ of the universal $p$-divisible group over 
$\SSh_{\eK}(\eG, X).$ If $\gM = \gM(T_p\ffG_{\xi}^\vee)$ then by Theorem \ref{classificationthm}, we 
have 
$$ s_{\alpha, {\rm dR},\xi} \in \O_K\otimes_{\gS}\varphi^*(\gM)^\otimes \iso \DD(\ffG_{\xi})(\O_K)^\otimes 
\iso \xi^*(\und \V)^\otimes.$$

It follows by Proposition \ref{keylemmapdiv} that $\xi^*(\widetilde {\SSh}_{\eK}(\eG, X))$ is a 
$\Gg$-torsor. Arguing as in the proof of Corollary \ref{isomoverR}, we see 
that $\widetilde {\SSh}_{\eK}(\eG, X)$ is a $\Gg$-torsor.
\end{proof}

\begin{thm}\label{localmodeldiagramThm}
Under the above assumptions, there exists a diagram of morphisms
\begin{equation}\label{locmoddiagram}
\xymatrix{
& {\widetilde\SSh_{\eK}(\eG, X)_{\O_E}} \ar[ld]_{\pi}\ar[rd]^{q} & \\
{\quad\quad  \SSh_{\eK}(\eG, X)_{\O_E}}   \quad\quad & &{\ \ \ {\rm M}^{\rm loc}_{G,X}\ }\, , \ \ \ \ 
 } \ \ \ \ \ 
\end{equation}
of $\O_{E}$-schemes, in which:
\begin{itemize}
\item $\pi$ is the $\Gg$-torsor given by Proposition \ref{torsorSh},
\item $q$ is $\Gg$-equivariant and smooth of relative dimension $\dim G$.
\end{itemize}
\end{thm}

\begin{proof}  Let $K/E^{\ur}$ be a finite extension, and $x \in \Sh_{\eK}(\eG,X)(K)$ be a point which admits 
a specialization $\bar x \in \SSh_{\eK}^-(\eG,X)(k).$ We use the notation introduced in \ref{basicsetup}. 
In particular, we have the orbit closure $\rM^{\loc}_G : = \overline{G_{K_0}\cdot y(x)} \subset \GL(\DD_{\bar x})/P.$
By Proposition \ref{WintenProp} and Lemma \ref{keylemmapdiv}, we have $s_{\alpha,0,\bar x} \in \DD_{\bar x}^\otimes$ and  
if $\widetilde \rM^{\loc}_G$ is the scheme over $\rM^{\loc}_G$ which parametrizes isomorphisms 
$f:\DD_{\bar x} \iso V_{\Z_p}^{\vee}\otimes_{\Z_p}W$ such that $f^\otimes(s_{\alpha,0,\bar x}) = s_{\alpha},$
then $\widetilde \rM^{\loc}_G = \calP \times \rM^{\loc}_G,$ where $\calP$ is a trivial $\calG_W$-torsor.
In particular $\widetilde \rM^{\loc}_G$ is a $\calG$-torsor over $\rM^{\loc}_G.$
We define a map of $\O_{E^{\ur}}$-schemes $q^{\loc}: \widetilde \rM^{\loc}_G \rightarrow \rM^{\loc}_{G,X}$ by taking $(f,\F)$ to $f^{-1}(\F).$
One sees easily that $q^{\loc}$ is a $\calG$-torsor. Thus we have a diagram 
\begin{equation}\label{locmoddiagramII}
\xymatrix{
& {\widetilde \rM^{\loc}_G} \ar[ld]_{\pi^{\loc}}\ar[rd]^{q^{\loc}} & \\
{\quad\quad \rM^{\loc}_G} \quad\quad & & {\ \ \ \rM^{\loc}_{G,X} }\, . \ \ \ \ 
 } \ \ \ \ \ 
\end{equation}
with $\pi^{\loc}$ and $q^{\loc}$ are $\calG$-torsors.

To construct the morphism $q,$ let $(x, f)$ be an $S$-valued point of $\widetilde \SSh_{\eK}(\eG, X)$.
We send $(x, f)$  to the inverse image $f^{-1}(\F)\subset V_{\Z_p}^{\vee}\otimes \O_S$ of the Hodge filtration $\F\subset \und \V(S)=R^1h_*\Omega^\bullet_{A/S},$ 
which is an $S$-valued point of $\GL(V_{\Z_p})/P_h.$ Now consider the diagram of $\O_E$-schemes 
\begin{equation}\label{locmoddiagramIII}
\xymatrix{
& {\widetilde\SSh_{\eK}(\eG, X)_{\O_E}} \ar[ld]_{\pi}\ar[rd]^{q} & \\
{\quad\quad  \SSh_{\eK}(\eG, X)_{\O_E}}   \quad\quad & &{\ \ \ \GL(V_{\Z_p})/P_h }\, . \ \ \ \ 
 } \ \ \ \ \ 
\end{equation}
We have to show that $q$ factors through $\rM^{\loc}_{G,X}$ and that the resulting morphism to $\rM^{\loc}_{G,X}$ is smooth of relative dimension $\dim G.$ 
To do this it suffices to show these properties for the corresponding morphism in the diagram of $\O_{E^{\ur}}$-schemes 
\begin{equation}\label{locmoddiagramIV}
\xymatrix{
& {\widehat U_{\bar x}\times_{\SSh_{\eK}(\eG, X)_{\O_E}}\widetilde\SSh_{\eK}(\eG, X)_{\O_E}} \ar[ld]_{\pi}\ar[rd] & \\
{\quad\quad \widehat U_{\bar x}}   \quad\quad & &{\ \ \ \GL(V_{\Z_p})/P_h }\, . \ \ \ \ 
 } \ \ \ \ \ 
\end{equation}
Here we have written $\widehat U_{\bar x}$ for the affine scheme with the same affine ring as the formal scheme $\widehat U_{\bar x},$ 
and the map on the left of the diagram is obtained by pulling back $\pi$ by $\widehat U_{\bar x} \rightarrow \SSh_{\eK}(\eG, X)_{\O_E}.$
One sees directly from the definitions that this last diagram can be identified with the one obtained from \ref{locmoddiagramII} by 
pulling back $\pi^{\loc}$ by the isomorphism $\widehat U_{\bar x} \iso \widehat \rM^{\loc}_G \rightarrow \rM^{\loc}_G,$ given by Proposition \ref{structureofbranches}. 
Since $q^{\loc}$ has the required properties, so does $q.$
\end{proof}

\begin{cor}\label{localstr}  Under the above assumptions, the scheme $\SSh_{\eK}(\eG, X)$ has reduced special fibre.
If $\eK_p=\eK_p^\circ$,  i.e. $\eK_p$ is parahoric, then the geometric
special fibre $\SSh_{\eK}(\eG, X)\otimes_{\O_{E}}k$
 admits a stratification with locally closed strata parametrized by the $\mu$-admissible
 set
 of Kottwitz and Rapoport (e.g. \cite{PaZhu}, 9.1.2); the closure of each stratum is normal and Cohen-Macaulay.
\end{cor}
\begin{proof}
This follows from the existence of the diagram (\ref{locmoddiagram})
and \cite{PaZhu} Theorem 1.1 by the standard argument (see \emph{loc. cit.} Theorem 1.2.) Indeed,
the stratification is obtained from a $\Gg$-stratification on the geometric special fibre
$\rM^{\rm loc}_{G,X}\otimes_{\O_E}k$
which is given by realizing this as a union of affine Schubert varieties in an affine
Grassmannian.  
\end{proof}

\begin{cor}\label{henselizations}
Under the above assumptions, including $\eK_p=\eK^\circ_p$, given a point $z\in \SSh_{\eK}(G, X)(\mF_q)$, $ k_E\subset \mF_q$,
there is $w \in \rM^{\rm loc}_{G,X}(\mF_q)$,  well defined up to the action of $\Gg(\mF_q)$ on $ \rM^{\rm loc}_{G,X}(\mF_q)$,
such that we have an isomorphism of henselizations
$$
\O_{\SSh_{\eK}(G, X), z}^{\rm h}\simeq \O_{\rM^{\rm loc}_{G,X}, w}^{\rm h}.
$$
\end{cor}
\begin{proof} For simplicity, set $\SSh=\SSh_{\eK}(\eG, X)$.
Lang's lemma applied to the torsor $\pi_{\mF_q}$ for the smooth connected 
group scheme $\Gg_{\mF_q}$ implies that there is $\tilde z\in 
\widetilde \SSh(\mF_q)$ that lifts $z$, \emph{i.e.} $\pi(\tilde z)=z$, and we take $w=q(\tilde z)$.
Since both $\pi$ and $q$ are smooth,
we there is a section $\Spec(\O^{\rm h}_{\SSh, z})\to \widetilde\SSh$ which extends $\tilde z$ and 
is such that the composition 
$b: \Spec(\O^{\rm h}_{\SSh, z})\to \widetilde\SSh\to \rM^{\rm loc}_{G,X}$ induces
an injection of tangent spaces at $z$ and $w$. This injection is necessarily 
an isomorphism and so $b$ is formally
\'etale. The result follows. 
\end{proof}

\begin{Remark}\label{remarkHodge} {\rm a) We expect that the map $q: \widetilde\SSh_{\eK}(\eG, X)\to {\rM}^{\rm loc}_{G,X}$
is surjective. However, we don't know how to uniquely characterize the model $\SSh_{\eK}(\eG, X)$ of $\Sh_{\eK}(\eG, X)$ in general, even after assuming this statement. 
However, see \ref{RemarkCharacterize} for a partial result in this direction.

b) Assume that $G$ is unramified over $\Q_p$, \emph{i.e.}   quasi-split over $\Q_p$ and split
over an unramified extension of $\Q_p$. Then given $x\in \B(G, \Q_p)$, such that 
$\Gg^\circ_x(\Z_p)$ is contained in a hyperspecial subgroup, there is $x'\in \B(G, \Q_p)$
such that $\Gg^\circ_x=\Gg_{x'}$. (One can take $x'$ to be a generic point
of the smallest facet that contains $x$, see \emph{e.g.} Lemma 7.0.2 of the Corrigendum for \cite{HainesBC}.) 
Therefore, in this case, Theorem \ref{localmodeldiagramThm} and its corollaries 
can be applied to Shimura varieties for all such parahoric subgroups of $G(\Q_p)$. 

}
\end{Remark}
\end{para}

\subsection{Integral models for parahoric level}

\begin{para} We will use the results of the previous section to construct integral models for Shimura varieties of Hodge type with {\em parahoric} 
level structure. That is, where the level structure at $p$ is given by $\calG^\circ(\Z_p).$ 
We keep the notation introduced above, and write $\eK_p^\circ = \calG^\circ(\Z_p)$ and $\eK^\circ = \eK_p^\circ\eK^p.$
 We denote by $\tilde G$ the universal cover of $G^{\der}.$
We begin with two lemmas.
\end{para}

\begin{lemma}\label{parahoricreciprocitylemma}
The composite of the maps 
$$ E^\times \overset {[\mu_h^{-1}]} \rightarrow G(E)/\tilde G(E) \rightarrow G(\Q_p)/\tilde G(\Q_p) \rightarrow G(\Q_p)/\tilde G(\Q_p)\eK^\circ_p.$$ 
is trivial on $\O_{E}^\times.$ Here the first map is induced by the conjugacy class of $\mu_h^{-1}$ for $h \in X,$ and the second map is given by the norm 
$N_{E/\Q_p}.$
\end{lemma}
\begin{proof} Let $E_0 \subset E$ be the maximal unramified subfield of $E,$ and set $\eK^\circ_{p,E_0} = \calG^\circ_x(\O_{E_0}).$ It suffices to show that the composite 
$$ E^\times \overset {[\mu_h^{-1}]} \rightarrow G(E)/\tilde G(E) \overset {N_{E/E_0}} \rightarrow G(E_0)/\tilde G(E_0)\eK^\circ_{p,E_0},$$ kills $\O_E^\times,$ 
since then the lemma follows by applying $N_{E_0/\Q_p}.$

To show this, we may replace $E$ by $E\cdot E_0',$ where $E_0'$ is a finite unramified extension of $E_0,$ 
and assume that $G$ is quasi-split over $E_0.$ 
Let $T$ be the centralizer of a maximal split torus in $G_{E_0}.$ Then $[\mu_h^{-1}]$ contains 
a cocharacter $\mu \in \xcoch(T),$ defined over $E.$
After replacing $\eK^\circ_{p,E_0}$ by a conjugate subgroup, 
we may assume that the point $x \in \calB(G,E_0)$ defining $\calG^\circ$ is in the apartment corresponding to $T.$

Let $\calT^\circ$ denote the connected N\'eron model of $T.$ Consider the composite 
$$ R_{E/E_0} \GG_m \overset \mu \rightarrow R_{E/E_0} T \overset {N_{E/E_0}} \rightarrow T.$$ 
The corresponding map on $E_0$-points sends $\O_E^\times$ to a bounded subgroup of $T(E_0).$ 
If $\Gamma_{E_0} = \Gal(\bar \Q_p/E_0),$ then $\pi_1(R_{E/E_0} \GG_m)_{\Gamma_{E_0}} = \mathbb Z$ is torsion free, 
and in particular, the image of $\O_E^\times$ in $\pi_1(R_{E/E_0} \GG_m)_{\Gamma_{E_0}}$ and $\pi_1(T)_{\Gamma_{E_0}}$ 
is trivial. Hence the above map sends $\O_E^\times$ into $\calT^\circ(\O_{E_0}).$
Since $\calT^\circ(\O_{E_0}) \subset \eK^\circ_{p,E_0},$ the lemma follows.
\end{proof}
\smallskip

\begin{para}
Now let $C = \ker(\tilde G \rightarrow G^{\der}).$ For $c \in H^1(\Q,C),$ and $l$ a finite prime, 
denote by $c_l \in H^1(\Q_l,C)$ the image of $c.$ 
Until further notice, we assume that $G$ satisfies the following condition.
\begin{equation}\label{cohvanishingcondn}
\text{If $c \in H^1(\Q,C)$ satisfies $c_l= 0$ for $l \neq p,$ then $c_p = 0.$}
\end{equation}
The condition will be removed at the end, and so does not appear in our final result.
\end{para}

\begin{lemma}\label{levelstr}
For $\eK^p$ sufficiently small 
$$ G(\Q)\cap \eK\subset \eK^\circ.$$ 
\end{lemma}
\begin{proof} Let $\rho: \tilde G \rightarrow G$ denote the natural map. 
By \cite[Cor. 2.0.5, 2.0.13]{DeligneCorvallis}, we may choose $\eK^p$ sufficiently small that  
$$\eK^p \cap G(\Q) \subset (\rho\tilde G(\AA_f^p)\cap G(\Q))\cdot U \subset G(\AA_f^p),$$
where $U$ is any subgroup of finite index in the $p$-units in $Z_G(\Q).$

By \cite[Prop. 2.0.4(ii)]{DeligneCorvallis}, our assumption that
(\ref{cohvanishingcondn}) holds implies that 
$$ \rho\tilde G(\AA_f^p)\cap G(\Q) = \rho\tilde G(\AA_f)\cap G(\Q).$$
Here the intersection on the left (resp.~right) is taken in $G(\AA_f^p)$ (resp. $G(\AA_f)$).
Thus $\eK \cap G(\Q) \subset \rho\tilde G(\AA_f)\cdot U.$ In particular for $U$ and $\eK^p$ sufficiently small 
$\eK \cap G(\Q) \subset \eK^\circ$ by \cite{HainesRapoportAppendix}, Prop.~3.
\end{proof}

\begin{para} Let $G(\Q)_+$ denote the preimage of $G^{\ad}(\R)^+$ in $G(\Q),$ and let 
$G(\Q)_+^-$ be the closure of $G(\Q)_+$ in $G(\AA_f).$ 
Denote by $\SSh_{\eK^\circ}(G,X)$ the normalization of  $\SSh_{\eK}(G,X)$
in $\Sh_{\eK^\circ}(G,X)$. 
\end{para}

\begin{prop}\label{parlevelstr} If  $\eK^p$  satisfies the smallness assumption imposed in 
(\ref{introHodgetype}) then the covering 
$ \SSh_{\eK^\circ}(G,X) \rightarrow \SSh_{\eK}(G,X)$
is \'etale. If $\eK^p$ is sufficiently small, this covering splits over an unramified extension of $\O_E.$
\end{prop}
\begin{proof} By \cite{DeligneCorvallis} 2.1.3.1, the connected components of 
$\Sh_{\eK^\circ}(G,X)$ (resp.~$\Sh_{\eK}(G,X)$) over $\bar \Q_p$ form a torsor under 
$G(\AA_f)/G(\Q)_+\eK^\circ$ (resp.~$G(\AA_f)/G(\Q)_+\eK$), which is an abelian group.
Suppose that $\eK^p$ is sufficiently small, so that the conclusion of Lemma \ref{levelstr} holds. 
Then 
$$\pi_0( \Sh_{\eK^\circ}(G,X)_{\bar \Q_p}) \rightarrow \pi_0(\Sh_{\eK}(G,X)_{\bar \Q_p}) $$
is a torsor under 
$$ G(\Q)_+\eK/G(\Q)_+\eK^\circ = G(\Q)_+\eK^p\eK_p/G(\Q)_+\eK^p\eK^\circ_p = \eK_p/(G(\Q)_+\eK^\circ\cap\eK_p) = \eK_p/\eK_p^\circ. $$
As $\eK_p/\eK_p^\circ$ transitively on the geometric fibres of 
\begin{equation}\label{parlvlstreqnI}
\Sh_{\eK^\circ}(G,X) \rightarrow \Sh_{\eK}(G,X)
\end{equation}
 this implies that (\ref{parlvlstreqnI}) is a $\eK_p/\eK_p^\circ$-torsor 
which becomes trivial over $\bar \Q_p.$ Lemma \ref{parahoricreciprocitylemma}   together with \cite{DeligneCorvallis}, 
Thm.~2.6.3, which 
describes the action of $\Gal(\bar \Q/\eE)$ on the geometrically connected components of $\Sh_{\eK^\circ_p}(G,X),$ 
now imply that this torsor actually splits after base changing to an unramified extension 
of $E$. This proves the  Proposition for $\eK^p$ sufficiently small.

To show the Proposition for any $\eK^p$ (still satisfying the smallness assumption imposed in 
(\ref{introHodgetype})), Let $\eK^{p\prime} \subset \eK^p$ such that the Proposition holds 
for $\eK' = \eK^{p\prime}\eK_p$ and $\eK^{\prime\circ} = \eK^{p\prime}\eK_p^\circ.$ Then, by what 
we have shown above together with Proposition \ref{structureofbranches}, the maps 
$$ \SSh_{\eK^{\prime\circ}}(G,X) \rightarrow \SSh_{\eK'}(G,X) \rightarrow \SSh_{\eK}(G,X)$$ 
are finite \'etale. Since the composite of these maps factors through 
$\SSh_{\eK^\circ}(G,X),$ it follows that 
$$ \SSh_{\eK^{\circ}}(G,X) \rightarrow \SSh_{\eK}(G,X) $$
is finite \'etale.
\end{proof}

\begin{cor}\label{geometricallyconnectedcomponentsunramified}
  The geometrically connected components of $\SSh_{\eK^\circ_p}(G,X)$ are defined over the maximal extension 
of $\eE$ that is unramified over primes dividing $p.$
\end{cor}
\begin{proof} This follows from \ref{parahoricreciprocitylemma}, as well as \cite{DeligneCorvallis}, Thm.~2.6.3, which 
describes the action of $\Gal(\bar \Q/\eE)$ on the geometrically connected components of $\Sh_{\eK^\circ_p}(G,X).$
\end{proof}

\begin{para} The pullback of the torsor $\pi$ introduced in Theorem \ref{localmodeldiagramThm}, by the morphism 
$\SSh_{\eK^{\circ}}(G,X) \rightarrow \SSh_{\eK}(G,X)$ produces a $\calG$-torsor 
$$\pi^\circ: \widetilde \SSh_{\eK^{\circ}}(G,X) \rightarrow \SSh_{\eK^{\circ}}(G,X).$$ 
We conjecture that this $\calG$-torsor has a reduction to a $\calG^\circ$-torsor, although we are unable to prove this.
\end{para}

\subsection{Twisting Abelian Varieties}

\begin{para} In the next three subsections, we deduce the consequences of the above results for Shimura varieties of abelian type. 
Many of the arguments of \cite{KisinJAMS} \S 3 in the hyperspecial case go over unchanged, so we discuss in detail only those points which 
do not. One of these concerns the definition of the action of $\eG^{\ad}(\bbQ)^+$ on the models $\SSh_{\eK_p}(\eG,X)$ constructed above. 
In the hyperspecial case the models $\SSh_{\eK_p}(\eG,X)$ satisfy Milne's extension property. This implies the action of $\eG^{\ad}(\bbQ)^+$ 
on the generic fibre extends to the whole model, and it sufficed in \cite{KisinJAMS} to give a description of this action on the level of abelian varieties 
up to isogeny. In the case considered here, we do not have an analogue of the extension property, and we need to give a direct description of the action 
of $\eG^{\ad}(\bbQ)^+.$ This requires a refined form of the twisting construction in \S 3.1 of {\em loc.~cit.}
\end{para}

\begin{para} Let $A$ be a commutative ring with identity, $Z$ a flat, affine group scheme over $\Spec A,$ and  $\calP$ a $Z$-torsor. 
Note that by flat base change, the coherent cohomology of $\calP$ vanishes, so $\calP$ is affine. 
We write $\O_{Z}$ and $\O_{\calP}$ for the affine rings of $Z$ and $\calP$ respectively. If $M$ is an $A$-module, a $Z$-action on $M$ 
is a map of fppf sheaves $Z \rightarrow \SAut M.$ Giving a $Z$-action on $M$ is equivalent to giving $M$ the structure of 
an $\O_Z$-comodule. For any such $M$ the subsheaf $M^Z$ may be regarded as an $A$-submodule of $M$ by descent. 
\end{para}

\begin{lemma}\label{twistbyaffinetorsor}
With the notation above, the natural map 
\begin{equation}\label{periodeqn}
(M\otimes_A\O_{\calP})^Z\otimes_A\O_{\calP} \rightarrow M\otimes_A\O_{\calP} 
\end{equation}
is an isomorphism
\end{lemma}
\begin{proof} Let $\pi_1, \pi_2$ be the morphisms $Z \times_{\Spec A} \calP \rightarrow \calP$ given by 
sending $(z,h)$ to $zh$ and $h$ respectively. A semi-linear action of $Z$ on an $\O_{\calP}$ module $N$ 
gives rise to an isomorphism $\pi_1^*(N) \iso \pi_2^*(N),$ which via the isomorphism 
$Z\times \calP \iso \calP\times \calP$ is nothing but a descent datum for the morphism $\calP \rightarrow \Spec A.$ 
We apply this to $N = M\otimes_A\O_{\calP}.$ 
The lemma now follows by faithfully flat descent, since $Z$ is flat over $A$ and hence so is $\calP.$
\end{proof}

\begin{para} We now suppose that $Z$ is of finite type, and that $A \subset \Q.$  
For $S$ a scheme we define the $A$-isogeny category of abelian schemes 
over $S$ to be the category obtained from the category of abelian schemes over $S$ by tensoring the $\Hom$ groups by 
$\otimes_{\mathbb Z}A.$ An object $\calA$ in this category is called an {\it abelian scheme up to $A$-isogeny} over $S.$ 
For $T$ an $S$-scheme we set $\calA(T) = \Hom_S(T,\calA)\otimes_{\mathbb Z} A.$ 

Let $\calA$ be an abelian scheme up to $A$-isogeny over $S.$ 
Denote by $\SAut_A(\calA)$ the $A$-group whose points in an $A$-algebra $R$ are given by 
$$\SAut_A(\calA)(R) = ((\End_S \calA)\otimes_{\Z}R)^\times.$$ 

Let $Z$ and $\PP$ be as above, and suppose that we are given a map of $A$-groups 
$Z \rightarrow \SAut_A(\calA).$ We define a pre-sheaf $\calA^{\PP}$ in the fppf topology of $S$ by setting 
$$ \calA^{\PP}(T) = (\calA(T)\otimes_{\Q}\O_{\PP})^Z.$$
\end{para}

\begin{lemma}\label{twistav}
$\calA^{\PP}$ is a sheaf, represented by an abelian scheme up to $A$-isogeny.
\end{lemma}
\begin{proof} By a result of Moret-Bailly, \cite{MBAnnENS} Thm.~1.6, there exists a finite, integral, torsion free $A$-algebra $A'$ such that 
$\PP(A')$ is non-empty. Specializing (\ref{periodeqn}) by the map $\O_{\PP} \rightarrow A'$ 
we obtain an isomorphism $\calA^{\PP}\otimes_AA' \iso \calA\otimes_AA'.$ 
Since $A'$ is a free $A$-module $\calA\otimes_AA'$ is an abelian scheme up to $A$-isogeny.

We may assume that 
$\Fr A'$ is Galois over $\Q,$ when $\calA^{\PP}$ is the $\Gal(\Fr A'/\Q)$-invariants of 
$\calA^{\PP}\otimes_AA'.$ Hence $\calA^{\PP}$ is the kernel of a map of abelian schemes up to $A$-isogeny. 
Write this map as $n^{-1}\cdot f$ where $n$ is an integer which is invertible in $A,$ and $f$ is a map of abelian schemes. 
Let $B$ be the connected component of the identity of $\ker (f),$ and view $B$ as an abelian scheme up to $A$-isogeny.
The cokernel of the natural inclusion $B \subset \calA^{\PP}$ is a torsion sheaf, so 
the natural map $B\otimes_AA' \rightarrow \calA\otimes_AA'$ induced by (\ref{periodeqn}) is an isomorphism, which implies 
that $B = \calA^{\PP}.$
\end{proof}

\begin{para} Keeping the above assumptions, denote by $\calA^*$ the dual abelian scheme. 
By an $A$-polarization, we mean an isomorphism $\calA \xrightarrow{\sim}\calA^*$ of abelian schemes up to $A$-isogeny, 
some multiple of which can be realized as a polarization of abelian schemes. Two $A$-polarizations are 
said to be equivalent if they differ by a multiplication by an element of $A^\times.$ A {\em weak} $A$-polarization 
is an equivalence class of $A$-polarizations.

Let $c: Z \rightarrow \GG_m$ be a character. We will denote by $\calA(c)$ the abelian scheme up to $A$-isogeny $\calA$ 
equipped with the map $Z \rightarrow \SAut_A\calA$ obtained by multiplying the natural action by $c.$
Let $\lambda: \calA \rightarrow \calA^*$ be a weak $A$-polarization. 
We have a canonical map $Z \rightarrow \SAut_A(\A^*).$ We say that $\lambda$ is a $c$-polarization 
if the induced map $\calA \rightarrow \calA^*(c)$ is compatible with $Z$-actions.
The same argument as in \cite{KisinJAMS} Lemma 3.1.5 proves the following 
\end{para}

\begin{lemma} There is a natural isomorphism $(\calA^*)^{\PP} \iso \calA^{\PP*}.$ 
If $\lambda: \calA \rightarrow \calA^*$ is a $c$-polarization, then there is a unique 
weak $A$-polarization $\lambda^{\PP}: \calA^{\PP} \rightarrow \calA^{\PP*}$ such that the diagram 
$$\xymatrix{ 
\calA^{\PP}\otimes_{A}\O_{\PP} \ar[r]^{\lambda^{\PP}\otimes 1}\ar[d]^\sim & \calA^{\PP*}\otimes_{A}\O_{\PP} \ar[d]^\sim \\
\calA\otimes_{A}\O_{\PP} \ar[r]^{\lambda\otimes 1} & \calA^*\otimes_{A}\O_{\PP}}
$$
commutes up to an element of $\O_{\PP}^\times.$ 
Here the map on the right is obtained by composing $\calA^{\PP*} \iso (\calA^*)^{\PP}$ with the isomorphism 
of (\ref{periodeqn}). 
\end{lemma}

\subsection{The adjoint group action}
\begin{para} We now return to the assumptions and notations of \S \ref{integralmodels}. In particular, 
we have an embedding of Shimura data $(\eG,X) \subset (\GSp, S^{\pm})$ where $\GSp = \GSp(V_{\Q})$ and 
$V_{\Q}$ is equipped with a lattice $V_{\mathbb Z},$ and an embedding 
$$\Sh_{\eK}(G,X) \hookrightarrow \Sh_{\eK'}(\GSp, S^{\pm}).$$

Let $\calB$ be an abelian scheme up to $\Z_{(p)}$-isogeny over a $\Z_{(p)}$-scheme $T.$ 
Set $\widehat V^p(\calB) = \varprojlim_{p\nmid n} \calB[n]$ and 
$$\widehat V^p(\calB)_{\Z_{(p)}} = \widehat V^p(\calB)\otimes_{\Z}\Z_{(p)} = \widehat V^p(\calB)\otimes_{\Z}\Q.$$
Suppose $\calB$ has dimension $\dim_{\Q} V_{\Q}/2$ and its equipped with a weak $\Z_{(p)}$-isogeny $\lambda.$
We denote by $\SIsom(V_{\AA_f^p},\widehat V^p(\calB)_{\Q})$ the pro-\'etale sheaf of isomorphisms 
$V_{\AA_f^p} \iso \widehat V^p(\calB)_{\Q}$ which are compatible with the pairings induced by $\psi$ and $\lambda$ up to a 
$\AA_f^{p,\times}$-scalar. Then $\SIsom(V_{\AA_f^p},\widehat V^p(\calB)_{\Q})/\eK^{p\prime}$ is an \'etale sheaf. 

A point $x \in \SSh_{\eK'}(\GSp, S^{\pm})(T)$ corresponds to a triple $(\A_x,\lambda_x,\varepsilon_x^p),$ where 
$\A_x$ is an abelian scheme over $T$ up to $\Z_{(p)}$-isogeny, equipped with a weak $\Z_{(p)}$-polarization
$\lambda_x,$ and 
$$ \varepsilon_x^p \in \Gamma(T,\SIsom(V_{\AA_f^p},\widehat V^p(\calA_x)_{\Q})/\eK^{p\prime}).$$
If $x \in \SSh_{\eK}(\eG,X)(T),$ then as in \cite{KisinJAMS} (3.2.4), $\varepsilon_x^p$ can be promoted to a section 
$ \varepsilon_x^p \in \Gamma(T,\SIsom(V_{\AA_f^p},\widehat V^p(\A_x)_{\Q})/\eK^p).$
Similarly, if $x \in \SSh_{\eK_p}(\eG,X)(T),$ then we obtain an
element 
$$ \varepsilon_x^p \in \ilim_{\eK^p} \Gamma(T,\SIsom(V_{\AA_f^p},\widehat V^p(\A_x)_{\Q})/\eK^p).$$
We denote by  $Z = Z_{G_{\Z_{(p)}}}$ the closure in $G_{\Z_{(p)}}$ of the center of the $\Q$-group $\eG.$
\end{para}

\begin{lemma}\label{actionofZ} If $x \in \SSh_{\eK}(\eG,X)(T),$ then there is a natural embedding 
$$ Z \hookrightarrow \SAut_{\Z_{(p)}}(\A_x).$$
\end{lemma}
\begin{proof} It suffices to construct the embedding for the universal point 
with $T = \SSh_{\eK}(\eG,X),$ and by \cite{FaltingsChai} I, 2.7 it suffices to 
consider $T = \Sh_{\eK}(\eG,X).$

By \cite{KisinJAMS} Lemma 3.2.2, 3.4.1 there is a natural embedding 
$Z\otimes_{\Z_{(p)}}\Q  \hookrightarrow \SAut_{\Q}(\A_x).$ For 
$y \in \Sh_{\eK}(\eG,X)(\mathbb C)$ this embedding specializes to the one induced 
by the natural action of $G$ on $H_1(\A_y(\mathbb C),\Q),$ obtained by choosing a lift 
$\tilde y \in X \times G(\AA_f)$ of $y.$ 
In particular, since 
$G_{\Z_{(p)}}$ is a subgroup of $\GL(V_{\Z_{(p)}}),$ we obtain maps 
$$
Z \rightarrow \SAut_{\Z_{(p)}}(\A_y) \rightarrow  \GL(V_{\Z_{(p)}}).
$$ 
As the composite is a closed embedding, so is the first map. 
Hence we get an embedding  $Z \hookrightarrow \SAut_{\Z_{(p)}}(\A_x).$
\end{proof}

\begin{para}  Let $G^{\ad}_{\Z_{(p)}}=  G_{\Z_{(p)}}/Z_{G_{\Z_{(p)}}},$ $\gamma \in G_{\Z_{(p)}}(\Z_{(p)})$  and $\P$ the fibre of 
$G_{\Z_{(p)}} \rightarrow G_{\Z_{(p)}}^{\ad}$ over $\gamma.$ Then $\P$ is a $Z_{G_{\Z_{(p)}}}$-torsor. 
Fix a Galois extension 
$F/\Q$ such that $\P$ admits an $\O_{F,(p)} = \O_F\otimes_{\Z} \Z_{(p)}$-point $\tilde\gamma.$ 
Such a point exists by the result of Moret-Bailly used in the proof of lemma \ref{twistav} above.
Applying lemma \ref{actionofZ} and specializing (\ref{periodeqn}) by $\tilde\gamma$ we obtain a $\Z_{(p)}$-isogeny
$$ \iota_{\tilde\gamma}: \A_x^{\P} \otimes_{\Z_{(p)}}\O_{F,(p)} \xrightarrow{\sim} \A_x \otimes_{\Z_{(p)}}\O_{F,(p)}.$$
\end{para}

\begin{lemma} The composite 
\begin{equation} V_{\AA^p_f}\otimes_{\Q}F \xrightarrow{\tilde \gamma^{-1}}   
V_{\AA^p_f}\otimes_{\Q} F \xrightarrow{\,\varepsilon^p_x\ }   
\widehat V^p(\A_x)\otimes_{\Q} F \xrightarrow{\iota_{\tilde \gamma}^{-1}}  
\widehat V^p(\A^{\PP}_x)\otimes_{\Q}F 
\end{equation}
is $\Gal(F/\Q)$-invariant and induces a section 
$$ \varepsilon^{p,\PP}_x\in \Gamma (T, \SIsom(V_{\AA^p_f}, \widehat V^p(\A^{\PP}_x)_{\Q})/\gamma \eK^p\gamma^{-1}) $$
\end{lemma}
\begin{proof}
This is identical to the proof of \cite{KisinJAMS} lemma 3.4.
\end{proof}

\begin{para} We recall the notation of \cite{DeligneCorvallis}. Let $H$ be a group equipped with an action of a group $\Delta,$ and 
$\Gamma \subset H$ a $\Delta$-stable subgroup. Suppose given a $\Delta$-equivariant 
map $\varphi: \Gamma \rightarrow \Delta$ where $\Delta$ acts on itself by inner automorphisms, 
and suppose that for $\gamma \in \Gamma,$ $\varphi(\gamma)$ acts on $H$ as inner conjugation by $\gamma.$
Then the elements of the form $(\gamma, \varphi(\gamma)^{-1})$ form a normal subgroup of the 
semi-direct product $H\rtimes\Delta.$ We denote by $H*_{\Gamma}\Delta$ the quotient of $H\rtimes\Delta$ 
by this subgroup. 

For a subgroup $H \subset G(\R)$ denote by $H_+$ the preimage in $H$ of $G^{\ad}(\R)^+,$ 
the connected component of the identity in $G^{\ad}(\R).$ 
As usual, we write $G^{\ad}(\Q)^+ = G^{\ad}(\Q)\cap G^{\ad}(\R)^+.$ 

There is an action of $\eG^{\ad}(\Q)^+$ on $\Sh(\eG,X)$ induced by the action by conjugation of $\eG$ on itself.
 Combining this with the action of $G(\AA_f)$ on $\Sh(\eG,X),$ 
gives rise to a right action of 
$$\fA(G) : = G(\AA_f)/Z(\Q)^-*_{G(\Q)_+/Z(\Q)}G^{\ad}(\Q)^+. $$
 on $\Sh(G,X)$ where $Z(\Q)^-$ denotes the closure of $Z_G(\Q)$ in $G(\AA_f).$ 

Let $G(\Q)^-_+$ denote the closure of $G(\Q)_+$ in $G(\AA_f)$ and set 
$$ \fA(G)^{\circ} =  G(\Q)^-_+/Z(\Q)^-*_{G(\Q)_+/Z(\Q)}G^{\ad}(\Q)^+. $$
This group depends only on $G^{\der}$ and not on $G;$ it is equal to the 
completion of $G^{\ad}(\Q)^+$ with respect to the topology whose open sets 
are images of congruence subgroups in $G^{\der}(\Q)$ \cite{DeligneCorvallis} 2.7.12. 
This definition will be used in the next subsection.

The action of $\eG^{\ad}(\Q)^+$ on $\Sh(\eG,X)$ induces an action of the group 
$\eG^{\ad}(\Z_{(p)})^+$ on $\Sh_{\eK_p}(\eG,X).$ This gives rise to an action of 
$$\fB(G_{\Z_{(p)}}) : = G(\AA_f^p)/Z(\Z_{(p)})^-*_{G(\Z_{(p)})_+/Z(\Z_{(p)})}G^{\ad}(\Z_{(p)})^+$$
on $\Sh_{\eK_p}(\eG,X).$ Here $Z(\Z_{(p)})^-$ denotes the closure of $Z(\Z_{(p)})$ in $G(\AA_f^p).$
\end{para}

\begin{lemma}\label{adjointaction} Let $\gamma \in \eG^{\ad}(\Z_{(p)})^+,$ and $\P$ the fibre of $G_{\Z_{(p)}} \rightarrow G_{\Z_{(p)}}^{\ad}$ 
over $\gamma.$ For $T$ a $\Z_{(p)}$-scheme and $x \in \SSh_{\eK_p}(\eG,X)(T),$ the assignment 
$$ (\A_x, \lambda_x, \varepsilon^p_x) \mapsto  (\A_x^{\PP}, \lambda_x^{\PP}, \varepsilon^{p,\PP}_x)$$
induces a map 
$$ \SSh_{\eK_p}(\eG,X) \rightarrow \SSh_{\eK_p}(\eG,X)$$
whose generic fibre agrees with the map induced by conjugation by $\gamma.$

Combining the $\eG^{\ad}(\Z_{(p)})^+$-action with the natural action of $G(\AA_f^p)$ on $\SSh_{\eK_p}(\eG,X)$ induces an action of 
$\fB(G_{\Z_{(p)}})$ on $\SSh_{\eK_p}(\eG,X).$
\end{lemma}
\begin{proof} The assignment induces a map 
$$ \SSh_{\eK_p}(\eG,X) \rightarrow \SSh_{\eK_p'}(\GSp,S^{\pm}).$$
The same argument as in \cite{KisinJAMS} Lemma 3.2.6 shows that on generic fibres 
this map factors through $\SSh_{\eK_p}(\eG,X)$ and induces the map obtained from 
the conjugation by $\gamma.$ The lemma follows from the definition of $\SSh_{\eK_p}(\eG,X)$ 
as the normalization of the closure of $\Sh_{\eK_p}(\eG,X)$ in $\SSh_{\eK_p'}(\GSp,S^{\pm}).$

The final claim follows from the analogous result on generic fibres, which is easily checked on complex points.
\end{proof}

\begin{para} Recall that in Theorem \ref{localmodeldiagramThm} we defined a $G_{\Z_{(p)}}$-torsor $\widetilde \SSh_{\eK}(\eG,X)_{\O_E},$ 
and we denote by $\widetilde \SSh_{\eK_p} = \widetilde \SSh_{\eK_p}(\eG,X),$ its pullback to $\SSh_{\eK_p}(\eG,X).$ 
Let $\widetilde \SSh_{\eK_p}^{\ad}$ be the $G_{\Z_{(p)}}^{\ad}$-torsor obtained from $\widetilde \SSh_{\eK_p}.$
We remark that the map $q$ in (\ref{locmoddiagram}) obviously factors through  $\widetilde \SSh_{\eK_p}^{\ad}.$

We will show that the action of $\fB(G_{\Z_{(p)}})$ on $\SSh_{\eK_p}(\eG,X)$ defined above, can be lifted to $\widetilde \SSh_{\eK_p}^{\ad}$.
\end{para}

\begin{lemma}\label{adjointactiontorsors} The action of  $G^{\ad}(\Z_{(p)})^+$ on $\SSh_{\eK_p}(\eG,X)$ lifts to an action 
on  $\widetilde \SSh_{\eK_p}^{\ad}$ as a $G^{\ad}_{\Z_p}$-torsor. If we equip ${\rm M}^{\rm loc}_{G,X}$ with the trivial  $\fB(G_{\Z_{(p)}})$-action, the maps in the diagram 
of $\O_E$-schemes 

\begin{equation}\label{locmoddiagramadjoint}
\xymatrix{
& \widetilde \SSh_{\eK_p}^{\ad}  \ar[ld]_{\pi}\ar[rd]^{q} & \\
\quad\quad   \SSh_{\eK_p}(\eG, X) \quad\quad & &{\ \ \ {\rm M}^{\rm loc}_{G,X}\ }\, , \ \ \ \ 
  \ \ \ \ \ 
}
\end{equation}
are $\fB(G_{\Z_{(p)}})$-equivariant. Moreover any sufficiently small $\eK^p \subset G(\AA_f^p)$ acts freely 
on $\widetilde \SSh_{\eK_p}^{\ad},$ and the map 
$$  \widetilde \SSh_{\eK_p}^{\ad}/\eK^p \rightarrow {\rm M}^{\rm loc}_{G,X}$$
induced by $q$ is smooth of relative dimension $\dim G^{\ad}.$
\end{lemma}  
\begin{proof} We begin by defining the action of $G^{\ad}(\Z_{(p)})^+$ on $\widetilde \SSh_{\eK_p}^{\ad}.$ 
Let $S$ be an $\O_E$-scheme, and $(x,f) \in \widetilde \SSh_{\eK_p}(S)$ where $x \in \SSh_{\eK_p}(\eG, X)(S),$ 
and $f$ is a map $f:V_{\Z_p}^\vee\otimes \O_S \iso R^1f_* \Omega^1_{\calA_x/S}$ .
Let $\gamma \in G^{\ad}(\Z_{(p)})^+,$ and $\cal P$ the corresponding $Z$-torsor. 
Choose a number field $F,$ Galois over $\Q,$ and a section $\tilde \gamma \in \calP(\O_F).$ 
For $\sigma \in \Gal(F/\Q)$ set $c_{\tilde \gamma}(\sigma) = \sigma(\tilde \gamma)\gamma^{-1}.$
Consider the composite
$$
 \tilde\gamma(f): V_{\Z_p}^\vee\otimes \O_S \otimes \O_F \underset{f}{\xrightarrow{\sim}} R^1f_* \Omega^1_{\calA_x/S} \otimes \O_F \underset {\iota_{\tilde \gamma}^{-1}}{\xrightarrow{\sim}} R^1f_* \Omega^1_{\calA_x^{\calP}/S}\otimes \O_F.
 $$ Then $(\gamma(x), \tilde \gamma(f)) \in \widetilde \SSh_{\eK_p}(S\otimes \O_F).$ 
Since $\sigma(\iota_{\tilde \gamma}) = c_{\tilde \gamma}(\sigma)^{-1}\iota_{\tilde\gamma}$ for $\sigma \in \Gal(F/\Q)$ by \cite{KisinJAMS}, Lemma 3.4.3, 
$(\gamma(x), \tilde \gamma(f))$ induces a point of $\widetilde \SSh^{\ad}_{\eK_p}(S),$ which depends only on the image of $(x,f)$ in  $\widetilde \SSh^{\ad}_{\eK_p}(S)$ and on $\gamma.$

This shows that $(x,f) \mapsto (\gamma(x), \gamma(f))$ induces an action of $G^{\ad}(\Z_{(p)})^+$ on $\widetilde \SSh^{\ad}_{\eK_p},$ lifting the action on 
$\SSh_{\eK_p}(\eG, X).$ That the map $q$ is  $G^{\ad}(\Z_{(p)})^+$-equivariant, follows from the fact that the map $\iota_{\tilde \gamma}^{-1}$ in the definition of 
$\gamma(f)$ respects Hodge filtrations, as it arises from a map of abelian varieties.

Now lift the action of $h \in G(\AA_f^p)$ to $\widetilde \SSh_{\eK_p}$ by sending $(x,f)$ to $(h(x), f).$ It remains to show that this, together with the action of  $G^{\ad}(\Z_{(p)})^+,$ 
defined above, defines an action of  $\fB(G_{\Z_{(p)}})$ in $\widetilde \SSh_{\eK_p}^{\ad}.$ To do this we have to check it defines an action of 
$G(\AA_f^p)\rtimes G^{\ad}(\Z_{(p)})^+$ such that $G(\Z_{(p)})_+$ acts trivially. If $h\in G(\AA_f^p)$ and $g \in G^{\ad}(\Z_{(p)})^+,$ then as $G(\AA_f^p)\rtimes G^{\ad}(\Z_{(p)})^+$ acts 
on $\SSh_{\eK_p}(G,X),$ we have 
$$g(h(g^{-1}(x,f))) = g(h((g^{-1}(x), g^{-1}(f)))) = g(h(g^{-1}(x)), g^{-1}(f)) = (ghg^{-1}(x),f),$$
which implies that $G(\AA_f^p)\rtimes G^{\ad}(\Z_{(p)})^+$ acts on $\widetilde \SSh_{\eK_p}.$ 

If $\gamma$ lifts to $\tilde \gamma \in G(\Z_{(p)})_+,$ then $(\tilde \gamma, \tilde\gamma^{-1})(x) = x.$ More precisely, $(\tilde \gamma, \tilde\gamma^{-1})(x)$ 
corresponds to the triple $(\calA_x^{\cal P}, \varepsilon_x^{p,\calP}\circ \tilde \gamma, \lambda_x^{\cal P}),$ and $\iota_{\tilde\gamma}: \calA^{\calP}_x \iso \calA_x$ induces 
an isomorphism of this triple with $(\calA_x, \varepsilon_x^p, \lambda_x).$ Similarly, $\iota_{\tilde\gamma}$ intertwines the isomorphisms $f$ and $\gamma(f),$ 
which proves that  $G(\Z_{(p)})_+$ acts trivially on $\widetilde \SSh_{\eK_p}.$

The final statement follows immediately from the construction of $\widetilde \SSh_{\eK_p}^{\ad}.$
\end{proof}

\subsection{Shimura varieties of abelian type}\label{subsec:AbelianType}

We continue to use the notation and assumptions of the previous subsection.

\begin{para} We will   define the analogues of $\fA(G)^\circ$ and $\fA(G)$ for $\Z_{(p)}$-valued points, but we need 
some preparation.

Suppose $S$ is an affine $\Q$-scheme, and let $S_{\Z_p}$ be a flat, affine $\Z_p$-scheme, with generic 
%\mar{GP: I think the ref. wants these paragraphs to go in 4.5 someplace. I agree that this is not needed because $G_{\Z_{(p)}}$ is defined directly via a Zariski closure...}
fibre $S\otimes\Q_p.$ Then there is a canonical $\Z_{(p)}$-scheme $S_{\Z_{(p)}}$ with generic fibre $S$ 
such that $S_{\Z_{(p)}}\otimes_{\Z_{(p)}} \Z_p = S_{\Z_p}.$ Indeed $S_{\Z_{(p)}}$ is the spectrum of the ring 
obtained by intersecting the global functions on $S_{\Z_p}$ and $S$ inside those on $S\otimes \Q_p.$

Recall the smooth model $G_{\Z_{(p)}}$ is defined by a point $x \in \B(G,\Q_p),$ 
and that we set $G^{\ad}_{\Z_{(p)}} =G_{\Z_{(p)}}/Z_{\Z_{(p)}}.$ Let 
$G^{\der}_{\Z_{(p)}}$ be the closure of $G^{\der}$ in $G_{\Z_{(p)}}.$
We denote by $G^\circ = G_{\Z_{(p)}}^\circ,$ the connected component of the identity. 
It is the parahoric group scheme attached to $x.$ More precisely, $x$ 
defines a parahoric group scheme over $\Z_p,$ which descends to $\Z_{(p)}$ via the general construction 
described above. Similarly, let $G^{\ad\circ} = G^{\ad\circ}_{\Z_{(p)}}$ and $G^{\der\circ} = G^{\der\circ}_{\Z_{(p)}}$ 
be the parahoric models of $G^{\ad}$ and $G^{\der}$ respectively, 
defined by the image, $x^{\ad},$ of $x$ in $\B(G^{\ad},\Q_p) = \B(G^{\der},\Q_p).$ 
Note that, in general, $G^{\ad\circ}_{\Z_{(p)}}$ is not equal to the neutral component $(G^{\rm ad}_{\Z_{(p)}})^\circ$ of $G^{\rm ad}_{\Z_{(p)}}$, but see (2) below.
 \end{para}

\begin{lemma}\label{parahoricrelations} We have 
\begin{enumerate}
\item $G^{\der}_{\Z_{(p)}}$ is the stabilizer of $x^{\ad}.$ In particular, $G^{\der\circ}_{\Z_{{(p)}}}$ is the connected component of the identity of $G^{\der}_{\Z_{(p)}}.$
\item Suppose that either the center $Z_G$ is connected
 or that $Z_{G^{\rm der}}$ has rank prime to $p$. Then 
 $G^{\ad\circ}_{\Z_{(p)}}$ is the connected component of the identity of 
 $G^{\rm ad}_{\Z_{(p)}} = G_{\Z_{(p)}}/Z_{\Z_{(p)}}$.
In particular, there is a map of $\Z_{(p)}$-group schemes $G^{\ad\circ}_{\Z_{(p)}} \rightarrow G^{\ad}_{\Z_{(p)}},$ extending the identity on generic fibres.
\end{enumerate}
\end{lemma}
\begin{proof} 
Let $G^{\der}_{\Z_p} = G^{\der}_{\Z_{(p)}}\otimes_{\Z_{(p)}}\Z_p.$ Let $x^{\der} \in \calB(G^{\der},\Q_p) \subset \calB(G,\Q_p)$ be the preimage 
of $x^{\ad}$ under the identification $\calB(G^{\der},\Q_p) = \calB(G^{\ad},\Q_p).$ Then $\calG_{x^{\der}} \hookrightarrow \mathcal {GL}(V_{Z_p})_{\iota_{x^{\der}}}$ 
is a closed embedding by Proposition \ref{miniscule}, and similarly for $\calG^{\der}_{x^{\der}}.$ Thus, the closure of $G^{\der}$ in $\calG_{x^{\der}}$ is smooth, 
and coincides with $\calG^{\der}_{x^{\der}},$ the group scheme stabilizer of $x^{\ad}.$
On the other hand, $\calG_{x^{\der}}$ can be naturally identified with $\calG_x$ (cf.~\cite{TitsCorvallis} 3.4.1). In particular, the closure of $G^{\der}$ in $\calG_x$ is smooth, 
and coincides with the group scheme stabilizer of $x^{\ad}.$ Now (1), which is the corresponding statement over $\Z_{(p)},$ follows.  

Let us consider (2). Note that by the functoriality of the group schemes stabilizing a point of the building, there is always a map $
  (G^{\rm ad}_{\Z_{(p)}})^\circ\to G^{\rm ad\circ}_{\Z_{(p)}}$ 
(see \ref{centralIso}).
If $Z_G$ is connected (2) follows immediately from Proposition \ref{cIsoProp}.
Suppose that $Z_{G^{\rm der}}$ has rank prime to $p$.   Then (1) together with \ref{cIsoProp} applied to $G^{\rm der}$
implies that $G^{\rm ad\circ}_{\Z_{(p)}}$ is the quotient of $G^{\rm der\circ}_{\Z_{(p)}}$
by the Zariski closure of the center $Z_{G^{\rm der}}$. This provides a  map $G^{\rm ad\circ}_{\Z_{(p)}}\to
 G_{\Z_{(p)}}/Z_{\Z_{(p)}}=G^{\rm ad}_{\Z_{(p)}}$ which gives the inverse
 $G^{\rm ad\circ}_{\Z_{(p)}}\to
  (G^{\rm ad}_{\Z_{(p)}})^\circ$. \end{proof}

\begin{para}\label{setupintegralaction}  Let $Z^\circ$ denote the Zariski closure of $Z$ in $G^{\circ}.$
We denote by $Z^\circ(\Z_{(p)})^-$ the closure of $Z^\circ(\Z_{(p)})$ in $Z^\circ(\AA_f).$ 
Note that the image of $Z^\circ(\Z_{(p)})^-$ in $Z^\circ(\AA_f^p)$ coincides with the closure 
of $Z^\circ(\Z_{(p)})$ in $Z^\circ(\AA^p_f).$

 Let 
$$ \tilde \fA(G_{\Z_{(p)}}) = [G(\AA_f^p)\times G_{\Z_{(p)}}^\circ(\Z_p)]/Z(\Z_{(p)})^-*_{G^\circ(\Z_{(p)})_+/Z(\Z_{(p)})}G^{\ad\circ}(\Z_{(p)})^+ \subset \fA(G) $$
and 
$$ \tilde \fA(G_{\Z_{(p)}})^{\circ} = G^\circ(\Z_{(p)})^{\sim}_+/Z(\Z_{(p)})^-*_{G^\circ(\Z_{(p)})_+/Z(\Z_{(p)})}G^{\ad\circ}(\Z_{(p)})^+ \subset \fA(G)^\circ$$ 
where $G^\circ(\Z_{(p)})_+^{\sim}$ denotes the closure of $G^\circ(\Z_{(p)})_+$ in $G(\AA_f^p)\times G(\Z_p).$

We set 
$$ \fA(G_{\Z_{(p)}}) = G(\AA_f^p)/Z(\Z_{(p)})^-*_{G^\circ(\Z_{(p)})_+/Z(\Z_{(p)})}G^{\ad\circ}(\Z_{(p)})^+ $$
and 
$$ \fA(G_{\Z_{(p)}})^{\circ} = G^\circ(\Z_{(p)})^-_+/Z(\Z_{(p)})^-*_{G^\circ(\Z_{(p)})_+/Z(\Z_{(p)})}G^{\ad\circ}(\Z_{(p)})^+, $$
where  $G^\circ(\Z_{(p)})^-_+$ is the closure 
of $G^\circ(\Z_{(p)})_+$ in $G(\AA_f^p).$ Note that the difference between $\fA(G_{\Z_{(p)}})$ and the group $\fB(G_{\Z_{(p)}})$ defined 
in the previous section is that the former is defined using parahoric models of $G$ and $G^{\ad}.$

In what follows, we will assume that either $Z=Z_G$ is connected
 or that $Z_{G^{\rm der}}$ has rank prime to $p$.  Under this assumption, by Lemma 
\ref{parahoricrelations} (2) and Lemma \ref{adjointaction}, the action of $\fA(G_{\Z_{(p)}})$ on 
$\Sh_{\eK_p}(G,X)$ extends to $\SSh_{\eK_p}(G,X).$  As in \S 4.3, we denote by 
$\SSh_{\eK_p^\circ}(G,X)$ the normalization of $\SSh_{\eK_p}(G,X)$ in $\Sh_{\eK_p^\circ}(G,X).$
Then the action of $\fA(G_{\Z_{(p)}})$ on $\Sh_{\eK_p^\circ}(G,X)$ extends to $\SSh_{\eK_p^\circ}(G,X).$
\end{para}

\begin{lemma}\label{independenceofA} We have 
\begin{enumerate}
\item $\tilde \fA(G_{\Z_{(p)}})^{\circ}$ is the closure of $G^{\ad\circ}(\Z_{(p)})^+$ in $\fA(G)^\circ.$
\item $\fA(G_{\Z_{(p)}})^{\circ}$ is the completion of $G^{\ad\circ}(\Z_{(p)})^+$ with respect to the 
topology generated by images of sets of the form $G^{\der\circ}(\Z_{(p)})_+\cap \eK^p.$
\end{enumerate}
\end{lemma}
\begin{proof} (1) is immediate from the definitions. Using \cite{DeligneCorvallis} 2.0.13, one sees that 
$\fA(G_{\Z_{(p)}})^{\circ}$ is the completion of $G^{\ad\circ}(\Z_{(p)})^+$ with respect to the 
topology generated by images of sets of the form 
$$ G^\circ(\Z_{(p)})_+\cap (\eK^p \cap G^{\der}(\Q))\cdot U$$ 
where $U$ is a finite index subgroup of the group of $p$-units in $Z_G(\Q).$
Suppose $g = hu \in  G^\circ(\Z_{(p)})_+$ with $h \in \eK^p \cap G^{\der}(\Q)$ and $u \in U.$ 
Then $h$ fixes $x^{\ad},$ so $h \in G^{\der}(\Z_{(p)})_+.$ As in the proof of Lemma \ref{levelstr}, this 
implies that $h \in G^{\der\circ}(\Z_{(p)})_+,$ for $\eK^p$ small enough.  Thus for $\eK^p$ small enough 
the image of the set above is equal to $\eK^p\cap  G^{\der\circ}(\Z_{(p)})_+.$ This proves (2).
\end{proof}

\begin{para} Fix a connected component $X^+ \subset X.$ We denote by $\Sh(G,X)^+ \subset \Sh(G,X)$ the connected Shimura variety corresponding to the choice of $X^+,$ and similarly for 
$\Sh_{\eK_p^\circ}(G,X)^+ \subset \Sh_{\eK_p^\circ}(G,X).$
Let $\eE^p \subset \bar \eE$ denote the maximal extension of $\eE$ which is unramified at primes 
dividing $p.$  By \ref{parahoricreciprocitylemma} and \cite{DeligneCorvallis}, Thm. 2.6.3, the action of $\Gal(\bar \eE/\eE)$ on $\Sh_{\eK_p^\circ}(G,X)^+$ factors through 
$\Gal(\eE^p/\eE).$ We again denote by $\Sh_{\eK_p^\circ}(G,X)^+$ the $\eE^p$-scheme obtained from $\Sh_{\eK_p^\circ}(G,X)^+$ by descent, 
and by $\SSh_{\eK_p^\circ}(G,X)^+ \subset \SSh_{\eK_p^\circ}(G,X)$ the corresponding component of $\SSh_{\eK_p^\circ}(G,X),$ which is defined over 
$\O_{\eE^p}\otimes_{\O}\O_{(v)}.$

Let $\fE(G^{\circ}_{\Z_{(p)}}) \subset \fA(G_{\Z_{(p)}}) \times \Gal(\eE^p/\eE)$ denote the 
stabilizer of $\Sh_{\eK_p^\circ}(G,X)^+ \subset \Sh_{\eK^\circ_p}(G,X)$ (viewed as $\eE^p$-schemes), 
and let $\tilde \fE(G^{\circ}_{\Z_{(p)}}) \subset \tilde\fA(G_{\Z_{(p)}}) \times \Gal(\bar \eE/\eE)$ 
denote the stabilizer of $\Sh(G,X)^+ \subset \Sh(G,X).$ 
\end{para}

\begin{lemma}\label{extensionstructure} We have 
\begin{enumerate} 
\item $\fE(G^{\circ}_{\Z_{(p)}})$ (resp.~$\tilde \fE(G^{\circ}_{\Z_{(p)}})$) 
is an extension of $\Gal(\eE^p/\eE)$ (resp. $\Gal(\bar \eE/\eE)$) 
by $\fA(G_{\Z_{(p)}})^\circ$ (resp.~$\tilde\fA(G_{\Z_{(p)}})^\circ$).
\item There are canonical isomorphisms
$$ \fA(G_{\Z_{(p)}})*_{\fA(G_{\Z_{(p)}})^\circ}\fE(G^{\circ}_{\Z_{(p)}}) \iso \fA(G_{\Z_{(p)}})\times \Gal(\eE^p/\eE)$$
$$ \tilde\fA(G_{\Z_{(p)}})*_{\tilde\fA(G_{\Z_{(p)}})^\circ}\tilde\fE(G^{\circ}_{\Z_{(p)}}) \iso \tilde\fA(G_{\Z_{(p)}})\times \Gal(\bar\eE/\eE).$$
where an element of $\fE(G^{\circ}_{\Z_{(p)}})$ (resp. $\tilde\fE(G^{\circ}_{\Z_{(p)}})$) acts on $\fA(G_{\Z_{(p)}})$ (resp. 
$\tilde\fA(G_{\Z_{(p)}})$) via conjugation by its image in $\fA(G_{\Z_{(p)}}).$
\end{enumerate}
\end{lemma}
\begin{proof} Let $\eE^{\times,+} = \eE^\times \cap (\eE\otimes_{\Q}\RR)^{\times,+}.$ 
Consider the composite map 
\begin{equation}\label{reciprocitymap}
\AA_{\eE}^\times/\eE^\times(\eE\otimes_{\Q}\RR)^{\times,+} = \AA_{\eE}^{f\times}/\eE^{\times,+} 
\overset {\mu_h^{-1}} \rightarrow G(\AA_{\eE}^f)/G(\eE)^-_+ \overset {N_{\eE/\Q}} \rightarrow 
G(\AA_f)/G(\Q)^-_+
\end{equation}
where $G(\eE)^-_+ = (R_{\eE/\Q} G)(\Q)^-_+.$
If $x \in \AA_{\eE}^{f\times}/\eE^{\times,+},$ then by weak approximation $x$ has a representative 
$(x_v) \in \AA_{\eE}^{f\times}$ with $x_v \in \O_{\eE_v}^\times$ for all $v|p.$ Hence, by 
Lemma \ref{parahoricreciprocitylemma}, 
the image of $x$ under (\ref{reciprocitymap}) is contained in 
$G(\AA_f^p)\times G_{\Z_{(p)}}^\circ(\Z_p)/G^\circ(\Z_{(p)})^\sim_+.$ 
By \cite{DeligneCorvallis}, Thm.~2.6.3,
 the action of $\Gal(\bar \eE/\eE)$ on the geometrically connected components of $\Sh(G,X)$ 
is given by the composite of (\ref{reciprocitymap}) and the class field theory isomorphism. 
This proves the claim that $\tilde \fE(G^{\circ}_{\Z_{(p)}})$ is an extension of 
$\Gal(\bar \eE/\eE)$ by $\tilde\fA(G_{\Z_{(p)}})^\circ.$

It follows that the action 
of $\Gal(\eE^p/\eE)$ on the geometrically connected components of $\Sh_{\eK^\circ_p}(G,X)$ 
is given by the induced map 
$$\Gal(\eE^p/\eE) \rightarrow G(\AA_f^p)\times G_{\Z_{(p)}}^\circ(\Z_p)/G^\circ(\Z_{(p)})^\sim_+ G_{\Z_{(p)}}^\circ(\Z_p)
\iso G(\AA_f^p)/G^\circ(\Z_{(p)})_+^-.$$
This  shows that $\fE(G^{\circ}_{\Z_{(p)}})$ is an extension of $\Gal(\eE^p/\eE)$ by $\fA(G_{\Z_{(p)}})^\circ.$
Now (2) follows easily.
\end{proof}

\begin{para} 
Let $G_2$ be a reductive group over $\Q$ equipped with a central isogeny $\al: G^{\der} \rightarrow G^{\der}_2.$ 
Let $x_2 \in  B(G_2,\Q_p)$ with $x_2^{\ad} = x^{\ad}.$ 
We denote by $\calG_2$ the model of $G_2$ defined as the stabilizer of $x_2,$ 
and by $G_{2,\Z_{(p)}}$ and $G_{2,\Z_{(p)}}^\circ$ the group schemes over $\Z_{(p)}$ corresponding to 
$\calG_2$ and $\calG_2^\circ.$ Write $K_{2,p}^\circ = \calG_2^\circ(\Z_p).$

Suppose that we have a Shimura datum $(G_2,X_2)$ such that $\al$ induces an isomorphism of Shimura data 
$$ (G^{\ad},X^{\ad}) \iso (G_2^{\ad},X_2^{\ad}).$$ By the real approximation theorem, after replacing $X_2$ by its conjugate by some element of 
$G_2^{\ad}(\Q),$ we may assume that the image of $X_2 \subset X_2^{\ad}$ contains $X^+.$
We denote by $\eE_2$ the reflex field $(G_2,X_2),$ 
and we set $\eE' = \eE\cdot\eE_2.$  We denote by $\fE_{\eE'}(G^{\circ}_{\Z_{(p)}})$ and $\tilde\fE_{\eE'}(G^{\circ}_{\Z_{(p)}})$ 
the pullbacks of $\fE(G^{\circ}_{\Z_{(p)}})$ and $\tilde \fE(G^{\circ}_{\Z_{(p)}})$ by 
$\Gal(\eE^{\prime p}/\eE') \rightarrow \Gal(\eE^p/\eE)$ and $\Gal(\bar \eE/\eE') \rightarrow \Gal(\bar \eE/\eE)$ 
respectively.

We have the groups $\fA(G_2)$ and $\tilde \fA(G_2)^\circ$ defined as above, and we set 
$$\fA(G_{2,\Z_{(p)}}) = G_2(\AA_f^p)/Z_{G_2}(\Z_{(p)})^-*_{G_2^\circ(\Z_{(p)})_+/Z_{G_2}(\Z_{(p)})} G^{\ad\circ}(\Z_{(p)})^+,$$
and similarly for $\tilde \fA(G_{2,\Z_{(p)}})^\circ.$
Note that the group $G^{\ad\circ}(\Z_{(p)})^+$ is exactly the same one which appeared in the definition 
of $\fA(G_{\Z_{(p)}}).$ 

As in Corollary \ref{geometricallyconnectedcomponentsunramified}, the geometrically connected components of $\Sh_{\eK_{2,p}^\circ}(G_2,X_2)$ 
are defined over $\eE_2^p.$ We define $\fE(G^\circ_{2,\Z_{(p)}}) \subset \fA(G_{2,\Z_{(p)}})\times \Gal(\eE_2^p/\eE_2)$ 
as the stabilizer of $\Sh_{\eK_{2,p}^\circ}(G_2,X_2)^+ \subset \Sh_{\eK_{2,p}^\circ}(G_2,X_2).$ As in the proof of, 
Lemma \ref{extensionstructure}, this is an extension of $\Gal(\eE^p_2/\eE_2)$ by $\fA(G_{2,\Z_{(p)}})^\circ.$

Similarly, we define $\tilde\fE(G^\circ_{2,\Z_{(p)}})$ as above. It is an extension of $\Gal(\bar \eE/\eE_2)$ by $\tilde \fA(G_{2,\Z_{(p)}})^\circ.$ 
\end{para}

\begin{lemma}\label{mapofextns} There exist natural maps of extensions 
$$ \text{$\fE_{\eE'}(G^\circ_{\Z_{(p)}}) \rightarrow \fE(G^\circ_{2,\Z_{(p)}})$ and $\tilde \fE_{\eE'}(G^\circ_{\Z_{(p)}}) \rightarrow \tilde \fE(G^\circ_{2,\Z_{(p)}})$}$$
\end{lemma}
\begin{proof} (cf. \cite{DeligneCorvallis} 2.5.6) 
Let $G_3$ be the connected component of the identity of $G\times_{G^{\ad}} G_2,$ and $X_3$ the conjugacy class 
of homomorphisms $\mathbb S \rightarrow G_{3,\mathbb R}$ induced by $X$ and $X_2.$ Repeating the above definitions 
for $G_3$ we obtain an extension $\fE(G^\circ_{3,\Z_{(p)}}).$ Note that the reflex field of $(G_3,X_3)$ is $\eE',$ 
and $G_3^{\der} = G^{\der}.$ Thus $G_3$ satisfies the condition (\ref{cohvanishingcondn}), as we are assuming $G$ does. 
In particular, we have $\fA(G_{3,\Z_{(p)}})^\circ \iso \fA(G_{\Z_{(p)}})^\circ$ by Lemma \ref{independenceofA}(2). It follows 
that the natural map $\fE(G^\circ_{3,\Z_{(p)}}) \rightarrow \fE_{\eE'}(G^\circ_{\Z_{(p)}})$ is an isomorphism of extensions. 
The first map of the lemma is given by the composite 
$$ \fE_{\eE'}(G^\circ_{\Z_{(p)}}) \iso \fE(G^\circ_{3,\Z_{(p)}}) \rightarrow \fE(G^\circ_{2,\Z_{(p)}}).$$
The construction for the second map is analogous.
\end{proof}

\begin{lemma}\label{injmap} The diagram 
\begin{equation}
\xymatrix{
\tilde \fA(G_{\Z_{(p)}})^\circ \ar[r] \ar[d] & \tilde \fA(G_{2,\Z_{(p)}}) \ar[d] \\
\fA(G)^\circ \ar[r] & \fA(G_2)
}
\end{equation}
is commutative and Cartesian. In particular, the morphism of complexes 
$$ (\tilde \fA(G_{\Z_{(p)}})^\circ \rightarrow \tilde \fA(G_{2,\Z_{(p)}})) \rightarrow 
(\fA(G)^\circ \rightarrow \fA(G_2)).
$$
induces a bijection on kernels and an injection on cokernels.
\end{lemma}
\begin{proof} (cf.~\cite{KisinJAMS} Lemma 3.3.3.) We remark that the top map is well defined by 
Lemma \ref{independenceofA}(1).  The diagram commutes, since $G^{\ad\circ}(\Z_{(p)})^+$ is naturally a subgroup of each term, 
is dense in $(\tilde \fA(G_{\Z_{(p)}})^\circ,$ and all the maps are the identity on $G^{\ad\circ}(\Z_{(p)})^+.$

Suppose that $(g,\gamma) \in \tilde \fA(G_{2,\Z_{(p)}})$ is in the image of $(g_1,\gamma_1) \in \fA(G)^\circ.$
Since $g_1$ may be approximated by an element of $G(\Q)_+,$ we may assume that $g_1$ is in the image of 
$G(\AA^p_f)\times G^\circ(\Z_p).$ Since $g\gamma = g_1\gamma_1$ 
in $G^{\ad}(\AA_f),$ we have $\gamma_1 = g_1^{-1}g\gamma  \in G^{\ad\circ}(\Z_{(p)})_+$ so 
$(g_1,\gamma_1) \in \tilde \fA(G_{\Z_{(p)}})^\circ.$ Thus the diagram in the lemma is Cartesian.
\end{proof}

\begin{para}

By Lemma \ref{injmap} we have an inclusion 
$$ \fA(G_{\Z_{(p)}})^\circ\backslash \fA(G_{2,\Z_{(p)}}) 
= \tilde \fA(G_{\Z_{(p)})})^\circ\backslash \tilde \fA(G_{2,\Z_{(p)}})/\eK_{2,p}^\circ \hookrightarrow \fA(G)^\circ\backslash \fA(G_2)/\eK_{2,p}^\circ.$$

Let $J \subset G_2(\Q_p)$ denote a set which maps bijectively to a set of coset representatives for the image of 
$\fA(G_{2,\Z_{(p)}})$ in $\fA(G)^\circ \backslash \fA(G_2)/\eK_{2,p}^\circ.$

Recall, we assume either that the center $Z$ of $G$ is connected or that $Z_{G^{\rm der}}$ has rank prime to $p$. 
\end{para}

\begin{lemma}\label{reconstructionI} There is an isomorphism of $\bar \eE$-schemes 
with $G_2(\AA_f^p)\times\Gal(\bar \eE/\eE')$-action
$$ \Sh_{\eK^\circ_{2,p}}(G_2,X_2) \iso \Big[ [\Sh_{\eK^\circ_p}(G,X)^+ \times \fA(G_{2,\Z_{(p)}}) ]/ \fA(G_{\Z_{(p)}})^\circ\Big]^{|J|}  $$
where $h \in \fA(G_{\Z_{(p)}})^\circ$ acts on $\fA(G_{\Z_{(p)}})$ by left multiplication by $h^{-1}.$
\end{lemma}
\begin{proof} 
By \cite{DeligneCorvallis} 2.5.6, there is a morphism of extensions 
$\fE_{\eE'}(G) \rightarrow \fE(G_2),$ and in particular, 
an isomorphism 
$$ \tilde\fA(G_2)*_{\tilde\fA(G)^\circ} \fE_{\eE'}(G)
\iso G_2(\AA_f)\times \Gal(\bar \eE/\eE').$$
We equip $\Sh(G,X)^+\times \fA(G_2)$ with a right action of 
$\fE_{\eE'}(G)$ given by $(s,a)\cdot e = (se, \bar e^{-1} a \bar e),$ 
and with the action of $\fA(G_2)$ induced by right multiplication of $\fA(G_2)$ 
on itself. Here $\bar e$ denotes the image of $e$ under
$$\fE_{\eE'}(G) \rightarrow \fE(G_2) \rightarrow \fA(G_2).$$
This induces an action of $\fA(G_2)\rtimes\fE_{\eE'}(G)$ 
on $\Sh(G,X)^+\times \fA(G_2),$ which descends to an action of 
$\fA(G_2)*_{\fA(G)^\circ} \fE_{\eE'}(G)$ on $ [\Sh(G,X)^+\times \fA(G_2)]/ \fA(G)^\circ.$

By \cite{DeligneCorvallis} 2.7.11, 2.7.13, using the above isomorphism gives 
 an isomorphism of 
$\bar \eE$-schemes with $G_2(\AA_f) \times \Gal(\bar \eE/\eE')$-action 
$$
 \Sh(G_2,X_2) \iso [\Sh(G,X)^+ \times \fA(G_2)]/\fA(G)^\circ.
$$
Dividing both sides by $\eK^\circ_{2,p}$ we obtain an isomorphism of 
$\bar \eE$-schemes with $G_2(\AA^p_f) \times \Gal(\bar \eE/\eE')$-action 
\begin{multline}
 \Sh_{\eK_{2,p}^\circ}(G_2,X_2) \iso [\Sh(G,X)^+ \times \fA(G_2)/\eK^\circ_{2,p}]/\fA(G)^\circ \\
\iso \coprod_j [\Sh(G,X)^+ \times \fA(G_{2,\Z_{(p)}})j]/\tilde \fA(G_{\Z_{(p)}})^\circ.
\end{multline}

Since  $\eK^\circ_p \cap \tilde \fA(G_{\Z_{(p)}})^\circ$ is contained in the kernel of the composite 
$$ \tilde \fA(G_{\Z_{(p)}})^\circ \rightarrow \fA(G_{\Z_{(p)}})^\circ \rightarrow \fA(G_{\Z_{2,(p)}})^\circ, $$
the final quotient above is equal to 
$$ \coprod_{j \in J} [\Sh_{\eK^\circ_p}(G,X)^+ \times \fA(G_{2,\Z_{(p)}})j]/\fA(G_{\Z_{(p)}})^\circ.$$
The lemma follows.
\end{proof}

\begin{cor}\label{reconstructionII} The $\O_{\eE^{\prime p},(v)} = \O_{\eE^{\prime p}}\otimes_{\O}\O_{(v)}$-scheme 
\begin{equation}\label{reconstructionIIeqn}
\SSh_{\eK^\circ_{2,p}}(G_2,X_2) = 
\big[ [\SSh_{\eK^\circ_p}(G,X)^+\times \fA(G_{2,\Z_{(p)}})]/\fA(G_{\Z_{(p)}})^\circ \big]^{|J|}
\end{equation} 
has a natural structure of a $\O'_{(v)} = \O_{\eE'}\otimes_{\O}\O_{(v)}$-scheme with $G_2(\AA_f^p)$-action, 
and is a model for $\Sh_{\eK^\circ_{2,p}}(G_2,X_2).$ 
The local rings on $\SSh_{\eK^\circ_{2,p}}(G_2,X_2)\otimes_{\O}\O_v$ are \'etale locally isomorphic to those on $\rM_{G,X}^{\loc}\otimes_{\O_v}\O'_v$.
\end{cor}

\begin{proof} 
As observed in \ref{setupintegralaction}, the action of $\fA(G_{\Z_{(p)}})$ on 
$\Sh_{\eK^\circ_p}(G,X)$ extends to $\SSh_{\eK^\circ_p}(G,X).$ Hence 
$\fE(G^{\circ}_{\Z_{(p)}})$ acts on the $\O_{\eE^p,(v)}$-scheme $\SSh_{\eK^\circ_p}(G,X)^+.$
Using Lemma \ref{mapofextns}, as above, there is an isomorphism 
$$ \fA(G_{2,\Z_{(p)}}) *_{\fA(G_{\Z_{(p)}})^\circ}\fE_{\eE'}(G^{\circ}_{\Z_{(p)}})
\iso G_2(\AA^p_f) \times \Gal(\eE^{\prime p}/\eE').$$
In particular, the right side of (\ref{reconstructionIIeqn}) is an $\O_{\eE^{\prime p},(v)}$-scheme 
with an action of $G_2(\AA^p_f) \times \Gal(\eE^{\prime p}/\eE').$ Hence by Galois descent 
it is naturally an $\O'_{(v)}$-scheme with an action of $G_2(\AA^p_f).$
The first statement is now a consequence of Lemma \ref{reconstructionI}. 

The second statement then follows from Theorem \ref{localmodeldiagramThm} and Proposition \ref{parlevelstr} 
once 
we show that 
$$\Delta(G,G_2) := \ker({\fA(G_{\Z_{(p)}})^\circ} \rightarrow \fA(G_{2,\Z_{(p)}}))$$ acts freely on 
$\SSh_{\eK_p}(G,X)^+.$ For this we follow the proof of \cite{KisinJAMS} Prop.~3.4.6, which can be modified to work in 
our present setting because we have defined the twisting construction $\fA \mapsto \fA^{\calP}$ 
on the level of abelian varieties and not just in the isogeny category.

Let $(h,\gamma^{-1}) \in \Delta(G,G_2)$ with $h \in G(\AA_f^p)$ and $\gamma \in G^{\ad\circ}(\Z_{(p)})^+.$ 
Denote by $\PP$ the $Z$-torsor associated to $\gamma,$ and  fix a Galois 
extension $F/\Q$ and a point $\tilde \gamma \in \calP(\O_{F,(p)})$ lifting $\gamma.$

Let $x \in \SSh_{\eK_p^\circ}(G,X)(T),$ where $T$ is the spectrum of an algebraically closed field, and suppose that $(h,\gamma^{-1})$ 
fixes $x.$ Write $(\A_x,\lambda,\varepsilon^p)$ for the corresponding triple. 
Then by Lemma \ref{adjointaction}, for every compact open subgroup $\eK^p \subset G(\AA_f^p)$ 
there exists a $\Z_{(p)}$-isogeny $\alpha = \alpha(\eK^p): \A_x \iso \A_x^{\PP}$ respecting weak 
$\Z_{(p)}$-polarizations, and such that the left hand square of the following diagram commutes modulo $\eK^p$ 
(That is up to multiplication by an element of $\eK^p$ on the bottom left hand corner.)
\begin{equation}\label{rigidifying}
\xymatrix{
{\widehat V}^p(\A_x)\otimes F \ar[r]^\sim_{\alpha\otimes 1} & {\widehat V}^p(\A_x^{\PP})\otimes F\ar[r]^{\iota_{\tilde \gamma}} 
& {\widehat V}^p(\A_x)\otimes F \\
V\otimes\AA_f^p\otimes F \ar[r]^{\tilde \gamma h\tilde \gamma^{-1}}\ar[u]^{\varepsilon^p} 
& V\otimes\AA_f^p\otimes F \ar[u]^{\varepsilon^{p,\PP}}\ar[r]^{\tilde \gamma^{-1} \cdot } & V\otimes\AA_f^p\otimes F\ar[u]^{\varepsilon^p}} 
\end{equation} 
while the right square commutes by the definition of $\varepsilon^{p,\PP}.$

For $\eK^p$ sufficiently small, the map $\alpha(\eK^p)$ is unique. Hence if $\eK^p$ is sufficiently 
small then $\alpha$ does not depend on $\eK^p,$ and we may assume that \ref{rigidifying} commutes. 

Note that the composite of the maps in the lower row of \ref{rigidifying} is $h\tilde\gamma^{-1}.$ 
Since $(h,\gamma^{-1}) \in \Delta(G,G_2),$ we have $h\tilde\gamma^{-1} \in Z(\AA_f^p\otimes F),$ so  
$$ \iota_{\tilde \gamma}\circ \alpha \in Z(\AA_f^p\otimes F) \cap (\SAut_{\Z_{(p)}} \A_x)(\O_{F,(p)}) = Z(\O_{F,{(p)}}) \subset (\SAut_{\Q} \A_x)(\AA_f^p\otimes F). $$ 
Hence $h\tilde \gamma^{-1} \in Z(\O_{F,{(p)}}),$ and after replacing $\tilde\gamma$ with another lift, we may assume that $h\tilde\gamma^{-1} = 1.$ 
Then $\tilde\gamma$ is $\Gal(F/\Q)$ invariant, so $\tilde\gamma \in G(\Z_{(p)})_+.$

In this case the action of $(h,\gamma^{-1})$ on $\Sh_{\eK^\circ_p}(G,X)$ is by the natural right action of $h\tilde\gamma^{-1} \in G(\AA_f),$ 
which is given by the action of $\tilde\gamma^{-1} \in G(\Q_p),$ since $h\tilde\gamma^{-1} = 1$ in $G(\AA^p_f).$
It follows from Lemma \ref{levelstr} that $\tilde\gamma^{-1} \in G^\circ(\Z_{(p)}),$ so $(h,\gamma^{-1}) = 1.$
\end{proof}
 
\begin{cor}\label{localmodeldiagII} Extend $v$ to an embedding $v':  \eE' \hookrightarrow E^{\ur},$ and set $E' = \eE'_{v'}.$ 
We equip ${\rm M}^{\rm loc}_{G,X}$ with the trivial $G_2(\AA_f^p)$-action.
There is a diagram of $\O_{E'}$-schemes with $G_2(\AA_f^p)$-action
\begin{equation}
\xymatrix{
& \widetilde \SSh_{\eK_{2,p}^\circ}^{\ad}  \ar[ld]_{\pi}\ar[rd]^{q} & \\
\quad\quad   \SSh_{\eK_{2,p}^\circ}(\eG_2, X_2) \quad\quad & &{\ \ \ {\rm M}^{\rm loc}_{G,X}\ }\, , \ \ \ \ 
  \ \ \ \ \ 
}
\end{equation}
where $\pi$ is a $G_{\Z_p}^{\ad}$-torsor, and $q$ is $G_{\Z_p}^{\ad}$-equivariant. 

Moreover,   any sufficiently small compact open $\eK_2^p\subset G_2(\AA_f^p)$ acts freely on 
$\widetilde\SSh_{\eK_{2,p}^\circ}^{\ad},$ and the morphism 
$\widetilde\SSh_{\eK_{2,p}^\circ}^{\ad}/\eK_2^p \rightarrow  {\rm M}^{\rm loc}_{G,X},$ induced by $q,$ 
is smooth of relative dimension $\dim G^{\ad}.$
\end{cor}
\begin{proof} By Lemma \ref{parahoricrelations}, 
the $\fB(G_{\Z_{(p)}})$-action on the $G^{\ad}_{\Z_p}$-torsor $\widetilde \SSh^{\ad}_{\eK_p}$ given by Lemma \ref{adjointactiontorsors}, restricts to an $\fA(G_{\Z_{(p)}})$-action. 
 Let $\widetilde \SSh_{\eK_p^\circ}^{\ad+}$ denote the pullback of $\widetilde \SSh_{\eK_p}^{\ad}$ to $\SSh_{\eK^\circ_p}(G,X)^+_{\O_{E^{\ur}}}.$
Denote by $\fE_{E''}(G^{\circ}_{\Z_{(p)}})$ the pullback of $\fE_{\eE'}(G^{\circ}_{\Z_{(p)}})$ by $\Gal(E^{\ur}/E') \rightarrow \Gal(\eE^p/\eE').$
Then the stabilizer of  $\widetilde \SSh_{\eK_p^\circ}^{\ad+}$ in $\fA(G_{\Z_{(p)}})\times \Gal(E^{\ur}/E')$ is $\fE_{E''}(G^{\circ}_{\Z_{(p)}}).$
Set 
\begin{equation}\label{reconstructionIIIeqn}
\widetilde \SSh_{\eK^\circ_{2,p}}^{\ad} = \big[ [\widetilde \SSh_{\eK^\circ_p}^{\ad+}\times \fA(G_{2,\Z_{(p)}})]/\fA(G_{\Z_{(p)}})^\circ \big]^{|J|}
\end{equation}
Note that the group $\ker(\fA(G_{\Z_{(p)}})^\circ \rightarrow  \fA(G_{2,\Z_{(p)}}))$ need not be finite, but this causes no difficulty in the formation of the quotient; see \cite[E.7]{KisinLR}.

As above, $\widetilde \SSh_{\eK^\circ_{2,p}}^{\ad}$ is equipped with an action of $G_2(\AA_f^p)\times\Gal(E^{\ur}/E'),$ and descends to an $\O_{E'}$-scheme with 
$G_2(\AA_f^p)$-action.  By construction, it is a $G_{\Z_p}^{\ad}$-torsor over $\SSh_{\eK_{2,p}^\circ}(G_2,X_2),$ and $q$ is $G_{\Z_p}^{\ad}$-equivariant. 
The final claim can be checked over $\O_{E^{\ur}},$ when it follows easily from Lemma \ref{adjointactiontorsors}.
\end{proof}

\begin{para} Let $(H,Y)$ be a Shimura datum with $H^{\ad}$ a classical group. Recall (\cite{DeligneCorvallis}, cf.~\cite{KisinJAMS} 3.4.13) that 
there is central isogeny $\tilde H \rightarrow H^{\sharp}$  (which depends also on $Y$ if  $H$ has a factor of type $D$) such that $(H,Y)$ 
is of abelian type if and only if $H^{\der}$ is a quotient of $H^{\sharp}.$ 

For the remainder of this subsection we let $(G_2, X_2)$ be a Shimura datum of abelian type with reflex field $\eE_2.$ 
We choose  $x_2 \in \calB(G_2,\Q_p),$ and we denote by $\eK_{2,p}^\circ \subset G_2(\Q_p)$ the corresponding compact open, parahoric subgroup
and by $\eK_{2,p} \subset G_2(\Q_p)$ the stabilizer of $x_2.$ As always we assume $p > 2.$
\end{para}

\begin{lemma}\label{existHodetype}  Suppose that $G_2$ splits over a tamely ramified extension of $\Q_p.$ 
Then there exists a Shimura datum of Hodge type $(G,X),$ together with a central isogeny  
$G^{\der} \rightarrow G_2^{\der},$ which induces an isomorphism $(G^{\ad}, X^{\ad}) \iso (G_2^{\ad}, X_2^{\ad}).$ Moreover, 
$(G,X)$ can be chosen to satisfy the following conditions.
\begin{enumerate}
\item  $\pi_1(G^{\der})$ is a $2$-group, and is trivial if $(G^{\ad},X^{\ad})$ has no factors of type $D^{\mathbb H}.$ Moreover 
$G$ satisfies the condition (\ref{cohvanishingcondn}).
\item $G$ splits over a tamely ramified extension of $\Q_p.$
\item If $\eE$ denote the reflex field of $(G,X)$ and $\eE' = \eE\cdot \eE_2,$ then any primes $v_2|p$ of $\eE_2$ splits completely in $\eE'.$ 
\item $Z_G$ is a torus.
\item $\xcoch(G^{\ab})_{I_{\Q_p}}$ is torsion free, where $I_{\Q_p} \subset G_{\Q_p}$ denotes the inertia subgroup.
\end{enumerate}
\end{lemma}
\begin{proof}
 Write $G_2^{\ad} = \prod_i G_i$ where each $G_i$ is $\Q$-simple and adjoint.  Then $G_i = \Res_{F_i/\Q} G_i'$ where $F_i$ 
is totally real, and $G_i'$ is absolutely simple over $F_i$ and adjoint. Since we are assuming that $G$ splits over a tamely ramified extension of 
$\Q_p$ each $F_i$ is tamely ramified over $p.$

Choose $(G,X)$ of Hodge type such that there exists a central isogeny $G^{\der} \rightarrow G^{\der}_2$ inducing an isomorphism 
of Shimura data $(G^{\ad},X^{\ad}) \iso (G_2^{\ad},X_2^{\ad}).$ 
By \cite{DeligneCorvallis} 2.3.10, we may choose $G$ so that $G^{\der} = \prod \Res_{F_i/\Q} G_i^{\prime\sharp}.$ 
Then $\pi_1(G^{\der})$ is an elementary $2$-group (the only contributions comes from factors $G_i'$ of type $D^{\mathbb H}$), 
and in particular $p\nmid |\pi_1(G^{\der}).$ Moreover, $\ker(\tilde G^{\der} \rightarrow G^{\der})$ has the form 
$\prod \Res_{F_i/\Q} C_i$ where each $C_i$ is either trivial, or $\mu_2.$ In particular one sees using Cebotarev density  
that the condition (\ref{cohvanishingcondn}) is satisfied.

Next we explain how to choose $(G,X)$ so that $G$ splits over a tamely ramified extension of $\Q_p$, and 
any prime $v|p$ of $\eE^{\ad} = E(G^{\ad},X^{\ad})$ splits completely in $\eE = E(G,X).$ Suppose first that $G^{\ad} = \Res_{F/\Q} G'.$ 
Following \cite{DeligneCorvallis} 2.3, let $I_c$ be the set of real places $v$ of $F$ such that $G'(F_v)$ is compact, and $I_{nc}$ 
the real places of $F$ not in $I_c.$ Let $K/F$ be a quadratic, totally imaginary extension of $F$ in which the primes above $p$ 
split completely. Fix an isomorphism $\mathbb C \iso \bar \Q_p.$ Then $\Gal(\bar \Q_p/\Q_p)$ acts on the embeddings $K \hookrightarrow \mathbb C.$
Let $T$ be a set of embeddings $K \hookrightarrow \mathbb C$ which map bijectively to $I_c,$ and such that 
if $\tau \in \Gal(\bar \Q_p/\Q_p)$ preserves $I_c$ then it preserves $T.$ 
This is possible since all the primes of $F$ above $p$ split completely in $K.$

Define a morphism 
$h_T: \mathbb C^\times \rightarrow K\otimes_{\Q}\mathbb R$ by requiring that $K\otimes_{\sigma}\mathbb C$ has type $(-1,0)$ 
if $\sigma \in T,$ type $(0,-1)$ if $\bar \sigma \in T,$ and type $(0,0)$ otherwise. If  $\tau \in \Gal(\bar \Q_p/\Q_p)$ fixes $\eE^{\ad},$ 
then in particular it preserves $I_c,$ and hence preserves $T,$ and fixes $E(K^\times, h_T).$ 
This implies that any prime of $\eE^{\ad}$ above $p$ splits completely in 
$\eE^{\ad}\cdot E(K^\times,h_T).$ By \cite{DeligneCorvallis} 2.3.10, $(G,X)$ may be chosen so that 
$\eE = \eE^{\ad}\cdot E(K^\times,h_T).$ In particular, any prime of $\eE^{\ad}$ over $p$ splits completely in $\eE.$ 
Moreover, the construction of {\it loc.~cit} produces a group such that $Z_G$ is contained in a product of $Z_{G^{\der}},$ $K^\times,$ 
and a torus which splits over the fixed field of the subgroup of $\Gal(\bar \Q/\Q)$ which acts trivially on the Dynkin diagram of $G^{\ad}.$
In particular $G$ splits over a tamely ramified extension of $\Q_p.$
In general, when $G^{\ad}$ is not assumed simple, the above construction applies to each of the factors $\Res_{F_i/\Q} G'_i.$
Finally, any prime $v_2|p$ of $\eE_2$ splits completely in $\eE' = \eE\cdot\eE_2.$

We will now show that we can  arrange so that, in addition, the center $Z$ of $G$ is connected. 
Let $(G, X)\hookrightarrow ({\rm GSp}(V), S^{\pm})$ be the Hodge embedding. 
Choose $h\in X$ corresponding to a special point;  there is a maximal torus 
$T_0$ in $G$ defined over $\Q$, such that $h$ factors through $T_{0\RR}$.
By an argument as in \cite{KisinJAMS},  proof of Prop. 2.2.4, one sees  
that $h$ and $T_0$ can be chosen so that $T_0$ splits over a tamely ramified extension of $\Q_p$.
Observe that $T_{0\RR} /w_h({\mathbb G}_m)$ is compact, as $T_{0\RR} /w_h({\mathbb G}_m)\hookrightarrow {\rm GSp}(V_{\RR})/{\rm diag}({\mathbb G}_m)$
and 
the centralizer of $h$ in ${\rm GSp}(V_{\RR})/{\rm diag}({\mathbb G}_m) $ is compact.
Consider $G' = (G \times T_0)/Z_G$. Then 
the center of $G'$ is $T_0$ (which is connected),
and $G$, $G'$ have the same derived group. 

Let $W = {\mathrm {Hom}}_{Z_G}(V,V)$ ($\Q$-linear maps which are $Z_G$-equivariant). The 
group $G'$ acts on $W$ via $((g,t)\cdot f)(x) = g f(t^{-1}x)$.  Since $W$ contains $G$, 
one sees easily see that this $G'$-action is faithful.  We   equip $W$
 with a Hodge structure by writing $W = {\mathrm {Hom}}_Z(V_2,V)$, where 
 $V_2$ is $V$ with trivial Hodge structure; the corresponding Deligne cocharacter $h'$
 of $G'$ is  given by $h\times 1$. Then $W$ has type $\{(-1,0),
  (0,-1)\}$. Since $T_{0\RR} /w_h({\mathbb G}_m)$ is compact, it follows that ${\rm ad}\, h'(i)$ gives a Cartan involution on $G'_{\RR}/w_{h'}({\mathbb G}_m)$.
Hence, we can apply \cite{DeligneCorvallis} Prop. 2.3.2, to obtain 
 an alternating form on $W$ and a corresponding Hodge embedding 
 $(G', X')\hookrightarrow ({\rm GSp}(W), S^{\pm}_W)$. Notice   again that all primes 
 of $\eE_2$ above $p$ split in $\eE_2\cdot \eE(G', X')$.

Finally, we show that $(G,X)$ may be chosen to satisfy the last condition. 
We may assume that $(G,X)$ already satisfies the first four conditions. In particular $Z_G$ is a torus, which splits 
over a tamely ramified extension of $\Q_p.$ Since $(G,X)$ is of Hodge type, $Z_G$ splits over a CM field $F.$ 
Let $F_0$ be the totally real subfield of $F.$ Let $F_1$ be a quadratic imaginary field in which $p$ splits, 
and which is linearly disjoint from $F,$ and set $F' = F\cdot F_1.$ The $F'$ is a CM field, and we denote by $F'_0$ 
its totally real subfield.

Let $T$ and $T_0$ be the tori whose $\Q$-points are given by $F^{\prime \times}$ and $\ker(F^{\prime\times} \rightarrow F^{\prime\times}_0)$ respectively. 
Fix an embedding $\bar \Q \hookrightarrow \bar \Q_p.$ The action of complex conjugation on $\xcoch(T)$ does not 
coincide with that of any element of the inertia subgroup $I_{\Q_p},$ since the latter acts trivially on $F_1^\times.$ 
In particular $ \xcoch(T_0)_{I_{\Q_p}} = \xcoch(T)^{c=-1}_{I_{\Q_p}}$ is torsion free since $X_*(T)$ is an induced Galois module. 
Here $c$ denotes complex conjugation.

Since $Z_G/w_h(\mathbb G_m)(\mathbb R)$ is contained in a compact real Lie group, $c$ acts by $-1$ on
$\xch(Z_{G^{\der}}).$ 
Hence for some integers $n,r$ there exists an 
embedding $\xch(Z_{G^{\der}}) \hookrightarrow \xch(T_0[n]^r).$ Denote by $T_1,$ the pushout of 
$$ 1 \rightarrow T_0[n]^r \rightarrow T_0^r \overset n \rightarrow T_0^r \rightarrow 1$$ by the corresponding map $T_0[n]^r \rightarrow Z_{G^{\der}}.$
Then $Z_{G^{\der}} \subset T_1,$ and $T_1/Z_{G^{\der}} \iso T_0^r.$ Let $(Z_G)_0 \subset Z_G$ be the subtorus corresponding to $\xcoch(Z_G)^{c=-1} \subset \xcoch(Z_G)$.
There is an embedding $(Z_G)_0 \hookrightarrow T_0^s$ for some integer $s.$ 
Set $G' = (G\times T_1 \times T_0^s)/(Z_{G^{\der}}\times (Z_G)_0),$ where $Z_{G^{\der}}\times (Z_G)_0$ acts on $G$ via the multiplication 
$Z_{G^{\der}}\times (Z_G)_0 \rightarrow (Z_G)_0.$ Let $X'$ be the $G'$-conjugacy class induced by $X.$ 
Then $(G',X')$ has the same reflex field as $(G,X),$
and satisfies the first four conditions of the Lemma. 
Moreover, one sees as above that $(G',X')$ is of Hodge type.

We have $G^{\prime\der} = G^{\der},$ and 
$ (Z_{G'})_0/Z_{G^{\der}} = (T_1\times T_0^s)/Z_{G^{\der}}$ with $z \in Z_{G^{\der}}$ acting by $(z,z^{-1}).$ 
Hence we have an exact sequence 
$$
 0 \rightarrow \xcoch(T_0^s) \rightarrow \xcoch((Z_{G'})_0/Z_{G^{\der}}) \rightarrow \xcoch(T_1/Z_{G^{\der}}) \rightarrow 0. 
 $$
In particular, $\xcoch((Z_{G'})_0/Z_{G^{\der}})_{I_{\Q_p}}$ is torsion free. Finally $\xcoch(G'^{\ab})$ is an extension of $\mathbb Z$ by $\xcoch((Z_{G'})_0/Z_{G^{\der}}),$ 
so  $\xcoch(G'^{\ab})_{I_{\Q_p}}$ is torsion free.
\end{proof}

\begin{thm}\label{mainthmI} 
If $G_2$ splits over a tamely ramified extension of $\Q_p,$ then there exists a Shimura datum 
of Hodge type $(G,X)$ such that the conditions of Corollary \ref{reconstructionII} are satisfied, and all primes $v_2|p$ of $\eE_2$ split completely in $\eE'.$
In particular, for any prime $v_2|p$ of $\eE_2,$ the construction in Corollary  \ref{reconstructionII} gives rise to a $G_2(\AA_f^p)$-equivariant $\O_{\eE_2,v_2}$-scheme, 
$\SSh_{\eK_p^\circ}(G_2,X_2)$ with the following properties.

\begin{enumerate}
\item $\SSh_{\eK_{2,p}^\circ}(G_2,X_2)$ is \'etale locally isomorphic to $\rM^{\loc}_{G,X}$.  
\item If $p\nmid|\pi_1(G^{\rm der}_{2})|,$ then $\SSh_{\eK_{2,p}^\circ}(G_2,X_2)$ is \'etale locally isomorphic to $\rM^{\loc}_{G_2,X_2}.$
\item For any discrete valuation ring $R$ of mixed characteristic
  $0,p$ the map 
$$ \SSh_{\eK_{2,p}^\circ}(G_2,X_2)(R) \rightarrow \SSh_{\eK_{2,p}^\circ}(G_2,X_2)(R[1/p])$$ 
is a bijection. 
\item If $(G_2^{\ad}, X^{\ad})$ has no factors of type $D^{\mathbb H},$ then $(G,X)$ 
can be chosen so that there exists a diagram of $\O_{E_2}$-schemes with $G_2(\AA^p_f)$-action 
\begin{equation}
\xymatrix{
& \widetilde \SSh_{\eK_{2,p}^\circ}^{\ad}  \ar[ld]_{\pi}\ar[rd]^{q} & \\
\quad\quad   \SSh_{\eK_{2,p}^\circ}(\eG_2, X_2) \quad\quad & &{\ \ \ {\rm M}^{\rm loc}_{G,X}\ }\, , \ \ \ \ 
  \ \ \ \ \ 
}
\end{equation}
where $\pi$ is a $G_{2,\Z_p}^{\ad\circ}$-torsor, $q$ is  $G_{2,\Z_p}^{\ad\circ}$-equivariant, and for any sufficiently small, compact open 
$\eK^p_2 \subset G(\AA_f^p),$ the map $\SSh_{\eK_{2,p}^\circ}^{\ad}/\eK^p_2 \rightarrow {\rm M}^{\rm loc}_{G,X},$ 
induced by $q$ is smooth of relative dimension $\dim G_2^{\ad}.$

In particular, if $\kappa'$ is a finite extension of $\kappa(v_2),$ and $y \in \SSh_{\eK_{2,p}^\circ}(\kappa'),$ 
then there exists $z \in {\rm M}^{\rm loc}_{G,X}(\kappa')$ such that
$$
\O^{\rm h}_{\SSh_{\eK_{2,p}^\circ},y} = \O^{\rm h}_{{\rm M}^{\rm loc}_{G,X},z}.
$$ 
\item If $G_{2,\Q_p}$ is unramified, and there exists $x'_2 \in \calB(G_2,\Q_p)$  with $\calG_{2,x^{\prime\ad}_2} = \calG^\circ_{2,x_2^{\ad}},$ 
then $(G,X)$ can be chosen so that the construction in Corollary \ref{reconstructionII} applies with $x'_2$ in place of $x_2,$ 
and gives rise to an $\O_{\eE_2,v_2}$-scheme $\SSh_{\eK_{2,p}^\circ}(G_2,X_2)$ satisfying the conclusion of (4) above.
%\mar{GP: there was a correction here; ``.. of $(4)$'' was `` .. of $(3)$'' before. Made some corr. corrections in the proof}
\end{enumerate}
\end{thm}
\begin{proof} We apply Lemma \ref{existHodetype} and choose $(G,X)$ satisfying conditions (1)-(4) of that Lemma. 
As before, after conjugating $X$ by an element of $G^{\ad}(\Q),$ we may assume that $X \subset X^{\ad} = X_2^{\ad}$ 
contains some connected component $X_2^+$ of $X_2.$ 
Then (2) and (4) imply that the conditions of Corollary \ref{reconstructionII} are satisfied, and (3) implies that $v_2$ extends to 
a place $v'$ of $\eE'$ such that $E' = E_{2,v_2}.$ The claim (1) in the Theorem then follows
from Corollary \ref{localmodeldiagII}.
The claim (2) follows by combining this with Proposition \ref{compLocModProp}.

Let 
$$\SSh_{\eK'_p}(\GSp, S^{\pm}) = \ilim_{\eK^{\prime p}} \SSh_{\eK'_p\eK^{\prime p}}(\GSp, S^{\pm}) $$ 
where $\SSh_{\eK'_p\eK^{\prime p}}(\GSp, S^{\pm})$ is defined in \ref{basicsetup}. 
Then $\SSh_{\eK'_p}(\GSp, S^{\pm})$ satisfies the extension property in (3), by the N\'eron-Ogg-Shafarevich criterion. 
Indeed a $R[1/p]$ point of $\SSh_{\eK'_p}(\GSp, S^{\pm})$ defines an abelian variety over $R[1/p],$ together with a trivialization 
of its $l$-adic Tate module for any $l \neq p.$ Hence the abelian variety has good reduction, and the $R[1/p]$-point comes from an $R$-point.
Now (3) follows for $(G_2,X_2)$ of Hodge type and then of abelian type, by construction.

To see (4), we choose $(G,X)$ satisfying (5) of Lemma \ref{existHodetype}. Since in this case $\pi_1(G^{\der}) = \{1\},$ 
we have $\pi_1(G) = \xcoch(G^{\ab}),$ and $\pi_1(G)_{I_{\Q_p}}$ is torsion free. In particular, if $x \in \calB(G,\Q_p)$ lifts $x_2^{\ad},$ 
the Kottwitz map $\kappa_G$ is trivial on $\calG_x(\Z_p^{\ur}),$ and $\calG_x = \calG^\circ_x.$ Hence the existence of the diagram in (4) of the Theorem follows by combining 
Corollary \ref{localmodeldiagII} with Lemma \ref{parahoricrelations}. The final claim in (4) follows by Lang's lemma.

For (5), choose $x_2^{\prime} \in \calB (G_2,\Q_p)$ such that $\calG_{2,x^{\prime\ad}_2} = \calG^\circ_{2,x_2^{\ad}}.$ 
%(the existence of such a point has been noted above) 
Then $\calG_{2,x_2'}(\Z_p^{\ur}) \subset G_2(\Q_p^{\ur})$ consists of those points 
which map to $\calG_{2,x^{\prime\ad}_2}(\Z_p^{\ur})$ and into the maximal bounded subgroup of $G_2^{\ab}(\Q_p^{\ur}).$  
In particular $\calG_{2,x_2'}(\Z_p^{\ur}) \subset \calG_{2,x_2}(\Z_p^{\ur})$ and for any $g \in \calG_{2,x_2'}(\Z_p^{\ur}),$  
$\kappa_{G_2}(g)$ maps to $0$ in $\pi_1(G_2^{\ad})$ and $\xcoch(G_2^{\ab}).$ Hence $\kappa_{G_2}(g) = 0,$ and $\calG_{2,x_2'} = \calG_{2,x_2}^\circ.$

Now choose $(G,X)$ satisfying (1)-(4) of Lemma \ref{existHodetype}. From the construction one sees easily that one may in fact 
choose $(G,X)$ so that $G_{\Q_p}$ is unramified. Choose $x \in \calB (G,\Q_p)$ lifting $x_2^{\prime\ad},$ The condition on $x_2'$ 
implies that the Kottwitz map $\kappa_G: \calG_x(\Z_p^{\ur}) \rightarrow \pi_1(G)$ factors through $\xcoch(Z_G),$ and hence is 
trivial, as the latter group has no torsion. Hence $\calG_x = \calG_x^\circ.$ Now (5) follows using the same argument as in (4).
\end{proof}
 
\begin{Remark} {\rm a)  If $\eK^p_2\subset G_2(\AA^p_f)$ is sufficiently small,  
the conclusions of Theorem \ref{mainthmI} about \'etale local structure 
also hold for the quotient $\SSh_{\eK^\circ_2}(G_2, X_2)= \SSh_{\eK^\circ_{2,p}}(G, X)/\eK^p_2$.

b) The form of Theorem \ref{mainthmI} is that a particular construction gives a model of the Shimura variety with parahoric level structure with the 
described properties. Unfortunately we do not know how to characterize the models constructed in Theorem \ref{mainthmI}, for example using an extension 
property as in the hyperspecial case. However, it seems to us that this should not be a problem in applications. For example, the methods of \cite{Madapusi} should show that 
the models we construct are proper when the group $G^{\ad}$ is anisotropic over $\Q$ (so the corresponding Shimura variety is proper). Secondly, the construction should be well adapted for applications involving computation of the zeta function (see \cite{KisinLR}) for the hyperspecial case.

c) Theorem \ref{ithmIII} of the Introduction follows from Theorem \ref{mainthmI}(4), (5), by
using Proposition \ref{compLocModProp} and Remark \ref{remarkHodge}.} 
\end{Remark}

\begin{cor}\label{corAbI}
Let $p > 2$ and $(G , X )$ be a Shimura datum of abelian type\footnote{Note that the group $G_2$ in Theorem \ref{mainthmI} is now denoted by $G.$} with reflex field $\eE,$ 
such that $G $ splits over a tamely ramified extension of $\Q_p.$ Let  $x \in \calB(G ,\Q_p)$ and $\eK_p^\circ \subset G (\Q_p)$
the corresponding compact open, parahoric subgroup. Suppose $v|p$ is a prime of $\eE,$ and assume that $\eK^p$ is sufficiently small.

Then the special fibre of the integral model $\SSh_{\eK^\circ}(G ,X )$ over $\O_{\eE, v}$
 constructed above is reduced; the geometric analytic branches of the special fibre at each point are normal and Cohen-Macaulay.

If $x$ is a special vertex in $\calB(G ,\Q^{\rm ur}_p)$ then the special fibre
in normal (hence analytically unibranch at each point) and Cohen-Macaulay.
\end{cor} 

 \begin{proof}
This follows from Theorem \ref{mainthmI}  combined with \cite{PaZhu} Theorem 9.1 and Corollary 9.4 applied to the 
local model for the corresponding Hodge type group (denoted by $G$ in Theorem 
 \ref{mainthmI}  and its proof.) In particular, $p>2$ implies the  assumption on the 
 fundamental group needed in \cite{PaZhu}.
\end{proof}

\begin{para}\label{RemarkCharacterize} For a Shimura datum $(G_2,X_2)$ of abelian type, as in \ref{mainthmI}, the construction of the integral model 
$\SSh_{\eK_{2}^\circ}(G_2,X_2)$ depends on the choice of Shimura datum $(G,X),$ as well as the choice of symplectic embedding 
$(G,X) \hookrightarrow (\GSp, S^{\pm}).$ It seems reasonable to conjecture that the resulting integral model is independent of all choices, but we do not know how to prove 
this in general; by the argument in \cite[\S 2]{MilnePoints}, the extension property in \ref{mainthmI}(3) is enough to guarantee only that two such models contain isomorphic open neighborhoods containing all generic points 
of the special fibre. We give show the independence in the special case: 
\end{para}

\begin{prop}\label{independencespecial} Let $(G,X)$ be as in Corollary \ref{corAbI}, and suppose that 
$\eK_{2,p}^\circ$ is a very special parahoric subgroup and that $G^{\ad}$ is absolutely simple. Then the model 
$\SSh_{\eK_{2}^\circ}(G,X)$ does not depend on the choices made in its construction.
\end{prop}
\begin{proof}
Consider any Shimura datum $(G,X),$  such that $G^{\ad}$ is absolutely simple. Let $h \in X^{\ad},$ and 
let $T \subset G_{\mathbb C}^{\ad}$ be a maximal torus with $\mu_h \in X_*(T),$ and $B \supset T$ a Borel subgroup of $G$ 
for which $\mu_h^{-1}$ is dominant.  We denote by $P \supset B$ the parabolic subgroup of $G$ corresponding to $\mu_h^{-1},$  and by $N \subset P$ the unipotent radical. 
Since $\mu_h$ is minuscule, and $G^{\ad}$ is absolutely simple, there is exactly one simple root which is contained in $\Lie (N),$ and it generates the 
space of characters $\xch(P)$ of $P.$ For any compact open subgroup $\eK \subset G(\AA_F),$ any such character $\chi$ gives rise to a line bundle $\omega_{\chi}$ on $\Sh_{\eK}(G,X)$ 
which is defined over the reflex field \cite{MilneAnnArbor}.  The group $\xch(P) = \mathbb Z,$ has a unique generator $\chi_0$ such that $\omega_{\chi_0}$ is ample, and we write $\omega_G$ 
for $\omega_{\chi_0}.$

Now let $(G,X)$ be as in Corollary \ref{corAbI}, and assume that  $\eK_p^\circ$ is very special. Then $\SSh_{\eK^\circ}(G,X)$ has normal special fibre. 
We claim that any irreducible component of $\SSh_{\eK^\circ}(G,X)$ has irreducible special fibre. From the construction, it suffices to prove this when $(G,X)$ 
is of Hodge type. By \cite{Madapusi} there is an open embedding $\SSh_{\eK^\circ}(G,X) \hookrightarrow \overline {\SSh}_{\eK^\circ}(G,X)$ whose complement is a relative Cartier divisor, 
and such that $\overline {\SSh}_{\eK^\circ}(G,X)$ has normal special fibre. This implies the claim, first for $\overline {\SSh}_{\eK^\circ}(G,X)$ in place of $\SSh_{\eK^\circ}(G,X)$ by Stein factorization, and then for $\SSh_{\eK^\circ}(G,X).$

In particular, any extension of a line bundle on  $\Sh_{\eK^\circ}(G,X)$ to $\SSh_{\eK^\circ}(G,X)$ is unique, in the sense that any two extensions differ by an automorphism of $\omega_G$ which is a scalar on any irreducible component of  $\Sh_{\eK^\circ}(G,X).$ We will first show  that for some $n > 0,$ $\omega_G^{\otimes n}$ extends to an ample line bundle $\calL$ on  $\SSh_{\eK^\circ}(G,X).$ For that one reduces easily to the case of $(G,X)$ of Hodge type. For any symplectic embedding 
$\iota: (G,X) \hookrightarrow (\GSp, S^{\pm}),$ $\iota^*(\omega_{\GSp})$ is ample, and corresponds to a character of $P.$ Hence it is equal to $\omega_G^{\otimes n}$ for some $n > 0.$ 
Since $\omega_{\GSp}$ has an ample extension, $\iota^*(\omega_{\GSp})$ has an ample extension, by construction.

Now suppose that we have two (a priori) different models $\SSh=\SSh_{\eK^\circ}(G,X)$ and $\SSh'=\SSh'_{\eK^\circ}(G,X)$ of $\Sh_{\eK^\circ}(G,X)$
with ample line bundles $\calL$ and $\calL'$ which extend the same power $\omega^{\otimes n}_G$. As remarked above, $\SSh$ and $\SSh'$ contain isomorphic open subschemes containing all generic points  of the special fibres. Moreover, as above, we can assume that the isomorphism between these open subschemes
lifts to an isomorphism between
the corresponding restrictions of $\calL$ and $\calL'$.
Assume that $U$ is  a common open neighborhood of $\SSh$ and $\SSh'$, which contains the generic fibre and all generic points  of the special fibres. Then, 
since $\SSh$ and $\SSh'$ are normal, we have $\Gamma(\SSh,\calL^{\otimes j})= 
\Gamma(U,\omega_G^{\otimes nj})= \Gamma(\SSh',\calL'^{\otimes
  j})$. Since both $\calL$ and $\calL'$ are ample, this allows us to
view both $\SSh$ and $\SSh'$ as open subschemes of $\Proj(A)$, with
$A=\oplus_{j \geq 0} \Gamma(U,\omega_G^{\otimes nj})$. 

To show $\SSh=\SSh'$, it is now enough to verify that $\SSh(k)=\SSh'(k)$, where $k$ is an algebraic closure of ${\mathbb F}_p$. By flatness each $k$-valued point of $\SSh$ lifts to an $R$-valued point of $\SSh$ where $R$ is some strictly henselian 
discrete valuation ring of mixed characteristic $0$, and similarly for $\SSh'$. Since $\SSh_{\eK^\circ_p}(G,X)\to \SSh$ is pro-\'etale, the $R$-valued point of $\SSh$ lifts to an $R$-valued point of $\SSh_{\eK^\circ_p}(G,X)$,
and similarly for $\SSh'$ and $\SSh'_{\eK^\circ_p}(G,X)$. The result now follows using the extension property \ref{mainthmI}(3) and the fact that $\Proj(A)$ is separated.
 \end{proof}

\subsection{Nearby cycles}

 We now give some results about the nearby cycles
of the integral models $\SSh_{\eK^\circ}(G ,X )$ over $\O_{\eE, v}$
which can  be obtained easily by combining the above with results 
of \cite{PaZhu} about the nearby cycles of local models.

\begin{para} Let $l\neq p$ be a prime. Set $S=\Spec(\O_E)$, $\eta=\Spec(E)$, $s=\Spec(k_E)$, so that $(S, s, \eta)$ is a henselian trait.
%\mar{MK: Was $\bar E$ ever defined ? It is defined in the introduction now, but I repeated it here.}
Let $\bar E$ be an algebraic closure of $E,$ with residue field $k_{\bar E},$ and write  $\bar\eta=\Spec(\bar E)$, $\bar s=\Spec(k_{\bar E})$ and 
$\bar S$ for the normalization of $S$ in $\bar\eta$.
If $f: X\to S$ is a scheme of finite type  
and $\F$  in $D^b_c(X_\eta, \bar\Q_l)$ we set  
$$
R\Psi(X, \F)=\bar i_* R \bar j^* (\F_{\bar\eta})
$$
for the ``complex" of nearby cycles. Here   $\bar i: X_{\bar s}\hookrightarrow X_{\bar S}$ and $\bar j: X_{\bar\eta}\hookrightarrow X_{\bar S}$
are the closed and open immersions of the
geometric special and generic fibres and $\F_{\bar\eta}$ is the pull-back of $\F$ to $X_{\bar\eta}$.
Recall that $R\Psi(X, \F)$ is an object in $D^b_c(X_{\bar s}, \bar\Q_l)$ 
which supports a   continuous action of $\Gamma_E={\rm Gal}(\bar\eta/\eta)$
 compatible with the action on $X_{\bar s}$ via $\Gamma_E={\rm Gal}(\bar \eta/\eta)\to {\rm Gal}(\bar s/s)$.
 See \cite{IllusieMonodromy} for more  details.
For a point $x\in X({\mathbb F}_q)$, ${\mathbb F}_q\supset k_E$, and corresponding  geometric  point $\bar x=\Spec(\bar s)\to X$, the inertia subgroup
$I_E\subset \Gamma_E$ acts on the stalk $R\Psi(X, \F)_{\bar x}$
and one can define \cite{RapoportAnnArbor} the semi-simple trace of Frobenius
$$
{\rm Tr}^{ss}({\rm Frob}_x, R\Psi(X, \F)_{\bar x}).
$$
We will denote $R\Psi(X,\bar\Q_l)$ simply by $R\Psi^X$.

We refer the reader to \cite{HainesGoertz}, \cite{HainesSurvey},
\cite{HainesNgoNearby}, for more details, additional references and background on
nearby cycles of integral models of Shimura varieties.
\end{para}

\begin{para} 
For the following result, the notation and hypotheses are as in Corollary \ref{corAbI}. In addition, we denote by $F \subset \bar E $ a 
tamely ramified, finite extension of $E$ over which $G$ splits. 
\footnote{Note that the conditions of Corollary \ref{corAbI} imply that $E$ is tamely ramified over $\Q_p,$ 
so $F$ is also.} In particular, we have $x \in \calB(G,\Q_p),$ and the group schemes $\calG = \calG_x$ and $\calG^{\ad} = \calG^{\ad}_{x^{\ad}}.$
For notational simplicity, we set $\SSh=\SSh_{\eK^\circ}(G, X)$.

 \begin{cor}
The inertia subgroup $I_F$ of ${\rm Gal}(\bar E/F)$ acts unipotently on all the stalks $R\Psi^\SSh_{\bar z}$,
for $\bar z$ in $\SSh(k_{\bar E}) $. If $x\in \calB(G, \Q_p)$ is a very special vertex,  then $I_F$
acts trivially on all the stalks $R\Psi^\SSh_{\bar z}$, $\bar z$ as above.
\end{cor}

\begin{proof}
This follows from Corollary \ref{corAbI}, \cite{PaZhu} Theorems 10.9 and 10.12, and the fact that
the stalk of the nearby cycles at $\bar z$ with its inertia action depends only on the strict henselization of the local ring
at $\bar z$.
\end{proof} 

\end{para}

\begin{para}\label{localModelCentralPar}
Now suppose that there exists a $G(\AA_f^p)$-equivariant local model diagram 

\begin{equation}
\xymatrix{
& \widetilde \SSh_{\eK_p^\circ}^{\ad}  \ar[ld]_{\pi}\ar[rd]^{q} & \\
\quad\quad   \SSh_{\eK_p^\circ}(\eG, X_2) \quad\quad & &{\ \ \ {\rm M}^{\rm loc}_{G,X}\ }\, , \ \ \ \ 
  \ \ \ \ \ 
}
\end{equation}
where $\pi$  a $\calG^{\ad\circ}$-torsor, $q$ is a $\calG^{\ad\circ}$-equivariant map,  
any sufficiently small compact open $\eK^p \subset G(\AA_f^p),$ acts freely on 
$\widetilde \SSh_{\eK_p^\circ}^{\ad},$ and 
the map $\widetilde \SSh_{\eK_p^\circ}^{\ad}/\eK^p \rightarrow {\rm M}^{\rm loc}_{G,X}$ induced by $q$ is smooth of relative dimension $\dim G^{\ad}.$
Such a diagram exists when $p\nmid |\pi_1(G^{\der})|$ and, either $(G^{\ad}, X^{\ad})$ has   no factors of type $D^{\mathbb H},$ or $G$ is unramified over $\Q_p$ and $\eK_p^\circ$ is contained in a hyperspecial, by Theorem \ref{ithmIII}.

Suppose $y\in \SSh(\mF_q)$, where $\mF_q\supset k_E$. Using Lang's Lemma we see that  there is a point $w  \in \rM^{\rm loc}_{G,X}(\mF_q)$, well-defined up
to the action of $\Gg^{\ad\circ}(\mF_q)$, such that we have an isomorphism of henselizations
\begin{equation}\label{isohensel}
\O^{\rm h }_{\SSh, y}\simeq \O^{\rm h }_{\rM^{\rm loc}_{G,X}, w}.
\end{equation}
This in turn implies an equality of semi-simple traces
\begin{equation}\label{sstraceidentity}
{\rm Tr}^{ss}({\rm Frob}_y, R\Psi^\SSh _{\bar y})={\rm Tr}^{ss}({\rm Frob}_w, R\Psi^{\rM^{\rm loc}_{G,X}}_{\bar w}).
\end{equation}
Set $r=[\mF_q: k_E]$. Consider the function
$$
\psi_r: \SSh(\mF_q)\to \bar\Q_l \ ;\quad \psi_r(y):= {\rm Tr}^{ss}({\rm Frob}_y, R\Psi^\SSh_{\bar y}).
$$
(The function $\psi_r$  appears in the Langlands-Kottwitz method for the calculation of the factor at $v$ 
of the semi-simple zeta function
of the Shimura variety $\Sh_{\eK^\circ}(G, X)$.)
By (\ref{sstraceidentity}) $\psi_r$ factors as a composition 
$$
\SSh(\mF_q)\xrightarrow{q} \Gg^\circ(\mF_q)\backslash \rM^{\rm loc}_{G,X}(\mF_q)\xrightarrow{\phi_r} \bar\Q_l
$$
where
$$
\phi_r: \rM^{\rm loc}_{G,X}(\mF_q)\to \bar\Q_l \ ;\quad \phi_r(w)= {\rm Tr}^{ss}({\rm Frob}_y, R\Psi^{\rM^{\rm loc}_{G,X}}_{\bar w}).
$$
\end{para}

\begin{para}\label{centralPar}
In \cite{PaZhu} there is a construction of a reductive group $G'$ over $\mathbb F_p\llps t\lrps$,  resp. a 
parahoric subgroup scheme $\Gg'$ over ${\mathbb F}_p\lps t\rps$, of the ``same type" as $G$, resp. $\Gg^\circ$; 
in particular, we can identify the special fibres of $\Gg^\circ$ and $\Gg'$ over 
${\mathbb F}_p$. We have an $\Gg'(\mF_q\lps t \rps)$-equivariant embedding $\rM^{\rm loc}_{G,X}(\mF_q)\subset 
G'(\mF_q\llps t\lrps )/\Gg'(\mF_q\lps t\rps),$ with $\Gg'(\mF_q\lps t \rps)$ acting on $\rM^{\rm loc}_{G,X}$ via 
 $\Gg^\circ(\mF_q).$ Set $\eP'_r=\Gg'(\mF_q\lps t\rps)$. We have 
$$
\Gg^\circ(\mF_q)\backslash \rM^{\rm loc}_{G,X}(\mF_q)\subset\eP'_r\backslash G'(\mF_q\llps t\lrps)/\eP'_r.
$$
We again denote by $\phi_r$ the extension by $0$ of $\phi_r$ to $G'(\mF_q\llps t\lrps)/\eP'_r.$
Then $\phi_r$ is $\eP'_r$-equivariant and so it gives  
$$
\phi_r \in \cal H_r(G', \Gg')=C_c(\eP'_r\backslash G'(\mF_q\llps t\lrps)/\eP'_r)
$$
in the parahoric Hecke algebra of compactly supported $\eP'_r$-bivariant locally constant $\bar\Q_l$-valued functions on $G'(\mF_q\llps t\lrps)$
under convolution (see \cite{PaZhu} \S 10.4.2.). By \emph{loc. cit.}, Theorem 10.14, for all $r\geq 1$, the function $\phi_r$ belongs to the center of the Hecke 
algebra $\cal H_r(G',\Gg')$. 
\end{para}

\begin{para}
Let us now assume in addition that $G$ is unramified over $\Q_p$, $p \nmid|\pi_1(G^{\der})|$,
and that either $(G^{\ad}, X^{\ad})$ has no factor of type $D^{\mathbb H},$ or $\eK_p^\circ$ is contained in a hyperspecial. Then we can apply the above 
discussion to the local model diagram given by Theorem \ref{mainthmI} (4), (5). 

The extension $E/\Q_p$ is unramified; 
denote by $E_r\subset L$ the unramified extension of $E$ of degree $r=[\mF_q: k_E]$ with residue field $\mF_q$
and by $\O_r=W(\mF_q)$ the ring of integers of $E_r$ Set $P_r=\Gg^\circ(\O_r)$, $\eP'_r= \Gg'(\mF_q\lps t\rps)$.
By the construction of $G'$, for each $r\geq 1$,
there is a natural bijection 
\begin{equation}\label{cosetbijection}
\eP'_r\backslash G'(\mF_q\llps t\lrps)/\eP'_r \cong  P_r\backslash G(E_r)/P_r
\end{equation}
which   gives
\begin{equation*}
\Gg^\circ(\mF_q)\backslash \rM^{\rm loc}_{G,X}(\mF_q)\hookrightarrow P_r\backslash G(E_r)/P_r.
\end{equation*}
 
Using this, we can view $\phi_r: \Gg^\circ(\mF_q)\backslash \rM^{\rm loc}_{G,X}(\mF_q)\to \bar\Q_l$
as an element of the parahoric Hecke algebra $C_c(P_r\backslash G(E_r)/P_r)$.
Set   $d= \dim(\Sh_{\eK^\circ}(G, X))$ and let $\mu$ be a cocharacter in the conjugacy class of 
$\mu_h$.

\begin{thm} (Kottwitz's conjecture) Suppose that $(G,X)$ is of abelian type with $G$ unramified, $p\nmid |\pi_1(G^{\der})|$,
and either $(G^{\ad}, X^{\ad})$ has no factor of type $D^{\mathbb H},$ or $\eK_p^\circ$ is contained in a hyperspecial.
For $y\in \SSh(\mF_q)$, we have
\begin{equation}\label{kottwitz}
{\rm Tr}^{ss}({\rm Frob}_y, R\Psi^\SSh _{\bar y})=q^{d/2}z_{\mu, r}(w)
\end{equation}
where $w \in \rM^{\rm loc}_{G,X}$ corresponds to $y$ and $z_{\mu, r}$ is the Bernstein function attached to $\mu$ in the center of the parahoric Hecke algebra $C_c(P_r\backslash G(E_r)/P_r)$.
\end{thm}

\begin{proof}
See \cite{LusztigChar} (also the work of Haines \cite{HainesBC, HainesSurvey, HainesBernsteinCenter}) for the definition of the Bernstein function $z_{\mu, r}$.

Let us pick an alcove in the apartment of the standard split torus $S$ (cf. \ref{quasisplitconstruction})
whose closure contains $x$ and let $x_0$ be a hyperspecial vertex
in the same closure. This alcove defines an Iwahori group scheme $\calI'$ over 
$\mF\lps t\rps $. Set $\eI'_r=\calI'(\mF_q\lps t\rps)$ so that $\eI'_r\subset \eP'_r$.
 We also  set $\eK'_r=\Gg'_0(\mF_q\lps t\rps)$, where
$\Gg'_0$ is the reductive group scheme corresponding to $x_0$. Then $\eK'_r$ is a maximal compact subgroup of $G'(\mF_q\llps t\lrps)$ and we also have $\eI'_r\subset \eK'_r $. 

By \cite{PaZhu} Theorem 10.16, for each $r\geq 1$, the function 
$\phi_r$ is the unique element of  the center $\calZ (C_c(\eP'_r\backslash G'(\mF_q\llps t\lrps)/\eP'_r))$ whose image under the Bernstein isomorphism $(-* 1_{\eK'_r})\cdot (-* 1_{\eP'_r})^{-1} $
obtained by composing the convolutions (\cite{HainesBC} Theorem 3.1.1) 
$$
\ \ -* 1_{\eP'_r}: \calZ (C_c(\eI'_r\backslash G'(\mF_q\llps t\lrps)/\eI'_r))\xrightarrow{\sim} \calZ (C_c(\eP'_r\backslash G'(\mF_q\llps t\lrps)/\eP'_r)),
$$
$$
-* 1_{\eK'_r} : \calZ (C_c(\eI'_r\backslash G'(\mF_q\llps t\lrps)/\eI'_r))\xrightarrow{\sim}  C_c(\eK'_r\backslash G'(\mF_q\llps t\lrps)/\eK'_r),\ \ 
$$
is the characteristic function   $1_{\eK'_r s_{\mu'} \eK'_r}$ where $s_{\mu'} \in G'(\mF_q\llps t\lrps )$ is the element determined by the coweight $\mu'$
of $G'$ which corresponds to $\mu$ as in \cite{PaZhu}. (Note that the result in \emph{loc. cit.} is given for the intersection complex
$\bar\Q_l[d](d/2)$.) It follows 
from the compatibility of the Bernstein and Satake isomorphisms that $\phi_r$ is $q^{d/2}z'_{\mu', r}$, where $z'_{\mu', r}$
is the Bernstein function in the center of the parahoric Hecke algebra $C_c(\eP'_r\backslash G'(\mF_q\llps t\lrps)/\eP'_r)$.
It remains to note that (\ref{cosetbijection}) induces an algebra isomorphism
\begin{equation}\label{heckecompa}
C_r(\eP'_r\backslash G'(\mF_q\llps t\lrps)/\eP'_r) \cong  C_c(P_r\backslash G(E_r)/P_r)
\end{equation}
which takes the Bernstein function $z'_{\mu', r}$ to $z_{\mu, r}$ and the result then follows from (\ref{sstraceidentity}).
\end{proof}

\end{para}

\providecommand{\bysame}{\leavevmode\hbox to3em{\hrulefill}\thinspace}
\providecommand{\href}[2]{#2}

%\bibliographystyle{hamsplain}
 
%\bibliography{ShimuraBiblioPara3}

\end{document}